\documentclass{article}

\usepackage{amsmath,amsthm,amssymb}
\usepackage{enumerate}
\usepackage{hyperref}
\usepackage{mathtools,pifont}
\usepackage{tikz-cd}
\usepackage{tikz,pgf}
\usepackage{bbding}

\numberwithin{equation}{section}

\newtheorem{theorem}{Theorem}[section]
\newtheorem{observation}[theorem]{Observation}
\newtheorem{lemma}[theorem]{Lemma}
\newtheorem{proposition}[theorem]{Proposition}
\newtheorem{corollary}[theorem]{Corollary}

\theoremstyle{definition}
\newtheorem{definition}[theorem]{Definition}

\newtheorem{example}[theorem]{Example}
\newtheorem{construction}[theorem]{Construction}

\theoremstyle{remark}
\newtheorem{remark}[theorem]{Remark}

\newcommand\R{\mathbb{R}}
\newcommand\C{\mathbb{C}}

\newcommand\N{\mathbb{N}}

\DeclareMathOperator{\id}{id}
\DeclareMathOperator{\rad}{rad}

\DeclareMathOperator{\Ran}{Ran}
\DeclareMathOperator{\Ad}{Ad}

\DeclareMathOperator{\im}{Im}

\DeclareMathOperator{\Alg}{Alg}

\DeclareMathOperator{\alg}{alg}
\DeclareMathOperator{\Hom}{Hom}

\DeclareMathOperator{\Var}{Var}
\DeclareMathOperator{\dil}{dil}

\DeclareMathOperator{\free}{free}
\DeclareMathOperator{\Bool}{Bool}
\DeclareMathOperator{\orth}{orth}
\DeclareMathOperator{\sub}{sub}
\DeclareMathOperator{\mono}{mono}

\DeclareMathOperator{\graph}{graph}
\DeclareMathOperator{\Walk}{Walk}
\DeclareMathOperator{\Tree}{Tree}
\DeclareMathOperator{\Digraph}{Digraph}
\DeclareMathOperator{\depth}{depth}
\DeclareMathOperator{\chain}{chain}
\DeclareMathOperator{\Perm}{Perm}
\DeclareMathOperator{\Mod}{Mod}
\DeclareMathOperator{\Mom}{Mom}
\DeclareMathOperator{\Lip}{Lip}
\DeclareMathOperator{\Func}{Func}
\DeclareMathOperator{\red}{red}
\DeclareMathOperator{\Ord}{Ord}

\DeclarePairedDelimiter{\norm}{\lVert}{\rVert}
\DeclarePairedDelimiter{\ip}{\langle}{\rangle}

\newcommand{\assemb}{\bigstar}

\DeclareFontFamily{U}{mathb}{\hyphenchar\font45}
\DeclareFontShape{U}{mathb}{m}{n}{
      <5> <6> <7> <8> <9> <10> gen * mathb
      <10.95> mathb10 <12> <14.4> <17.28> <20.74> <24.88> mathb12
      }{}
\DeclareSymbolFont{mathb}{U}{mathb}{m}{n}
\DeclareMathSymbol{\boxright}{3}{mathb}{'151}

\begin{document}

\title{An Operad of Non-commutative Independences Defined by Trees}
\author{David Jekel and Weihua Liu}

\maketitle

\begin{abstract}
We study notions of $N$-ary non-commutative independence, which generalize free, Boolean, and monotone independence.  For every rooted subtree $\mathcal{T}$ of the $N$-regular tree, we define the $\mathcal{T}$-free product of $N$ non-commutative probability spaces and the $\mathcal{T}$-free additive convolution of $N$ non-commutative laws.

These $N$-ary additive convolution operations form a topological symmetric operad which includes the free, Boolean, monotone, and anti-monotone convolutions, as well as the orthogonal and subordination convolutions.  Using the operadic framework, the proof of convolution identities such as $\mu \boxplus \nu = \mu \rhd (\nu \boxright \mu)$ can be reduced to combinatorial manipulations of trees.  In particular, we obtain a decomposition of the $\mathcal{T}$-free convolution into iterated Boolean and orthogonal convolutions, which generalizes work of Lenczewski.

We also develop a theory of $\mathcal{T}$-free independence that closely parallels the free, Boolean, and monotone cases, provided that the root vertex has more than one neighbor.  This includes combinatorial moment formulas, cumulants, a central limit theorem, and classification of infinitely divisible distributions (in the case of bounded support).  In particular, we study the case where the root vertex of $\mathcal{T}$ has $n$ children and each other vertex has $d$ children, and we relate the $\mathcal{T}$-free convolution powers to free and Boolean convolution powers and the Belinschi-Nica semigroup.
\end{abstract}

\newpage

\tableofcontents

\newpage

\section{Introduction}

\subsection{Non-commutative Independences}

Recall that a $\mathrm{C}^*$-non-commutative probability space is a pair $(\mathcal{A},\phi)$ where $\mathcal{A}$ is a unital $\mathrm{C}^*$-algebra and $\phi: \mathcal{A} \to \C$ is a state (that is, a positive linear functional with $\phi(1) = 1$).  The elements of $\mathcal{A}$ are viewed as random variables and $\phi$ is viewed as the expectation.

Beginning with the seminal papers \cite{Voiculescu1985} \cite{Voiculescu1986} of Voiculescu, free probability theory has developed a systematic analogy between classical independence and free independence.  This analogy was later extended to include two other types of independence, Boolean independence \cite{SW1997} and monotone independence \cite{Muraki2000} \cite{Muraki2001}.  For each of these independences, the following notions and results were developed (though the list is certainly not exhaustive):
\begin{itemize}
	\item \emph{moment conditions} which characterize the independence of algebras $\mathcal{A}_1$, \dots, $\mathcal{A}_N$;
	\item a \emph{combinatorial theory} which describes the mixed moments of independent random variables in terms of certain partitions of $\{1,\dots,n\}$ and certain functionals called \emph{cumulants};
	\item a \emph{product operation} on non-commutative probability spaces which provide a way to independently join probability spaces $(\mathcal{A}_1,\phi_1)$, \dots, $(\mathcal{A}_n,\phi_n)$;
	\item a \emph{convolution operation} on probability measures which describes the law of a sum of independent random variables.
	\item \emph{analytic transforms} which aid in the computation of convolutions (e.g.\ the Fourier transform in the classical case and the $R$-transform in the free case);
	\item a \emph{central limit theorem} which describes asymptotic behavior of $(X_1 + \dots + X_N) / \sqrt{N}$ where $X_j$ are i.i.d.\ random variables with zero mean and finite variance;
	\item a \emph{Levy-Khintchine formula} that characterizes analytically the distributions which are infinitely divisible with respect to each type of convolution;
	\item a \emph{Fock space} defined using tensor powers of a given Hilbert space $\mathcal{K}$, and certain operators on the Fock space that can be used to realize the central limit distribution, and more generally any infinitely divisible distribution.
\end{itemize}

Several important papers have undertaken to unify or connect these different independences.  First, Speicher \cite{Speicher1997} showed that classical, free, and Boolean independence give rise to the only universal product operations on non-commutative probability spaces (that is, the only binary product operations that are functorial, commutative, and associative).  Extending this result, Muraki \cite{Muraki2003,Muraki2013} showed that the only natural (that is, functorial and associative) product operations are those that arise from classical, free, Boolean, monotone, and anti-monotone independence.

Second, Bercovici and Pata \cite{BP1999} studied the bijection between classical, free, and boolean infinitely divisible laws (which was already implicit in the L{\'e}vy-Hin{\v c}in formulas) and even showed that this arises from a bijection between the domains of attraction (e.g.\ the probability distributions which satisfy the classical CLT correspond under their bijection to those which satisfy the free CLT).  This result was extended to the monotone case in \cite{AW2014}.

Third, Hasebe and Saigo \cite{HS2011b} gave an axiomatic characterization of cumulants and a proof of their existence and uniqueness that only depended on certain axioms about independence, and hence applied equally well to each type of independence without using casework.  And fourth, the papers \cite{Lehner2002}, \cite{BN2008b}, \cite{AHLV2015} gave combinatorial formulas to convert between the classical, free, Boolean, and monotone cumulants.

Here we will study a family of non-commutative independences, introduced by the second author in \cite{Liu2018}, which include free, Boolean, and monotone independence, and for which most of the notions itemized above makes sense.  More precisely, for each $N$ and for every rooted subtree $\mathcal{T}$ of the rooted $N$-regular tree, we will define a $\mathcal{T}$-free product operation on non-commutative probability spaces, and hence define a $\mathcal{T}$-free additive convolution operation $\boxplus_{\mathcal{T}}(\mu_1,\dots,\mu_N)$ on probability distributions.

For such independences, we will discuss product operations and convolutions (\S \ref{sec:constructconvolution}), as well as combinatorics of moments (\S \ref{sec:combinatorics}).  Assuming that the root vertex has more than one neighbor, we will discuss cumulants (\S \ref{sec:cumulants}), the central limit theorem (\S \ref{sec:CLT}), characterization of infinitely divisible laws and Fock spaces (\S \ref{sec:infinitelydivisible}), and Bercovici-Pata-type bijections (\S \ref{subsec:BPbijection}).

For each $N$, specific choices of the tree $\mathcal{T}$ will yield the $N$-ary free, Boolean, and monotone products.  Thus, in particular, our paper constitutes a unified treament of the free, Boolean, and monotone cases.  But our framework also includes mixtures of free, Boolean, and monotone independence as in \cite{Wysoczanski2010} and \cite{KW2013} (see \S \ref{subsec:digraphoperad}), and many other new types of independence.

For instance, in the case where $\mathcal{T}$ is an \emph{$(n,d)$-regular tree}, meaning that root vertex has $n$ children and each other vertex has $d$ children, we find the coefficients in the moment-cumulant formula explicitly (see \S \ref{subsec:treecoefficients}).  These coefficients interpolate between the free and Boolean cases.  Convolution powers for $\mathcal{T}$ can be expressed in terms of free and Boolean convolution powers, and the Bercovici-Pata bijection from Boolean independence to $\mathcal{T}$-free independence is precisely the Belinschi-Nica semigroup $\mathbb{BN}_t(\mu) = (\mu^{\boxplus (1 + t)})^{\uplus 1/(1 + t)}$ at time $t = d / (n - 1)$ (see \S \ref{subsec:BPbijection}).

We include sharp operator-norm estimates throughout, and in particular, we give some of the best known central limit estimates for the operator-valued free, Boolean, and monotone settings with a new ``coupling'' proof (\S \ref{subsec:CLTestimate}).

\subsection{Convolution Identities}

Besides studying the properties of $\mathcal{T}$-free convolution for a fixed $\mathcal{T}$, we will also prove certain identities relating these convolution operations using the language of operads (see \S \ref{sec:operad} - \S \ref{sec:convolutionidentities}).  Let $\Tree(N)$ be the collection of rooted subtrees of the rooted $N$-regular tree.  For $\mathcal{T} \in \Tree(k)$ and $\mathcal{T}_1 \in \Tree(n_1)$, \dots, $\mathcal{T}_k \in \Tree(n_k)$, we will define in \S \ref{sec:operad} the composition
\[
\mathcal{T}(\mathcal{T}_1,\dots,\mathcal{T}_k) \in \Tree(n_1 + \dots + n_k),
\]
and prove that
\begin{multline*}
\boxplus_{\mathcal{T}(\mathcal{T}_1,\dots,\mathcal{T}_k)}(\mu_{1,1}, \dots, \mu_{1,n_1}, \dots \dots , \mu_{k,1},\dots,\mu_{k,n_k}) \\
= \boxplus_{\mathcal{T}}(\boxplus_{\mathcal{T}_1}(\mu_{1,1},\dots,\mu_{1,n_1}), \dots, \boxplus_{\mathcal{T}_k}(\mu_{k,1},\dots,\mu_{k,n_k}));
\end{multline*}
see Corollary \ref{cor:operadmorphism}.  In other words, the map $\mathcal{T} \mapsto \boxplus_{\mathcal{T}}$ respects operad composition.  We will also show that this map intertwines the natural actions of the symmetric group $\Perm(N)$ on $\Tree(N)$ and on $N$-ary functions on the space of laws (Corollary \ref{cor:convolutionidentity}) and that it is continuous with respect to certain natural topologies on the domain and target space (\S \ref{subsec:continuity}).  Thus, it defines a morphism of topological symmetric operads.

These results provide a unified framework for proving various convolution identities.  For example, consider the identity
\[
(\mu_1 \boxplus \mu_2) \boxplus \mu_3 = \mu_1 \boxplus (\mu_2 \boxplus \mu_3)
\]
which expresses associativity of the binary free convolution operation.  If $\mathcal{T}_{2,\free}$ is the tree representing free convolution, then this identity says that
\[
\boxplus_{\mathcal{T}_{2,\free}}(\boxplus_{\mathcal{T}_{2,\free}}(\mu_1,\mu_2),\mu_3) = \boxplus_{\mathcal{T}_{2,\free}}(\mu_1,\boxplus_{\mathcal{T}_{2,\free}}(\mu_2,\mu_3)),
\]
and in light of the operad morphism described above, this follows from the identity
\[
\mathcal{T}_{2,\free}(\mathcal{T}_{2,\free},\id) = \mathcal{T}_{2,\free}(\id,\mathcal{T}_{2,\free}),
\]
which is simply a combinatorial manipulation.  The same applies to Boolean and monotone independence.  Similarly, the identity $\mu_1 \boxplus \mu_2 = \mu_2 \boxplus \mu_1$ follows from the permutation invariance of the tree $\mathcal{T}_{2,\free}$, and the same holds for Boolean independence.

The operad morphism provides a systematic way to prove convolution identities by combinatorially manipulating trees, since our general results already do the work of converting manipulations of trees into manipulations of Hilbert spaces and random variables.  To give a few other examples, there are trees $\mathcal{T}_{\sub} \in \Tree(2)$ and $\mathcal{T}_{\orth} \in \Tree(2)$ representing respectively the subordination convolution $\boxright$ and the orthogonal convolution $\vdash$.  Using Corollary \ref{cor:convolutionidentity} and manipulation of trees, one can show the identity
\[
\mu \boxplus \nu = (\mu \boxright \nu) \lhd \nu
\]
where $\lhd$ is anti-monotone convolution (see Example \ref{ex:subordination}); this identity was studied in \cite{Lenczewski2007,Nica2009,Liu2018} and relates to analytic subordination for the additive free convolution.  Similarly, two other identities studied in those papers, namely
\[
\mu \rhd \nu = (\mu \vdash \nu) \uplus \nu
\]
and
\[
(\mu_1 \boxplus \mu_2) \boxright \nu = (\mu_1 \boxright \nu) \boxplus (\mu_2 \boxright \nu),
\]
can also be deduced from our results.

A more general example is as follows (see \S \ref{subsec:BOdecomp}).  If $\mathcal{T}$ is a finite tree, then $\boxplus_{\mathcal{T}}(\mu_1,\dots,\mu_N)$ can be expressed using iterated Boolean and orthogonal convolutions of the laws $\mu_1$, \dots, $\mu_N$.  This provides an algorithm for computing the Cauchy-Stieltjes transform of the $\mathcal{T}$-free convolution.  Moreover, finite trees are dense in $\Tree(N)$, and thus this algorithm also gives an approximation for the Cauchy-Stieltjes transform even when $\mathcal{T}$ is infinite.  This result generalizes the decomposition for free convolution given by \cite{Lenczewski2007,ALS2007}.

\subsection{Scope and Approach}

There are three prominent viewpoints on non-commutative independences and convolution operations.  The \emph{operatorial viewpoint} models distributions using operators explicitly constructed on Hilbert spaces (e.g.\ free product Hilbert spaces, Fock spaces).  The \emph{combinatorial viewpoint} studies moment and cumulant formulas using non-crossing partitions.  The \emph{complex-analytic viewpoint} studies probability distributions through their Cauchy-Stieltjes transforms, differential equations, and functional equations.  This paper will focus primarily on the operatorial and combinatorial aspects of $\mathcal{T}$-free independence, leaving the full development of the complex-analytic viewpoint for future work.

Throughout the paper, we will work in the setting of $\mathrm{C}^*$-algebraic operator-valued non-commutative probability, where the scalars are replaced by a $\mathrm{C}^*$-algebra $\mathcal{B}$, introduced in \cite[\S 5]{Voiculescu1985}, \cite{Voiculescu1995}.  Most of the results mentioned in the first two sections of the introduction have been adapted to the operator-valued setting (individual references given throughout).

Moreover, we study here only operator-valued laws of a single operator $X$ rather than laws of a tuple $X_1$, \dots, $X_N$.  There is no loss of generality because the $\mathcal{B}$-valued law of a tuple $X_1$, \dots, $X_N$ can be packaged into a single $M_N(\mathcal{B})$-valued law; this is one of several tricks using matrix amplification that are now standard in operator-valued non-commutative probability (see \cite{HMS2018}).  We refer to \cite{Liu2018} for a detailed explanation of how to reduce the study of tuples to the study of single operators in the context of convolution operations.

We restrict ourselves here to the study of bounded operators (hence probability distributions with bounded support), because the basic operatorial setup for the unbounded theory is not well-understood in the operator-valued setting, and because in the scalar-valued setting convolutions of unbounded laws would be better handled from the complex-analytic viewpoint.

\subsection{Overview}

The broad structure of the paper is as follows:  \S \ref{sec:preliminaries} reviews preliminaries, \S \ref{sec:constructconvolution} - \ref{sec:combinatorics} give the definition and basic properties of $\mathcal{T}$-free products, \S \ref{sec:operad} - \ref{sec:convolutionidentities} study how the $\mathcal{T}$-free convolution operations relate to each other in the framework of operads, \S \ref{sec:cumulants} - \S \ref{sec:infinitelydivisible} develop the theory of $\mathcal{T}$-free convolution for a fixed $\mathcal{T}$ in parallel with the free, Boolean, and monotone cases, and \S \ref{sec:conclusion} gives concluding remarks and future research directions.

In more detail, the role of each section is as follows.  In \S \ref{sec:preliminaries}, because our paper will handle the operator-valued setting, we review background on $\mathcal{B}$-valued probability spaces and Hilbert modules, and we establish some notation.

In \S \ref{sec:constructconvolution}, we define our main objects of study, $\mathcal{T}$-free products of Hilbert modules and the resulting $\mathcal{T}$-free products of $\mathcal{B}$-valued probability spaces and $\mathcal{T}$-free convolution of $\mathcal{B}$-valued laws.  In \S \ref{sec:combinatorics}, we give a combinatorial formula for joint moments in the $\mathcal{T}$-free product space, expressed in terms of the Boolean cumulants and language of non-crossing partitions.

In \S \ref{sec:operad}, we define a topological symmetric operad $\Tree$ where the objects of arity $N$ are rooted subtrees of the $N$-regular tree.  We also define a topological symmetric operad $\Func(\mathcal{B})$ where the objects of arity $N$ are $N$-ary functions on the space of non-commutative laws.  We show that mapping $\mathcal{T}$ to the $\mathcal{T}$-free convolution operation defines a morphism of topological symmetric operads.

In \S \ref{sec:convolutionidentities}, we apply our framework to reprove several convolution identities from the literature using combinatorical manipulations of trees.  Moreover, generalizing work of Lenczewski \cite{Lenczewski2007}, we discuss a decomposition of $\mathcal{T}$-free convolutions using iterated Boolean and orthogonal convolution.

Next, assuming that the root vertex has more than one neighbor, we develop a theory of $\mathcal{T}$-free convolution that closely parallels the free, Boolean, and monotone cases.  In \S \ref{sec:cumulants}, we define the $\mathcal{T}$-free cumulants.  We show that they satisfy the same axioms that characterize the free, Boolean, and monotone cumulants \cite[Theorem 3.1]{HS2011b}.

In \S \ref{sec:CLT}, we prove a central limit theorem for $\mathcal{T}$-free independence.  We first present a proof based on cumulants, and then more refined estimates obtained from coupling different random variables on the same Hilbert module.

In \S \ref{sec:infinitelydivisible}, we study the laws which are infinitely divisible (with bounded support) with respect to $\mathcal{T}$-free convolution.  In particular, we give a model for such laws on a $\mathcal{T}$-free Fock space.  We thus obtain in \S \ref{subsec:BPbijection} generalized Bercovici-Pata bijections between laws that are infinitely divisible with bounded support for every $\mathcal{T}$ where the root vertex has more than one neighbor.

Finally, in \S \ref{sec:conclusion}, we suggest some directions for future research.

\subsection{Acknowledgements}

We thank Hari Bercovici and Dima Shlyakhtenko for discussion and advice.  We thank Octavio Arizmendi, Serban Belinschi, Ian Charlesworth, Takahiro Hasebe, Franz Lehner, and Roland Speicher for helpful conversations and/or pointing out typos.  The first author thanks the Banff International Research Station and the Mathematische Forschungsinstitut Oberwolfach for their hospitality.  He also acknowledges the support of the NSF grant DMS-1500035.

\section{Preliminaries} \label{sec:preliminaries}

Here we summarize some background material for $\mathrm{C}^*$-algebra operator-valued non-commutative probability for the reader to refer to as necessary.  Most importantly, since the scalars are replaced by a unital $\mathrm{C}^*$-algebra $\mathcal{B}$, we will use \emph{$\mathcal{A}$-$\mathcal{B}$-correspondences}, which are roughly speaking ``representations of $\mathcal{A}$ on a Hilbert space with $\mathcal{B}$-valued inner product'' (see Definition \ref{def:bimodule}).  We also use the less standard notation $\rad(\mu)$ for the ``support radius'' of a $\mathcal{B}$-valued law $\mu$ (Definition \ref{def:expbound}).

\subsection{$\mathrm{C}^*$-correspondences}

We assume familiarity with the basic theory of unital $\mathrm{C}^*$-algebras, matrices over a $\mathrm{C}^*$-algebra, and completely positive maps.  We refer to \cite[Chapter II]{Blackadar2006} for a summary of results and references.

Non-commutative probability often uses explicit representations of $\mathrm{C}^*$-algebras on Hilbert spaces.  In $\mathcal{B}$-valued non-commutative probability, we use an analogue of Hilbert spaces where the inner product is $\mathcal{B}$-valued, which is a called a right Hilbert $\mathcal{B}$-module.  For background, see \cite{Paschke1973}, \cite{Lance1995}, \cite[\S II.7.1 - II.7.2]{Blackadar2006}, and the references therein.

\begin{definition}
Let $\mathcal{B}$ be a unital $\mathrm{C}^*$-algebra.  If $\mathcal{H}$ is a right $\mathcal{B}$-module, then a \emph{$\mathcal{B}$-valued semi-inner product} is a map $\ip{\cdot,\cdot}: \mathcal{H} \times \mathcal{H} \to \mathcal{B}$ such that for $h$, $h_1$, $h_2 \in \mathcal{H}$.
\begin{enumerate}[(1)]
	\item $h_2 \mapsto \ip{h_1,h_2}$ is a right $\mathcal{B}$-module map.
	\item $\ip{h_2,h_1} = \ip{h_1,h_2}^*$.
	\item $\ip{h,h} \geq 0$.
\end{enumerate}
\end{definition}

One can show that the semi-inner product must satisfy an analogue of the Cauchy-Schwarz inequality and hence $\norm{h} := \norm{\ip{h,h}}_{\mathcal{B}}^{1/2}$ defines a semi-norm on $\mathcal{H}$.  We also have $\norm{hb} \leq \norm{h} \norm{b}$ for $h \in \mathcal{H}$ and $b \in \mathcal{B}$.

\begin{definition}
If $\mathcal{H}$ is a Banach space with respect to this norm, then we say that $\mathcal{H}$ is a \emph{right Hilbert $\mathcal{B}$-module}.  In general, if $\mathcal{H}$ has a $\mathcal{B}$-valued semi-inner product, then the completion of $\mathcal{H} / \{h: \norm{h} = 0\}$ is a right Hilbert $\mathcal{B}$-module with the right $\mathcal{B}$-action and the $\mathcal{B}$-valued inner product induced in the natural way from those of $\mathcal{H}$.  We refer to this module as the \emph{separation-completion} of $\mathcal{H}$ with respect to $\ip{\cdot,\cdot}$.
\end{definition}

\begin{definition}
Let $\mathcal{H}_1$ and $\mathcal{H}_2$ be Hilbert $\mathcal{B}$-modules, we say that a linear map $T: \mathcal{H}_1 \to \mathcal{H}_2$ is \emph{right $\mathcal{B}$-modular} if $T(hb) = (Th)b$ for $h \in \mathcal{H}_1$ and $b \in \mathcal{B}$.  We say that $T$ is \emph{adjointable} if there exists a map $T^*: \mathcal{H}_2 \to \mathcal{H}_1$ such that
\[
\ip{Th_1,h_2} = \ip{h_1,T^*h_2} \text{ for all } h_1 \in \mathcal{H}_1 \text{ and } h_2 \in \mathcal{H}_2.
\]
We denote by $\mathcal{L}(\mathcal{H})$ the space of bounded, right $\mathcal{B}$-modular, adjointable operators on a right Hilbert $\mathcal{B}$-module $\mathcal{H}$.  One can check that $\mathcal{L}(\mathcal{H})$ is a $\mathrm{C}^*$-algebra (see for instance \cite[p.\ 8]{Lance1995}).
\end{definition}

A \emph{$\mathcal{B}$-valued representation} of a $\mathrm{C}^*$-algebra $\mathcal{A}$ is a $*$-homomorphism $\pi: \mathcal{A} \to \mathcal{L}(\mathcal{H})$ for some right Hilbert $\mathcal{B}$-module $\mathcal{H}$.  Such a representation endows $\mathcal{H}$ with the structure of an $\mathcal{A}$-$\mathcal{B}$-bimodule.  Since the bimodule viewpoint will be notationally convenient, we make the following definition.

\begin{definition} \label{def:bimodule}
A \emph{$\mathcal{A}$-$\mathcal{B}$-correspondence} is an $\mathcal{A}$-$\mathcal{B}$ bimodule $\mathcal{H}$ with a $\mathcal{B}$-valued inner product, such that $\mathcal{H}$ right Hilbert $\mathcal{B}$-module with respect to the right action of $\mathcal{B}$, and the left action of $\mathcal{A}$ defines a $*$-homomorphism $\mathcal{A} \to \mathcal{L}(\mathcal{H})$.  We refer to $\mathcal{H}$ generically as a \emph{$\mathrm{C}^*$-correspondence} if $\mathcal{A}$ and $\mathcal{B}$ are unspecified or clear from context.
\end{definition}

\begin{remark}
The term ``$\mathcal{A}$-$\mathcal{B}$-correspondence'' is standard in $\mathrm{C}^*$-algebra theory and has been used before in non-commutative probability, for instance by \cite[\S 5]{Voiculescu1985} \cite{PV2013}.  We caution that this definition is asymmetrical since the left and right actions are of different natures, and that this is not the only notion of ``Hilbert bimodule'' one might study in the context of operator algebras.
\end{remark}

We will also need the following notion of \emph{tensor products of $\mathrm{C}^*$-correspondences}.

\begin{construction} \label{const:bimoduletensorproduct}
Let $\mathcal{B}_0$, \dots, $\mathcal{B}_n$ be unital $\mathrm{C}^*$-algebras, and suppose that $\mathcal{H}_j$ is a Hilbert $\mathcal{B}_{j-1}$-$\mathcal{B}_j$-correspondence for each $j$.  We can form the algebraic tensor product
\[
\mathcal{H}_1 \otimes_{\alg, \mathcal{B}_1} \dots \otimes_{\alg, \mathcal{B}_{n-1}} \mathcal{H}_n
\]
in the sense of algebraic bimodules.  We define a semi-inner product by
\[
\ip{h_1 \otimes \dots \otimes h_n, h_1' \otimes \dots \otimes h_n'} = \ip{h_n, \ip{h_{n-1}, \dots \ip{h_1, h_1'} \dots h_{n-1}'} h_n'}.
\]
In other words, we first evaluate $\ip{h_1,h_1'} \in \mathcal{B}_1$, then evaluate $\ip{h_1,h_1'} h_2$ using the left $\mathcal{B}_1$-module structure on $\mathcal{H}_2$, then compute $\ip{h_2, \ip{h_1,h_1'} h_2'} \in \mathcal{B}_2$ and so forth.  Positivity of the inner product is checked by using complete positivity in the standard way.

We denote the separation-completion with respect to this inner product by
\[
\mathcal{H}_1 \otimes_{\mathcal{B}_1} \dots \otimes_{\mathcal{B}_{n-1}} \mathcal{H}_n
\]
and one can show that this is a Hilbert $\mathcal{B}_0$-$\mathcal{B}_n$-bimodule in the obvious way.  As one would expect, these tensor products satisfy the associativity up to a canonical isomorphism, and they distribute over direct sums of correspondences in each argument.
\end{construction}

\subsection{$\mathcal{B}$-valued Probability Spaces}

\begin{definition}
Let $\mathcal{B} \subseteq \mathcal{A}$ be a unital inclusion of unital $\mathrm{C}^*$-algebras.  Then a \emph{conditional expectation} $\mathcal{A} \to \mathcal{B}$, or simply a \emph{$\mathcal{B}$-valued expectation}, is a linear map $E: \mathcal{A} \to \mathcal{B}$ that is unital, completely positive, and $\mathcal{B}$-$\mathcal{B}$-bimodular (that is, $E[b_1ab_2] = b_1 E[a] b_2$ for $a \in \mathcal{A}$ and $b_1, b_2 \in \mathcal{B}$).
\end{definition}

\begin{definition} \label{def:NCprobspace}
Suppose $\mathcal{B}$ is a unital $\mathrm{C}^*$-algebra.  A \emph{$\mathcal{B}$-valued (non-commutative) probability space} is a pair $(\mathcal{A},E)$, where $\mathcal{A}$ is a unital $\mathrm{C}^*$-algebra with a specified unital inclusion $\mathcal{B} \subseteq \mathcal{A}$ and where $E: \mathcal{A} \to \mathcal{B}$ is a conditional expectation, such that for each $a \in \mathcal{A}$,
\begin{equation} \label{eq:nondegeneracy}
E[a_1aa_2] = 0 \text{ for all } a_1, a_2 \in \mathcal{A} \implies a = 0.
\end{equation}
We refer to the elements of $\mathcal{A}$ as \emph{(bounded) $\mathcal{B}$-valued random variables}.
\end{definition}

If $(\mathcal{A},E)$ is a $\mathcal{B}$-valued probability space, then we have the following canonical representation of $\mathcal{A}$ on a right Hilbert $\mathcal{B}$-module.

\begin{construction} \label{const:canonicalrep}
Note that $\mathcal{A}$ is a right $\mathcal{B}$-module and we can define a $\mathcal{B}$-valued semi-inner product on $\mathcal{A}$ by $\ip{a_1,a_2} = E[a_1^*a_2]$.  We denote the separation-completion with respect to this inner product by $L^2(\mathcal{A},E)$.  One can check that $L^2(\mathcal{A},E)$ is a $\mathcal{A}$-$\mathcal{B}$-correspondence.  If we denote by $\xi$ the equivalence class of the vector $1 \in \mathcal{A}$, then we have
\[
E[a] = \ip{\xi, a \xi}.
\]
If we denote by $\pi: \mathcal{A} \to \mathcal{L}(L^2(\mathcal{A},E))$ the corresponding representation, then the non-degeneracy condition \eqref{eq:nondegeneracy} means that $\pi$ is injective.
\end{construction}

In particular, this shows that given a $\mathcal{B}$-valued non-commutative probability space $(\mathcal{A},E)$, there is an $\mathcal{A}$-$\mathcal{B}$-correspondence $\mathcal{H}$ and a vector $\xi$ such that $E[a] = \ip{\xi, a\xi}$.  Conversely, given an $\mathcal{A}$-$\mathcal{B}$-correspondence $\mathcal{H}$ and $\xi \in \mathcal{H}$, we can define $E: \mathcal{A} \to \mathcal{B}$ by $E[a] = \ip{\xi, a\xi}$, and the next two lemmas describe sufficient conditions for $(\mathcal{A},E)$ to be a $\mathcal{B}$-valued probability space.

\begin{lemma}[{\cite[Lemma 2.10]{Liu2018}}] \label{lem:unitvector}
Let $\mathcal{B} \subseteq \mathcal{A}$ be a unital inclusion of unital $\mathrm{C}^*$-algebras.  Let $\mathcal{H}$ be a $\mathcal{A}$-$\mathcal{B}$-correspondence and $\xi \in \mathcal{H}$ and define $E: \mathcal{A} \to \mathcal{B}$ by $\Phi(a) = \ip{\xi, a \xi}$.  Then the following are equivalent:
\begin{enumerate}[(1)]
	\item $E$ is a $\mathcal{B}$-valued expectation.
	\item $\ip{\xi, b\xi} = b$ for every $b \in \mathcal{B}$.
	\item $\ip{\xi,\xi} = 1$ and $b \xi = \xi b$ for every $b \in \mathcal{B}$. 
\end{enumerate}
In this case, we say that $\xi$ is a \emph{$\mathcal{B}$-central unit vector}.
\end{lemma}

\begin{proof}
If (1) holds, then $\ip{\xi, b\xi} = \Phi(b) = b$, so that (2) holds.

Suppose (2) holds.  Then $\ip{\xi,\xi} = \ip{\xi, 1 \xi} = 1$.  Also, for $b \in \mathcal{B}$, we have
\begin{align*}
\ip{b \xi - \xi b, b \xi - \xi b} &= \ip{b \xi, b \xi} - \ip{b \xi, \xi b} - \ip{\xi b, b \xi} + \ip{\xi b, \xi b} \\
&= \ip{\xi, b^*b \xi} - \ip{\xi, b^* \xi} b - b^* \ip{\xi, b \xi} + b^* \ip{\xi,\xi} b \\
&= b^*b - b^*b - b^*b + b^*b = 0.
\end{align*}
Therefore, $b \xi = \xi b$, so that (3) holds.

Suppose that (3) holds.  Then $E$ is unital since $\ip{\xi, 1 \xi} = 1$.  Moreover, $E$ is $\mathcal{B}$-$\mathcal{B}$-bimodular because
\[
\ip{\xi, b_1 a b_2 \xi} = \ip{b_1^* \xi, ab_2 \xi} = \ip{\xi b_1^*, a \xi b_2} = b_1 \ip{\xi, a\xi} b_2.
\]
Thus, (1) holds.
\end{proof}

\begin{lemma} \label{lem:faithfulness}
Suppose that $\mathcal{B} \subseteq \mathcal{A}$ is a unital inclusion.  Suppose $\mathcal{H}$ is a $\mathcal{A}$-$\mathcal{B}$-correspondence and $\xi \in \mathcal{H}$ is a $\mathcal{B}$-central unit vector.  If the representation $\pi: \mathcal{A} \to \mathcal{L}(\mathcal{H})$ is injective and if $\mathcal{A} \xi$ is dense in $\mathcal{H}$, then the non-degeneracy condition \eqref{eq:nondegeneracy} holds and hence $(\mathcal{A},E)$ is a $\mathcal{B}$-valued probability space.
\end{lemma}

\begin{proof}
Let $a \in \mathcal{A}$.  If $\ip{\xi, a_1aa_2 \xi} = 0$ for all $a_1$, $a_2$, then we have $\ip{a_1 \xi, a a_2\xi} = 0$ for all $a_1$ and $a_2$.  Since $\mathcal{A} \xi$, is dense, it follows that $\pi(a) = 0$.  Thus, $a = 0$ by assumption.
\end{proof}

We need one more fact for our construction of product spaces.  If $\mathcal{H}$ is a $\mathcal{A}$-$\mathcal{B}$-correspondence and $\mathcal{K} \subseteq \mathcal{H}$ is an $\mathcal{A}$-$\mathcal{B}$-submodule, then the orthogonal complement $\mathcal{K}^\perp = \{h: \ip{h,k} = 0 \text{ for all } k \in \mathcal{K}\}$ is also a Hilbert $\mathcal{A}$-$\mathcal{B}$-submodule, but $\mathcal{K} + \mathcal{K}^\perp$ might not span all of $\mathcal{H}$.  However, we do have a such a decomposition in the special case where $\mathcal{K}$ is the span of a $\mathcal{B}$-central unit vector in a $\mathcal{B}$-$\mathcal{B}$-correspondence.  The following lemma is proved as in \cite[Proof of Remark 3.3]{PV2013}.

\begin{lemma} \label{lem:orthocomplement}
Let $\mathcal{H}$ be a $\mathcal{B}$-$\mathcal{B}$-correspondence and $\xi$ a $\mathcal{B}$-central unit vector.  Let $\mathcal{H}^\circ = \{h \in \mathcal{H}: \ip{h,\xi} = 0\}$.  Then we have $\mathcal{H} = \mathcal{B} \xi \oplus \mathcal{H}^\circ$ as $\mathcal{B}$-$\mathcal{B}$-correspondences.
\end{lemma}

\subsection{$\mathcal{B}$-valued Laws}

The law of a self-adjoint random variable $X$ in a $\mathcal{B}$-valued probability space is defined as follows, as in \cite{Voiculescu1995}, \cite{PV2013}, \cite{AW2016}.  As motivation, recall that classically the law of a bounded real random variable $X$ is completely captured by its moments, or in other words by the map $\C[x] \to \C$ given by $p \mapsto E[p(X)]$.

\begin{definition}
Let $\mathcal{B}$ be a unital $\mathrm{C}^*$-algebra.  We define the \emph{non-commutative polynomial algebra} $\mathcal{B}\ip{X}$ to be the universal unital $*$-algebra generated by $\mathcal{B}$ and a self-adjoint indeterminate $X$.  As a vector space, $\mathcal{B}\ip{X}$ is spanned by the non-commutative monomials $b_0 X b_1 \dots X b_k$ for $k \geq 0$ and $b_j \in \mathcal{B}$.  Note that $\mathcal{B} \subseteq \mathcal{B}\ip{X}$ as unital $*$-algebras and in particular $\mathcal{B}\ip{X}$ is a $\mathcal{B}$-$\mathcal{B}$-bimodule.
\end{definition}

\begin{definition}
Let $Y$ be a self-adjoint random variable in the $\mathcal{B}$-valued probability space $(\mathcal{A},E)$.  Then the \emph{law of $Y$} is the map $\mathcal{B}\ip{X} \to \mathcal{B}$ given by $p \mapsto E[p(Y)]$.  More generally, suppose that $\mathcal{B} \subseteq \mathcal{A}$ unitally, $\Phi: \mathcal{A} \to \mathcal{B}$ is completely positive, and $Y \in \mathcal{A}$ is self-adjoint.  Then the \emph{law of $Y$} is the map $p \mapsto E[p(Y)]$.
\end{definition}

There is an abstract description of the maps $\mathcal{B}\ip{X} \to \mathcal{B}$ which can be realized as the law of some self-adjoint element $Y$ as above.

\begin{definition}
We say that $\sigma: \mathcal{B}\ip{X} \to \mathcal{B}$ is \emph{completely positive} if for every $n \geq 1$, for every $P(X) \in M_n(\mathcal{B}\ip{X})$, we have $\sigma^{(n)}(P(X)^*P(X)) \geq 0$ in $M_n(\mathcal{B})$.
\end{definition}

\begin{definition} \label{def:expbound}
We say that $\sigma: \mathcal{B}\ip{X} \to \mathcal{B}$ is \emph{exponentially bounded} if there exists $M$ and $R > 0$ such that
\[
\norm{\sigma(b_0 X b_1 \dots X b_\ell)} \leq M R^\ell \norm{b_0} \dots \norm{b_\ell} \text{ for all } \ell \geq 0 \text{ and } b_j \in \mathcal{B}.
\]
We denote by $\rad(\sigma)$ the infimum of all values of $R$ such that this inequality holds for some $M$.
\end{definition}

If $\sigma$ is the law of a self-adjoint random variable $Y$ in $(\mathcal{A},E)$, then $\sigma$ is completely positive, exponentially bounded, unital, and $\mathcal{B}$-$\mathcal{B}$-bimodular.  More generally, if $\sigma$ is the distribution of a self-adjoint element $Y$ with respect to completely positive map $\Phi: \mathcal{A} \to \mathcal{B}$, then $\sigma$ is completely positive and exponentially bounded.  The exponential bound is given explicitly by
\begin{align*}
\norm{\sigma(b_0 X b_1 \dots X b_\ell)} &= \norm{\Phi(b_0 Y b_1 \dots Y b_\ell)} \\
&\leq \norm{\Phi(1)} \norm{Y}^\ell \norm{b_0} \dots \norm{b_\ell} \\
&= \norm{\sigma(1)} \norm{Y}^\ell \norm{b_0} \dots \norm{b_\ell},
\end{align*}
so that $\rad(\sigma) \leq \norm{Y}$.

Conversely, the next result shows that every completely positive and exponentially bounded $\sigma: \mathcal{B}\ip{X} \to \mathcal{B}$ can be realized as the distribution of some self-adjoint element.  This is an adaptation of \cite[Proposition 1.2]{PV2013} and Williams \cite[Proposition 2.9]{Williams2017}.  However, we do not assume that $\sigma|_{\mathcal{B}} = \id$ and we give a sharper bound on the operator norm.  Thus, for completeness, we include some details of the proof.

\begin{theorem} \label{thm:realizationoflaw}
Let $\sigma: \mathcal{B}\ip{X} \to \mathcal{B}$ be completely positive and exponentially bounded.  Then there exists a unital $\mathrm{C}^*$-algebra $\mathcal{A}$ which contains $\mathcal{B}$ unitally, a completely positive map $\Phi: \mathcal{A} \to \mathcal{B}$, and a self-adjoint $Y \in \mathcal{A}$ such that $\norm{Y} = \rad(\sigma)$ and $\sigma$ is the distribution of $Y$ with respect to $\Phi$.  Furthermore, if $\sigma|_{\mathcal{B}} = \id$, then $(\mathcal{A},\Phi)$ can be chosen to be a $\mathcal{B}$-valued probability space (and in particular, $\sigma$ must be a $\mathcal{B}$-$\mathcal{B}$-bimodule map).
\end{theorem}

\begin{proof}
We define $\mathcal{B}\ip{X} \otimes_\sigma \mathcal{B}$ to be the right Hilbert $\mathcal{B}$-module which is the separation-completion of $\mathcal{B}\ip{X} \otimes_{\alg} \mathcal{B}$ with respect to the semi-inner product
\[
\ip{p(X) \otimes b, p'(X) \otimes b'} = b^* \sigma(p(X)^* p'(X)) b'.
\]
The positivity of this semi-inner product follows from complete positivity of $\sigma$ as in \cite[Proposition 4.5]{Lance1995}.

We claim that for each $p(X) \in \mathcal{B}\ip{X}$, the left multiplication by $p(X)$ defines an operator in $\mathcal{L}(\mathcal{B}\ip{X} \otimes_\sigma \mathcal{B})$.  By taking sums and products, it suffices to prove the case where $p(X) = b \in \mathcal{B}$ or $p(X) = X$.  For the easier case $p(X) = b$, we refer to the references cited above or to \cite[Theorem 5.2]{Paschke1973}.

To show that multiplication by $X$ yields a well-defined bounded operator $Y$ on $\mathcal{B}\ip{X} \otimes_\sigma \mathcal{B}$ with $\norm{Y} \leq \rad(\sigma)$, it suffices to show that for every $R > \rad(\sigma)$ and $h \in \mathcal{B}\ip{X} \otimes_{\alg} \mathcal{B}$, we have
\begin{equation} \label{eq:multiplicationclaim}
\ip{h, (R^2 - X^2) h} \geq 0.
\end{equation}
We want to write $R^2 - X^2$ as $g(X)^* g(X)$ for some function $g(X)$.  We will define $g(X)$ using the power series of $\sqrt{R^2 - X^2}$; we will show that this makes sense in a certain analytic completion of $\mathcal{B}\ip{X}$, defined as follows.

Fix $R > \rad(\sigma)$ and choose $R_0$ such that $R > R_0 > \rad(\sigma)$.  For a monomial $b_0 X b_1 \dots X b_\ell$, define
\[
\mathfrak{p}_{R_0}(b_0 X b_1 \dots X b_\ell) = R_0^\ell \norm{b_0} \dots \norm{b_\ell}.
\]
For $f \in \mathcal{B}\ip{X}$, define
\[
\norm{f}_{R_0} = \inf \left \{ \sum_{j=1}^n \mathfrak{p}_{R_0}(f_j): f_j \text{ monomials and } f = \sum_{j=1}^n f_j \right\}
\]
One can check that $\norm{\cdot}_{R_0}$ is a norm, $\norm{f_1 f_2}_{R_0} \leq \norm{f_1}_{R_0} \norm{f_2}_{R_0}$, and $\norm{f^*}_{R_0} = \norm{f}_{R_0}$.  Hence, the completion of $\mathcal{B}\ip{X}$ with respect to this norm, which we denote by $\mathcal{B}\ip{\ip{X}}_{R_0}$, is a Banach $*$-algebra.

Because $\sigma$ is exponentially bounded with $\rad(\sigma) < R_0$, there exists an $M > 0$ such that $\norm{\sigma(b_0 X b_1 \dots X b_\ell)} \leq M R_0^\ell \norm{b_0} \dots \norm{b_\ell}$.  Hence, for $f \in \mathcal{B}\ip{X}$, we have $\norm{\sigma(f)} \leq M \norm{f}_{\mathfrak{p}}$.  Therefore, $\sigma$ extends uniquely to a bounded map $\mathcal{B}\ip{\ip{X}}_{R_0} \to \mathcal{B}$, and this extended map is completely positive.  Similarly, for every $h \in \mathcal{B}\ip{X} \otimes_{\alg} \mathcal{B}$, the map $f \mapsto \ip{h, f(X)h}$ extends to be bounded on $\mathcal{B}\ip{\ip{X}}_{R_0}$.

Fix $R > R_0$.  Let $\sum_{j=0}^\infty \alpha_j x^j$ be the power series expansion of the (scalar-valued) function $\sqrt{R^2 - x^2}$ about the point zero.  The radius of convergence of this series is $R$.  Since $R_0 < R$, it follows that  $g(X) = \sum_{j=0}^\infty \alpha_j X^j$ converges absolutely in $\mathcal{B}\ip{\ip{X}}_{R_0}$.  Because $\mathcal{B}\ip{\ip{X}}_{R_0}$ is a Banach $*$-algebra, we may compute the square of the absolutely convergent series $g(X)$ by multiplying it out term by term.  It follows that $g(X)^*g(X) = g(X)^2 = R^2 - X^2$ and hence for $h \in \mathcal{B}\ip{X} \otimes_{\alg} \mathcal{B}$,
\[
\ip{h, (R^2 - X^2) h} = \ip{h, g(X)^2 h} = \ip{g(X) h, g(X) h} \geq 0.
\]
Therefore, the operator $Y$ of multiplication by $X$ is well-defined and bounded on $\mathcal{B}\ip{X} \otimes_\sigma \mathcal{B}$ with $\norm{Y} \leq \rad(\sigma)$.  The opposite inequality $\rad(\sigma) \leq \norm{Y}$ is immediate.  The self-adjointness (hence adjointability) of $Y$ follows by direct computation.

Thus, we can take $\mathcal{A}$ to be the $\mathrm{C}^*$-subalgebra of $\mathcal{L}(\mathcal{B}\ip{X} \otimes_\sigma \mathcal{B})$ generated by $Y$ and $\mathcal{B}$, let $\xi = [1 \otimes 1] \in \mathcal{B}\ip{X} \otimes \mathcal{B}$, and take $\Phi: \mathcal{A} \to \mathcal{B}$ to be the map $\Phi(a) = \ip{\xi, a \xi}$.

If $\sigma|_{\mathcal{B}} = \id$, then we have $\ip{\xi, b\xi} = b$ for every $b$ and hence by Lemma \ref{lem:unitvector}, $\Phi$ is a $\mathcal{B}$-valued expectation.  To show that $(\mathcal{A},\Phi)$ is a $\mathcal{B}$-valued probability space is suffices by Lemma \ref{lem:faithfulness} to show that $\mathcal{A} \xi$ is dense in $\mathcal{B}\ip{X} \otimes_\sigma \mathcal{B}$.  By construction, vectors for the form $f(X) \otimes b$ are dense in $\mathcal{B}\ip{X} \otimes_\sigma \mathcal{B}$, but since $b \xi = \xi b$, we have $f(X) \otimes b = f(Y) \xi b = f(Y)b \xi \in \mathcal{A}\xi$.
\end{proof}

\begin{remark}
If we assume that $\sigma$ is unital and $\mathcal{B}$-$\mathcal{B}$-bimodular, then we can replace $\mathcal{B}\ip{X} \otimes \mathcal{B}$ with the module $L^2(\mathcal{B}\ip{X},\sigma)$ which is defined to be the completion of $\mathcal{B}\ip{X}$ with respect to $\ip{p(X), p'(X)} = \sigma(p(X)^* p'(X))$.  Moreover, we have $\mathcal{B}\ip{X} \otimes_\sigma \mathcal{B} \cong L^2(\mathcal{B}\ip{X},\sigma)$ in this case.
\end{remark}

The special case where $\mu$ is a unital $\mathcal{B}$-$\mathcal{B}$-bimodule map is essential for the rest of the paper, and we therefore introduce the following notation.

\begin{definition} \label{def:laws}
A \emph{(bounded) $\mathcal{B}$-valued law} is a unital, completely positive, exponentially bounded, $\mathcal{B}$-$\mathcal{B}$-bimodule map $\mu: \mathcal{B}\ip{X} \to \mathcal{B}$.  We denote the set of such laws by $\Sigma(\mathcal{B})$.  We also denote $\Sigma_R(\mathcal{B}) = \{\mu \in \Sigma(\mathcal{B}): \rad(\mu) \leq R\}$.
\end{definition}

We caution that some authors do not include the assumption of exponential boundedness in their definition of $\Sigma(\mathcal{B})$.  A topology on $\Sigma_R(\mathcal{B})$ of convergence of moments will be discussed in \S \ref{subsec:continuity}.

\section{$\mathcal{T}$-free Products and Convolutions} \label{sec:constructconvolution}

\subsection{Definitions}

\begin{definition}
For $N \in \N$, let $[N] = \{1,\dots,N\}$.  A \emph{string} on the alphabet $[N]$ is a finite sequence $j_1 \dots j_\ell$ with $j_i \in [N]$.  We denote by the $i$th letter of a string $s$ by $s(i)$.  Given two strings $s_1$ and $s_2$, we denote their concatenation by $s_1 s_2$.
\end{definition}

\begin{definition} \label{def:string}
A string is called \emph{alternating} if $j_i \neq j_{i+1}$ for every $i \in \{1,\dots,\ell-1\}$.  For a string $s$, we define the \emph{alternating reduction} $\red(s)$ to be the alternating string obtained by replacing consecutive occurrences of the same letter by a single occurrence of that letter; for instance,
\[
\red(112331) = 1231, \qquad \red(1221311) = 12131.
\]
\end{definition}

\begin{definition} \label{def:stringtree}
Let $\mathcal{T}_{N,\free}$ be the (simple) graph whose vertices are the alternating strings on the alphabet $[N]$ and where the edges are given by $s \sim j s$ for every letter $j$ and every string $s$ that does not begin with $j$.  Note that $\mathcal{T}_{N,\free}$ is an infinite $N$-regular tree.  We denote the empty string by $\emptyset$, and we view $\emptyset$ as the preferred root vertex of the graph $\mathcal{T}_{N,\free}$.
\end{definition}

\begin{definition}
We denote by $\Tree(N)$ the set of rooted subtrees of $\mathcal{T}_{N,\free}$ (that is, connected subgraphs containing the vertex $\emptyset$).  We denote by $\Tree'(N)$ the set of rooted subtrees that contain all of the singleton strings $1$, \dots, $N$.  Note that if $\mathcal{T} \in \Tree(N)$, then the edge set is uniquely determined by the vertex set and vice versa.  Thus, we may treat $\mathcal{T}$ merely as a set of vertices when it is notationally convenient.
\end{definition}

Let $\mathcal{T} \in \Tree(N)$, and let $(\mathcal{H}_1,\xi_1)$, \dots, $(\mathcal{H}_N,\xi_N)$ be $\mathcal{B}$-$\mathcal{B}$-correspondences with $\mathcal{B}$-central unit vectors (using the terminology established in Definition \ref{def:bimodule} and Lemma \ref{lem:unitvector}).  We will describe how to define a $\mathcal{B}$-$\mathcal{B}$-correspondence with $\mathcal{B}$-central unit vector
\[
(\mathcal{H},\xi) = \assemb_\mathcal{T}[(\mathcal{H}_1,\xi_1), \dots, (\mathcal{H}_N,\xi_N)]
\]
together with inclusion maps
\[
\lambda_{\mathcal{T},j}: \mathcal{L}(\mathcal{H}_j) \to \mathcal{L}(\mathcal{H}).
\]
Let $\mathcal{H}_j^\circ$ be the orthogonal complement of $\xi_j$ (see Lemma \ref{lem:orthocomplement}) and for a string $s = j_1 \dots j_\ell \in \mathcal{T}_{N,\free}$, define (by Construction \ref{const:bimoduletensorproduct})
\[
\mathcal{H}_s^\circ = \begin{cases} \mathcal{B}, & \ell = 0, \\ \mathcal{H}_{j_1}^\circ \otimes_{\mathcal{B}} \dots \otimes_{\mathcal{B}} \mathcal{H}_{j_\ell}^\circ, & \text{otherwise.} \end{cases}
\]
Note that $\mathcal{H}_{s_1 s_2}^\circ \cong \mathcal{H}_{s_1}^\circ \otimes_{\mathcal{B}} \mathcal{H}_{s_2}^\circ$.  Now we define
\[
\mathcal{H} = \bigoplus_{s \in \mathcal{T}} \mathcal{H}_s^\circ,
\]
where the sum is taken over all vertices $j_1 \dots j_\ell$ of $\mathcal{T}$.

In order to define $\lambda_{\mathcal{T},j}$, let us denote
\begin{align}
S_{\mathcal{T},j} &= \{s \in \mathcal{T}: s(1) \neq j, js \in \mathcal{T}\} \nonumber \\
S_{\mathcal{T},j}' &= \{s \in \mathcal{T}: s(1) \neq j, js \not \in \mathcal{T}\}. \label{eq:defineSGJ}
\end{align}
Every vertex of $\mathcal{T}$ is either in $S_{\mathcal{T},j}$, in $S_{\mathcal{T},j}'$, or it is the concatenation of $j$ with an element of $S_{\mathcal{T},j}$.  Therefore, we have
\[
\mathcal{H} \cong \bigoplus_{s \in S_{\mathcal{T},j}} (\mathcal{H}_s^\circ \oplus \mathcal{H}_{js}^\circ) \oplus \bigoplus_{s \in S_{\mathcal{T},j}'} \mathcal{H}_s^\circ.
\]
Noting that
\[
\mathcal{H}_s^\circ \oplus \mathcal{H}_{js}^\circ \cong (\mathcal{B} \oplus \mathcal{H}_j^\circ) \otimes_{\mathcal{B}} \mathcal{H}_s^\circ \cong \mathcal{H}_j \otimes_{\mathcal{B}} \mathcal{H}_s^\circ,
\]
we have a unitary isomorphism
\[
U_{\mathcal{T},j}: \mathcal{H} \to \left[ \mathcal{H}_j \otimes \left( \bigoplus_{s \in S_{\mathcal{T},j}} \mathcal{H}_s^\circ \right) \right] \oplus \left( \bigoplus_{s S_{\mathcal{T},j}} \mathcal{H}_s^\circ \right).
\]
We define
\[
\lambda_{\mathcal{T},j}(x) = U_{\mathcal{T},j}^*([x \otimes \id] \oplus 0)U_{\mathcal{T},j} \text{ for } x \in \mathcal{L}(\mathcal{H}_j).
\]
The map $x \otimes \id$ is well-defined because $\mathcal{H}_j$ is an $\mathcal{L}(\mathcal{H}_j)$-$\mathcal{B}$-correspondence, and hence the left action of $\mathcal{L}(\mathcal{H}_j)$ on $\mathcal{H}_j \otimes \bigoplus_{s \in S_{\mathcal{T},j}} \mathcal{H}_s^\circ$ by $x \otimes \id$ is well-defined and bounded.  Note that $\lambda_{\mathcal{T},j}$ is a $*$-homomorphism which is not necessarily unital.

\begin{definition}
We denote the Hilbert $\mathcal{B}$-$\mathcal{B}$-module $(\mathcal{H},\xi)$ constructed above by
\[
\assemb_{\mathcal{T}}[(\mathcal{H}_1,\xi_1),\dots,(\mathcal{H}_N,\xi_N)],
\]
and we call it the \emph{$\mathcal{T}$-free product} of $(\mathcal{H}_1,\xi_1)$, \dots, $(\mathcal{H}_N,\xi_N)$.
\end{definition}

\begin{definition}
Given algebras $(\mathcal{A}_1,E_1)$, \dots, $(\mathcal{A}_N,E_N)$ and $\mathcal{T} \subseteq \mathcal{T}_{N,\free}$, let $\mathcal{H}_j = L^2(\mathcal{A}_j,E_j)$, let $\xi_j = 1 \in \mathcal{H}_j$, and let $\pi_j: \mathcal{A}_j \to \mathcal{L}(\mathcal{H}_j)$ be the canonical representation as in Construction \ref{const:canonicalrep}.  Let $\mathcal{H}$ be the $\mathcal{T}$-free product of $(\mathcal{H}_1,\xi_1)$, \dots, $(\mathcal{H}_N,\xi_N)$.

Then we define the \emph{$\mathcal{T}$-free product} of $(\mathcal{A}_1,E_1)$, \dots, $(\mathcal{A}_N,E_N)$ as the unital $\mathrm{C}^*$-subalgebra $\mathcal{A}$ of $\mathcal{L}(\mathcal{H})$ generated by the images $\lambda_{\mathcal{T},j} \circ \pi_j(\mathcal{A}_j)$, equipped with the expectation $E$ given by the $\mathcal{B}$-central unit vector $\xi$.  It follows from Lemmas \ref{lem:unitvector} and \ref{lem:faithfulness} that $(\mathcal{A},E)$ is a $\mathcal{B}$-valued probability space.  We denote this $\mathcal{B}$-valued non-commutative probability space by $\assemb_{\mathcal{T}}[(\mathcal{A}_1,E_1), \dots, (\mathcal{A}_N,E_N)]$.
\end{definition}

\begin{definition}
Let $\mu_1$, \dots, $\mu_N \in \Sigma(\mathcal{B})$.  Let $\mathcal{H}_j = L^2(\mathcal{B}\ip{X_j},\mu_j)$ and let $X_j$ be the multiplication operator on $\mathcal{H}_j$.  We define the \emph{additive $\mathcal{T}$-free convolution} of $\mu_1$, \dots, $\mu_N$ as the law of $\lambda_{\mathcal{T},1}(X_1) + \dots + \lambda_{\mathcal{T},N}(X_N)$.
\end{definition}

\begin{remark}
The $\mathcal{T}$-free products $\mathcal{B}$-$\mathcal{B}$-correspondences and of $\mathcal{B}$-valued probability spaces make perfect sense when $[N]$ is replaced by a different alphabet, even an infinite alphabet.  However, as our focus will be on finitary convolution operations, it will be convenient for us always to use the index set $[N]$.
\end{remark}

\begin{remark}
The original presentation in \cite{Liu2018} did not define the space $(\mathcal{H},\xi) = \assemb_{\mathcal{T}}[(\mathcal{H}_1,\xi_1),\dots,(\mathcal{H}_N,\xi_N)]$, but rather defined representations of $\mathcal{L}(\mathcal{H}_j)$ on the free product $\mathrm{C}^*$-correspondence.  If we denote $\tilde{H}$ the free product $\mathrm{C}^*$-correspondence and by $\tilde{\lambda}_j$ the representation defined in \cite{Liu2018}, then $\mathcal{H}$ is a $\mathcal{B}$-$\mathcal{B}$-submodule of the free product $\mathrm{C}^*$-correspondence.  In fact, $\mathcal{H}$ is the cyclic subspace generated by $\xi$ under the actions of $\tilde{\lambda}_j(\mathcal{L}(\mathcal{H}_j))$, and $\lambda_j(x)$ is the restriction of $\tilde{\lambda}_j(x)$ to this submodule.  The definition in \cite{Liu2018} was also phrased in terms of the index sets $S_{\mathcal{T},j} \cup jS_{\mathcal{T},j}$ rather than the tree $\mathcal{T}$.
\end{remark}

\subsection{Examples} \label{subsec:productexamples}

\begin{example}
By taking $\mathcal{T} = \mathcal{T}_{N,\free}$, we obtain the free product with amalgamation and the free convolution of $N$ variables over $\mathcal{B}$.  The free convolution of $\mu_1$, \dots, $\mu_N$ is denoted by $\mu_1 \boxplus \dots \boxplus \mu_N$.  See \cite[\S 5]{Voiculescu1985}, \cite[p.\ 351 - 353]{AGZ2009}.
\end{example}

\begin{example}
By taking $\mathcal{T}_{N,\Bool} = \{\emptyset, 1, 2, \dots, N\}$, we obtain the Boolean product and the Boolean convolution of $N$ variables.  The Boolean convolution is denoted by $\mu_1 \uplus \dots \uplus \mu_N$.  See \cite{Bercovici2006}, \cite[Remark 3.3]{PV2013}.
\end{example}

\begin{example}
Let $\mathcal{T}_{N,\mono}$ be the subtree of $\mathcal{T}_{N,\free}$ whose vertex set consists of all strings which are strictly decreasing (that is, $s_1 > \dots > s_\ell$).  Then we obtain the monotone product and monotone convolution of $N$ variables.  The monotone convolution of $\mu_1$, \dots, $\mu_N$ is denoted by $\mu_1 \rhd \dots \rhd \mu_N$.  See \cite[\S 2]{Muraki2000}, \cite[\S 4.1]{Popa2008a}.
\end{example}

\begin{example}
Symmetrically, the anti-monotone convolution is obtained using the tree consisting of all strings which are strictly \emph{increasing}.  We denote this tree by $\mathcal{T}_{N, \mono \dagger}$ and the monotone convolution of laws by $\mu_1 \lhd \dots \lhd \mu_N$.
\end{example}

\begin{example}
Let $\mathcal{T}_{\orth} \subseteq \mathcal{T}_{2,\free}$ be the subtree with vertex set $\{\emptyset, 1, 21\}$.  Then we obtain the orthogonal convolution $\vdash$.  See \cite[Def.\ 4.2, Thm.\ 4.1]{Lenczewski2007}.
\end{example}

\begin{example}
Let $\mathcal{T}_{\sub} \subseteq \mathcal{T}_{2,\free}$ be the subtree consisting of all strings which do not end with $2$.  Then we obtain the subordination free convolution $\mu \boxright \nu$.  This convolution was introduced by Lenczewski \cite[\S 7]{Lenczewski2007}, who showed the identity $\mu \boxplus \nu = \mu \rhd (\nu \boxright \mu)$, which relates to the analytic subordination property of free convolution.  We will discuss this further in Example \ref{ex:subordination}.
\end{example}

\begin{example} \label{ex:cfree}
In \cite{ABT2019}, the authors define a product operation which takes as input two \emph{pairs} of pointed Hilbert spaces and outputs another pair of pointed Hilbert spaces; this is called the \emph{c-free product} because it relates to c-free independence.  This product operation generalizes without difficulty to the operator-valued setting with $\mathcal{B}$-$\mathcal{B}$-correspondences and to $N$ pairs rather than two pairs of Hilbert spaces.  We can fit the $c$-free product into our framework as follows.  Consider the index set $[2N] = \{1,\dots,2N\}$.  For each $j \in [N]$, let us write $j' = j + N$, and similarly, for a string $s = j_1, \dots j_\ell \in \mathcal{T}_{N,\free}$,  let us write $s' = j_1' \dots j_\ell'$.   Consider pairs
\[
[(\mathcal{H}_j, \xi_j), (\mathcal{H}_{j'}, \xi_{j'})] \text{ for } j = 1, \dots, N.
\]
Then the $c$-free product is the pair
\[
[\assemb_{\mathcal{T}_1}[(\mathcal{H}_j,\xi_j)_{j=1}^{2N}], \assemb_{\mathcal{T}_2}[(\mathcal{H}_j,\xi_j)_{j=1}^{2N}]],
\]
where
\[
\mathcal{T}_1 = \{\emptyset\} \cup \{ s'j: s \in \mathcal{T}_{N,\free} \text{ and } j \in [N] \text{ such that } sj \in \mathcal{T}_{N,\free}\}.
\]
and
\[
\mathcal{T}_2 = \{s': s \in \mathcal{T}_{N,\free} \}.
\]
Thus, for instance, when $N = 2$,
\begin{align*}
\mathcal{T}_1 &= \{\emptyset\} \cup \{1,2'1,1'2'1,2'1'2'1\dots\} \cup \{2, 1'2, 2'1'2,1'2'1'2 \dots\} \\
&= \{\emptyset\} \cup \{1,41,341,4341,\dots\} \cup \{2, 32, 432, 3432, \dots\}
\end{align*}
and $\mathcal{T}_2$ consists of all alternating strings on $\{3,4\}$.  For $S_j \in \mathcal{L}(\mathcal{H}_j)$ and $T_j \in \mathcal{L}(\mathcal{H}_{j'})$, the authors of \cite{ABT2019} define an operator $\Lambda_{(S_j,T_j)}$ on first space in the $c$-free product pair, and this operator in our notation is precisely
\[
\lambda_{\mathcal{T}_1,j}(S_j) + \lambda_{\mathcal{T}_1,j'}(T_j).
\]
\end{example}

We will discuss the free, Boolean, and monotone cases in detail throughout the paper as we develop each aspect of the general theory.  For instance, the moment conditions typically used as the definition of these independences will be discussed in \S \ref{subsec:FBMcombinatorics}.  The associative property of these convolution operations will be discussed in \S \ref{subsec:digraphoperad}.  The cumulants will be discussed in \S \ref{subsec:cumulantexamples}, and infinitely divisible laws and Fock spaces in \S \ref{subsec:infdivexamples}.  We also reference the free, Boolean, monotone, orthogonal, and subordination cases in the examples throughout.

However, we will leave any further discussion of c-free convolution for future work.  Because it is a convolution operation for $N$ pairs of laws rather for $N$ laws, it would be better handled in a modified version of our framework that uses pairs of trees on the alphabet $[2N]$ to convolve $N$ pairs of laws, which we will not develop here, both for the sake of time and to minimize distraction from the main ideas.

The free, Boolean, and monotone cases fit into a general class of examples where $\mathcal{T}$ arises as the set of walks in a simple directed graph with vertex set $[N]$.
 
\begin{definition}
A \emph{simple digraph $G$ on the vertex set $[N]$} is a given by a relation $\sim_G$ on $[N]$ which is irreflexive (that is, $j \not \sim_G j$, or equivalently the relation is a subset of $[N] \times [N]$ that does not intersect the diagonal).  The relation $\sim_G$ is called the \emph{adjacency relation of $G$}.  Each pair $i \sim_G j$ in the relation will be called a \emph{directed edge from $i$ to $j$}.  We denote the set of such digraphs by $\Digraph(N)$.
\end{definition}

\begin{definition}
For a simple digraph $G$, a \emph{walk of length $\ell$} is a sequence of vertices $j_0$, \dots, $j_\ell$ with $j_i \sim_G j_{i+1}$.  We denote by $\Walk(G)$ the subtree of $\mathcal{T}_{N,\free}$ consisting of $\emptyset$ and every string $j_1 \dots j_\ell$ such that $j_\ell$, $j_{\ell-1}$, \dots, $j_0$ is a walk on $G$ (note how the order of indices is reversed).
\end{definition}

For every simple digraph $G$ on $[N]$, we can define the $\Walk(G)$-free product.  For instance,
\begin{itemize}
	\item Let $K_N$ be the complete graph on $[N]$, or in other words the adjacency relation is $\neq$.  Then $\Walk(K_N) = \mathcal{T}_{N,\free}$, which yields the free convolution.
	\item Let $K_N^c$ be the totally disconnected graph on $[N]$, or in other words the adjacency relation is $\varnothing$.  Then $\Walk(K_N^c) = \mathcal{T}_{N,\Bool}$, which yields the Boolean convolution.
	\item Let $K_N^{<}$ be the digraph on $[N]$ where the adjacency relation is given by $<$.  Then $\Walk(G) = \mathcal{T}_{N,\mono}$, which yields the monotone convolution.
	\item The anti-monotone convolution arises from the graph $K_N^{>}$ defined in the symmetrical way.
\end{itemize}
We depict the digraphs for the binary free, Boolean, monotone, and anti-monotone convolution operations in Figure \ref{fig:FBMA}.  Further discussion can be found in \S \ref{subsec:digraphoperad}.

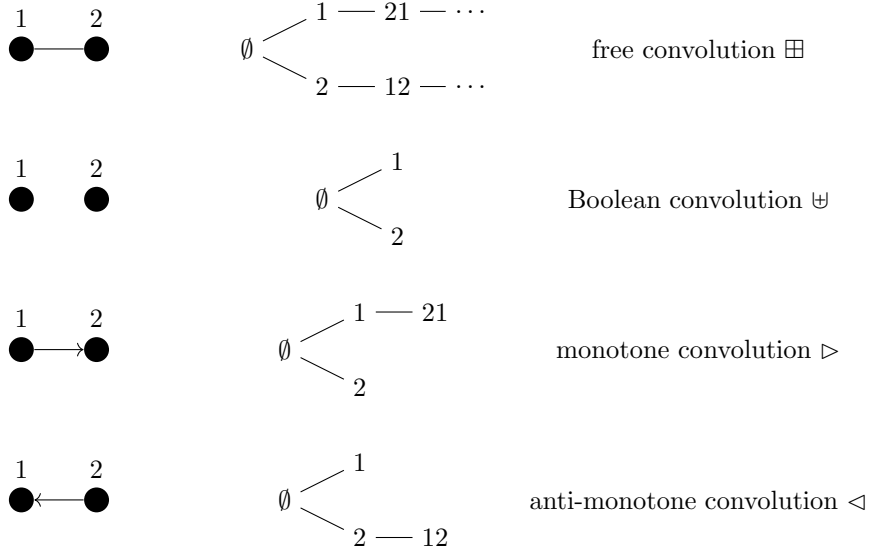
\begin{figure}

\begin{center}
\begin{tikzpicture}

\begin{scope}
	\node[circle,fill,label=above: $1$] (A) at (0,0) {};
	\node[circle,fill,label=above: $2$] (B) at (1,0) {};
	\draw (A) to (B);
	
	\node (0) at (3,0) {$\emptyset$};
	\node (1) at (4,0.5) {$1$};
	\node (2) at (4,-0.5) {$2$};
	\node (21) at (5,0.5) {$21$};
	\node (12) at (5,-0.5) {$12$};
	\node (121) at (6,0.5) {$\dots$};
	\node (212) at (6,-0.5) {$\dots$};

	\draw (0) to (1) to (21) to (121);
	\draw (0) to (2) to (12) to (212);
	
	\node at (9,0) {free convolution $\boxplus$};
\end{scope}

\begin{scope}[shift = {(0,-2)}]
	\node[circle,fill,label=above: $1$] (A) at (0,0) {};
	\node[circle,fill,label=above: $2$] (B) at (1,0) {};
	
	\node (0) at (4,0) {$\emptyset$};
	\node (1) at (5,0.5) {$1$};
	\node (2) at (5,-0.5) {$2$};

	\draw (0) to (1);
	\draw (0) to (2);
	
	\node at (9,0) {Boolean convolution $\uplus$};
\end{scope}

\begin{scope}[shift = {(0,-4)}]
	\node[circle,fill,label=above: $1$] (A) at (0,0) {};
	\node[circle,fill,label=above: $2$] (B) at (1,0) {};
	\draw[->] (A) to (B);
	
	\node (0) at (3.5,0) {$\emptyset$};
	\node (1) at (4.5,0.5) {$1$};
	\node (2) at (4.5,-0.5) {$2$};
	\node (21) at (5.5,0.5) {$21$};

	\draw (0) to (1);
	\draw (0) to (2);
	\draw (1) to (21);
	
	\node at (9,0) {monotone convolution $\rhd$};
\end{scope}

\begin{scope}[shift = {(0,-6)}]
	\node[circle,fill,label=above: $1$] (A) at (0,0) {};
	\node[circle,fill,label=above: $2$] (B) at (1,0) {};
	\draw[<-] (A) to (B);
	
	\node (0) at (3.5,0) {$\emptyset$};
	\node (1) at (4.5,0.5) {$1$};
	\node (2) at (4.5,-0.5) {$2$};
	\node (12) at (5.5,-0.5) {$12$};

	\draw (0) to (1);
	\draw (0) to (2);
	\draw (2) to (12);
	
	\node at (9,0) {anti-monotone convolution $\lhd$};
\end{scope}

\end{tikzpicture}
\end{center}

\caption{The four simple digraphs on $\{1,2\}$ (left), the corresponding trees $\Walk(G)$ (center), and the resulting binary convolution operations (right).} \label{fig:FBMA}

\end{figure}

\subsection{Bounds on the Operator Norm}

We have the following estimate for the norm of a sum of ``$\mathcal{T}$-free independent'' random variables with expectation zero.  This is a generalization of the estimate proved in the free case by \cite[Lemma 3.2]{Voiculescu1986}.

\begin{proposition} \label{prop:operatornormbound}
Let $\mathcal{T} \in \Tree(N)$, suppose that $(\mathcal{H},\xi) = \assemb_{\mathcal{T}}[(\mathcal{H}_1,\xi_1),\dots,(\mathcal{H}_N,\xi_N)]$, and let $\lambda_{\mathcal{T},j}: \mathcal{L}(\mathcal{H}_j) \to \mathcal{L}(\mathcal{H})$ be the corresponding $*$-homomorphism.  Suppose that $a_j \in \mathcal{L}(\mathcal{H}_j)$ with $\ip{\xi_j, a_j \xi_j} = 0$.  Then we have
\[
\norm*{\sum_{j=1}^N \lambda_{\mathcal{T},j}(a_j) } \leq \max_{s \in \mathcal{T}} \norm*{ \sum_{j \in [N]: j s \in \mathcal{T}} \ip{a_j \xi_j, a_j \xi_j} }^{1/2} + \max_{s \in \mathcal{T}} \norm*{ \sum_{j \in [N]: j s \in \mathcal{T}} \ip{a_j^* \xi_j, a_j^* \xi_j} }^{1/2} + \max_j \norm{a_j}.
\]
\end{proposition}

\begin{remark}
Although $\mathcal{T}$ may be infinite, the two quantities on the right-hand side are really maxima rather than suprema.  The reason is that even though there could be infinitely many possible values of $s \in \mathcal{T}$, there are at most $2^N$ possible values of $\{j \in [N]: js \in \mathcal{T}\}$ since it is a subset of $[N]$.
\end{remark}

\begin{proof}
Let $P_j \in \mathcal{L}(\mathcal{H}_j)$ be the projection onto $\xi_j$ and let $Q_j = 1 - P_j$.  Because $\ip{\xi_j, a_j \xi_j} = 0$, we have $P_j a_j P_j = 0$, and hence
\[
a_j = Q_j a_j P_j + P_j a_j Q_j + Q_j a_j Q_j.
\]
First, let us estimate $\sum_j \lambda_j(Q_j a_j P_j)$.  Note $\lambda_j(Q_j)$ is the projection onto the direct sum of the spaces $\mathcal{H}_s^\circ$ with $s(1) = j$.  Thus, the ranges of $\lambda_j(Q_j)$ are orthogonal, and hence
\[
\left( \sum_{j=1}^N \lambda_j(Q_ja_jP_j) \right)^* \left( \sum_{j=1}^N \lambda_j(Q_ja_jP_j) \right) = \sum_{j=1}^N \lambda_j(P_j a_j^* Q_j a_j P_j).
\]
Now $P_j a_j P_j = 0$ implies that
\[
(Q_j a_j P_j)^*(Q_j a_j P_j) = P_j a_j^* Q_j a_j P_j = P_j a_j^* a_j P_j = \ip{a_j \xi_j, a_j \xi_j} P_j.
\]
Thus,
\[
\norm*{ \sum_{j=1}^N \lambda_j(Q_ja_jP_j) } = \norm*{\sum_{j=1}^N \ip{a_j \xi_j, a_j \xi_j} \lambda_j(P_j) }^{1/2}
\]
Now $\sum_{j=1}^N \ip{a_j \xi_j, a_j \xi_j} \lambda_j(P_j)$ maps each direct summand $\mathcal{H}_s^\circ$ into itself.  Also,
\[
\sum_{j=1}^N \ip{a_j \xi_j, a_j \xi_j} \lambda_j(P_j) \biggr|_{\mathcal{H}_s^\circ} = \sum_{j: js \in \mathcal{T}} \ip{a_j \xi_j, a_j \xi_j} \id.
\]
Therefore, we have
\[
\norm*{ \sum_{j=1}^N \lambda_j(Q_j a_j P_j) } = \sup_{s \in \mathcal{T}} \norm*{ \sum_{j \in [N]: j s \in \mathcal{T}} \ip{a_j \xi_j, a_j \xi_j} }^{1/2}.
\]
Similarly,
\[
\norm*{ \sum_{j=1}^N \lambda_j(P_j a_j Q_j) } = \norm*{ \sum_{j=1}^N \lambda_j(Q_j a_j^* P_j) } = \sup_{s \in \mathcal{T}} \norm*{ \sum_{j \in [N]: j s \in \mathcal{T}} \ip{a_j^* \xi_j, a_j^* \xi_j} }^{1/2}.
\]
Finally, because the $\lambda_j(Q_j)$'s have orthogonal ranges, we have
\[
\norm*{ \sum_{j=1}^N \lambda_j(Q_j a_j Q_j) } = \max_j \norm{\lambda_j(Q_ja_jQ_j)} \leq \max_j \norm{a_j}.
\]
Adding the estimates for the three terms together completes the proof.
\end{proof}

\begin{corollary}
If $\ip{\xi_j, a_j \xi_j} = 0$, then we have
\[
\norm*{ \sum_{j=1}^N \lambda_j(a_j) } \leq 2 \norm*{ \sum_{j=1}^N \norm{a_j}^2 }^{1/2} + \max_j \norm{a_j}.
\]
\end{corollary}

\begin{corollary}
Let $d = \sup_{s \in \mathcal{T}} |\{j: j s \in \mathcal{T}\}|$, that is the maximum degree of $\mathcal{T}$ where only the edges that increase the length of the string are counted.  If $\ip{\xi_j, a_j \xi_j} = 0$, then
\[
\norm*{ \sum_{j=1}^n \lambda_j(a_j) } \leq (2 \sqrt{d} + 1) \max_j \norm{a_j}.
\]
\end{corollary}

\section{Combinatorial Computation of Moments} \label{sec:combinatorics}

In this section, we will show that there is a universal rule for computing the joint moments of variables in the $\mathcal{T}$-free product (Theorem \ref{thm:combinatorics}).  In order to state this rule, we first review the machinery of non-crossing partitions and the Boolean cumulants.

Non-crossing partitions were introduced into non-commutative probability by Speicher \cite{Speicher1994,Speicher1998}.  They have been used by many authors for many types of non-commutative independence; see for instance the references given in \S \ref{sec:cumulants} regarding cumulants.  We especially recommend \cite[\S 3]{ABFN2013} as a clear and efficient exposition of the background material on non-crossing partitions which we will cover mostly in \S \ref{subsec:partitions} and \S \ref{subsec:partitioncomposition} of this paper.

\subsection{Non-Crossing Partitions} \label{subsec:partitions}

Here we review basic terminology for non-crossing partitions of the $[\ell] = \{1,\dots,\ell\}$.  It will be convenient for the sake of notation to work more generally with partitions of a totally ordered finite set $S$, even though this makes no difference to the content of the results.

\begin{definition}
If $S$ is a totally ordered finite set, then a \emph{partition} of $S$ is collection of nonempty subsets $V_1$, \dots, $V_k$ such that $S = \bigsqcup_{j=1}^k V_j$.  We call the subsets $V_j$ \emph{blocks}.  We denote by $|\pi|$ the number of blocks.  We denote the collection of partitions by $\mathcal{P}(S)$, and we also write $\mathcal{P}(\ell) = \mathcal{P}([\ell])$.
\end{definition}

\begin{definition}
If $\pi \in \mathcal{P}(\ell)$, we say that $i \sim_\pi j$ if $i$ and $j$ are in the same block of $\pi$.
\end{definition}

\begin{definition}
Let $\pi$ be a partition of a totally ordered finite set $S$.  A \emph{crossing} is a set of indices $i_1 < j_1 < i_2 < j_2$ such that $i_1$ and $i_2$ are in the same block $V$ and $j_1$ and $j_2$ are in the same block $W \neq V$.  A partition is said to be \emph{non-crossing} if it has no crossings.  We denote the set of non-crossing partitions of $[\ell]$ by $\mathcal{NC}(\ell)$.
\end{definition}

\begin{definition}
Let $V$ and $W$ be blocks in a non-crossing partition $\pi$.  We say that \emph{$V$ is nested inside $W$}, or $V \succ W$, if there exist $j, k \in W$ with $V \subseteq \{j+1,\dots,k-1\}$.  As a consequence of $\pi$ being non-crossing, $\prec$ is a strict partial order on the blocks of $\pi$.
\end{definition}

\begin{remark}
We adopt the convention that $\mathcal{P}(\varnothing) = \mathcal{NC}(\varnothing)$ consists of the single partition $\varnothing$, which is a partition with zero blocks.  Here we mean that the number of blocks is zero, not that size of each block is zero, because blocks are required to be nonempty.
\end{remark}

\begin{definition}
We say that a partition $\pi \in \mathcal{NC}(S)$ is \emph{irreducible} if we have $\min S \sim_\pi \max S$, that is, the first and last elements of $S$ are in the same block of $\pi$.  We denote the set of irreducible partitions by $\mathcal{NC}^\circ(S)$.
\end{definition}

\begin{remark}
Note that $\pi \in \mathcal{NC}(S)$ is irreducible if and only if it has a unique minimal block with respect to $\prec$.  Moreover, a partition is reducible (i.e.\ not irreducible) if and only if there exists a decomposition of $S$ into $S_1 \sqcup S_2$ and $\pi = \pi_1 \sqcup \pi_2$ with $\pi_j \in \mathcal{NC}(S_j)$, where every element of $S_1$ is less than every element of $S_2$ (in other words, $\pi$ is expressed by ``concatenating'' $\pi_1$ and $\pi_2$ from left to right).  More generally, every non-crossing partition can be expressed uniquely as a concatenation of some number of irreducible partitions.
\end{remark}

\subsection{Partitions as Composition Diagrams} \label{subsec:partitioncomposition}

\begin{definition}
Let $\pi \in \mathcal{NC}(S)$ and let $V$ be a block of $\pi$.  Then we denote by $\pi \setminus V$ the partition of $S \setminus V$ given by deleting $V$ from $\pi$.  We say that a block $V$ is an \emph{interval} if it has the form $\{i: j < i \leq k\}$ for some $j < k$ in $S$.
\end{definition}

\begin{definition}
Let $\mathcal{A}$ and $\mathcal{A}'$ be $\mathcal{B}$-$\mathcal{B}$-correspondences.  A multlinear form $\Lambda: \mathcal{A}^k \to \mathcal{A}'$ will be called an \emph{$\mathcal{B}$-quasi-multlinear} if we have for $a_1$, \dots, $a_k \in \mathcal{A}$ and $b \in \mathcal{B}$ that
\begin{align*}
\Lambda[ba_1,a_2,\dots,a_k] &= b \Lambda[a_1,\dots,a_k] \\
\Lambda[a_1,\dots,a_{k-1},a_kb] &= \Lambda[a_1,\dots,a_{k-1},a_k] b \\
\Lambda[a_1,\dots,a_jb,a_{j+1},\dots,a_k] &= \Lambda[a_1,\dots,a_j, ba_{j+1},\dots, a_k].
\end{align*}
\end{definition}

\begin{definition} \label{def:picomposition}
Let $\mathcal{A}$ be a $\mathcal{B}$-$\mathcal{B}$-correspondence.  Let $\Lambda_\ell: \mathcal{A}^\ell \to \mathcal{B}$ be a sequence of $\mathcal{B}$-quasi-multilinear forms.  For $\pi \in \mathcal{NC}(S)$, we define $\Lambda_\pi: \mathcal{A}^{|S|} \to \mathcal{B}$ by the following recursive relation.  Suppose that $V$ is an interval block of $\pi$ and thus $V$ can be written as $\phi^{-1}(\{j+1,\dots,k\})$ where $\phi: S \to [|S|]$ is the unique order-preserving bijection and $j < k$ in $[|S|]$.  Then
\[
\Lambda_\pi[a_1,\dots,a_\ell] = \begin{cases} \Lambda_{\pi \setminus V}[a_1,\dots,a_i, \Lambda_{k-j}[a_{j+1},\dots,a_k]a_{k+1},\dots,a_\ell], & k < \ell \\
\Lambda_{\pi \setminus V}[a_1,\dots,a_i] \Lambda_{\ell-i}[a_{j+1},\dots,a_\ell], & k = \ell. \end{cases}
\]
\end{definition}

To show that this is well-defined, first note that every partition must have some interval block because a maximal block with respect to $\prec$ must be an interval.  Moreover, by the associativity properties of composition and the fact that $\Lambda_\ell$ is $\mathcal{B}$-quasi-multilinear, the resulting multilinear form $\Lambda_\pi$ is independent of the sequence of recursive steps taken to evaluate it.  Moreover, it is straighforward to check by induction that $\Lambda_\pi$ is $\mathcal{B}$-quasi-multilinear.

If $|\pi| = 1$, then $\Lambda_\pi = \Lambda_{|S|}$.  Moreover, if $S$ and $S'$ are isomorphic as totally ordered sets, and $\pi \in \mathcal{NC}(S)$ and $\pi' \in \mathcal{NC}(S')$ correspond under this isomorphism, then $\Lambda_\pi = \Lambda_{\pi'}$.  Thus, it would be sufficient to define $\Lambda_\pi$ only for $\pi \in \mathcal{NC}(\ell)$.

The following fact about M{\"o}bius inversion is well-known in this context.

\begin{lemma} \label{lem:partitionMobiusinversion}
Let $\mathcal{A}$ be an algebra containing $\mathcal{B}$.  Let $\Gamma_\ell: \mathcal{A}^n \to \mathcal{B}$ be a $\mathcal{A}$-quasi-multilinear form.  For each non-crossing partition $\pi$, let $\alpha_\pi \in \C$, and assume that $\alpha_\pi \neq 0$ when $\pi$ consists of a single block.  Then there exist unique $\mathcal{B}$-quasi-multilinear forms $\Lambda_\ell: \mathcal{A}^n \to \mathcal{B}$ such that
\[
\Gamma_\ell[a_1,\dots,a_\ell] = \sum_{\pi \in \mathcal{NC}(\ell)} \alpha_\pi \Lambda_\pi[a_1,\dots,a_\ell].
\]
\end{lemma}

\subsection{The Boolean Cumulants}

\begin{definition}
A partition $\pi \in \mathcal{P}(S)$ is an \emph{interval partition} if every block $V$ has the form $V = \{i: j \leq i \leq k\}$ for some $j \leq k$.  We denote the set of interval partitions by $\mathcal{I}(S)$.  Note that every interval partition is non-crossing.
\end{definition}

\begin{definition} \label{def:Booleancumulants}
Let $(\mathcal{A},E)$ be a $\mathcal{B}$-valued probability space.  We define the Boolean cumulants $K_{\Bool,\ell}: \mathcal{A}^\ell \to \mathcal{B}$ implicitly by the relation
\[
E[a_1 \dots a_\ell] = \sum_{\pi \in \mathcal{I}(\ell)} K_{\Bool,\pi}[a_1,\dots,a_\ell],
\]
which makes sense by Lemma \ref{lem:partitionMobiusinversion}.
\end{definition}

The following lemma was proved in the scalar-valued case for one variable in \cite[Prop.\ 5.1]{Lenczewski2007}, and the general case is no harder.

\begin{lemma} \label{lem:booleancumulants}
Let $\mathcal{H}$ be a $\mathcal{B}$-$\mathcal{B}$-correspondence with a $\mathcal{B}$-central unit vector $\xi$, and let $E_\xi[a] = \ip{\xi,a\xi}$ for $a \in \mathcal{L}(\mathcal{H})$.  Let $P$ be the projection onto $\mathcal{B}\xi$ and let $Q$ be the projection onto its orthogonal complement.  For $a_1$, \dots, $a_\ell \in \mathcal{L}(\mathcal{H})$, we have
\[
P a_1 Q a_2 Q a_3 \dots Q a_\ell P = K_{\Bool,\ell}[a_1,\dots,a_\ell] P.
\]
\end{lemma}

\begin{proof}
Define
\[
\Lambda_n[a_1,\dots,a_\ell] = \ip{\xi, a_1 Q a_2 \dots Q a_\ell \xi},
\]
and note that
\[
P a_1 Q a_2 Q a_3 \dots Q a_\ell P = \Lambda_n[a_1,\dots,a_\ell] P.
\]
To show that $ \Lambda_\ell = K_{\Bool,\ell}$, it suffices to show that
\[
E[a_1 \dots a_\ell] = \sum_{\pi \in \mathcal{I}(\ell)} \Lambda_\pi[a_1,\dots,a_\ell],
\]
where $\Lambda_\pi$ is given by Definition \ref{def:picomposition}.  Observe that
\begin{align*}
E[a_1 \dots a_\ell] &= E[(P+Q)a_1(P+Q)a_2 \dots (P+Q)a_\ell(P+Q)] \\
&= E[Pa_1(P+Q)a_2 \dots (P+Q)a_\ell P].
\end{align*}
We expand this expression by multilinearity of multiplication.  Each of the resulting terms corresponds to a string of length $\ell - 1$ in the letters $P$ and $Q$.  We can define a correspondence between these strings and the interval partitions of $\ell$ given by placing the letter $P$ between $a_j$ and $a_{j+1}$ if they are in different blocks and the letter $Q$ between $a_j$ and $a_{j+1}$ if they are in the same block.  Then the expectation of the string corresponding to a partition $\pi$ is exactly $\Lambda_\pi[a_1,\dots,a_n]$.
\end{proof}

\subsection{Combinatorial Formula for the $\mathcal{T}$-free Product}

\begin{definition}
For a partition $\pi$, we define $\graph(\pi)$ to be the (simple undirected) graph with vertex set $\pi \sqcup \{\emptyset\}$ and with edges given by
\begin{itemize}
	\item $\emptyset \sim V$ for every block $V$ that is minimal with respect to $\prec$.
	\item $V \sim W$ whenever $V \prec W$ and there is no block $U$ strictly between $V$ and $W$.
\end{itemize}
We view $\graph(\pi)$ as a rooted graph with $\emptyset$ as the root vertex.
\end{definition}

\begin{remark}
As a consequence of the fact that $\pi$ is non-crossing, every block $W$ has a unique immediate predecessor with respect to $\prec$, and therefore $\graph(\pi)$ is a tree, also known as the \emph{nesting tree of $\pi$}.  An example of $\graph(\pi)$ is shown in Figure \ref{fig:partitiongraph}.
\end{remark}

\begin{definition}
For a block $V$ of $\pi$, let us denote
\[
\chain(V) = (V,V_2,\dots,V_d),
\]
where $V \succ V_2 \succ \dots \succ V_d$ are all the blocks surrounding $V$.  We also define the \emph{depth of $V$ in} $\pi$, denoted $\depth_\pi(V)$ as the number $d$.  Equivalently, $\chain(V)$ is the unique path from $V$ to the root vertex in $\graph(\pi)$, and $\depth_\pi(V)$ is the distance of $V$ from the root vertex.
\end{definition}

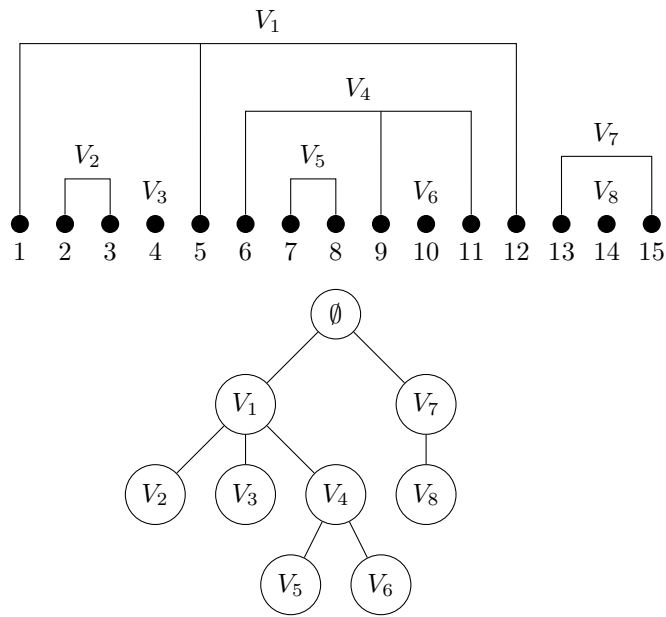
\begin{figure}

\begin{center}

\begin{tikzpicture}[scale = 0.6]

\begin{scope}

	\foreach \j in {1,...,15} {
		\node[label=below:{\j}] ({\j}) at (\j,0) {};
		\fill (\j,0) circle (0.2);
	}

	\draw (1) to (1,4) to (12,4) to (12);
	\draw (5) to (5,4);
	\node at (6.5,4.5) {$V_1$};

	\draw (2) to (2,1) to (3,1) to (3);
	\node at (2.5,1.5) {$V_2$};

	\node at (4,0.7) {$V_3$};

	\draw (6) to (6,2.5) to (11,2.5) to (11);
	\draw (9) to (9,2.5);
	\node at (8.5,3) {$V_4$};

	\draw (7) to (7,1) to (8,1) to (8);
	\node at (7.5,1.5) {$V_5$};

	\node at (10,0.7) {$V_6$};

	\draw (13) to (13,1.5) to (15,1.5) to (15);
	\node at (14,2) {$V_7$};

	\node at (14,0.7) {$V_8$};

\end{scope}

\begin{scope}[shift = {(8,-2)}]
	\node[circle,draw] (V0) at (0,0) {$\emptyset$};
	\node[circle,draw] (V1) at (-2,-2) {$V_1$};
	\node[circle,draw] (V2) at (-4,-4) {$V_2$};
	\node[circle,draw] (V3) at (-2,-4) {$V_3$};
	\node[circle,draw] (V4) at (0,-4) {$V_4$};
	\node[circle,draw] (V5) at (-1,-6) {$V_5$};
	\node[circle,draw] (V6) at (1,-6) {$V_6$};
	\node[circle,draw] (V7) at (2,-2) {$V_7$};
	\node[circle,draw] (V8) at (2,-4) {$V_8$};
	
	\draw (V0) to (V1) to (V2);
	\draw (V1) to (V3);
	\draw (V1) to (V4) to (V5);
	\draw (V4) to (V6);
	\draw (V0) to (V7) to (V8);
\end{scope}

\end{tikzpicture}

\caption{An example of a non-crossing partition $\pi$ of $[15]$ into 8 blocks (above) together with $\graph(\pi)$ (below).  In this example, we have $\chain(V_5) = (V_5,V_4,V_1)$ and $\chain(V_4) = (V_4,V_1)$ and $\chain(V_7) = V_7$.}  \label{fig:partitiongraph}

\end{center}

\end{figure}

\begin{definition}
Let $S$ be a totally ordered finite set.  An \emph{$N$-coloring of $S$} is a function $\chi: S \to [N]$.  We say that a partition $\pi$ is \emph{compatible with $\chi$} if $\chi$ is constant on each block of $\pi$.  We denote the set of partitions compatible with $\chi$ by $\mathcal{NC}(\chi)$.  If $\pi \in \mathcal{NC}(\chi)$, then for each block $V \in \pi$, we denote its color by $\chi(V) \in [N]$.
\end{definition}

\begin{definition}
Let $\chi$ be an $N$-coloring of $S$, and let $\mathcal{T}$ be a rooted subtree of $\mathcal{T}_{N,\free}$.  Suppose that $\pi \in \mathcal{NC}(\chi)$ and $V \in \pi$.  If $\chain(V) = (V,V_1,\dots,V_d)$, then we define $\chi(\chain(V))$ to be the string
\[
\chi(\chain(V)) = \chi(V) \chi(V_1) \dots \chi(V_d).
\]
We say that $\pi \in \mathcal{NC}(\chi)$ is \emph{compatible with $\mathcal{T}$} if $\chi(\chain(V)) \in \mathcal{T}$ for every $V \in \pi$.  We denote the set of such partitions by $\mathcal{NC}(\chi,\mathcal{T})$.
\end{definition}

\begin{remark} \label{rem:graphhomo}
If $\pi \in \mathcal{NC}([\ell])$ and $\chi: [\ell] \to [N]$ are compatible, then there is a unique rooted graph homomorphism $\phi_{\chi,\pi}: \graph(\pi) \to \mathcal{T}_{N,\free}$ such that $\phi(V)$ begins with $\chi(V)$ for every $V \in \pi$ and the length of $\phi(V)$ as a string is equal to $\depth_\pi(V)$.  Indeed, this homomorphism is given by $\phi_{\chi,\pi}(V) = \chi(V) \chi(V_1) \dots \chi(V_d)$, where $\chain(V) = (V,V_1,\dots,V_d)$.  The condition that $\pi \in \mathcal{NC}(\chi,\mathcal{T})$ is equivalent to saying that $\phi_{\chi,\pi}(\graph(\pi)) \subseteq \mathcal{T}$.
\end{remark}

\begin{theorem} \label{thm:combinatorics}
Let $\mathcal{T}$ be a rooted subtree of $\mathcal{T}_{N,\free}$.  Suppose that
\[
(\mathcal{H},\xi) = \assemb_{\mathcal{T}}[(\mathcal{H}_1,\xi_1), \dots, (\mathcal{H}_N,\xi_N)],
\]
and let $\lambda_j = \lambda_{\mathcal{T},j}: \mathcal{L}(\mathcal{H}_j) \to \mathcal{L}(\mathcal{H})$ be the corresponding $*$-homomorphism.  Let $\chi$ be an $N$-coloring of $[\ell]$.  Let $a_j \in \mathcal{L}(\mathcal{H}_{\chi(j)})$ for $j = 1$, \dots, $\ell$.  Then
\begin{equation} \label{eq:maincombinatorialformula}
\ip{\xi, \lambda_{\chi(1)}(a_1) \dots \lambda_{\chi(\ell)}(a_\ell) \xi} = \sum_{\pi \in \mathcal{NC}(\chi,\mathcal{T})} \Lambda_{\chi,\pi}(a_1,\dots,a_\ell),
\end{equation}
where the maps $\Lambda_{\chi,\pi}: \mathcal{L}(\mathcal{H}_{\chi(1)}) \times \dots \times \mathcal{L}(\mathcal{H}_{\chi(\ell)}) \to \mathcal{B}$ are defined recursively by the following conditions.

Suppose that $S$ is a totally ordered set, $\chi: S \to [N]$ is a coloring, $\phi: S \to [\ell]$ is an order-preserving isomorphism, and $V = \phi^{-1}(\{j+1,\dots,k\})$ is a block of $\pi$.  Then $\Lambda_{\chi,\pi}: \prod_{j \in S} \mathcal{L}(\mathcal{H}_{\chi(j)}) \to \mathcal{B}$ is given by
\[
\Lambda_{\chi,\pi}[a_1,\dots,a_\ell] = \Lambda_{\chi|_{[\ell] \setminus V}, \pi \setminus V}[a_1,\dots,a_j, K_{\Bool,k-j}[a_{j+1},\dots,a_k]a_{k+1},\dots,a_\ell],
\]
if $k < \ell$ and
\[
\Lambda_{\chi,\pi}[a_1,\dots,a_\ell] = \Lambda_{\chi|_{[\ell] \setminus V}, \pi \setminus V}[a_1,\dots,a_j] K_{\Bool,k-j}[a_{j+1},\dots,a_k],
\]
if $k = \ell$, where $K_{\Bool,k-j}: \mathcal{L}(\mathcal{H}_{\chi(V)})^{k-j} \to \mathcal{B}$ is the Boolean cumulant given by Definition \ref{def:Booleancumulants}.
\end{theorem}

\begin{remark}
The definition of $\Lambda_\pi$ here requires a slight modification of Definition \ref{def:picomposition} because the maps $K_{\Bool,j-i}$ on $\mathcal{L}(\mathcal{H}_{\chi(V)})$ do not all have the same domain.  In the case where the singleton string $\chi(V)$ is not in $\mathcal{T}$, it is important to use the Boolean cumulant $K_{\Bool,j-i}$ of $\mathcal{L}(\mathcal{H}_{\chi(V)})$ rather than that of $\mathcal{L}(\mathcal{H})$ because the map $\lambda_{\mathcal{T},j}: \mathcal{L}(\mathcal{H}_j) \to \mathcal{L}(\mathcal{H})$ will not be expectation-preserving.  For example, this occurs for orthogonal independence for the index $j = 2$.

On the other hand, if we assume that $\mathcal{T} \in \Tree'(N)$, then the maps $\lambda_{\mathcal{T},j}$ are expectation-preserving, and hence it makes no difference whether we compute the Boolean cumulants in $\mathcal{L}(\mathcal{H}_j)$ or $\mathcal{L}(\mathcal{H})$.  Thus, we can express \eqref{eq:maincombinatorialformula} in the simpler form
\begin{equation} \label{eq:maincombinatorialformula2}
\ip{\xi, \lambda_{\mathcal{T},\chi(1)}(a_1) \dots \lambda_{\mathcal{T},\chi(\ell)}(a_\ell) \xi} = \sum_{\pi \in \mathcal{NC}(\chi,\mathcal{T})} K_{\Bool,\pi}[\lambda_{\mathcal{T},\chi(1)}(a_1), \dots , \lambda_{\mathcal{T},\chi(\ell)}(a_\ell)],
\end{equation}
where $K_\ell$ denotes the Boolean cumulant of $\mathcal{L}(\mathcal{H})$ with respect to $\xi$.
\end{remark}

To outline the proof of Theorem \ref{thm:combinatorics}, let $P_i$ and $Q_i$ in $\mathcal{L}(\mathcal{H}_i)$ be the projections onto $\mathcal{B}\xi_i$ and its orthogonal complement respectively.  Let us write
\[
a_j = a_j^{(0,0)} + a_j^{(0,1)} + a_j^{(1,0)} + a_j^{(1,1)},
\]
where
\begin{align*}
a_j^{(0,0)} &= P_{\chi(j)}a_jP_{\chi(j)} \\
a_j^{(0,1)} &= P_{\chi(j)}a_jQ_{\chi(j)} \\
a_j^{(1,0)} &= Q_{\chi(j)}a_jP_{\chi(j)} \\
a_j^{(1,1)} &= Q_{\chi(j)}a_jQ_{\chi(j)}.
\end{align*}
Then we have by multilinearity that
\begin{equation} \label{eq:multilinearexpansion}
\ip{\xi, \lambda_{\chi(1)}(a_1) \dots \lambda_{\chi(\ell)}(a_\ell) \xi} = \sum_{(\delta_1,\epsilon_1), \dots, (\delta_\ell,\epsilon_\ell)} \ip{\xi, \lambda_{\chi(1)}(a_1^{(\delta_1,\epsilon_1)}) \dots \lambda_{\chi(\ell)}(a_\ell^{(\delta_\ell,\epsilon_\ell)}) \xi}.
\end{equation}
Each operator $\lambda_{\chi(j)}(a_j^{(\delta_j,\epsilon_j)})$ maps each the direct summands $\mathcal{H}_s^\circ$ of the product space either to zero or to another one of the direct summands $\mathcal{H}_{s'}^\circ$ such that $s'$ is equal to $s$ or adjacent to $s$ in $\mathcal{T}$.  Each term on the right hand side of \eqref{eq:multilinearexpansion} will either vanish or correspond to a ``path'' in the tree $\mathcal{T}$, where the notion of path is expanded to allow consecutive repetitions of the same vertex (see Lemma \ref{lem:compatible} for precise statement).  We will show in Lemma \ref{lem:bijection} that such paths are in bijection with $\mathcal{NC}(\chi,\mathcal{T})$ using a generalization of the well-known bijection between non-crossing pair partitions and Dyck paths.  Finally, in Lemma \ref{lem:evaluation}, we will evaluate the term corresponding to each path as the term $\Lambda_{\chi,\pi}$ in the Theorem.

To make these ideas precise, we first introduce some notation.  Let $\mathcal{T}'$ be the graph obtained by adding a self-loop to each vertex of $\mathcal{T}$.  Let us define four sets of oriented edges by
\begin{align*}
\mathcal{E}_i^{(0,0)} &= \{(s,s): is \in \mathcal{T}\} \\
\mathcal{E}_i^{(1,1)} &= \{(s,s): s \in \mathcal{T}, s(1) = i\} \\
\mathcal{E}_i^{(0,1)} &= \{(s,is): is \in \mathcal{T} \} \\
\mathcal{E}_i^{(1,0)} &= \{(is,s): is \in \mathcal{T} \}
\end{align*}
Note that for $\delta, \epsilon \in \{0,1\}$, the operator $\lambda_{\chi(j)}(a_j^{(\delta,\epsilon)})$ maps $\mathcal{H}_t^\circ$ into $\mathcal{H}_s^\circ$ if $(s,t) \in \mathcal{E}_{\chi(j)}^{(\delta,\epsilon)}$ and it vanishes on $\mathcal{H}_s^\circ$ if it is not the source of some edge in $\mathcal{E}_{\chi(j)}^{(\delta,\epsilon)}$.

Let us say that a sequence $(\delta_1,\epsilon_1)$, \dots, $(\delta_\ell,\epsilon_\ell)$ and a path $\emptyset = s_0$, $s_1$, \dots, $s_\ell = \emptyset$ in $\mathcal{T}'$ are \emph{compatible} (with respect to $\chi$) if $(s_{j-1},s_j) \in \mathcal{E}_{\chi(j)}^{(\delta_j,\epsilon_j)}$.  Note that in this case, $(\delta_j,\epsilon_j)$ is uniquely determined by $(s_{j-1},s_j)$ (and $\chi$).  Conversely, if $(\delta_1,\epsilon_1)$, \dots, $(\delta_\ell, \epsilon_\ell)$ has a compatible path $s_0$, \dots, $s_\ell$, then $s_j$ can be determined inductively by $s_{j-1}$ and $(\delta_j,\epsilon_j)$ and $\chi(j)$.

\begin{lemma} \label{lem:compatible}
If $(\delta_1,\epsilon_1)$, \dots, $(\delta_\ell, \epsilon_\ell)$ does not have a compatible path in $\mathcal{T}'$ from $\emptyset$ to $\emptyset$, then
\[
\ip{\xi, \lambda_{\chi(1)}(a_1^{(\delta_1,\epsilon_1)}) \dots \lambda_{\chi(\ell)}(a_\ell^{(\delta_\ell,\epsilon_\ell)}) \xi} = 0.
\]
Therefore,
\[
\ip{\xi, \lambda_{\chi(1)}(a_1) \dots \lambda_{\chi(\ell)}(a_\ell) \xi} = \sum_{\substack{(\delta_1,\epsilon_1), \dots, (\delta_\ell,\epsilon_\ell) \\ \text{ with a compatible path}}} \ip{\xi, \lambda_{\chi(1)}(a_1^{(\delta_1,\epsilon_1)}) \dots \lambda_{\chi(\ell)}(a_\ell^{(\delta_\ell,\epsilon_\ell)}) \xi}.
\]
\end{lemma}

\begin{proof}
The behavior of the operators $\lambda_{\chi(j)}(a_j^{(\delta_j,\epsilon_j)})$ can be described as follows:
\begin{itemize}
	\item $\lambda_{\chi(j)}(a_j^{(0,1)})$ maps $\mathcal{H}_s^\circ$ into $\mathcal{H}_{\chi(j) s}^\circ$ provided that $\chi(j) s \in \mathcal{T}$ and otherwise it vanishes on $\mathcal{H}_s^\circ$.
	\item $\lambda_{\chi(j)}(a_j^{(1,0)})$ maps $\mathcal{H}_{s'}^\circ$ into $\mathcal{H}_s^\circ$ provided that $s' = \chi(j) s \in \mathcal{T}$ and otherwise it vanishes on $\mathcal{H}_{s'}^\circ$.
	\item $\lambda_{\chi(j)}(a_j^{(0,0)})$ maps $\mathcal{H}_s^\circ$ into itself provided that $\chi(j) s \in \mathcal{T}$ and otherwise it vanishes on $\mathcal{H}_s$.
	\item $\lambda_{\chi(j)}(a_j^{(1,1)})$ maps $\mathcal{H}_s^\circ$ into itself provided that $s \in \mathcal{T}$ begins with $\chi(j)$ and otherwise it vanishes on $\mathcal{H}_s^\circ$.
\end{itemize}
One can argue by backward induction that for $j = \ell, \ell - 1, \dots, 1$, we have
\[
\lambda_{\chi(j)}(a_j^{(\delta_j,\epsilon_j)}) \dots \lambda(a_j^{(\delta_\ell,\epsilon_\ell)}) \xi = 0
\]
unless there is a compatible path $s_{j-1}$, $s_{j+1}$, \dots, $s_\ell = \emptyset$ in $\mathcal{T}'$.  If there is a compatible path, then this vector is in $\mathcal{H}_{s_{j-1}}^\circ$.  If there is a compatible path $s_0$, \dots, $s_\ell = \varnothing$ at the end of the induction, then we either have $s_0 = \emptyset$ or else
\[
\ip{\xi, \lambda_{\chi(j)}(a_j^{(\delta_j,\epsilon_j)}) \dots \lambda(a_j^{(\delta_\ell,\epsilon_\ell)}) \xi} = 0.
\]
\end{proof}

\begin{lemma} \label{lem:bijection}
Let us say that a path $\emptyset = s_0$, $s_1$, \dots, $s_\ell = \emptyset$ in $\mathcal{T}'$ is \emph{admissible} if there exists a compatible sequence of indices $(\delta_1,\epsilon_1)$, \dots, $(\delta_\ell,\epsilon_\ell)$.  Then there is a bijection between the set $P(\chi,\mathcal{T})$ of admissible paths and $\mathcal{NC}(\chi,\mathcal{T})$ (described explicitly in the proof).
\end{lemma}

\begin{proof}
Define a map $f: P(\chi,\mathcal{T}) \to \mathcal{NC}(\chi,\mathcal{T})$ as follows.  Fix a path $s_0$, \dots, $s_\ell$.  Let $\mathcal{J}$ be the set of indices where $(\delta_j,\epsilon_j) = (0,0)$.  For each $j \in \mathcal{J}$, we define a singleton block $V_j = \{j\}$.  Let $\mathcal{J}'$ be the set of indices such that the length of $s_j$ is greater than the length of $s_{j-1}$ (and hence $s_j = \chi(j) s_{j-1}$).  For each $j \in \mathcal{J}$, let
\[
j' = \min \{i > j: s_{i-1} = s_j, s_i = s_{j-1}\}.
\]
Note that $\chi(j')$ must equal $\chi(j)$ and this must be the first letter of $s_j$.  For each $j \in \mathcal{J}$, we define the block
\[
V_j = \{j,j'\} \cup \{k: j < k < j', s_{k-1} = s_k = s_j, \chi(k) = \chi(j) \}.
\]
Then $f((s_0,\dots,s_\ell))$ is defined to be the partition $\pi = \{V_j\}_{j \in \mathcal{J} \cup \mathcal{J}'}$.

To show that the partition $\pi$ is non-crossing, suppose that $a, b \in V_j$ and $c, d \in V_k$ with $a < c < b < d$.  In between $a$ and $b$, the path cannot reach the vertex $s_{j-1}$, and in particular, $s_k$ must have the form $t s_j$ for some nonempty string $t$.  Now $s_c$ or $s_{c-1}$ (the longer of the two values) must equal $s_k$.  Yet $s_b$ or $s_{b-1}$ must be equal to $s_j$ and in particular the path must go from the vertex $s_k$ to the vertex $s_{k-1}$ before the time index $b$.  But this implies that $k' \leq b$, and hence $d \leq k' \leq b$, which contradicts the assumption of crossing.

Thus, $\pi$ is non-crossing.  Moreover, by construction $\chi$ is constant on each block of $\pi$, so that $\pi \in \mathcal{NC}(\chi)$.  To check that $\pi \in \mathcal{NC}(\chi,\mathcal{T})$, choose a block $V_j$.  Denoting $\chain(V_j) = (V_j, W_1,\dots,W_d)$, we have
\[
\chi(V_j) \chi(W_1) \dots \chi(W_d) = \begin{cases} s_j, & |V_j| > 1 \\ \chi(V_j) s_j, & |V_j| = 1, \end{cases}
\]
which follows by a straightforward induction argument on $V_j$ with respect to the ordering $\prec$.  In the first case $|V_j| > 1$, it is clear that $\chi(V_j) \chi(W_1) \dots \chi(W_d) \in \mathcal{T}$, and in the second case $|V_j| = 1$, we observe that because $(\delta_j,\epsilon_j) = (0,0)$ we must have $(s_j,s_j) \in \mathcal{E}_{\chi(j)}^{(0,0)}$ and hence $\chi(j) s_j \in \mathcal{T}$.  Thus, $f$ defines a map $P(\chi,\mathcal{T}) \to \mathcal{NC}(\chi,\mathcal{T})$ as desired.

Conversely, we define a map $g: \mathcal{NC}(\chi,\mathcal{T}) \to P(\chi,\mathcal{T})$ as follows. Given a partition $\pi$ and index $k$, let $V$ be the block of $\pi$ containing $k$ and let $\chain(V) = (V,V_1,\dots,V_m)$, and define
\[
s_k = \begin{cases} \chi(V) \chi(V_1) \dots \chi(V_m), & k < \max(V), \\ \chi(V_1) \dots \chi(V_m), & k = \max(V). \end{cases}
\]
The reader may check that $g$ maps into $P(\chi,\mathcal{T})$, where the corresponding indices $(\delta_k,\epsilon_k)$ are given by
\[
(\delta_j,\epsilon_j) =
\begin{cases}
(0,0), & V = \{j\}, \\
(0,1), & j = \min(V) < \max(V), \\
(1,0), & j = \max(V) > \min(V) \\
(1,1), & \text{otherwise.}
\end{cases}
\]
Moreover, $g$ is the inverse function of $f$.
\end{proof}

\begin{lemma} \label{lem:evaluation}
Suppose that $s_0$, \dots, $s_\ell$ is an admissible path with respect to $\chi$, and let $(\delta_1,\epsilon_1)$, \dots, $(\delta_\ell,\epsilon_\ell)$ be the corresponding compatible sequence of indices.  Let $\pi$ by the corresponding partition in $\mathcal{NC}(\chi,\mathcal{T})$ under the bijection in the previous lemma.  Then
\[
\ip{\xi, \lambda_{\chi(1)}(a_1^{(\delta_1,\epsilon_1)}) \dots \lambda_{\chi(\ell)}(a_\ell^{(\delta_\ell,\epsilon_\ell)}) \xi} = \Lambda_{\chi,\pi}[a_1,\dots,a_\ell].
\]
\end{lemma}

\begin{proof}
We proceed by induction on $|\pi|$ (the number of blocks of $\pi$), where we allow $\ell$ to vary, and begin with the base case $\pi = \varnothing$.  Choose a partition $\pi$ with $|\pi| \geq 1$, and recall that $V$ must have an interval block $V = \{j+1,\dots,k\}$.  Let
\[
\zeta = \lambda_{\chi(k+1)}(a_{k+1}^{(\delta_{k+1},\epsilon_{k+1})}) \dots \lambda_{\chi(\ell)}(a_\ell^{(\delta_\ell,\epsilon_\ell)}) \xi \in \mathcal{H}_{s_k}^\circ.
\]
Observe that
\begin{align*}
\lambda_{\chi(j+1)}(a_{j+1}^{(\delta_{j+1},\epsilon_{j+1})}) \dots \lambda_{\chi(k)}(a_k^{(\delta_k,\epsilon_k)})
&= \lambda_{\chi(V)}(a_{j+1}^{(0,1)} a_{j+2}^{(1,1)} \dots a_{k-1}^{(1,1)} a_k^{(1,0)}) \\
&= \lambda_{\chi(V)}(K_{\Bool,k-j}[a_{j+1},\dots,a_k] P_{\chi(V)}),
\end{align*}
where the last equality follows from Lemma \ref{lem:booleancumulants}, and the inside expression is to be interpreted as $a_k^{(0,0)}$ in the case when $j+1=k$ and $|V| = 1$.  By our assumptions on $\pi$, we know that $s_k$ does not begin with $\chi(V)$ and that $\chi(V) s_k \in \mathcal{T}$, and hence
\[
\lambda_{\chi(V)}(K_{\Bool,k-j}[a_{j+1},\dots,a_k] P_{\chi(V)}) \zeta = K_{\Bool,k-j}[a_{j+1},\dots,a_k] \zeta.
\]
Therefore, the expectation we want to compute becomes
\[
\ip{\xi, \lambda_{\chi(1)}(a_1^{(\delta_1,\epsilon_1)}) \dots \lambda_{\chi(j)}(a_j^{(\delta_j,\epsilon_j)}) K_{k-j}[a_{j+1},\dots,a_k] \lambda_{\chi(k+1)}(a_{k+1}^{(\delta_{k+1},\epsilon_{k+1})}) \dots \lambda_{\chi(\ell)}(a_\ell^{(\delta_\ell,\epsilon_\ell)}) \xi}.
\]
Applying the inductive hypothesis to $\pi \setminus V$, we obtain obtain $\Lambda_{\chi,\pi}[a_1,\dots,a_\ell]$ as desired.
\end{proof}

Theorem \ref{thm:combinatorics} follows from Lemmas \ref{lem:compatible}, \ref{lem:bijection}, and \ref{lem:evaluation}.

\begin{corollary} \label{cor:combinatorics2}
With the setup of Theorem \ref{thm:combinatorics}, the Boolean cumulants in $\mathcal{L}(\mathcal{H})$ with respect to $\xi$ are given by
\[
K_{\Bool,\ell}[\lambda_{\chi(1)}(a_1), \dots, \lambda_{\chi(\ell)}(a_\ell)]= \sum_{\pi \in \mathcal{NC}^\circ(\chi,\mathcal{T})} \Lambda_{\chi,\pi}(a_1,\dots,a_\ell),
\]
where $\mathcal{NC}^\circ(\chi,\mathcal{T})$ denotes the set of irreducible partitions.
\end{corollary}

\begin{proof}
Let $P$ be the projection onto $\xi$ in $\mathcal{H}$, and let $Q = 1 - P$.  Then the Boolean cumulant we want to compute is
\[
\ip{\xi, \lambda_{\chi(1)}(a_1) Q \dots Q \lambda_{\chi(\ell)}(a_\ell) \xi}.
\]
We proceed simiarly as in the proof Theorem \ref{thm:combinatorics}, except that if some $s_j = \emptyset$ for $0 < j < \ell$, then the corresponding term is eliminated by the projection $Q$.  We thus obtain a sum over all paths in $\mathcal{T}'$ that are admissible with respect to $\chi$ and do not visit the root vertex between the start and end times.  These paths correspond, under the bijection $f$ constructed above, to partitions in which $1 \sim_\pi \ell$, or in other words irreducible partitions.
\end{proof}

\subsection{$\mathcal{T}$-free Independence and Convolution}

Theorem \ref{thm:combinatorics} implies several well-definedness properties of the moments of random variables in the $\mathcal{T}$-free product.  For instance, the joint moments of variables in $\mathcal{T}$-free product spaces are independent of the particular representations of the random variables.

\begin{corollary}
Let $\mathcal{T} \in \Tree(N)$.  Let $(\mathcal{A},E) = \assemb_\mathcal{T}[(\mathcal{A}_1,E_1), \dots, (\mathcal{A}_N,E_N)]$ and let $\lambda_{\mathcal{T},j}: \mathcal{A}_j \to \mathcal{A}$ be the $\mathcal{T}$-free product inclusion.  Then $E[\lambda_{\mathcal{T},i_1}(a_1) \dots \lambda_{\mathcal{T},i_\ell}(a_\ell)]$ is uniquely determined by $\mathcal{T}$ and the joint moments of $\{a_k: i_k = j\}$ in $(\mathcal{A}_j,E_j)$, independently of the specific choice of algebras $(\mathcal{A}_j,E_j)$.
\end{corollary}

\begin{corollary} \label{cor:convolvedlaw}
With the set up of the previous corollary, if $X_j \in \mathcal{A}_j$ is a bounded self-adjoint operator with law $\mu_j$ for $j = 1$, \dots, $N$, then the law of $\sum_{j=1}^N \lambda_{\mathcal{T},j}(X_j)$ is $\boxplus_{\mathcal{T}}(\mu_1,\dots,\mu_N)$.
\end{corollary}

If we assume that $\mathcal{T} \in \Tree'(N)$, then the maps $\lambda_{\mathcal{T},j}$ are expectation-preserving and hence injective (by the faithfulness assumption in our definition of non-commutative probability space).  In this case, it makes sense to define $\mathcal{T}$-free independence.

\begin{definition}
Let $\mathcal{A}_1$, \dots, $\mathcal{A}_N$ be $\mathcal{B}$-$*$-subalgebras of $(\mathcal{A},E)$, not necessarily closed.  We assume that each $\mathcal{A}_j$ has an internal unit, but the inclusion $\mathcal{A}_j \to \mathcal{A}$ is not necessarily unital.  We say that $\mathcal{A}_1$, \dots, $\mathcal{A}_N$ are \emph{$\mathcal{T}$-free independent over $\mathcal{B}$} if whenever $\chi: [\ell] \to [N]$ and $a_j \in \mathcal{A}_{\chi(j)}$ for $j = 1$, \dots, $N$, we have
\begin{equation} \label{eq:maincombinatorialformula3}
E[a_1 \dots a_\ell] = \sum_{\pi \in \mathcal{NC}(\chi,\mathcal{T})} K_{\Bool,\pi}[a_1,\dots,a_\ell],
\end{equation}
where $K_{\Bool,j}$ denotes the $j$th Boolean cumulant and $K_\pi$ is the $\pi$-composition of the $K_j$'s.
\end{definition}

It follows from the foregoing arguments that if $(\mathcal{A},E)$ is the $\mathcal{T}$-free product of $(\mathcal{A}_1,E_1)$, \dots, $(\mathcal{A},E_N)$, then the algebras $\lambda_{\mathcal{T},j}(\mathcal{A}_j)$ are $\mathcal{T}$-freely independent in $(\mathcal{A},E)$.

Furthermore, $\mathcal{A}_1$, \dots, $\mathcal{A}_N$ are $\mathcal{T}$-free independent in $\mathcal{A}$ and $X_j \in \mathcal{A}_j$ is self-adjoint, then the law of $X_1 + \dots + X_j$ only depends on $E|_{\mathcal{B}\ip{X_j}}$ and hence it is the $\mathcal{T}$-free convolution of the laws of $X_1$, \dots, $X_N$.

Another consequence of Theorem \ref{thm:combinatorics} concerns what happens when we restrict to a subset of the indices $[N]$.

\begin{corollary} \label{cor:indexsubset}
Let $\mathcal{T} \in \Tree(N)$.  Let $N' \leq N$.  Let $\mathcal{T}'$ be the set of alternating strings on the alphabet $[N'] \subseteq [N]$ that are contained in $\mathcal{T}$.  Let $(\mathcal{H}_1,\xi_1)$, \dots, $(\mathcal{H}_N,\xi_N)$ be $\mathcal{B}$-$\mathcal{B}$-correspondences with $\mathcal{B}$-central unit vectors.  Let
\begin{align*}
(\mathcal{H},\xi) &= \assemb_{\mathcal{T}}[(\mathcal{H}_1,\xi_1),\dots,(\mathcal{H}_N,\xi_N)]
(\mathcal{H}',\xi') &= \assemb_{\mathcal{T}'}[(\mathcal{H}_1,\xi_1),\dots,(\mathcal{H}_{N'},\xi_{N'})].
\end{align*}
Let $\chi: [\ell] \to [N']$ and let $a_j \in \mathcal{L}(\mathcal{H}_{\chi(j)})$ for $j = 1$, \dots, $\ell$. Then
\[
\ip{\xi', \lambda_{\mathcal{T}',\chi(1)}(a_1) \dots \lambda_{\mathcal{T}',\chi(\ell)}(a_\ell) \xi'} = \ip{\xi, \lambda_{\mathcal{T},\chi(1)}(a_1) \dots \lambda_{\mathcal{T},\chi(\ell)}(a_\ell) \xi}.
\]
\end{corollary}

\begin{proof}
In light of Theorem \ref{thm:combinatorics}, it suffices to show that if $\chi: [\ell] \to [N']$, then $\mathcal{NC}(\chi,\mathcal{T}') = \mathcal{NC}(\chi,\mathcal{T})$.  This is immediate, because if $\pi$ is a partition compatible with $\chi$ and if $V \in \pi$, then $\chi(\chain(V))$ is a string on the alphabet $[N']$, and hence $\chi(\chain(V))$ is in $\mathcal{T}$ if and only if it is in $\mathcal{T}'$.
\end{proof}

Corollary \ref{cor:indexsubset} easily implies the following statements.  Suppose that $\mathcal{T} \in \Tree'(N)$ and $n \leq N$ and $\mathcal{T}'$ is as in Corollary \ref{cor:indexsubset}.  If $\mathcal{A}_1$, \dots, $\mathcal{A}_N$ are $\mathcal{T}$-freely independent, then $\mathcal{A}_1$, \dots, $\mathcal{A}_{N'}$ are $\mathcal{T}'$-freely independent.  Moreover, if $\mathcal{T}$, $N'$, and $\mathcal{T}'$ are as in the Corollary \ref{cor:indexsubset}, the
\[
\boxplus_{\mathcal{T}'}(\mu_1,\dots,\mu_{N'}) = \boxplus_{\mathcal{T}}(\mu_1,\dots,\mu_{N'},\delta_0,\dots,\delta_0),
\]
where $\delta_0$ is the $\mathcal{B}$-valued law of the zero operator.

\subsection{The Free, Boolean, and Monotone Cases} \label{subsec:FBMcombinatorics}

Interrupting the general exposition, we now explain how Theorem \ref{thm:combinatorics} relates to the moment conditions used in previous literature to define free, Boolean, and monotone independence.

\begin{proposition}
Let $\mathcal{A}_1$, \dots, $\mathcal{A}_N$ be $\mathcal{B}$-$*$-subalgebras of $(\mathcal{A},E)$ with internal units.  The following are equivalent:
\begin{enumerate}[(1)]
	\item Let $(\tilde{\mathcal{A}},\tilde{E})$ be the free product of $(\mathcal{A}_1,E_1)$, \dots, $(A_N,E_N)$ and $\lambda_{\free,j}: \mathcal{A}_j \to \tilde{\mathcal{A}}$ the corresponding inclusions.  Then for every $\chi: [\ell] \to [N]$ and $a_j \in \mathcal{A}_{\chi(j)}$, we have
	\[
	E[a_1 \dots a_\ell] = \tilde{E}[\lambda_{\free,\chi(1)}(a_1) \dots \lambda_{\free,\chi(\ell)}(a_\ell)].
	\]
	\item Given $\chi: [\ell] \to [N]$ and given $a_j \in \mathcal{A}_{\chi(j)}$, we have
	\[
	E[a_1 \dots a_\ell] = \sum_{\pi \in \mathcal{NC}(\chi,\mathcal{T}_{N,\free})} K_{\Bool,\pi}[a_1, \dots, a_\ell].
	\]
	\item Given $\chi: [\ell] \to [N]$ such that $\chi$ is alternating (that is, $\chi(j+1) \neq \chi(j)$) and given $a_j \in \mathcal{A}_{\chi(j)}$ with $E[a_j] = 0$, we have
	\[
	E[a_1 \dots a_\ell] = 0.
	\]
	This is the formulation given by \cite[Definition 1.2]{Voiculescu1995} (see \cite{Voiculescu1985} for the scalar case).
\end{enumerate}
In (2) above, $\mathcal{NC}(\chi,\mathcal{T}_{N,\free})$ consists of partitions $\pi \in \mathcal{NC}(\chi)$ such that adjacent blocks in $\graph(\pi)$ have distinct colors.
\end{proposition}

\begin{proof}
(1) $\implies$ (2) follows from Theorem \ref{thm:combinatorics}.

To show (2) $\implies$ (3), suppose that $\chi: [\ell] \to [N]$ is alternating and that $a_j \in \mathcal{A}_{\chi(j)}$ with $E[a_j] = 0$.  Every partition $\pi \in \mathcal{NC}(\chi,\mathcal{T}_{N,\free})$ must have some interval block.  The coloring $\chi$ must be constant on this block.  But because $\chi$ is alternating, this forces the interval block to have size one.  Since $K_{\Bool,1}(a_j) = E[a_j] = 0$, this means that $K_{\Bool,\pi}[a_1,\dots,a_\ell] = 0$.  Hence, all the terms on the right hand side of (2) vanish, and thus (3) holds.

Finally, to prove that (3) $\implies$ (1), observe that condition (3) uniquely determines $E|_{\Alg(\mathcal{A}_1,\dots,\mathcal{A}_N)}$ by an inductive argument which can be found in \cite{Voiculescu1995}.   The algebras $\mathcal{A}_j$ in $(\mathcal{A},E)$ satisfy (3), and the algebras $\lambda_{\free,j}(\mathcal{A}_j)$ in $(\tilde{\mathcal{A}},\tilde{E})$ also satisfy (3) because (1) $\implies$ (3).  Therefore, the joint moments in $(\mathcal{A},E)$ and $(\tilde{A},\tilde{E})$ must agree.  Thus, (1) holds.

Finally, the claim concerning $\mathcal{NC}(\chi,\mathcal{T}_{N,\free})$ follows immediately from the definition.
\end{proof}

\begin{proposition}
Let $\mathcal{A}_1$, \dots, $\mathcal{A}_N$ be $\mathcal{B}$-$*$-subalgebras of $(\mathcal{A},E)$ with internal units.  The following are equivalent:
\begin{enumerate}[(1)]
	\item Let $(\tilde{\mathcal{A}},\tilde{E})$ be the Boolean product of $(\mathcal{A}_1,E_1)$, \dots, $(A_N,E_N)$ and $\lambda_{\Bool,j}: \mathcal{A}_j \to \tilde{\mathcal{A}}$ the corresponding inclusions.  Then for every $\chi: [\ell] \to [N]$ and $a_j \in \mathcal{A}_{\chi(j)}$, we have
	\[
	E[a_1 \dots a_\ell] = \tilde{E}[\lambda_{\Bool,\chi(1)}(a_1) \dots \lambda_{\Bool,\chi(\ell)}(a_\ell)].
	\]
	\item Given $\chi: [\ell] \to [N]$ and given $a_j \in \mathcal{A}_{\chi(j)}$, we have
	\[
	E[a_1 \dots a_\ell] = \sum_{\pi \in \mathcal{NC}(\chi,\mathcal{T}_{N,\Bool})} K_{\Bool,\pi}[a_1, \dots, a_\ell].
	\]
	\item Given $\chi: [\ell] \to [N]$ such that $\chi$ is alternating (that is, $\chi(j+1) \neq \chi(j)$) and given $a_j \in \mathcal{A}_{\chi(j)}$, we have
	\[
	E[a_1 \dots a_\ell] = E[a_1] \dots E[a_\ell].
	\]
	This is the formulation given by \cite[\S 4.1]{Popa2009} (see \cite{SW1997} for the scalar case).
\end{enumerate}
Moreover, in (2) above, $\mathcal{NC}(\chi,\mathcal{T}_{N,\Bool})$ consists of the interval partitions in $\mathcal{NC}(\chi)$.
\end{proposition}

\begin{proof}
(1) $\implies$ (2) follows from Theorem \ref{thm:combinatorics}.

To show (2) $\implies$ (3), suppose that $\chi: [\ell] \to [N]$ is alternating and that $a_j \in \mathcal{A}_{\chi(j)}$.  Because $\mathcal{T}_{N,\Bool}$ contains only the root vertex and its neighbors, a partition can only be in $\mathcal{NC}(\chi,\mathcal{T}_{N,\Bool})$ if all the blocks all have depth $1$, and hence $\mathcal{NC}(\chi,\mathcal{T}_{N,\Bool})$ consists of the interval partitions compatible with $\chi$.  But if $\chi$ is alternating, then there is only one interval partition, namely the partition where every block is a singleton.  Thus, in this case (2) reduces to the formula (3).

Finally, to prove that (3) $\implies$ (1), observe that condition (3) uniquely determines $E|_{\Alg(\mathcal{A}_1,\dots,\mathcal{A}_N)}$.  Hence, the proof works the same as in the free case.
\end{proof}

\begin{proposition} \label{prop:monotonemoments}
Let $\mathcal{A}_1$, \dots, $\mathcal{A}_N$ be $\mathcal{B}$-$*$-subalgebras of $(\mathcal{A},E)$ with internal units.  The following are equivalent:
\begin{enumerate}[(1)]
	\item Let $(\tilde{\mathcal{A}},\tilde{E})$ be the monotone product of $(\mathcal{A}_1,E_1)$, \dots, $(A_N,E_N)$ and $\lambda_{\mono,j}: \mathcal{A}_j \to \tilde{\mathcal{A}}$ the corresponding inclusions.  Then for every $\chi: [\ell] \to [N]$ and $a_j \in \mathcal{A}_{\chi(j)}$, we have
	\[
	E[a_1 \dots a_\ell] = \tilde{E}[\lambda_{\mono,\chi(1)}(a_1) \dots \lambda_{\mono,\chi(\ell)}(a_\ell)].
	\]
	\item Given $\chi: [\ell] \to [N]$ and given $a_j \in \mathcal{A}_{\chi(j)}$, we have
	\[
	E[a_1 \dots a_\ell] = \sum_{\pi \in \mathcal{NC}(\chi,\mathcal{T}_{N,\mono})} K_{\Bool,\pi}[a_1, \dots, a_\ell].
	\]
	\item Suppose $\chi: [\ell] \to [N]$ and that $a_j \in \mathcal{A}_{\chi(j)}$.  Suppose that $k \in [\ell]$ such that $\chi(k) > \chi(k - 1)$ (if $k \neq 1$) and $\chi(k) > \chi(k+1)$ (if $k < \ell$).  Then
	\[
	E[a_1 \dots a_\ell] = E[a_1 \dots a_{k-1} E[a_k] a_{k+1} \dots a_\ell].
	\]
	This is the formulation given by \cite[Def.\ 2.2]{HS2014} (see \cite[Def.\ 2.5]{HS2011a} for the scalar case).  Equivalent formulations were given earlier by \cite[Def.\ 1.1]{Muraki2000} \cite[eq.\ (2.4)]{Skeide2004}.
\end{enumerate}
Moreover, in (2) above, $\mathcal{NC}(\chi,\mathcal{T}_{N,\mono})$ consists of the partitions in $\mathcal{NC}(\chi)$ such that $V \prec W$ implies $\chi(V) < \chi(W)$ for every $V, W \in \pi$.
\end{proposition}

\begin{proof}
(1) $\implies$ (2) follows from Theorem \ref{thm:combinatorics}.

To show (2) $\implies$ (3), suppose that $\chi: [\ell] \to [N]$.  For $\pi$ to be in $\mathcal{NC}(\chi,\mathcal{T}_{N,\mono})$ means that if $V \in \pi$ and $\chain(V) = (V,V_1,\dots,V_d)$, then $\chi(V) \chi(V_1) \dots \chi(V_d) \in \mathcal{T}_{N,\mono}$.  By definition of $\mathcal{T}_{N,\mono}$, this means that $\chi(V) > \chi(V_1) > \dots > \chi(V_d)$.  It is straightforward to see that this is equivalent to the condition $V \prec W \implies \chi(V) < \chi(W)$.

Now suppose that $a_j \in \mathcal{A}_{\chi(j)}$ and that $k$ is as in (3).  Suppose that $\pi \in \mathcal{NC}(\chi,\mathcal{T}_{N,\mono})$ and $V$ is the block of $\pi$ containing $k$.  If $V$ contained any index $j < k$, then the block $W$ containing $k - 1$ would be satisfy $W \succ V$, and hence $\chi(k - 1) > \chi(k)$, which is a contradiction.  Thus, $V$ cannot contain any indices below $k$.  By a symmetrical argument, $V$ cannot contain any indices above $k$.  Therefore, $V = \{k\}$.  Hence, all of the partitions $\pi$ used on the right hand side of (2) isolate $k$ in a singleton block.  For such partitions,
\[
K_{\Bool,\pi}[a_1, \dots, a_\ell] = \begin{cases} K_{\Bool,\pi \setminus \{k\}}[a_1, \dots, a_{k-1}, E[a_k] a_{k+1}, \dots, a_\ell], & k < \ell \\ K_{\Bool,\pi \setminus \{k\}}[a_1,\dots,a_{\ell-1}] E[a_\ell], & k = \ell. \end{cases}
\]
Now $\pi \mapsto \pi \setminus \{k\}$ defines a bijection $\mathcal{NC}(\chi, \mathcal{T}_{N,\mono}) \to \mathcal{NC}(\chi|_{[\ell] \setminus \{k\}}, \mathcal{T}_{N,\mono})$.  Therefore, the right hand side in (2) becomes (in the case where $k < \ell$)
\[
\sum_{\pi \in \mathcal{NC}(\chi|_{[\ell] \setminus \{k\}}, \mathcal{T}_{N,\mono})} K_{\Bool,\pi}[a_1, \dots, a_{k-1}, E[a_k] a_{k+1}, \dots, a_\ell] = E[a_1 \dots a_{k-1} E[a_k] a_{k+1} \dots a_\ell],
\]
and the case where $k = \ell$ is handled similarly.  Thus, (3) holds.

As before, to prove that (3) $\implies$ (1), it suffices to show that $E|_{\Alg(\mathcal{A}_1,\dots,\mathcal{A}_N)}$ is uniquely determined by (3).  We prove this for strings $a_1 \dots a_\ell$ with $a_j \in \mathcal{A}_{\chi(j)}$ by induction on the length $\ell$.  The base case $\ell = 1$ is trivial.  Suppose $\ell > 1$, and let $k$ be an index such that $\chi(k)$ is maximal.  If $\chi(k) = \chi(k-1)$ or $\chi(k) = \chi(k+1)$, then we may group $a_{k-1} a_k$ or $a_k a_{k+1}$ into a single letter and thus view $a_1 \dots a_\ell$ as a string of length $\ell - 1$ and apply the inductive hypothesis.  Otherwise, $\chi(k)$ satisfies the assumptions of (3) and therefore using (3) we may replace $a_k$ by $E[a_k]$.  We may group $E[a_k]$ together with either $a_{k-1}$ or $a_{k+1}$ since $\mathcal{A}_{\chi(k-1)}$ and $\mathcal{A}_{\chi(k+1)}$ are $\mathcal{B}$-$\mathcal{B}$-bimodules.  We thus reduce to a string of length $\ell - 1$ to which the inductive hypothesis applies.
\end{proof}

\section{Operad Properties} \label{sec:operad}

\subsection{The Operad of Rooted Trees}

Recall that $\Tree(N)$ the set of rooted subtrees of $\mathcal{T}_{N,\free}$.  Our goal is to define a topological symmetric operad $\Tree$ where $\Tree(N)$ is the set of elements of arity $N$.  We will then show that $\mathcal{T} \mapsto \boxplus_{\mathcal{T}}$ defines a morphism of topological symmetric operads from $\Tree$ to a certain operad of functions on tuples of laws.  A topological symmetric operad is defined as follows (see e.g.\ \cite{Leinster2004} for general background on operads).

\begin{definition}
A \emph{(plain) operad} consists of a sequence $(P(n))_{n \in \N}$ of sets, an element $\id \in P(1)$, and composition maps
\[
\circ_{k,n_1,\dots,n_k}: P(k) \times P(n_1) \times \dots \times P(n_k) \to P(n_1 + \dots + n_k)
\]
denoted 
\[
(f, f_1,\dots, f_k) \mapsto f(f_1,\dots,f_k),
\]
such that the following axioms hold:
\begin{itemize}
	\item \emph{Identity:} For $f \in P(k)$, we have $f(\id,\dots,\id) = f$ and $\id(f) = f$.
	\item \emph{Associativity:} Given $f \in P(k)$ and $f_j \in P(n_j)$ for $j = 1$, \dots, $k$ and $f_{j,i} \in P(m_{j,i})$ for $i = 1$, \dots, $I_j$ and $j = 1$, \dots, $n$, we have
	\begin{multline*}
	f(f_1(f_{1,1},\dots,f_{1,I_1}), \dots, f_k(f_{k,1}, \dots, f_{k,I_k})) \\
	= [f(f_1,\dots,f_k)](f_{1,1},\dots, f_{1,I_1}, \dots \dots , f_{1,k}, \dots, f_{k,I_k}).
	\end{multline*}
\end{itemize}
The elements of $P(k)$ are said to be \emph{$k$-ary} or have \emph{arity $k$}.
\end{definition}

\begin{definition}
A \emph{symmetric operad} consists of an operad $(P(n))$ together with a right action of the symmetric (permutation) group $\Perm(k)$ on $P(k)$, denoted $(f,\sigma) \mapsto f_\sigma$, satisfying the following axioms:
\begin{itemize}
	\item Let $f \in P(k)$ and $f_j \in P(n_j)$ for $j = 1,\dots, k$.  Let $\sigma \in \Perm(k)$, and let $\tilde{\sigma} \in S_{n_1+\dots+n_k}$ denote the element that rearranges the order of the blocks $\{1,\dots,n_1\}$, $\{n_1+1,\dots,n_1 + n_2\}$, $\{n_1+n_2+1,\dots,n_1 + n_2 + n_3\}$ according to $\sigma$.  Then
	\[
	f_\sigma(f_{\sigma(1)},\dots,f_{\sigma(k)}) = [f(f_1,\dots,f_k)]_{\tilde{\sigma}}.
	\]
	\item Let $f$ and $f_j$ be as above.  Let $\sigma_j \in \Perm(n_j)$, and let $\sigma \in \Perm(n_1+\dots+n_k)$ be the element which permute the elements within each block $\{n_1+\dots+n_{j-1}+1,\dots, n_1 + \dots + n_j\}$ by the permutation $s_j$, without changing the order of the blocks.  Then
	\[
	f((f_1)_{\sigma_1},\dots,(f_k)_{\sigma_k}) = f(f_1,\dots,f_k)_\sigma.
	\]
\end{itemize}
\end{definition}

\begin{definition}
A \emph{topological symmetric operad} consists of a symmetric operad together with a specified topology on $P(k)$ for each $k$, such that the composition and permutation operations of the symmetric operad are continuous.  A \emph{morphism of topological symmetric operads} $P \to Q$ is a sequence of continuous maps $P(k) \to Q(k)$ which respect the composition operations and permutation actions.
\end{definition}

In order to define the operad $\Tree$ where $\Tree(k)$ is the set of elements of arity $k$, we will first describe the composition operation.  Let $\mathcal{T} \in \Tree(k)$ and $\mathcal{T}_1 \in \Tree(n_1)$, \dots, $\mathcal{T}_k \in \Tree(n_k)$.  Let $N_j = n_1 + \dots + n_j$ and $N = N_k$.  Define $\iota_j: [n_j] \to [N]$ by $\iota_j(i) = N_{j-1} + i$, so that $[N] = \bigsqcup_{j=1}^k \iota_j([n_j])$.  For a string $s \in \mathcal{T}_{n_j,\free}$, let $(\iota_j)_*(s)$ denote the string obtained by applying $\iota_j$ to each letter of $s$.  Then we define $\mathcal{T}(\mathcal{T}_1,\dots,\mathcal{T}_k) \in \mathcal{T}_{N,\free}$ to be the rooted subtree with vertex set
\[
\bigcup_{\ell \geq 0} \bigcup_{i_1 \dots i_\ell \in \mathcal{T}} \bigcup_{\substack{s_j \in \mathcal{T}_{i_j} \setminus \{\emptyset\} \\ \text{for } j \in [\ell]}} (\iota_{i_1})_*(s_1) \dots (\iota_{i_\ell})_*(s_\ell).
\]
In other words, the strings in $\mathcal{T}(\mathcal{T}_1, \dots, \mathcal{T}_k)$ are obtained by taking a string $t = i_1 \dots i_\ell$ in $\mathcal{T}$ and replacing each letter $i_j$ by a string $s_j$ from $\mathcal{T}_{i_j}$, with the indices appropriately shifted by $\iota_j: [n_j] \to [N]$.

One can check that the vertex set $\mathcal{T}(\mathcal{T}_1, \dots, \mathcal{T}_k)$ defines a connected subgraph of $\mathcal{T}_{N,\free}$.  Indeed, every final substring of a string in $\mathcal{T}(\mathcal{T}_1, \dots, \mathcal{T}_k)$ will also be in $\mathcal{T}(\mathcal{T}_1, \dots, \mathcal{T}_k)$.  Hence, if $s \in \mathcal{T}(\mathcal{T}_1,\dots,\mathcal{T}_k)$, then we may define a path from $s$ to $\emptyset$ by deleting the first letter of $s$, then the second letter, and so forth.

\begin{observation}
We may define an operad $\Tree$ by letting $\Tree(k)$ be the set of rooted subtrees of $\mathcal{T}_{k,\free}$ using the composition operation above.
\end{observation}

Checking the operad associativity property is a routine exercise in cumbersome notation, which we leave to the reader.  Note that $\mathcal{T}(1) \in \Tree(1)$ has only the two vertices $\emptyset$ and $1$, and $\mathcal{T}(1)$ acts as the identity of the operad.  We next turn to the symmetric structure of the operad.  

We equip $\Tree(k)$ with a right action of $\Perm(k)$ as follows.  Note that there is a left action of $\Perm(k)$ by graph automorphisms on $\mathcal{T}_{k,\free}$, where $\sigma \in \Perm(k)$ acts by permuting the letters $\{1,\dots,k\}$, that is, if $s = j_1 \dots j_\ell$ is a vertex of $\mathcal{T}_{k,\free}$, then $\sigma(j_1,\dots,j_\ell) = (\sigma(j_1),\dots,\sigma(j_\ell))$.  Then we define $\mathcal{T}_\sigma$ to be the image of $\mathcal{T}$ under $\sigma^{-1}$.  It is straightforward to check that this makes $\Tree$ into a symmetric operad; indeed, this reduces to examining how $\Perm(k)$ acts on the labels $\{1,\dots,k\}$.

Furthermore, we claim $\Tree$ can be equipped with the structure of a \emph{topological} symmetric operad.  This comes from the following two observations.

\begin{observation}
For a rooted tree $\mathcal{T} \subseteq \mathcal{T}_{N,\free}$ and $\ell \geq 0$, let $B_\ell(\mathcal{T}) \subseteq \mathcal{T}_{N,\free}$ be set of strings in $\mathcal{T}$ of length $\leq \ell$ (or equivalently the closed ball of radius $\ell$ in the graph metric).  Define $\rho_N: \mathcal{T}_{N,\free} \times \mathcal{T}_{N,\free} \to \R$ by
\[
\rho_N(\mathcal{T}, \mathcal{T}') = \exp(-\sup \{\ell \geq 0: B_\ell(\mathcal{T}) = B_\ell(\mathcal{T}')\}).
\]
Then $\rho_N$ defines a metric on $\Tree(N)$ (and in fact an ultrametric), which makes $\Tree(N)$ into a compact metric space.
\end{observation}

\begin{observation}
Let $\mathcal{T}$, $\mathcal{T}' \in \Tree(k)$ and let $\mathcal{T}_j$, $\mathcal{T}_j' \in \Tree(n_j)$ for $j = 1$, \dots, $k$.  Let $N = n_1 + \dots + n_k$.  Then we have
\begin{multline*}
\rho_N\bigl( \mathcal{T}(\mathcal{T}_1, \dots, \mathcal{T}_k), \mathcal{T}'(\mathcal{T}_1', \dots, \mathcal{T}_k') \bigr) \leq \\
\max\left( \rho_k(\mathcal{T}, \mathcal{T}'), \rho_{n_1}(\mathcal{T}_1,\mathcal{T}_1'), \dots, \rho_{n_k}(\mathcal{T}_k,\mathcal{T}_k') \right)
\end{multline*}
This follows because every string of length $\leq \ell$ in $\mathcal{T}(\mathcal{T}_1, \dots, \mathcal{T}_k)$ is formed by concatenating $\leq \ell$ strings from each of the subgraphs, each of which has length $\leq \ell$.
\end{observation}

\begin{remark}
$\Tree'(N)$ is a closed subspace of $\Tree(N)$, and the sets $\Tree'(N)$ are closed under composition and permutation, so that $\Tree'$ also forms a topological symmetric operad.
\end{remark}

\subsection{Continuity of the Convolution Operations} \label{subsec:continuity}

Next, we define the topological symmetric operad which will serve as the target space of the map $\mathcal{T} \mapsto \boxplus_{\mathcal{T}}$, and we show continuity of this map.  The elements of arity $N$ in the target space will be certain functions $\Sigma(\mathcal{B})^N \to \Sigma(\mathcal{B})$ which are homogeneous with respect to dilation, and the topology will be given in terms of the moments of laws in $\Sigma_1(\mathcal{B})$.

\begin{definition}
We recall that $\Sigma_R(\mathcal{B})$ is the set of $\mathcal{B}$-valued laws with radius bounded by $R$.  We denote by $\Mom_\ell(\mu)$ the multilinear map $\mathcal{B}^{\ell+1} \to \mathcal{B}$ given by
\[
\Mom_\ell(\mu)[b_0,\dots,b_\ell] = \mu(b_0Xb_1 \dots X b_\ell).
\]
We define the norm of a multilinear map $\Lambda: \mathcal{B}^\ell \to \mathcal{B}$ by
\[
\norm{\Lambda} = \sup_{\norm{b_j} \leq 1} \norm{\Lambda[b_1,\dots,b_\ell]}.
\]
We define $d_R$ on $\Sigma_R(\mathcal{B})$ by
\[
d_R(\mu,\nu) = \sup_{\ell \geq 1} \frac{1}{(2R)^\ell} \norm{\Mom_\ell(\mu) - \Mom_\ell(\nu)}
\]
Note $(\Sigma_R(\mathcal{B}), d_R)$ is a complete metric space and that $\mu_n \to \mu$ in $d_R$ if and only if $d^{(\ell)}(\mu_n,\mu) \to 0$ for every $\ell$.
\end{definition}

\begin{definition} \label{def:dilation}
Let $\mu$ be a $\mathcal{B}$-valued distribution and $c \in \R$.  We define the dilation $\dil_c(\mu)$ by $\dil_c(\mu)(f(X)) = \mu(f(cX))$.
\end{definition}

\begin{definition} \label{def:functionoperad}
Let $\Func(\mathcal{B},N)$ denote the set of functions $F: \Sigma(\mathcal{B})^N \to \Sigma(\mathcal{B})$ satisfying
\begin{enumerate}[(1)]
	\item $\rad(F(\mu_1,\dots,\mu_N)) \leq \rad(\mu_1) + \dots + \rad(\mu_N)$.
	\item $F(\dil_c(\mu_1),\dots,\dil_c(\mu_N)) = \dil_c(F(\mu_1,\dots,\mu_N))$.
	\item $F$ restricts to a uniformly continuous function $\Sigma_1(\mathcal{B})^N \to \Sigma_N(\mathcal{B})$.
\end{enumerate}
We equip $\Func(\mathcal{B},N)$ with the metric
\[
d_{\Func(\mathcal{B},N)}(F,G) = \sup_{\mu_1, \dots, \mu_N \in \Sigma_1(\mathcal{B})} d_N(F(\mu_1,\dots,\mu_N), G(\mu_1,\dots,\mu_N)).
\]
\end{definition}

Since $F \in \Func(\mathcal{B},N)$ is homogeneous, continuity on $\Sigma_1(\mathcal{B})^N$ implies continuity of $F$ on $\Sigma_R(\mathcal{B})^N$ for every $R$.  Moreover, one can check directly that conditions (1), (2), and (3) are preserved under composition, so that $(\Func(\mathcal{B}),N)_{N \in \N}$ forms an operad.  Furthermore, it is a symmetric operad under the permutation action
\[
F_\sigma(\mu_1,\dots,\mu_N) = F(\mu_{\sigma^{-1}(1)}, \dots, \mu_{\sigma^{-1}(N)}).
\]
Finallly, the composition operations on $\Func(\mathcal{B},N)$ are continuous, so that $(\Func(\mathcal{B},N))_{N \in \N}$ is a topological symmetric operad.

One of the main goals in this section is to show that the map $\Tree(N) \to \Func(\mathcal{B},N)$ given by $\mathcal{T} \mapsto \boxplus_{\mathcal{T}}$ defines a morphism of topological symmetric operads.  The following observation is the first step.

\begin{lemma} \label{lem:operadcontinuity}
The map $\mathcal{T} \mapsto \boxplus_{\mathcal{T}}$ defines a continuous function $\Tree(N) \to \Func(\mathcal{B},N)$.
\end{lemma}

\begin{proof}
First, we must show that $\boxplus_{\mathcal{T}} \in \Func(\mathcal{B},N)$.  Let $X_j$ be an operator on $(\mathcal{H}_j,\xi_j)$ with law $\mu_j$ and $\norm{X_j} = \rad(\mu_j)$.  Let $(\mathcal{H},\xi)$ be the $\mathcal{T}$-free product of $(\mathcal{H}_1,\xi_1)$, \dots, $(\mathcal{H}_N,\xi_N)$ with the $*$-homomorphisms $\lambda_{\mathcal{T},j}: \mathcal{L}(\mathcal{H}_j) \to \mathcal{L}(\mathcal{H})$.  Then
\[
\norm*{\sum_{j=1}^N \lambda_{\mathcal{T},j}(X_j) } \leq \sum_{j=1}^N \norm{X_j},
\]
which implies that $\rad(\boxplus_{\mathcal{T}}(\mu_1,\dots,\mu_N)) \leq \rad(\mu_1) + \dots + \rad(\mu_N)$, so that (1) of Definition \ref{def:functionoperad} holds.  Moreover, (2) holds because we have
\[
c \sum_{j=1}^N \lambda_{\mathcal{T},j}(X_j) = \sum_{j=1}^N \lambda_{\mathcal{T},j}(cX_j).
\]
Next, to show the uniform continuity condition (3), it suffices to show that for every $\ell$, the moment $\Mom_\ell(\boxplus_{\mathcal{T}}(\mu_1,\dots,\mu_N))$ is a uniformly continuous function of $\mu_1$, \dots, $\mu_N \in \Sigma_1(\mathcal{B})$.  Letting $\mu = \boxplus_{\mathcal{T}}(\mu_1,\dots,\mu_N)$ and letting $X_1$, \dots, $X_\ell$ be as above, by Theorem \ref{thm:combinatorics}, we have
\[
\Mom_\ell(\mu)[b_0,\dots,b_\ell] = \sum_{\chi \in [\ell]^{[N]}} \sum_{\pi \in \mathcal{NC}(\chi,\mathcal{T})} b_0 \Lambda_\pi[X_{\chi(1)} b_1, \dots, X_{\chi(\ell)} b_\ell],
\]
where $\Lambda_\pi$ is given as in Theorem \ref{thm:combinatorics}.  Let us denote
\[
\kappa_{\Bool,\chi,\pi}(\mu_1,\dots,\mu_N)[b_1,\dots,b_{\ell-1}] = \Lambda_\pi[X_{\chi(1)} b_1, X_{\chi(2)} b_2, \dots, X_{\chi(\ell)}].
\]
Then it suffices to show that for each partition $\pi$, the quantity $\kappa_{\Bool,\chi,\pi}(\mu_1,\dots,\mu_N)$ depends continuously on $\Mom_k(\mu_j)$ for $j \in [N]$ and $k \leq \ell$ with respect to the norm on multilinear forms.  This follows from the fact that $\kappa_{\Bool,\chi,\pi}(\mu_1,\dots,\mu_N)$ depends continuously on the Boolean cumulants $\kappa_{\Bool,k}(\mu_j)$ for $j \in [N]$ and $k \leq \ell$, while the Boolean cumulants depend continuously on the moments of $\mu_1$, \dots, $\mu_N$ of degree $\leq \ell$.  We leave the details of these estimates to the reader.

Finally, to show that $\mathcal{T} \mapsto \boxplus_{\mathcal{T}}$ is continuous, note that if $B_\ell(\mathcal{T}) = B_\ell(\mathcal{T}')$, then by Theorem \ref{thm:combinatorics} the first $\ell$ moments of $\boxplus_{\mathcal{T}}(\mu_1,\dots,\mu_N)$ and $\boxplus_{\mathcal{T}'}(\mu_1,\dots,\mu_N)$ agree.  Hence, because these laws have radius $\leq N$, we obtain
\[
d_N(\boxplus_{\mathcal{T}}(\mu_1,\dots,\mu_N), \boxplus_{\mathcal{T}'}(\mu_1,\dots,\mu_N)) \leq \sum_{\ell' > \ell} \frac{1}{(2N)^{\ell'}} \cdot 2N^{\ell'} \leq \frac{1}{2^\ell},
\]
which is a uniform estimate for $\mu_1$, \dots, $\mu_N \in \Sigma_1(\mathcal{B})$.
\end{proof}

\subsection{Convolution and Operad Composition}

Next, we show that the map $\mathcal{T} \mapsto \boxplus_{\mathcal{T}}$ is an operad morphism $\Tree \to \Func(\mathcal{B})$.  In other words, we show that it respects composition in the sense that
\[
\boxplus_{\mathcal{T}(\mathcal{T}_1,\dots,\mathcal{T}_k)} = \boxplus_{\mathcal{T}}(\boxplus_{\mathcal{T}_1},\dots,\boxplus_{\mathcal{T}_k}).
\]
To accomplish this, we show that the operations $\assemb_{\mathcal{T}(\mathcal{T}_1,\dots,\mathcal{T}_k})$ and $\assemb_{\mathcal{T}}(\assemb_{\mathcal{T}_1},\dots,\assemb_{\mathcal{T}_k})$ produce isomorphic $\mathrm{C}^*$-correspondences in the following sense.

\begin{theorem} \label{thm:composition}
Let $\mathcal{T} \in \Tree(k)$ and $\mathcal{T}_j \in \Tree(n_j)$ for $j = 1, \dots, k$.  Let $\mathcal{T}' = \mathcal{T}(\mathcal{T}_1,\dots,\mathcal{T}_k)$.  Let $N_j = n_1 + \dots + n_j$, let $N = N_k$, and let $\iota_j: [n_j] \to [N]$ be the map $i \mapsto i + N_{j-1}$.

Let $(\mathcal{H}_{j,i}, \xi_{j,i})$ be a $\mathcal{B}$-$\mathcal{B}$-correspondence with a $\mathcal{B}$-central unit vector for $j = 1$, \dots, $k$ and $i = 1, \dots, n_j$.  Let
\begin{align*}
(\mathcal{H},\xi) &= \assemb_\mathcal{T}[\assemb_{\mathcal{T}_1}[(\mathcal{H}_{1,1},\xi_{1,1}),\dots,(\mathcal{H}_{1,n_1},\xi_{1,n_1})], \dots, \assemb_{\mathcal{T}_k}[(\mathcal{H}_{k,1},\xi_{k,1}),\dots,(\mathcal{H}_{k,n_k},\xi_{k,n_k})]] \\
(\mathcal{K},\zeta) &= \assemb_{\mathcal{T}(\mathcal{T}_1,\dots,\mathcal{T}_k)}[(\mathcal{H}_{1,1},\xi_{1,1}),\dots,(\mathcal{H}_{1,n_1},\xi_{1,n_1}), \dots \dots , (\mathcal{H}_{k,1},\xi_{k,1}),\dots,(\mathcal{H}_{k,n_k},\xi_{k,n_k})],
\end{align*}
and let us also denote
\[
(\mathcal{H}_j,\xi_j) = \assemb_{\mathcal{T}_j}[(\mathcal{H}_{j,1},\xi_{j,1}), \dots, (\mathcal{H}_{j,n_j}, \xi_{j,n_j})].
\]
Then there is a unique unitary isomorphism $\Phi: (\mathcal{H},\xi) \to (\mathcal{K},\zeta)$ of $\mathcal{B}$-$\mathcal{B}$-correspondences with $\mathcal{B}$-central unit vectors such that for every $j \in [k]$ and $i \in [n_j]$ the diagram
\begin{equation} \label{eq:compositionCD}
\begin{tikzcd}
\mathcal{L}(\mathcal{H}_{j,i}) \arrow{r}{\lambda_{\mathcal{T}_j,i}} \arrow[swap]{d}{\lambda_{\mathcal{T}',\iota_j(i)}} & \mathcal{L}(\mathcal{H}_j) \arrow{d}{\lambda_{\mathcal{T},j}} \\
\mathcal{L}(\mathcal{K}) \arrow[swap]{r}{\Ad_{\Phi}} & \mathcal{L}(\mathcal{H})
\end{tikzcd}
\end{equation}
commutes, where $\Ad_\Phi(x) = \Phi x \Phi^*$.
\end{theorem}

\begin{proof}
For $i \in [n_j]$, let us denote
\[
(\mathcal{K}_{\iota(i)},\zeta_{\iota(i)}') = (\mathcal{H}_{j,i},\xi_{j,i}), \qquad \mathcal{K}_{\iota(i)}^\circ = \mathcal{H}_{j,i}^\circ
\]
so that
\[
(\mathcal{K},\zeta) = \assemb_{\mathcal{T}'}[(\mathcal{K}_1,\zeta_1),\dots,(\mathcal{K}_N,\zeta_N)].
\]
Observe that
\[
\mathcal{H}_j^\circ = \bigoplus_{s \in \mathcal{T}_j \setminus \{\emptyset\}} \mathcal{H}_{j,s}^\circ,
\]
where
\[
\mathcal{H}_{j,s}^\circ = \mathcal{H}_{j,s(1)}^\circ \otimes_{\mathcal{B}} \dots \otimes_{\mathcal{B}} \mathcal{H}_{j,s(\ell)}^\circ,
\]
for an alternating string $s$ of length $\ell$.  Now $\mathcal{H}^\circ$ is the direct sum of
\[
\mathcal{H}_{j_1,\dots,j_\ell}^\circ = \mathcal{H}_{j_1}^\circ \otimes_{\mathcal{B}} \dots \otimes_{\mathcal{B}} \mathcal{H}_{j_\ell}^\circ
\]
over all strings $j_1$, \dots, $j_\ell$ in $\mathcal{T}$.  Substituting in the definition of $\mathcal{H}_j^\circ$ and distributing tensor products over direct sums, we obtain (up to canonical isomorphism) the direct sum of all terms of the form
\[
\mathcal{H}_{j_1,s_1}^\circ \otimes_{\mathcal{B}} \dots \otimes_{\mathcal{B}} \mathcal{H}_{j_\ell,s_\ell}^\circ,
\]
where $s_i \in \mathcal{T}_{j_i} \setminus \{\emptyset\}$.  By definition of $\mathcal{T}' = \mathcal{T}(\mathcal{T}_1,\dots,\mathcal{T}_k)$, this is equivalent to the direct sum of all the terms $\mathcal{K}_s^\circ$, where $s \in \mathcal{T}' \setminus \{\emptyset\}$.  We thus obtain a canonical isomorphism $\Phi: (\mathcal{H},\xi) \to (\mathcal{K},\zeta)$.

To check \eqref{eq:compositionCD}, fix $i \in [n_j]$.  For $x \in \mathcal{L}(\mathcal{H}_{j,i}) = \mathcal{L}(\mathcal{K}_{\iota_j(i)})$, the operator $\lambda_{\mathcal{T}',\iota(i)}(x)$ is define to act by $x \otimes \id$ on every direct summand of the form
\[
\mathcal{K}_s^\circ \otimes \mathcal{K}_{\iota_j(i) s}^\circ \cong \mathcal{K}_{\iota(i)} \otimes \mathcal{K}_s^\circ.
\]
Consider such a direct summand, let $r$ be the largest index such that $s(1)$, \dots, $s(r) \in \iota_j([n_j])$ (which may be zero), and let us write $s(1)$, \dots, $s(r)$ as $\iota_j(s_0)$ for some $s_0 \in \mathcal{T}_j$.  The remaining substring $s(r+1)$, \dots, $s(\ell)$ can then be expressed as $\iota_{j_1}(s_1)$, \dots, $\iota_{j_w}(s_w)$ where $j j_1 \dots j_w \in \mathcal{T}$ and $s_1 \in \mathcal{T}_{j_1}$, \dots, $s_w \in \mathcal{T}_{j_w}$.  Then we have
\[
\mathcal{K}_s^\circ \oplus \mathcal{K}_{\iota_j(i) s}^\circ \subseteq \Phi \left( (\mathcal{B} \oplus \mathcal{H}_j^\circ) \otimes_{\mathcal{B}} \mathcal{H}_{j_1}^\circ \otimes_{\mathcal{B}} \dots \otimes_{\mathcal{B}} \mathcal{H}_{j_w}^\circ \right).
\]
Now $\lambda_{\mathcal{T},j}(\lambda_{\mathcal{T}_j,i}(x))$ acts on this direct summand of the space $\assemb[(\mathcal{H}_1,\xi_1),\dots,(\mathcal{H}_k,\xi_k)]$ by $\lambda_{\mathcal{T}_j,i}(x) \otimes \id$, where $\mathcal{B} \oplus \mathcal{H}_j^\circ$ is viewed as a copy of $\mathcal{H}_j$.  Within this copy of $\mathcal{H}_j$, the subspace $(\mathcal{B} \oplus \mathcal{H}_{j,i}^\circ) \otimes_{\mathcal{B}} \mathcal{H}_{s_0}^\circ$ corresponds to the space $(\mathcal{B} \oplus \mathcal{K}_{\iota_j(i)}^\circ) \otimes_{\mathcal{B}} \mathcal{K}_{(\iota_j)_*(s_0)}^\circ$.  The action of $\lambda_{\mathcal{T}_j,i}(x)$ on this subspace is defined through the action of $x$ on $\mathcal{B} \oplus \mathcal{H}_{j,i}^\circ$.

The other direct summands of $\mathcal{K}$ have the form $\mathcal{K}_s^\circ$ where $s(1) \neq \iota_j(i)$ and $\iota_j(i) s \not \in \mathcal{T}'$.  On this subspace, the operator $\lambda_{\iota_j(i)}(x)$ acts by zero, and one can show that $\lambda_{\mathcal{T},j} \circ \lambda_{\mathcal{T}_j,i}(x)$ also acts by zero on the corresponding subspace of $\mathcal{H}$.  Thus, the action of $\lambda_{\mathcal{T},j} \circ \lambda_{\mathcal{T}_j,i}(x)$ corresponds under the isomorphism $\Phi$ to the action of $\lambda_{\mathcal{T}',\iota_j(i)}(x)$ as desired.

Finally, to show that the isomorphism $\Phi$ mapping $\xi$ to $\zeta$ and satisfying \eqref{eq:compositionCD} is unique, it suffices to note that $\zeta$ is a cyclic vector for the action on $\mathcal{K}$ of the algebra generated by $\lambda_{\mathcal{T}',i}(\mathcal{L}(\mathcal{K}_i))$ for $i \in [N]$.
\end{proof}

\begin{corollary} \label{cor:operadmorphism}
If $\mathcal{T} \in \Tree(k)$ and $\mathcal{T} \in \Tree(n_j)$ for $j = 1$, \dots, $k$, then
\[
\boxplus_{\mathcal{T}}(\boxplus_{\mathcal{T}_1}, \dots, \boxplus_{\mathcal{T}_k}) = \boxplus_{\mathcal{T}(\mathcal{T}_1,\dots,\mathcal{T}_k)}.
\]
In other words, $\mathcal{T} \mapsto \boxplus_{\mathcal{T}}$ is an operad morphism.
\end{corollary}

\begin{proof}
Let $\mu_{j,i}$ be a non-commutative law for each $j \in [k]$ and each $i \in [n_j]$.  Then there exists some $(\mathcal{H}_{j,i},\xi_{j,i})$ and $X_{j,i} \in \mathcal{L}(\mathcal{H}_{j,i})$ self-adjoint such that the law of $X_{j,i}$ is $\mu_{j,i}$.  Let $\mathcal{H}_j$, $\mathcal{H}$, $\mathcal{K}$, etc., be as in the previous proposition.  Then by Corollary \ref{cor:convolvedlaw}, the operator
\[
\sum_{j=1}^k \lambda_{\mathcal{T},j} \left( \sum_{i=1}^{n_j} \lambda_{\mathcal{T}_j,i}(X_{j,i}) \right) = \sum_{j=1}^k \sum_{i=1}^{n_j} \lambda_{\mathcal{T},j} \circ \lambda_{\mathcal{T}_j,i}(X_{j,i}) \in \mathcal{L}(\mathcal{H})
\]
has the law
\[
\boxplus_{\mathcal{T}}(\boxplus_{\mathcal{T}_1}(\mu_{1,1},\dots,\mu_{1,n_1}), \dots, \boxplus_{\mathcal{T}_k}(\mu_{k,1},\dots,\mu_{k,n_k})).
\]
By the previous proposition, the corresponding operator in $\mathcal{L}(\mathcal{K})$ has the same law.  This operator is
\[
\sum_{j=1}^k \sum_{i=1}^{n_j} \lambda_{\mathcal{T}',\iota_j(i)}(X_{j,i})
\]
which by Corollary \ref{cor:convolvedlaw} has the law
\[
\boxplus_{\mathcal{T}'}(\mu_{1,1},\dots,\mu_{1,n_1}, \dots, \mu_{k,1},\dots,\mu_{k,n_k}).
\]
\end{proof}

\subsection{Permutation Equivariance and Convolution Identities}

To complete the proof that $\mathcal{T} \mapsto \boxplus_{\mathcal{T}}$ is a morphism of topological symmetric operads, it only remains to check permutation equivariance.  As in our study of composition, we will proceed by manipulating the $\mathcal{T}$-free product $\mathrm{C}^*$-correspondences.

In fact, these manipulations work in a greater level of generality where we replace a permutation $\sigma: [N] \to [N]$ by an arbitrary map $\psi: [N'] \to [N]$.  Thus, our main result Theorem \ref{thm:permutation} has several applications besides permutation invariance.  As we will see below, the case where $\psi$ is surjective enables us to prove identities relating several convolution operations (Corollary \ref{cor:convolutionidentity}), while the case where $\psi$ is injective relates to the study of conditional expectations (Remark \ref{rem:conditionalexpectations}).

\begin{theorem} \label{thm:permutation}
Let $\psi$ be a function $[N'] \to [N]$ and let $\psi_*$ be the function from strings on the alphabet $[N']$ to strings on the alphabet $[N]$ given by $\psi_*(j_1 \dots j_\ell) = \psi(j_1) \dots \psi(j_\ell)$ for every string $j_1 \dots j_\ell \in \mathcal{T}_{N',\free}$.  Let $\mathcal{T}_{\Ran(\psi)} \subset \mathcal{T}_{N,\free}$ be the tree consisting of all alternating strings on the alphabet $\Ran(\psi) = \psi([N'])$.

Suppose that $\mathcal{T} \in \Tree(N)$ and $\mathcal{T}' \in \Tree(N')$ are such that $\psi_*$ defines a bijection $\mathcal{T}' \to \mathcal{T} \cap \mathcal{T}_{\Ran(\psi)}$.  (In particular, this requires that $\psi_*(s)$ is alternating for every $s \in \mathcal{T}'$.)

Let $(\mathcal{H}_1,\xi_1)$, \dots, $(\mathcal{H}_N, \xi_N)$ be $\mathcal{B}$-$\mathcal{B}$-correspondences with $\mathcal{B}$-central unit vectors.  Then there is a unique unitary embedding of $\mathcal{B}$-$\mathcal{B}$-correspondences with $\mathcal{B}$-central unit vectors
\[
\Psi: \assemb_{\mathcal{T}'}[(\mathcal{H}_{\psi(1)}, \xi_{\psi(1)}), \dots, (\mathcal{H}_{\psi(N)}, \xi_{\psi(N)})] \to \assemb_{\mathcal{T}}[(\mathcal{H}_1,\xi_1), \dots, (\mathcal{H}_N, \xi_N)]
\]
such that for $j \in \Ran \psi$, the diagram
\begin{equation} \label{eq:permutationCD}
\begin{tikzcd}[row sep = large]
\mathcal{L}(\mathcal{H}_j) \arrow[swap]{d}{\lambda_{\mathcal{T},j}} \arrow{rd}{\sum_{i \in \psi^{-1}(j)} \lambda_{\mathcal{T}',i}} & \\
\mathcal{L}(\mathcal{H}) \arrow[swap]{r}{\Ad_{\Psi^*}} & \mathcal{L}(\mathcal{K})
\end{tikzcd}
\end{equation}
commutes, where
\begin{align*}
(\mathcal{H},\xi) &= \assemb_{\mathcal{T}}[(\mathcal{H}_1,\xi_1),\dots,(\mathcal{H}_N,\xi_N)] \\
(\mathcal{K},\zeta) &= \assemb_{\mathcal{T}'}[(\mathcal{H}_{\psi(1)}, \xi_{\psi(1)}), \dots, (\mathcal{H}_{\psi(N')}, \xi_{\psi(N')}].
\end{align*}
and $\Ad_{\Psi^*}(x) = \Psi^* x \Psi$.  Moreover, we have
\begin{equation} \label{eq:conditionalexpectationproperty}
\Ad_{\Psi^*}[a_1 a a_2] = \Ad_{\Psi^*}[a_1] \Ad_{\Psi^*}[a] \Ad_{\Psi^*}[a_2]
\end{equation}
provided that $a_1$, $a_2 \in \Alg(\lambda_{\mathcal{T},j}(\mathcal{L}(\mathcal{H}_j)): j \in \Ran \psi)$.
\end{theorem}

\begin{proof}
Let $(\mathcal{K}_j, \zeta_j) = (\mathcal{H}_{\psi(j)}, \xi_{\psi(j)})$, so that
\[
(\mathcal{K},\zeta) = \assemb_{\mathcal{T}'}[(\mathcal{K}_1, \zeta_1), \dots, (\mathcal{K}_{N'}, \zeta_{N'})].
\]
For each $s \in \mathcal{T}'$, we have $\psi_*(s) \in \mathcal{T}$ by our assumptions about $\psi$, $\mathcal{T}$, and $\mathcal{T}$, and we also have
\[
\mathcal{K}_s^\circ = \mathcal{H}_{\psi_*(s)}^\circ.
\]
Since $\psi_*$ defines a bijection $\mathcal{T} \to \mathcal{T} \cap (\Ran \psi_*)$, we have $s \neq s' \implies \psi_*(s) \neq \psi_*(s')$.  Thus, we may define an injective unitary map $\Psi: \mathcal{K} \to \mathcal{H}$ by mapping $\mathcal{K}_s^\circ$ onto $\mathcal{H}_{\psi_*(s)}^\circ$ for each $s \in \mathcal{T}'$.  Clearly, $\Psi$ maps the given unit vector $\zeta \in \mathcal{K}$ to the given unit vector $\xi \in \mathcal{H}$.

Suppose $j \in \Ran \psi$ and let us check \eqref{eq:permutationCD}.  For $x \in \mathcal{L}(\mathcal{H}_j)$, we must show that
\[
\Psi \lambda_{\mathcal{T},j}(x) \Psi^* = \sum_{i \in \psi^{-1}(j)} \lambda_{\mathcal{T}',i}(x).
\]
Let us consider the action of each of these operators on $\mathcal{K}_s^\circ \oplus \mathcal{K}_{is}^\circ$ where $i \in \psi^{-1}(j)$.  Under the map $\Psi$, we have
\[
\mathcal{K}_s^\circ \oplus \mathcal{K}_{is}^\circ \cong \mathcal{H}_{\psi_*(s)}^\circ \oplus \mathcal{H}_{j \psi_*(s)}^\circ.
\]
The action of $\lambda_{\mathcal{T},j}(x)$ on this space is given by $x \otimes \id_{\mathcal{H}_{\psi_*(s)}^\circ}$.  This is equivalent to the action of $\lambda_{\mathcal{T}',i}(x)$ on $\mathcal{K}_s^\circ \otimes \mathcal{K}_{is}$.  Moreover, if $i' \neq i$ is in $\psi^{-1}(j)$, then by our assumptions on $\psi$, the strings $i's$ and $i'is$ are not in $\mathcal{T}'$.  Thus, the action of $\lambda_{\mathcal{T}',i'}(x)$ on $\mathcal{K}_s^\circ \otimes \mathcal{K}_{is}^\circ$ is zero.  Therefore, we have $\Psi \lambda_{\mathcal{T},j}(x) \Psi^* = \sum_{i \in \psi^{-1}(j)} \lambda_{\mathcal{T}',i}(x)$ when restricted to this subspace.

The other direct summands of $(\mathcal{K},\zeta)$ have the form $\mathcal{K}_s$ where $s(1) \not \in \psi^{-1}(j)$ and $is \not \in \mathcal{T}'$ for $i \in \psi^{-1}(j)$.  The operators $\lambda_{\mathcal{T}',i}(x)$ act by zero on this subspace.  Our assumptions on $\psi$ guarantee that $\psi_*(s)$ does not begin with $j$ and $j \psi_*(s)$ is not in $\mathcal{T}$.  Thus, $\lambda_{\mathcal{T},j}(x)$ also acts by zero on this subspace.  Thus, \eqref{eq:permutationCD} commutes as desired.

Next, we show uniqueness of $\Psi$.  By our assumptions, $\Psi(\mathcal{K})$ is the direct sum of $\mathcal{H}_s^\circ$ for $s \in \mathcal{T} \cap (\Ran \psi_*)$.  Now $\Alg(\lambda_j(\mathcal{L}(\mathcal{H}_j)): j \in \Ran \psi) \xi$ is dense in this subspace.  It follows that $\Alg( (\sum_{i \in \psi^{-1}} \lambda_{\mathcal{T}',i})(\mathcal{L}(\mathcal{H}_j)): j \in \Ran \psi) \xi$ is dense in $\mathcal{K}$.  If a map $\Psi': \mathcal{K} \to \mathcal{H}$ satisfies \eqref{eq:permutationCD}, then $\Psi$ and $\Psi'$ must agree on $\Alg( (\sum_{i \in \psi^{-1}} \lambda_{\mathcal{T}',i})(\mathcal{L}(\mathcal{H}_j)): j \in \Ran \psi) \xi$ and hence on all of $\mathcal{K}$.

Finally, to prove \eqref{eq:conditionalexpectationproperty}, observe that $\Ran \Psi$ is an invariant subspace for $\lambda_{\mathcal{T},j}(x)$ when $j \in \Ran \psi$.  It follows that $\Ran \Psi$ is an invariant subspace for every element of $\Alg(\lambda_{\mathcal{T},j}(\mathcal{L}(\mathcal{H}_j)): j \in \Ran \psi)$.  Moreover, $\Psi \Psi^*$ is the projection onto the image of $\Psi$.  Thus, if $a_1$, $a$, and $a_2$ are as in \eqref{eq:conditionalexpectationproperty}, then
\[
a_2 \Psi = \Psi \Psi^* a_2 \Psi
\]
and
\[
\Psi^* a_1 = (a_1^* \Psi)^* = (\Psi \Psi^* a_1^* \Psi)^* = \Psi^* a_1 \Psi^* \Psi
\]
so that
\[
\Psi^* a_1 a a_2 \Psi = (\Psi^* a_1 \Psi \Psi^*) a (\Psi \Psi^* a_2 \Psi) = (\Psi^* a_1 \Psi)(\Psi^* a \Psi)(\Psi^* a_2 \Psi).
\]
\end{proof}

The next corollary follows from Proposition \ref{thm:permutation} using similar reasoning as in the proof of Corollary \ref{cor:operadmorphism}.

\begin{corollary} \label{cor:convolutionidentity}
Let $\mathcal{T} \in \Tree(N)$ and $\mathcal{T}' \in \Tree(N')$.  Suppose that $\psi: [N'] \to [N]$ is surjective, and suppose that $\psi_*$ restricts to a bijection $\mathcal{T}' \to \mathcal{T}$.  Then we have for non-commutative laws $\mu_1$, \dots, $\mu_N$ that
\[
\boxplus_{\mathcal{T}'}(\mu_{\psi(1)}, \dots, \mu_{\psi(N')}) = \boxplus_{\mathcal{T}}(\mu_1,\dots,\mu_N).
\]
In particular, if $\sigma: [N] \to [N]$ is a permutation, then
\[
\boxplus_{\mathcal{T}_\sigma} = (\boxplus_{\mathcal{T}})_\sigma.
\]
Hence, $\mathcal{T} \mapsto \boxplus_{\mathcal{T}}$ defines a morphism of topological symmetric operads.
\end{corollary}

\begin{example} \label{ex:FBMpermutation1}
The tree $\mathcal{T}_{N,\free}$ used to define $N$-ary free convolution is invariant under permutations of the labels $[N]$ and hence the operation of free convolution is independent of the ordering; in particular, the binary free convolution operation is commutative.  The same holds for the tree $\emptyset \cup [N]$ used for Boolean convolution.  In the monotone case, the permutation $\sigma: i \mapsto N - i + 1$ maps the tree for monotone convolution to the tree for anti-monotone convolution.  Hence, monotone convolution of $\mu_1$, \dots, $\mu_N$ is equivalent to anti-monotone convolution of $\mu_N$, \dots, $\mu_1$.
\end{example}

Particular applications of Corollary \ref{cor:convolutionidentity} to prove convolution identities will be discussed in \S \ref{subsec:convolutionidentities}.  Another important special case of Theorem \ref{thm:permutation} is when $\psi: [N'] \to [N]$ is injective.  This case will furnish another proof of Corollary \ref{cor:indexsubset} and relates to conditional expectations.

\begin{example} \label{ex:indexsubset}
Let $N' \leq N$ and let $\psi: [N'] \to [N]$ be injective.  Let $\mathcal{T} \in \Tree(N)$ and let $\mathcal{T}'$ be the rooted subtree consisting of strings on the alphabet $\psi([N'])$.  By permutation equivariance, it suffices to consider the case where $\psi$ is the standard inclusion $[N'] \to [N]$, so that we are in the same situation as Corollary \ref{cor:indexsubset}.  Let $(\mathcal{H}_1,\xi_1)$, \dots, $(\mathcal{H}_N,\xi_N)$ be given $\mathcal{B}$-$\mathcal{B}$-correspondences with $\mathcal{B}$-central unit vectors, and let
\begin{align*}
(\mathcal{H},\xi) &= \assemb_{\mathcal{T}}[(\mathcal{H}_1,\xi_1),\dots,(\mathcal{H}_N,\xi_N)]
(\mathcal{H}',\xi') &= \assemb_{\mathcal{T}'}[(\mathcal{H}_1,\xi_1),\dots,(\mathcal{H}_{N'},\xi_{N'})].
\end{align*}
Let $\Psi: (\mathcal{H}',\xi') \to (\mathcal{H},\xi)$ be given by Theorem \ref{thm:permutation}.

Let $\chi: [\ell] \to [N']$ and let $a_j \in \mathcal{L}(\mathcal{H}_{\chi(j)})$ for $j = 1$, \dots, $\ell$.  Then we have by \eqref{eq:conditionalexpectationproperty} and the fact that $\Psi \xi' = \xi$ that
\begin{align*}
\ip{\xi', \lambda_{\mathcal{T}',\chi(1)}(a_1) \dots \lambda_{\mathcal{T}',\chi(\ell)}(a_\ell) \xi'} &= \ip{\xi', \Psi^* \lambda_{\mathcal{T},\chi(1)}(a_1) \dots \lambda_{\mathcal{T},\chi(\ell)}(a_\ell) \Psi \xi} \\
&= \ip{\xi, \lambda_{\mathcal{T},\chi(1)}(a_1) \dots \lambda_{\mathcal{T},\chi(\ell)}(a_\ell) \xi}.
\end{align*}
Hence, we have an alternative proof of Corollary \ref{cor:indexsubset}.
\end{example}

\begin{remark} \label{rem:conditionalexpectations}
Examining the last example and the statement of the theorem, we might hope that the map $\Ad_{\Psi^*}$ defines a conditional expectation from $\Alg(\lambda_{\mathcal{T},j}(\mathcal{L}(\mathcal{H}_j)): j \in [N])$ to $\Alg(\lambda_{\mathcal{T}',j}(\mathcal{L}(\mathcal{H}_j)): j \in [N'])$.

More precisely, let $\mathcal{A} \subseteq \mathcal{L}(\mathcal{H})$ be the $\mathrm{C}^*$-algebra generated by $\lambda_{\mathcal{T},j}(\mathcal{L}(\mathcal{H}_j))$ for $j = 1$, \dots, $N$.  Let $\mathcal{C} \subseteq \mathcal{L}(\mathcal{H})$ be the $\mathrm{C}^*$-algebra generated by $\lambda_{\mathcal{T},j}(\mathcal{L}(\mathcal{H}_j))$ for $j = 1$, \dots, $N'$, and let $\mathcal{C}' \subset \mathcal{T}(\mathcal{H}')$ be the $\mathrm{C}^*$-algebra generated by $\lambda_{\mathcal{T}',j}(\mathcal{L}(\mathcal{H}_j))$ for $j = 1$, \dots, $N'$.

Suppose that it happens that $\Ad_{\Psi^*}$ maps $\mathcal{A}$ into $\mathcal{C}'$ and that it restricts to an isomorphism $\mathcal{C} \to \mathcal{C}'$.  Then we may identify $\mathcal{C}'$ with $\mathcal{C} \subseteq \mathcal{A}$, and then \eqref{eq:conditionalexpectationproperty} says that $\Ad_{\Psi^*}$ defines a conditional expectation from $\mathcal{A}$ onto the subalgebra $\mathcal{C}$.  For example, this holds if $N' = 1$ and $1$ is in $\mathcal{T}$; indeed, $\Ad_{\Psi^*}$ maps $\mathcal{A}$ into $\mathcal{C}'$ since $\mathcal{C}' = \mathcal{L}(\mathcal{H}_{\psi(1)}) = \mathcal{C}$, and it restricts to an isomorphism $\mathcal{C} \to \mathcal{C}'$ since it is the identity map on $\mathcal{L}(\mathcal{H}_1)$.  Also, in the free, Boolean, and monotone cases, $\Ad_{\Psi^*}$ maps $\mathcal{A}$ into $\mathcal{C}$ and restricts to an isomorphism $\mathcal{C} \to \mathcal{C}'$ for every value of $N'$ and $N$.

For general $\mathcal{T}$, we do not know whether $\Ad_{\Psi^*}$ maps $\mathcal{A}$ into $\mathcal{C}'$ or whether it is injective on $\mathcal{C}$.  However, \eqref{eq:conditionalexpectationproperty} says that $\Ad_{\Psi^*}$ is a $*$-homomorphism on $\mathcal{C}$ and that it is a $\mathcal{C}$-$\mathcal{C}$-bimodule map (where the right and left actions of $\mathcal{C}$ on $\mathcal{C}'$ are given by first applying $\Ad_{\Psi^*}$ to the elements of $\mathcal{C}$).
\end{remark}

\subsection{The Case of Digraphs} \label{subsec:digraphoperad}

We discussed in \S \ref{subsec:productexamples} the case where $\mathcal{T} = \Walk(G)$ for some $G \in \Digraph(N)$.  It turns out that $\Digraph = (\Digraph(k))_{k \geq 1}$ can be made into an operad with a composition operation compatible with that of $\Tree$.

\begin{definition}
Let $G \in \Digraph(k)$ and suppose that $G_1 \in \Digraph(n_1)$, \dots, $G_k \in \Digraph(n_k)$.  Let $N_j = n_1 + \dots + n_j$ and $N = N_k$.  Define $\iota_j: [n_j] \to [N]$ by $\iota_j(i) = N_{j-1} + i$, so that $[N] = \bigsqcup_{j=1}^k \iota_j([n_j])$.  We define the composition $G(G_1,\dots,G_k) \in \Digraph(N)$ as the digraph with edge set
\[
\{(\iota_i(v), \iota_j(w)): i \sim_G j, v \in [n_i], w \in [n_j]\} \cup \{(\iota_j(v), \iota_j(w): v, w \in [n_j], v \sim_{G_j} w\}.
\]
\end{definition}

In other words, to construct $G(G_1,\dots,G_k)$, we take the disjoint union of $G_1$, \dots, $G_k$ (with the appropriate relabeling of the vertices), and then for each directed edge $(i,j)$ in $G$, we add a directed edge from every vertex of $G_i$ to every vertex of $G_j$.

\begin{example}
Let $K_N$ be the complete graph, $K_N^c$ the totally disconnected graph, and $K_N^{<}$ the directed complete graph as in \S \ref{subsec:productexamples}.  Given two digraphs $G$ and $G'$, the composition $K_2^c(G,G')$ is the disjoint union of $G$ and $G'$.  The composition $K_2(G,G')$ is obtained from the disjoint union by adding an undirected edge from every vertex of $G$ to every vertex of $G'$.  In particular, $K_2(K_n^c,K_m^c)$ is the complete bipartite graph $K_{n,m}$.  The composition $K_2^{<}$ is obtained by adding a directed edge from every vertex of $G$ to every vertex of $G'$.
\end{example}

One can check that the composition operation on $\Digraph$ defined above satisfies operad associativity.  Moreover, the graph $\bullet$ with one vertex acts as the identity in $\Digraph(1)$.  Thus, $\Digraph$ is an operad.  Moreover, just as in the case of $\Tree$, the symmetric group $\Perm(N)$ acts on $\Digraph(N)$ by permutation of the labels $1$, \dots, $N$, and this action endows $\Digraph$ with the structure of a symmetric operad.

One can also check that $G \mapsto \Walk(G)$ is an injective function $\Digraph(N) \to \Tree(N)$.  Furthermore, we have
\[
\Walk(G(G_1,\dots,G_k)) = \Walk(G)(\Walk(G_1), \dots, \Walk(G_k)).
\]
Indeed, suppose $i_0$, \dots, $i_\ell$ is a walk in $G$ and let $s_j$ be a walk in $G_{i_j}$ for each $j$.  Then $\iota_{i_0}(s_0) \dots \iota_{i_\ell}(s_\ell)$ is a walk in $G(G_1,\dots,G_k)$.  It begins in $G_0$ by following the walk $s_0$, then cross an edge from $G_{i_0}$ to $G_{i_1}$, then follow the walk $s_1$ in $G_{i_1}$, and so forth.  Every walk in $G(G_1,\dots,G_k)$ can be constructed in this way.  As a consequence, we have the following proposition.

\begin{proposition} \label{prop:digraphmorphism}
The maps $G \mapsto \Walk(G)$ defines an operad morphism $\Digraph \to \Tree$, and in fact this is a morphism of symmetric operads.  Consequently, the map $G \mapsto \boxplus_{\Walk(G)}$ is also a morphism of symmetric operads.
\end{proposition}

\begin{example}
Let $\bullet$ denote the graph with a single vertex.  Then, in light of Proposition \ref{prop:digraphmorphism}, the identity
\[
K_2(\bullet,K_2) = K_3 = K_2(K_2,\bullet)
\]
implies the associativity of the binary operation $\boxplus$.  Furthermore, we have
\[
K_N = K_2(\bullet, K_2(\bullet, \dots ))
\]
which implies that the $N$-ary free convolution operation is obtained by iterating binary free convolution operation.  The same observations hold for Boolean independence and the totally disconnected graph $K_N^c$, and for monotone independence and the digraph $K_N^{<}$.
\end{example}

\begin{example}
The permutation invariance property of $K_N$ and $K_N^c$ implies that free and Boolean convolution are permutation-invariant, and in particular the associated binary operations are commutative.  Moreover, the order-reversing permutation sends the digraph for monotone convolution to the digraph for anti-monotone convolution.
\end{example}

Our construction has the following behavior with respect to subgraphs.  Let $G \in \Digraph(N)$ and suppose that $\mathcal{A}_1$, \dots, $\mathcal{A}_N$ are $\Walk(G)$-freely independent in $(\mathcal{A},E)$.  Let $N' \leq N$, and let $G'$ be the induced sub-digraph of $G$ on the vertex set $[N']$.  Then $\mathcal{A}_1$, \dots, $\mathcal{A}_{N'}$ are $\Walk(G')$-independent as a consequence of Corollary \ref{cor:indexsubset}.  By permutation equivariance, we can say more generally that if $\psi: [N'] \to [N]$ is injective and if $G'$ is the digraph on $[N']$ given by $i \sim_{G'} j$ if and only if $\psi(i) \sim_G \psi(j)$, then $\mathcal{A}_{\psi(1)}$, \dots, $\mathcal{A}_{\psi(N')}$ are $\Walk(G')$-freely independent.

In particular, this allows us to describe the pairwise interaction of the algebras $\mathcal{A}_i$ and $\mathcal{A}_j$ for $i \neq j$.  Indeed, we can apply the above argument to the function $\psi: \{1,2\} \to [N]$ given by $\psi(1) = i$ and $\psi(2) = j$.  It follows that if there are directed edges from $i$ to $j$ and $j$ to $i$, then $\mathcal{A}_i$ and $\mathcal{A}_j$ are freely independent.  If there are no edges between $i$ and $j$, then $\mathcal{A}_i$ and $\mathcal{A}_j$ are Boolean independent.  If there is a directed edge from $i$ to $j$, but no edge from $j$ to $i$, then $\mathcal{A}_i$ and $\mathcal{A}_j$ are monotone independent.

The digraph construction thus produces a mixture of free, Boolean, and monotone independence similar to several constructions in previous work.  If we assume that the digraph $G$ forms a poset, that is, $\sim_G$ is a strict partial order, then we obtain the construction of Wysocza{\'n}ski \cite{Wysoczanski2010}.  In particular, if $S$ is a totally ordered subset of the vertices, then the algebras $(\mathcal{A}_i)_{i \in S}$ are monotone independent (when the indices are ordered according to the partial order $\sim_G$).

Next, suppose the digraph $G$ is an undirected graph, that is, $i \sim_G j$ if and only if $j \sim_G i$.  Then each pair of algebras is either freely independent or Boolean independent.  If $S \subseteq [N]$ is a clique (that is, the induced subgraph $G'$ is a complete graph), then $(\mathcal{A}_i: i \in S)$ are freely independent.  If $S$ is an anti-clique (that is, the induced subgraph $G'$ is a totally disconnected digraph), then $(\mathcal{A}_i: i \in S)$ are Boolean independent.  This is the same construction as that of Kula and Wysocza{\'n}ski \cite{KW2013} except that it is phrased in terms of an undirected graph rather than a poset.

We also remark that, given an undirected graph $G$, the $\Lambda$-free product of M{\l}otkowski \cite{Mlotkowski2004} (further studied in \cite{SW2016}) allows us to join $N$ algebras $\mathcal{A}_1$, \dots, $\mathcal{A}_N$, so that each pair $\mathcal{A}_i$ and $\mathcal{A}_j$ is classically independent if $i \sim_G j$ and freely independent if $i \not \sim_G j$ (in the scalar-valued setting).  However, our framework does not include classical independence or this construction.

\begin{remark}
Our operad $\Digraph$ is reminiscent of the graph operad used in the study of traffic freeness by Male \cite[\S 4.2]{Male}, although it is neither the same object nor used in the same way here.  Indeed, Male uses a graph operad to describe other algebraic operations besides addition and multiplication that can be performed on random variables, whereas we use the graph operation to describe the structure of independence.
\end{remark}

\begin{remark}
The theory developed here is distinct from, but similar in spirit to, work of Accardi, Lenczewski, and Sa{\l}apata \cite{ALS2007} that describes operations on rooted graphs that produce the free, Boolean, monotone, orthogonal, and subordination convolutions of the spectral measures of the adjacency operators.  For instance, the free product of rooted graphs will lead to the free convolution of the spectral measures associated to the adjacency operators, the star product corresponds to Boolean convolution, and the comb product to monotone convolution.  In our paper, the graphs themselves describe the product operations rather than being the objects that we take the products of.  However, it is worth investigating whether there is a $\mathcal{T}$-free product of graphs that corresponds to the convolution operations in our paper.
\end{remark}

\section{Convolution Identities and Decomposition} \label{sec:convolutionidentities}

\subsection{Background on Analytic Transforms}

For the sake of connecting the examples in this section with previous literature, we will use the analytic characterizations of the free, Boolean, monotone, and orthogonal convolutions.  Recall that the Cauchy-Stieltjes transform of a probability measure $\mu$ on $\R$ is given by
\[
G_\mu(z) = \int_{\R} \frac{1}{z - t}\,d\mu(t) = E[(z - X)^{-1}],
\]
where $\im z > 0$ and where $X$ is a random variable having the law $\mu$.  In the $\mathcal{B}$-valued setting, in order for $\mu$ to be characterized by $G_\mu$, it is necessary to view $G_\mu$ as a function defined not only for elements of $\mathcal{B}$ but also for $n \times n$ matrices over $\mathcal{B}$ for every $n$, or more precisely, a fully matricial function; see \cite[\S 5 - 6]{Voiculescu2004} \cite[\S 5.2 - 5.3]{PV2013} \cite{KVV2014} \cite{Williams2017} for more information.

Let $\mu \in \Sigma(\mathcal{B})$ and let $X$ be a self-adjoint variable in $(\mathcal{A},E)$ realizing the law $\mu$.  Let $X^{(n)} \in M_n(\mathcal{A})$ be the matrix with $X$ on the diagonal and zero elsewhere.  If $z \in M_n(\mathcal{B})$ with $\im z = (z - z^*) / 2i \geq \epsilon > 0$, then we define
\[
G_\mu^{(n)}(z) = E^{(n)}[(z - X^{(n)})^{-1}],
\]
where $E^{(n)}$ is the map $M_n(\mathcal{A}) \to M_n(\mathcal{B})$ obtained by applying $E$ entrywise.

The following results can be deduced from the moment formulas for each type of independence (discussed in \S \ref{sec:combinatorics}), and since the arguments are well-known and particular to each case, we defer them to the references cited.  A convenient summary is also found in \cite[\S 3]{ALS2007}.  We caution that conventions may differ slightly in some papers.

\begin{theorem} ~
\begin{enumerate}[(1)]
	\item {\bf Free Case:} Let $F_\mu^{(n)}(z) = G_\mu^{(n)}(z)^{-1}$ (the multiplicative inverse).  Let $(F_\mu^{(n)})^{-1}(z)$ denote the functional inverse of $F_\mu^{(n)}$, and set $\Phi_\mu^{(n)}(z) = (F_\mu^{(n)})^{-1}(z) - z$.  Then $\Phi_\mu^{(n)}$ is defined for $\im z$ sufficiently large (depending on $\rad(\mu)$) and we have
	\[
	\Phi_{\mu \boxplus \nu}^{(n)} = \Phi_\mu^{(n)} + \Phi_\nu^{(n)}.
	\]
	See \cite[\S 4.11]{Voiculescu1995}.
	
	\item {\bf Boolean Case:} Let $K_\mu^{(n)}(z) = z - F_\mu^{(n)}(z)$.  Then
	\[
	K_{\mu \uplus \nu}^{(n)} = K_\mu^{(n)} + K_\nu^{(n)}.
	\]
	See \cite[\S 2]{SW1997}, \cite[Theorem 2.2]{Bercovici2006}, \cite[Corollary 4.6]{Popa2009}.
	
	\item {\bf Monotone Case:} We have
	\[
	F_{\mu \rhd \nu}^{(n)} = F_\mu^{(n)} \circ F_\nu^{(n)}.
	\]
	See \cite[Theorem 3.1]{Muraki2000}, \cite{Bercovici2005}, \cite[Theorem 3.2]{Popa2008a}.
	
	\item {\bf Orthogonal Case:} We have
	\[
	K_{\mu \vdash \nu}^{(n)} = K_\mu^{(n)} \circ F_\nu^{(n)}.
	\]
	See \cite[Thm.\ 6.2, Cor.\ 6.3]{Lenczewski2007}.
\end{enumerate}
\end{theorem}

\subsection{Some Convolution Identities} \label{subsec:convolutionidentities}

As an application of Corollary \ref{cor:convolutionidentity}, we discuss several convolution identities that were studied in previous literature, often from the analytic viewpoint.

\begin{example}
The identity $\mu \rhd \nu = (\mu \vdash \nu) \uplus \nu$ studied in \cite[Cor.\ 6.6]{Lenczewski2007} is a special case of Corollary \ref{cor:convolutionidentity}.  Let $\psi: [3] \to [2]$ given by $\psi(1) = 1$, $\psi(2) = 2$, $\psi(3) = 2$.  Then we claim that $\psi_*$ defines an isomorphism $\mathcal{T}_{2,\Bool}(\mathcal{T}_{\orth}, \id) \to \mathcal{T}_{2,\mono}$.

To compute $\mathcal{T}_{2,\Bool}(\mathcal{T}_{\orth}, \id)$, let $\iota_1: [2] \to [3]$ and $\iota_2: [1] \to [3]$ be the inclusions given by $\iota_1(j) = j$ and $\iota_2(1) = 3$.  Because $\mathcal{T}_{2,\Bool} = \{\emptyset, 1, 2\}$, we have
\[
\mathcal{T}_{2,\Bool}(\mathcal{T}_{\orth}, \id) = \{\emptyset\} \cup \{(\iota_1)_*(s): s \in \mathcal{T}_{\orth} \setminus \{\emptyset\}\} \cup \{(\iota_2)_*(1)\}.
\]
Since $\mathcal{T}_{\orth} = \{\emptyset, 1, 21\}$, we obtain
\[
\mathcal{T}_{2,\Bool}(\mathcal{T}_{\orth}, \id) = \{\emptyset, 1, 21, 3\}.
\]
The map $\psi_*$ defines a bijection from $\mathcal{T}_{2,\Bool}(\mathcal{T}_{\orth}, \id)$ to $\{\emptyset, 1, 2, 21\} = \mathcal{T}_{2,\mono}$.  Thus, the Corollary implies that $\mu \rhd \nu = (\mu \vdash \nu) \uplus \nu$.
\end{example}

\begin{remark}
This identity is also easy to prove in terms of analytic transforms (as done in \cite{Lenczewski2007}) since it simply says that $F_\mu \circ F_\nu = z - K_\mu \circ F_\nu - K_\nu$.
\end{remark}

\begin{example} \label{ex:subordination}
The identity
\begin{equation} \label{eq:subordinationconvolution}
\mu \boxplus \nu = (\mu \boxright \nu) \lhd \nu
\end{equation}
studied in \cite[\S 7]{Lenczewski2007}, \cite{Nica2009}, \cite[Proposition 7.2]{Liu2018} can be deduced from Corollary \ref{cor:convolutionidentity} as follows.  Let $\mathcal{T}' = \mathcal{T}_{2,\mono \dagger}(\mathcal{T}_{\sub},\id)$.  Let $\psi: \{1,2,3\}$ be given by $\psi(1) = 1$, $\psi(2) = 2$, $\psi(3) = 2$.  Then we claim that $\psi_*$ defines a graph isomorphism from $\mathcal{T}'$ to $\mathcal{T}_{2,\free}$; by the Corollary, this will be sufficient to establish \eqref{eq:subordinationconvolution}.

To compute the composed tree $\mathcal{T}'$ as an element of $\Tree(3)$, let $\iota_1: [2] \to [3]$ and $\iota_2: [1] \to [3]$ be the inclusions given by $\iota_1(j) = j$ and $\iota_2(1) = 3$.  Because $\mathcal{T}_{2,\mono \dagger} = \{\emptyset, 1, 2, 12 \}$, we evaluate $\mathcal{T}_{2,\mono \dagger}(\mathcal{T}_{\sub}, \id)$ as
\begin{align*}
&\{\emptyset\} \cup \{(\iota_2)_*(1)\} \cup \{(\iota_1)_*(s): s \in \mathcal{T}_{\sub} \setminus \{\emptyset\}\} \cup \{(\iota_1)_*(s) (\iota_2)_*(1): s \in \mathcal{T}_{\sub} \setminus \{\emptyset\}\} \\
=& \{\emptyset\} \cup \{3\} \cup \{s: s \in \mathcal{T}_{\sub} \setminus \{\emptyset\}\} \cup \{s3: s \in \mathcal{T}_{\sub} \setminus \{\emptyset\}\}.
\end{align*}
Recall that $\mathcal{T}_{\sub} \setminus \{\emptyset\}$ consists of all alternating strings on $\{1,2\}$ which end with $1$.  When we apply $(\psi)_*$ to $\mathcal{T}_{2,\mono \dagger}(\mathcal{T}_{\sub}, \id)$, then the $3$ is replaced by a $2$.  So out of the four terms above, the first term produces the empty string, the second produces $2$, the third term produces all alternating strings on $\{1,2\}$ that end in $1$, and then the fourth term produces all alternating strings on $\{1,2\}$ that end in $12$.  Therefore, $\psi_*$ defines an isomorphism $\mathcal{T}(\id,\mathcal{T}_{\sub}) \to \mathcal{T}_{2,\free}$, which proves our claim.
\end{example}

\begin{remark}
In terms of analytic transforms, \eqref{eq:subordinationconvolution} translates to $F_{\mu \boxplus \nu} = F_\mu \circ F_{\nu \boxright \mu}$ and $G_{\mu \boxplus \nu} = G_\mu \circ F_{\mu \boxright \nu}$.  The fact that $G_{\mu \boxplus \nu} = G_\mu \circ F$ for some analytic $F$ from the upper half-plane to the upper half-plane was first observed by Voiculescu \cite[Proposition 4.4]{VoiculescuFE1} and the theory was further developed by \cite[Theorem 3.1]{Biane1998}, \cite{Voiculescu2000}, \cite{Voiculescu2002b}, and \cite{BMS2013}.  The approach of studying the subordination convolution itself is due to Lenczewski \cite[\S 7]{Lenczewski2007}, the multivariable case was handled in \cite{Nica2009}, and the operator-valued case was studied in \cite[Proposition 7.2]{Liu2018}.
\end{remark}

\begin{example}
Other identities that can be deduced from Corollary \ref{cor:convolutionidentity} in a similar fashion include \cite[eq.\ (1.7)]{Lenczewski2007} \cite[Prop.\ 7.4]{Liu2018}
\[
\mu \boxplus \nu = (\mu \boxright \nu) \uplus (\nu \boxright \mu)
\]
and \cite[\S 9]{Lenczewski2007} \cite[Prop.\ 7.8]{Liu2018}
\[
\mu \boxright \nu = \mu \vdash (\nu \boxright \mu)
\]
and \cite[Prop.\ 8.1]{Liu2018}
\[
(\mu_1 \boxplus \mu_2) \boxright \nu = (\mu_1 \boxright \nu) \boxplus (\mu_2 \boxright \nu).
\]
We leave the details as an exercise.
\end{example}

\begin{remark}
Our proofs are in some sense not new.  Indeed, the Hilbert module manipulations used in previous work precisely correspond to the manipulations of strings used here.  Our point is exactly that Corollary \ref{cor:convolutionidentity} reduces the work to manipulations of strings.
\end{remark}

\subsection{Boolean-Orthogonal Decomposition Theorem} \label{subsec:BOdecomp}

Lenczewski \cite{Lenczewski2007} considered decompositions of the free convolution into iterated Boolean and orthogonal convolutions (and similar results for product operations on graphs were given in \cite{ALS2007}).  We now show that there are Boolean-orthogonal decompositions for general $\mathcal{T}$-free convolution operations.  These decompositions will be obtained inductively from the following result.

\begin{proposition} \label{prop:BOdecomp}
Let $\mathcal{T} \in \Tree(N)$.  For $j \in [N] \cap \mathcal{T}$, let
\[
\mathcal{T}_j = \{s \in \mathcal{T}_{N,\free}: sj \in \mathcal{T}\},
\]
that is, $\mathcal{T}_j$ is the branch of $\mathcal{T}$ rooted at the vertex $j$.  Then we have
\begin{equation} \label{eq:BOdecomp}
\boxplus_{\mathcal{T}}(\mu_1,\dots,\mu_N) = \biguplus_{j \in [N] \cap \mathcal{T}} [\mu_j \vdash \boxplus_{\mathcal{T}_j}(\mu_1,\dots,\mu_N)].
\end{equation}
\end{proposition}

\begin{proof}
Note that the operation $\uplus_{j \in [N] \cap \mathcal{T}}$ on the right hand side of \eqref{eq:BOdecomp} is well-defined because Boolean convolution is commutative and associative.

Let $n = |[N] \cap \mathcal{T}|$.  Using permutation invariance, we may assume without loss of generality that $[N] \cap \mathcal{T} = \{1,\dots,n\}$ in order to simplify notation.

Let $N' = n(N+1)$.  Define $\iota_0: [n] \to [N']$ by
\[
\iota_0(i) = (N+1)(i-1) + 1
\]
and for $j = 1, \dots, n$, define $\iota_j: [N] \to [N']$ by
\[
\iota_j(i) = (N+1)(j-1) + i.
\]
In other words, the maps $\iota_j$ are defined so that
\[
(1,\dots,N') = (\iota_0(1), \iota_1(1), \dots, \iota_1(N), \iota_0(2), \iota_2(1), \dots, \iota_2(N), \dots \dots , \iota_0(n), \iota_n(1), \dots, \iota_n(N)).
\]
Note that $[N']$ is the disjoint union of the index sets $\iota_j([N])$ for $j = 1$, \dots, $n$.  Define $\psi: [N'] \to [N]$ by $\psi \circ \iota_j(i) = i$.  Let
\[
\mathcal{T}' = \mathcal{T}_{n,\Bool}(\mathcal{T}_{\orth}(\id, \mathcal{T}_1), \dots, \mathcal{T}_{\orth}(\id, \mathcal{T}_N))
\]

We claim that $\psi_*$ restricts to a graph isomorphism $\mathcal{T}' \to \mathcal{T}$.  Recall that $\mathcal{T}_{n,\Bool} = \{\emptyset, 1,\dots, n\}$ and $\mathcal{T}_{\orth} = \{\emptyset, 1, 21\}$.  In the composition $\mathcal{T}'$, the index $j \in \mathcal{T}_{n,\Bool}$ is replaced by strings from $\mathcal{T}_{\orth}(\id, \mathcal{T}_j)$ on the indices $\iota_0(j)$, $\iota_j(1)$, \dots, $\iota_j(N)$ (referring to their labels in the overall product $\mathcal{T}'$ rather than $\mathcal{T}_{\orth}(\id,\mathcal{T}_j)$), which means that $\mathcal{T}' \setminus \{\emptyset\}$ is the disjoint union of the sets of strings that arise from each terms $\mathcal{T}_{\orth}(\id,\mathcal{T}_j)$.  Meanwhile, the indices $1$ and $2$ in the $j$th copy of $\mathcal{T}_{\orth}$ are replaced respectively by $\iota_0(j)$ and by $\iota_j$ of strings in $\mathcal{T}_j$.  This means that $\mathcal{T}_{\orth}(\id,\mathcal{T}_j)$ contributes to $\mathcal{T}'$ the strings
\[
\{\iota_0(j)\} \cup \{(\iota_j)_*(s) \iota_0(j), s \in \mathcal{T}_j \setminus \{\emptyset\}\}
\]
or in other words $\{(\iota_j)_*(s) \iota_0(j), s \in \mathcal{T}_j \}$.  When we apply $\psi_*$ to these strings, we obtain precisely the strings of the form $sj$ where $s \in \mathcal{T}_j$.  Since $\mathcal{T}$ is the disjoint union of $\mathcal{T}_j \cdot j$ for $j = 1, \dots, n$, we see that $\psi_*$ defines an isomorphism $\mathcal{T}' \to \mathcal{T}$ as asserted.

It follows from Corollary \ref{cor:convolutionidentity} that
\[
\boxplus_{\mathcal{T}}(\mu_1,\dots,\mu_N) = \boxplus_{\mathcal{T}'}(\mu_{\psi(1)},\dots, \mu_{\psi(N')}),
\]
which is exactly \eqref{eq:BOdecomp}.
\end{proof}

Proposition \ref{prop:BOdecomp} can be used iteratively to obtain decompositions into Boolean and orthogonal convolutions for every finite tree $\mathcal{T}$.  Indeed, by the proposition, $\boxplus_{\mathcal{T}}$ can be decomposed into Boolean and orthogonal convolutions together with the convolution operations $\boxplus_{\mathcal{T}_j}$.  We then apply the proposition again to decompose $\boxplus_{\mathcal{T}_j}$ in terms of the convolutions for each branch of $\mathcal{T}_j$.  Continuing inductively, we will obtain a decomposition of $\mathcal{T}$ into Boolean and orthogonal convolutions, because each step will decrease the depth of the remaining branches, and when the depth of a branch becomes zero, this branch is simply the tree $\{\emptyset\}$.  (Here by the \emph{depth} of a rooted tree, we mean the maximum distance of any vertex from the root.)

Because the Boolean and orthogonal convolutions have simple descriptions in terms of the $K$ transform $K_\mu(z) = z - F_\mu(z)$, this decomposition technique provides a formula for computing the Cauchy transform of $\boxplus_{\mathcal{T}}(\mu_1,\dots,\mu_N)$ for every finite $\mathcal{T} \in \Tree(N)$.  Explicitly, \eqref{eq:BOdecomp} yields
\begin{equation} \label{eq:BOdecomp2}
K_{\boxplus_{\mathcal{T}}(\mu_1,\dots,\mu_N)}(z) = \sum_{j \in [N] \cap \mathcal{T}} K_{\mu_j}(z - K_{\boxplus_{\mathcal{T}_j}(\mu_1,\dots,\mu_N)}(z)).
\end{equation}

In fact, this allows us to approximate the $K$-transform of $\boxplus_{\mathcal{T}}(\mu_1,\dots,\mu_N)$ even when $\mathcal{T}$ is infinite because $\mathcal{T}$ can be approximated by finite trees.  For instance, we could let $\mathcal{T}_{(d)}$ be the truncation of $\mathcal{T}$ to depth $d$.  Then $\boxplus_{\mathcal{T}_{(d)}}(\mu_1,\dots,\mu_N) \to \boxplus_{\mathcal{T}}(\mu_1,\dots,\mu_N)$ as $d \to +\infty$.

The Boolean-orthogonal decompositions for $\mathcal{T}_{(d+1)}$ and $\mathcal{T}_{(d)}$ are closely related; indeed, the decomposition for $\mathcal{T}_{(d)}$ is obtained from the decomposition for $\mathcal{T}_{(d+1)}$ by replacing each of the branches at level $d+1$ by $\{\emptyset\}$.  In the formula for the Cauchy transform, this amounts to replacing the $K_{\boxplus_{\mathcal{S}}(\mu_1,\dots,\mu_N)}$ by $0$ for every branch $\mathcal{S}$ at level $d + 1$.  Intuitively, the sequence of decompositions for $\mathcal{T}_{(d)}$ can be viewed in the limit as a ``continued convolution decomposition'' for $\mathcal{T}$ analogous to the way that continued fractions are obtained from iterated addition and division operations (see \cite[p.\ 347-349]{Lenczewski2007}).

The case of digraphs is again especially interesting because the continued convolution decomposition can be expressed in terms of a fixed point equation system.

\begin{proposition}
Let $G$ be a digraph on the vertex set $[N]$.  For each $j$, let $\Walk_j(G)$ be the tree consisting of walks starting at vertex $j$, that is,
\[
\Walk_j(G) = \{\emptyset\} \cup \{j_1 \dots j_{\ell - 1} j: j \sim_G j_{\ell-1} \sim_G \dots \sim_G j_1\}.
\]
Let $\mu_1$, \dots, $\mu_N$ be non-commutative laws and let
\begin{align*}
\nu &= \boxplus_{\Walk(G)}(\mu_1,\dots,\mu_N), &
\nu_j &= \boxplus_{\Walk_j(G)}(\mu_1,\dots,\mu_N).
\end{align*}
Then we have
\begin{align*}
\nu &= \biguplus_{j \in [N]} \nu_j, &
\nu_j &= \mu_j \vdash \biguplus_{i: j \sim_G i} \nu_i.
\end{align*}
Thus, the $K$-transforms satisfy the relations
\begin{align*}
K_\nu(z) &= \sum_{j=1}^N K_{\nu_j}(z), &
K_{\nu_j}(z) &= K_{\mu_j} \left(z - \sum_{i: j \sim_G i} K_{\nu_j}(z) \right).
\end{align*}
\end{proposition}

\begin{proof}
For $S \subseteq [N]$, let us denote
\[
\Walk_S(G) = \bigcup_{j \in S} \Walk_j(G),
\]
which describes all the walks that begin at a vertex in $S$.  We also denote $S(j) = \{i: j \sim_G i\}$.  If we fix $S$ and take $\mathcal{T} = \Walk_S(G)$ in Proposition \ref{prop:BOdecomp}, then we obtain
\[
\mathcal{T}_j = \begin{cases} \bigcup_{i: j \sim_G i} \Walk_i(G) = \Walk_{S(j)}(G), & j \in S \\ \{\emptyset\}, & \text{otherwise.} \end{cases}
\]
Therefore,
\[
\boxplus_{\Walk_S(G)}(\mu_1, \dots, \mu_N) = \biguplus_{j \in S} \mu_j \vdash \boxplus_{\Walk_{S(j)}}(\mu_1,\dots,\mu_N).
\]
In particular, taking $S = \{j\}$, we get
\[
\nu_j = \boxplus_{\Walk_j(G)}(\mu_1,\dots,\mu_N) = \mu_j \vdash \boxplus_{\Walk_{S(j)}}(\mu_1,\dots,\mu_N),
\]
and substituting this back into the previous equation,
\[
\boxplus_{\Walk_S(G)}(\mu_1,\dots,\mu_N) = \biguplus_{j \in S} \nu_j.
\]
By taking $S = [N]$, we get $\nu = \biguplus_{j=1}^N \nu_j$.  Also, by combining the previous relations,
\[
\nu_j = \mu_j \vdash \boxplus_{\Walk_{S(j)}}(\mu_1,\dots,\mu_N) = \mu_j \vdash \biguplus_{i \in S(j)} \nu_i,
\]
which proves the desired convolution identities, and the relation for the analytic transforms follows immediately.
\end{proof}

This proposition provides a strategy to compute the $\Walk(G)$-free convolution of $\mu_1$, \dots, $\mu_N$.  We first find $(K_{\nu_j}(z))_{j \in [N]}$ by solving the fixed-point equation system
\[
K_{\nu_j}(z) = K_{\mu_j}(z + \sum_{i: j \sim_G i} K_{\nu_j}(z)), \qquad j = 1, \dots, N,
\]
and then obtain the $K$-transform of the convolution as $\sum_{j=1}^N K_{\nu_j}(z)$. This fixed-point equation system is a generalization of the fixed-point equations used to compute the free convolution of two laws (that case corresponds to $G = K_2$).  This suggests as an avenue for future research that the complex-analytic and numerical tools used for free convolution in \cite{Voiculescu2002b,BMS2013} should also be applied for the convolution operations associated to digraphs.

\section{The $\mathcal{T}$-free Cumulants} \label{sec:cumulants}

Up to this point, the paper has focused on the basic properties of $\mathcal{T}$-free convolutions as well as how such convolutions relate to each other through the operads $\Tree$ and $\Func(\mathcal{B})$.  The remaining sections of the paper will, for a fixed choice of $\mathcal{T}$, lay out a theory of $\mathcal{T}$-free independence that closely parallels the free, Boolean, and monotone cases.

Theorem \ref{thm:combinatorics} above provides a way to compute joint moments using the Boolean cumulants.  In this section, we construct $\mathcal{T}$-free cumulants that will aid in the analysis of $\boxplus_{\mathcal{T}}(\mu,\dots,\mu)$.  This construction generalizes the operator-valued free, Boolean, and monotone cumulants; see \S \ref{subsec:cumulantexamples} for discussion and references.

In the following, we fix $\mathcal{T}$ and denote by $n$ the number of singleton strings contained in $\mathcal{T}$, that is, $n = |\{j \in [N]: j \in \mathcal{T}\}|$, or more succinctly $n = |[N] \cap \mathcal{T}|$.  We assume throughout the section that $n \geq 2$.

\subsection{Definition of the Cumulants}

If $S$ is a totally ordered finite set and $\pi \in \mathcal{NC}(S)$, it will be convenient to work sometimes with colorings $\chi: S \to [N]$ and sometimes with colorings defined on the blocks of $\pi$, that is, functions $\pi \to [N]$.  Note that for $\pi \in \mathcal{NC}(S)$, the space $[N]^\pi$ can be canonically identified with the subspace of $[N]^S$ consisting of functions which are constant on each block.  We use this identification throughout the rest of the paper.

\begin{definition}
Let $\pi \in \mathcal{NC}(S)$ and $\chi \in [N]^\pi \subseteq [N]^S$.  The \emph{$\chi$-components of $\pi$} are the connected components of the graph formed by removing from $\graph(\pi)$ the root vertex $\emptyset$ and all the edges between blocks with different colors under $\chi$.   Each $\chi$-component $\pi'$ can be viewed a subset of $\pi$, and $\pi$ is the disjoint union of its $\chi$-components.   If $\chi \in [N]^\pi$, we define $\pi / \chi$ to be the partition obtained by joining each of the $\chi$-components of $\pi$ into a single block.  Note that $\pi / \chi$ is still non-crossing and $\chi$ also defines a coloring on $\pi / \chi$.  Moreover, if $\pi'$ is a $\chi$-component of $\pi$, then $\pi'$ is a non-crossing partition of the corresponding block in $\pi / \chi$.  For example, see Figure \ref{fig:chicomponents}.
\end{definition}

\begin{figure}

\begin{center}

\begin{tikzpicture}[scale = 0.6]

\begin{scope}

	\node at (-2,2.5) {$\pi$};

	\foreach \j in {1,...,15} {
		\node[label=below:{\j}] ({\j}) at (\j,0) {};
	}

	\begin{scope}
		\fill (1,0) circle (0.2);
		\fill (5,0) circle (0.2);
		\fill (12,0) circle (0.2);
		\draw (1) to (1,4) to (12,4) to (12);
		\draw (5) to (5,4);
		\node at (6.5,4.5) {$V_1$};
	\end{scope}

	\begin{scope}
		\draw (2,0) circle (0.2);
		\draw (3,0) circle (0.2);
		\draw[dashed] (2) to (2,1) to (3,1) to (3);
		\node at (2.5,1.5) {$V_2$};
	\end{scope}

	\begin{scope}
		\fill (4,0) circle (0.2);
		\node at (4,0.7) {$V_3$};
	\end{scope}

	\begin{scope}
		\draw (6,0) circle (0.2);
		\draw (9,0) circle (0.2);
		\draw (11,0) circle (0.2);
		\draw[dashed] (6) to (6,2.5) to (11,2.5) to (11);
		\draw[dashed] (9) to (9,2.5);
		\node at (8.5,3) {$V_4$};
	\end{scope}

	\begin{scope}
		\draw (7,0) circle (0.2);
		\draw (8,0) circle (0.2);
		\draw[dashed] (7) to (7,1) to (8,1) to (8);
		\node at (7.5,1.5) {$V_5$};
	\end{scope}

	\begin{scope}
		\fill (10,0) circle (0.2);
		\node at (10,0.7) {$V_6$};
	\end{scope}
	
	\begin{scope}
		\fill (13,0) circle (0.2);
		\fill (15,0) circle (0.2);
		\draw (13) to (13,1.5) to (15,1.5) to (15);
		\node at (14,2) {$V_7$};
	\end{scope}

	\begin{scope}
		\draw (14,0) circle (0.2);
		\node at (14,0.7) {$V_8$};
	\end{scope}
\end{scope}

\begin{scope}[shift = {(8,-2)}]
	\node at (-10,-3) {$\graph(\pi)$};

	\node[circle,draw,dotted] (V0) at (0,0) {$\emptyset$};
	\node[circle,draw,thick] (V1) at (-2,-2) {$V_1$};
	\node[circle,draw,thick] (V2) at (-4,-4) {$V_2$};
	\node[circle,draw,thick] (V3) at (-2,-4) {$V_3$};
	\node[circle,draw] (V4) at (0,-4) {$V_4$};
	\node[circle,draw] (V5) at (-1,-6) {$V_5$};
	\node[circle,draw,thick] (V6) at (1,-6) {$V_6$};
	\node[circle,draw,thick] (V7) at (2,-2) {$V_7$};
	\node[circle,draw] (V8) at (2,-4) {$V_8$};
	
	\draw[dotted] (V0) to (V1);
	\draw[dotted] (V1) to (V2);
	\draw (V1) to (V3);
	\draw[dotted] (V1) to (V4);
	\draw[dashed] (V4) to (V5);
	\draw[dotted] (V4) to (V6);
	\draw[dotted] (V0) to (V7) to (V8);
\end{scope}

\begin{scope}[shift = {(0,-14)}]

	\node at (-2,2.5) {$\pi / \chi$};

	\foreach \j in {1,...,15} {
		\node[label=below:{\j}] ({\j}) at (\j,0) {};
	}

	\begin{scope}
		\fill (1,0) circle (0.2);
		\fill (4,0) circle (0.2);
		\fill (5,0) circle (0.2);
		\fill (12,0) circle (0.2);
		\draw (1) to (1,3) to (12,3) to (12);
		\draw (4) to (4,3);
		\draw (5) to (5,3);
		\node at (6.5,3.5) {$V_1 \cup V_3$};
	\end{scope}

	\begin{scope}
		\draw (2,0) circle (0.2);
		\draw (3,0) circle (0.2);
		\draw[dashed] (2) to (2,1) to (3,1) to (3);
		\node at (2.5,1.5) {$V_2$};
	\end{scope}

	\begin{scope}
		\draw (6,0) circle (0.2);
		\draw (7,0) circle (0.2);
		\draw (8,0) circle (0.2);
		\draw (9,0) circle (0.2);
		\draw (11,0) circle (0.2);
		\draw[dashed] (6) to (6,1.5) to (11,1.5) to (11);
		\draw[dashed] (7) to (7,1.5);
		\draw[dashed] (8) to (8,1.5);
		\draw[dashed] (9) to (9,1.5);
		\node at (8.5,2) {$V_4 \cup V_5$};
	\end{scope}

	\begin{scope}
		\fill (10,0) circle (0.2);
		\node at (10,0.7) {$V_6$};
	\end{scope}
	
	\begin{scope}
		\fill (13,0) circle (0.2);
		\fill (15,0) circle (0.2);
		\draw (13) to (13,1.5) to (15,1.5) to (15);
		\node at (14,2) {$V_7$};
	\end{scope}

	\begin{scope}
		\draw (14,0) circle (0.2);
		\node at (14,0.7) {$V_8$};
	\end{scope}
\end{scope}

\end{tikzpicture}

\caption{At top, we show a non-crossing partition $\pi$ of $[15]$ into 8 blocks, together with a coloring $\chi: \pi \to \{1,2\} \cong \{\text{black}, \text{white}\}$, where the elements with $\chi = 1$ are dark with plain lines and the elements with $\chi = 2$ are lighter with dashed lines.  At middle, we depict $\graph(\pi)$. We make the nodes for blocks with $\chi = 1$ darker and the lines connecting the adjacent blocks for $\chi = 2$ dashed.  We draw the root vertex and edges between different colors with dotted lines.  At bottom, we show $\pi / \chi$.  The $\chi$-components in this example are $\{V_1,V_3\}$, $\{V_2\}$, $\{V_4,V_5\}$, $\{V_6\}$, $\{V_7\}$, $\{V_8\}$.  Note that although $V_1$ and $V_4$ are next to each other and the same color, they are not adjacent in the graph and are in different $\chi$-components.}  \label{fig:chicomponents}

\end{center}

\end{figure}
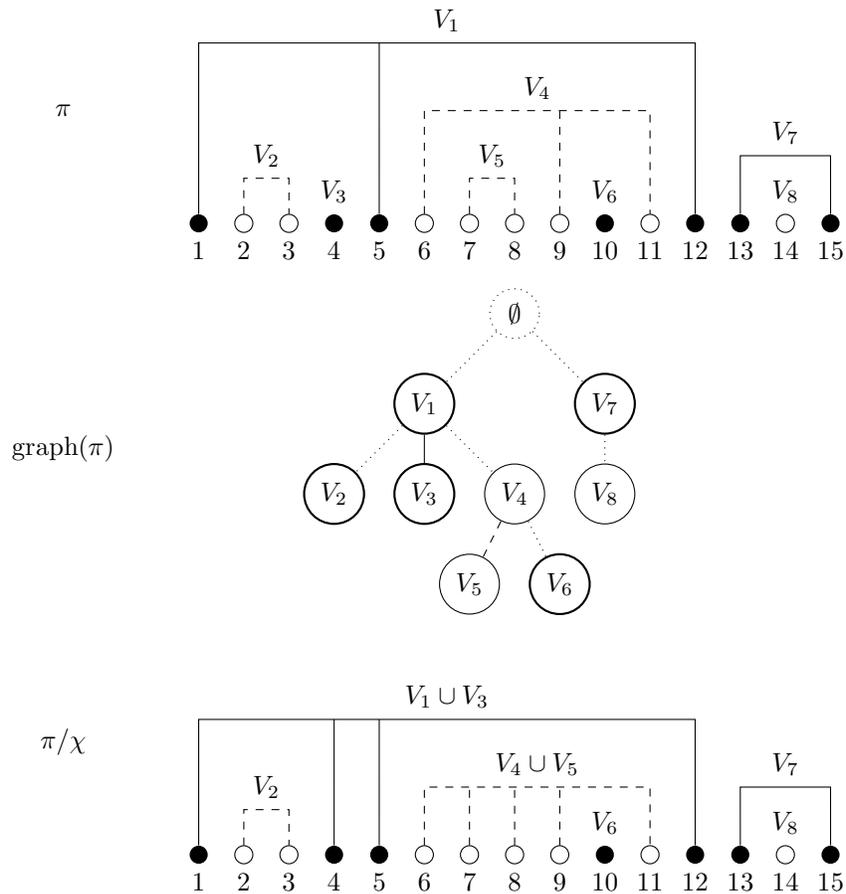

\begin{definition}
Let $S$ be a finite totally ordered set and $\chi: S \to [N]$.  We define
\[
\mathcal{NC}_w(\chi,\mathcal{T}) = \{\pi \in \mathcal{NC}(\chi): \pi / \chi \in \mathcal{NC}(\chi,\mathcal{T})\}.
\]
We also define
\[
\mathcal{X}_w(\pi,\mathcal{T}) = \{\chi \in [N]^S: \pi \in \mathcal{NC}_w(\chi,\mathcal{T})\}.
\]
By definition, we have $\pi \in \mathcal{NC}_w(\chi,\mathcal{T})$ if and only if $\chi \in \mathcal{X}_w(\pi,\mathcal{T})$.  In this case, we say that $\pi$ is \emph{weakly compatible with } $\chi$ and $\mathcal{T}$.  (The ``$w$'' in the above notations stands for ``weak.'')
\end{definition}

\begin{remark} \label{rem:weaklycompatible}
This condition can be equivalently expressed as follows.  Recall that $\red(s)$ denotes the alternating reduction of a string $s$ (Definition \ref{def:string}).  We have $\pi \in \mathcal{NC}_w(\chi,\mathcal{T})$ if and only if for every $V \in \pi$ with $\chain(V) = (V,V_1,\dots,V_d)$, we have
\[
\red[\chi(\chain(V))] = \red[\chi(V) \chi(V_1) \dots \chi(V_d)] \in \mathcal{T}.
\]
The reason for this is that if $V \in \pi$ and $V'$ is the block of $\pi / \chi$ containing $V$, then we have $\chi(\chain(V')) = \red[\chi(\chain(V))]$.  Indeed, whenever there are a repeated consecutive letters in the string $\chi(\chain(V))$, the corresponding blocks of $\pi$ will be put into the same $\chi$-component, and hence they will become a single block in $\pi / \chi$.
\end{remark}

\begin{lemma} \label{lem:cumulantcoefficients}
There exist unique coefficients $\alpha_{\mathcal{T},\pi}$ for partitions $\pi$ of totally ordered finite sets $S$ such that
\begin{equation} \label{eq:partitioncoefficientnormalization}
|\pi| = 1 \implies \alpha_{\mathcal{T},\pi} = 1
\end{equation}
and
\begin{equation} \label{eq:partitioncoefficientidentity}
\alpha_{\mathcal{T},\pi} = \frac{1}{n^{|\pi|}} \sum_{\chi \in \mathcal{X}_w(\pi,\mathcal{T})} \prod_{\text{$\chi$-components } \pi'} \alpha_{\mathcal{T},\pi'}.
\end{equation}
Moreover, we have $\alpha_{\mathcal{T},\pi} \geq 0$.  Of course, the coefficients are invariant under changing coordinates from $S$ to another set $S'$ by an order-preserving bijection.
\end{lemma}

\begin{proof}
We define the coefficients $\alpha_{\mathcal{T},\pi}$ by induction on $|\pi|$.  In this case $|\pi| = 1$, we set $\alpha_{\mathcal{T},\pi} = 1$.

Now suppose that $|\pi| > 1$.  Suppose that $\pi$ is reducible (that is, $\pi \not \in \mathcal{NC}^\circ(S)$).  Then $\graph(\pi) \setminus \{\emptyset\}$ has multiple components.  Thus, every coloring of $\pi$ will produce multiple $\chi$-components, so on the right hand side of \eqref{eq:partitioncoefficientidentity} every value of $\pi'$ satisfies $|\pi'| < |\pi|$.  By induction, $\alpha_{\mathcal{T},\pi'}$ is defined, and we define $\alpha_{\mathcal{T},\pi}$ by \eqref{eq:partitioncoefficientidentity}.

On the other hand, suppose that $\pi \in \mathcal{NC}^\circ(S)$.  Then $\alpha_{\mathcal{T},\pi}$ occurs both on the left and the right hand sides of \eqref{eq:partitioncoefficientidentity}; it occurs on the right hand for each coloring $\chi$ which is constant.  For constant $\chi$, we have $\pi \in \mathcal{NC}_w(\chi,\mathcal{T})$ if and only if the constant value of $\chi$ is one of the singleton strings in $\mathcal{T}$.  Thus, there are $n$ terms for the constant colorings $\chi$ on the right hand side.  Thus, subtracting $(n / n^{|\pi|})\alpha_{\mathcal{T},\pi}$ on both sides, we see that \eqref{eq:partitioncoefficientidentity} is equivalent to
\begin{equation} \label{eq:partitioncoefficientidentity2}
\left(1 - \frac{1}{n^{|\pi|-1}} \right) \alpha_{\mathcal{T},\pi} = \frac{1}{n^{|\pi|}} \sum_{\text{non-constant } \chi \in \mathcal{X}_w(\pi,\mathcal{T})} \prod_{\text{$\chi$-components } \pi'} \alpha_{\mathcal{T},\pi'}.
\end{equation}
For non-constant $\chi$, each of the $\chi$-components $\pi'$ satisfies $|\pi'| < |\pi|$, so that $\alpha_{\mathcal{T},\pi'}$ is well-defined by induction hypothesis.  On the left hand side, because $|\pi| > 1$ and $n \geq 2$, we have $1 - 1 / n^{|\pi| -1} \neq 0$ and therefore there is a unique value of $\alpha_{\mathcal{T},\pi}$ satisfying the equation.

Therefore, there is a unique collection of coefficients $\alpha_{\mathcal{T},\pi}$ satisfying \eqref{eq:partitioncoefficientnormalization} for $|\pi| = 1$ and \eqref{eq:partitioncoefficientidentity} for $|\pi| > 1$.  Moreover, \eqref{eq:partitioncoefficientidentity} also holds for $|\pi| = 1$ because there are exactly $n$ colorings of $\pi$ that are compatible with $\mathcal{T}$.

The fact that $\alpha_{\mathcal{T},\pi} \geq 0$ follows by induction from \eqref{eq:partitioncoefficientidentity2}.  The invariance under change of coordinates is left as an exercise.
\end{proof}

\begin{definition}
Let $\mathcal{T} \in \Tree(N)$ and suppose that $n \geq 2$ as above, and let $(\mathcal{A},E)$ be a $\mathcal{B}$-valued probability space.  We define the \emph{$\mathcal{T}$-free cumulants} $K_{\mathcal{T},\ell}: \mathcal{A}^\ell \to \mathcal{B}$ by the relations
\begin{equation} \label{eq:momentcumulantformula}
E[a_1 \dots a_\ell] = \sum_{\pi \in \mathcal{NC}(\ell)} \alpha_{\mathcal{T},\pi} K_{\mathcal{T},\pi}[a_1,\dots,a_\ell].
\end{equation}
\end{definition}

These cumulants are well-defined by M\"obius inversion (Lemma \ref{lem:partitionMobiusinversion}) because $\alpha_{\mathcal{T},\pi} = 1$ for every partition $\pi$ with one block.  In \S \ref{subsec:extensivity}, we show that these cumulants satisfy similar axioms to those of \cite{HS2011b}, which justifies our choice of definition.  In \S \ref{subsec:cumulantexamples} below, we verify directly that in the free, Boolean, and monotone cases, our definition reduces to the definitions given in previous literature.

Since our general moment formula (Theorem \ref{thm:combinatorics}) relies on the Boolean cumulants, it will be convenient to express the Boolean cumulants in terms of the $\mathcal{T}$-free cumulants.  In the lemma below, $K_{\Bool,\ell}$ is given by Definition \ref{def:Booleancumulants}, since we have not yet shown that this agrees with $\mathcal{T}_{N,\Bool}$-free cumulants.

\begin{lemma} \label{lem:TBcumulantconversion}
Let $(\mathcal{A},E)$ be a $\mathcal{B}$-valued probability space.  Then we have
\[
K_{\Bool,\ell}[a_1,\dots,a_\ell] = \sum_{\pi \in \mathcal{NC}^\circ(\ell)} \alpha_\pi K_{\mathcal{T},\pi}[a_1,\dots,a_\ell].
\]
\end{lemma}

\begin{remark}
The conversion between classical, free, Boolean, and monotone cumulants is known (see \cite{Lehner2002}, \cite{BN2008b}, \cite{AHLV2015}), and the lemma here includes the two easiest cases, namely converting free or monotone cumulants to Boolean cumulants.
\end{remark}

\begin{proof}[Proof of Lemma \ref{lem:TBcumulantconversion}]
Let $\Gamma_\ell[a_1,\dots,a_\ell]$ be the quantity on the right hand side.  By uniqueness of the Boolean cumulants, it suffices to show that
\[
E[a_1 \dots a_\ell] = \sum_{\pi \in \mathcal{I}(\ell)} \Gamma_\pi[a_1,\dots,a_\ell].
\]
Suppose that $\pi \in \mathcal{I}(\ell)$ and that the blocks of $\pi$ are listed as $V_1$, \dots, $V_{|\pi|}$ in order from the left to right.  Then we have
\[
\Gamma_\pi[a_1,\dots,a_\ell] = \Gamma_{V_1}[a_1,\dots,a_{\max V_1}] \dots \Gamma_{V_{|\pi|}}[a_{\min V_{|\pi|}}, \dots, a_\ell]
\]
For brevity, we denote this as
\[
\Gamma_\pi[a_1,\dots,a_\ell] = \prod_{V \in \pi} \Gamma_V[a_j: j \in V],
\]
where terms in the product are understood to be multiplied from left to right and the indices $(a_j: j \in V)$ are understood to run from left to right.  Substituting in the definition of $\Gamma_\ell$, we obtain
\begin{align*}
\Gamma_\pi[a_1,\dots,a_\ell] &= \prod_{V \in \pi} \sum_{\tau_V \in \mathcal{NC}^\circ(V)} K_{\mathcal{T},\tau_V}[a_j: j \in V] \\
&= \sum_{\substack{\tau_V \in \mathcal{NC}^\circ(V) \\ \text{for each } V}} \prod_{V \in \pi} \alpha_{\tau_V} K_{\mathcal{T},\tau_V}[a_j: j \in V].
\end{align*}
Let $\tau$ be the partition $\tau = \bigsqcup_{V \in \pi} \tau_V$; in other words, $\tau$ is the partition obtained by subdividing each block $V$ of $\pi$ according to $\tau_V$.  Then $\tau$ is a non-crossing partition of $[\ell]$ and we have
\[
\prod_{V \in \pi} K_{\mathcal{T},\tau_V}[a_j: j \in V] = K_{\mathcal{T},\tau}[a_1,\dots,a_\ell].
\]
We also claim, and will verify at the end of the proof, that
\[
\alpha_{\mathcal{T},\tau} = \prod_{V \in \pi} \alpha_{\mathcal{T},\tau_V},
\]
so that
\[
\prod_{V \in \pi} \alpha_{\tau_V} K_{\mathcal{T},\tau_V}[a_j: j \in V] = \alpha_{\mathcal{T},\tau} K_{\mathcal{T},\tau}[a_1,\dots,a_\ell].
\]
Every partition $\tau$ can be obtained uniquely in this way from an interval partition $\pi$ and a tuple of partitions $\tau_V \in \mathcal{NC}^\circ(V)$ for each block $V$ of $\pi$.  Indeed, the partitions $\tau_V$ correspond to the components of $\graph(\pi) \setminus \{\emptyset\}$.  Therefore, we have
\[
\sum_{\pi \in \mathcal{I}(\ell)} \Gamma_\pi[a_1,\dots,a_\ell] = \sum_{\tau \in \mathcal{NC}(\ell)} \alpha_{\mathcal{T},\tau} K_\tau[a_1,\dots,a_\ell] = E[a_1 \dots a_\ell]
\]
as desired.

It remains to show that $\alpha_{\mathcal{T},\tau} = \prod_{V \in \pi} \alpha_{\mathcal{T},\tau_V}$ whenever $\pi$, $\tau_V$, and $\tau$ are as above.  From \eqref{eq:partitioncoefficientidentity}, we have
\[
\alpha_{\mathcal{T},\tau} = \frac{1}{n^{|\tau|}} \sum_{\chi \in \mathcal{X}_w(\tau,\mathcal{T})} \prod_{\text{$\chi$-components } \tau'} \alpha_{\mathcal{T},\tau'}.
\]
There is a bijective correspondence between colorings $\chi$ of $\tau$ and tuples of colorings $\chi_V$ of $\tau_V$ for each $V \in \pi$, given by $\chi_V = \chi|_V$.  One can check that
\begin{itemize}
	\item $\chi \in \mathcal{X}_w(\tau,\mathcal{T})$ if and only if $\chi_V \in \mathcal{X}_w(\tau_V,\mathcal{T})$ for each $V \in \pi$;
	\item every $\chi$-component of $\tau$ is contained in some block $V$ of $\pi$, and in fact the $\chi$-components of $\tau$ contained in $V$ are precisely the $\chi_V$-components of $\tau_V$;
	\item we have $|\tau| = \sum_{V \in \pi} |\tau_V|$ and hence $n^{|\tau|} = \prod_{V \in \pi} n^{|\tau_V|}$.
\end{itemize}
Therefore, altogether
\[
\alpha_{\mathcal{T},\tau} = \prod_{V \in \pi} \frac{1}{n^{|\tau_V|}} \sum_{\chi_V \in \mathcal{X}_w(\tau_V,\mathcal{T}) \text{ for each } V} \prod_{\chi_V \text{-components } \tau_V'} \alpha_{\mathcal{T},\tau_V'} = \prod_{V \in \pi} \alpha_{\mathcal{T},\tau_V}.
\]
\end{proof}

\subsection{Extensivity and Axiomatic Characterization} \label{subsec:extensivity}

The next theorem shows that the $\mathcal{T}$-free cumulants satisfy the extensivity property that was discussed in \cite[\S 3]{HS2011b} in the free, Boolean, and monotone cases.  This property implies that the cumulants can be used to compute iterated convolutions of a $\mathcal{B}$-valued law, which will be useful in the next section for the central limit theorem (see Observation \ref{obs:lawcumulantextensivity}).

\begin{theorem} \label{thm:extensivity}
Suppose that $\mathcal{T} \in \Tree(N)$.  Let $(\mathcal{A},E)$ be a $\mathcal{B}$-valued probability space.  Let $\tilde{\mathcal{A}}$ be the $\mathcal{T}$-free product of $N$ copies of $\mathcal{A}$ and let $\lambda_{\mathcal{T},j}: \mathcal{A} \to \tilde{\mathcal{A}}$ be the map from the $j$th factor into the product.  Then we have
\begin{equation} \label{eq:extensivity}
K_{\mathcal{T},\ell}\left[ \sum_{j=1}^N \lambda_{\mathcal{T},j}(a_1), \dots, \sum_{j=1}^N \lambda_{\mathcal{T},j}(a_\ell) \right] = n K_{\mathcal{T},\ell}[a_1,\dots,a_\ell].
\end{equation}
\end{theorem}

\begin{proof}
Because the $\mathcal{T}$-cumulants can be recovered from the moments by M\"obius inversion, it suffices to show that
\begin{equation} \label{eq:extensivityproof1}
E \left[ \left( \sum_{j=1}^N \lambda_j(a_1) \right) \ldots \left( \sum_{j=1}^N \lambda_j(a_\ell) \right) \right] = \sum_{\pi \in \mathcal{NC}(\ell)} \alpha_\pi n^{|\pi|} K_{\mathcal{T},\pi}[a_1,\dots,a_\ell].
\end{equation}
We expand the left hand side by multilinearity then apply Theorem \ref{thm:combinatorics} to conclude that
\begin{align*}
\sum_{\chi: [\ell] \to [N]} E[\lambda_{\chi(1)}(a_1) \dots \lambda_{\chi(\ell)}(a_\ell)] &= \sum_{\chi: [\ell] \to [N]} \sum_{\pi \in \mathcal{NC}(\chi,\mathcal{T})} K_{\Bool,\pi}[a_1,\dots,a_\ell] \\
&= \sum_{\pi \in \mathcal{NC}(\ell)} \sum_{\substack{\chi \in [N]^\pi \\ \pi \in \mathcal{NC}(\chi,\mathcal{T})}} K_{\Bool,\pi}[a_1,\dots,a_\ell].
\end{align*}
For each block $V$ of $\pi$, we can express $K_{\Bool,|V|}$ in terms of the $\mathcal{T}$-free cumulants by Lemma \ref{lem:TBcumulantconversion}.  This results in a sum over partitions $\tau_V \in \mathcal{NC}^\circ(V)$ for each block $V \in \pi$.  Given partitions $\tau_V \in \mathcal{NC}(V)$ for each $V \in \pi$, we view $\tau_V$, the union $\bigsqcup_{V \in \pi} \tau_V$ defines a non-crossing partition of $[\ell]$, which is the partition obtained by subdividing each block $V$ of $\pi$ according to the partition $\tau_V$.  The above expression then becomes
\[
\sum_{\pi \in \mathcal{NC}(\ell)} \sum_{\substack{\chi \in [N]^\pi \\ \pi \in \mathcal{NC}(\chi,\mathcal{T})}} \sum_{\substack{\tau_V \in \mathcal{NC}(V) \\ \text{for each } V \in \pi}} \prod_{V \in \pi} \alpha_{\mathcal{T},\tau_V} K_{\mathcal{T},\bigsqcup_{V \in \pi} \tau_V}[a_1,\dots,a_\ell].
\]
For each choice of $\pi$, $\chi$, and $(\tau_V)_{V \in \pi}$, let $\tau = \bigsqcup_{V \in \pi} \tau_V$.  Then $\chi$ defines a coloring of $\tau$.  Moreover, we have $\tau / \chi = \pi$ and the $\chi$-components of $\tau$ are precisely $\{\tau_V: V \in \pi\}$.  Conversely, every choice of $\tau$ and $\chi \in [N]^\tau$ arises in this way from a unique choice of $\pi$ and $(\tau_V)_{V \in \pi}$.  Thus, the left hand side of \eqref{eq:extensivityproof1} is equal to
\[
\sum_{\tau \in \mathcal{NC}(\ell)} \left( \sum_{\substack{\chi \in [N]^\tau \\ \tau / \chi \in \mathcal{NC}(\chi,\mathcal{T})}} \prod_{\chi \text{-components } \tau' \text{ of } \tau} \alpha_{\mathcal{T},\tau} \right) K_{\mathcal{T},\tau}[a_1,\dots,a_\ell].
\]
Now $\tau / \chi \in \mathcal{NC}(\chi,\mathcal{T})$ is equivalent to $\tau \in \mathcal{NC}_w(\chi,\mathcal{T})$.  The condition that $\tau / \chi \in \mathcal{NC}(\chi,\mathcal{T})$ is equivalent by definition to $\chi \in \mathcal{X}_w(\chi,\mathcal{T})$.  So using \eqref{eq:partitioncoefficientidentity}, this becomes
\[
\sum_{\tau \in \mathcal{NC}(\ell)} n^{|\tau|} \alpha_{\mathcal{T},\tau} K_{\mathcal{T},\tau}[a_1,\dots,a_\ell],
\]
which is the right hand side of \eqref{eq:extensivityproof1}.
\end{proof}

In fact, the cumulants are uniquely characterized by extensivity and polynomial dependence of the moments on the cumulants.  The following proposition is a generalization of \cite[Theorem 3.1]{HS2011b} to other independences and to the operator-valued setting, with essentially the same proof.

\begin{proposition} \label{prop:cumulantaxioms}
Let $\mathcal{B}$ be a given $\mathrm{C}^*$-algebra and let $\mathcal{T} \in \Tree(N)$ with $n = |[N] \cap \mathcal{T}| \geq 2$.  Suppose we are given, for every $\mathcal{B}$-valued probability space $(\mathcal{A},E)$ and every $\ell \geq 1$, a map $\Gamma_\ell: \mathcal{A}^\ell \to \mathcal{B}$ (where the dependence on $\mathcal{A}$ is suppressed in the notation) such that the following axioms are satisfied:
\begin{enumerate}[(1)]
	\item \emph{Multilinearity:} $\Gamma_\ell$ is $\mathcal{B}$-quasi-multilinear.
	\item \emph{Polynomiality:} For each $\ell \geq 1$, there exists a $J \geq 0$, natural numbers $\ell_1$, \dots, $\ell_J$, complex numbers $\beta_j$, maps $\psi_j: [\ell_j] \to [\ell]$, and partitions $\pi_j \in \mathcal{NC}(\ell_j)$ with $|\pi_j| > 1$, such that
	\[
	E[a_1 \dots a_\ell] = \Gamma_\ell[a_1, \dots, a_\ell] + \sum_{j=1}^J \beta_j \Gamma_{\pi_j}[a_{\psi(1)}, \dots, a_{\psi(\ell_j)}],
	\]
	where the objects $J$, $\ell_j$, $\psi_j$, and $\pi_j$ are independent of the algebra $\mathcal{A}$.
	\item \emph{Extensivity:} If $\tilde{\mathcal{A}}$ is the $\mathcal{T}$-free product of $N$ copies of $\mathcal{A}$ and if
	\[
	\Gamma_\ell \left[ \sum_{j=1}^N \lambda_{\mathcal{T},j}(a_1), \dots, \sum_{j=1}^N \lambda_{\mathcal{T},j}(a_\ell) \right] = n \Gamma_\ell[a_1,\dots,a_\ell].
	\]
\end{enumerate}
Then $\Gamma_\ell = K_{\mathcal{T},\ell}$.
\end{proposition}

\begin{remark}
Regarding axiom (2), we remark that ``polynomial dependence'' requires a modified statement in the operator-valued setting.  Indeed, in the scalar-valued setting, the cumulants are $\C$-multilinear, so that $K_\pi[a_1,\dots,a_\ell]$ is the product of $K_V[a_j: j \in V]$ over all blocks $V$ of $\pi$, but in the operator-valued setting, the composition of multilinear forms is not as simple as a product.
\end{remark}

\begin{remark}
Here we have stated the polynomial dependence of the moments on the cumulants rather than polynomial dependence of the cumulants on the moments as in \cite{HS2011b}.  However, if one assumes in (2) that the blocks of $\pi_j$ all have size $< \ell$, then these two conditions are equivalent by M{\"o}bius inversion.
\end{remark}

\begin{proof}[Proof of Proposition \ref{prop:cumulantaxioms}]
We have shown that $K_{\mathcal{T},\ell}$ satisfies the three axioms, so it suffices to show uniqueness of $\Gamma_\ell$'s satisfying the three axioms.  Let $(\Gamma_\ell)_{\ell \in \N}$ and $(\tilde{\Gamma}_\ell)_{\ell \geq 1}$ be two such sequences of multilinear forms.

Fix $\ell$ and fix $a_1$, \dots, $a_\ell$ in a $\mathcal{B}$-valued probability space $(\mathcal{A},E)$.  Let $\mathcal{A}^{(k)}$ be defined inductively by saying $\mathcal{A}^{(1)} = \mathcal{A}$ and that $\mathcal{A}^{(k+1)}$ is the $\mathcal{T}$-free product of $N$ copies of $\mathcal{A}^{(k)}$.  Let $a_i^{(k)} \in \mathcal{A}^{(k)}$ be defined inductively by saying that $a_i = a_i$ and $a_i^{(k+1)} = \sum_{j=1}^N \lambda_{\mathcal{T},j}(a_i^{(k+1)})$.

It follows from the axioms that
\[
E[a_1^{(k)} \dots a_\ell^{(k)}] = n^k \Gamma_\ell[a_1,\dots,a_\ell] + \sum_{j=1}^J n^{k|\pi_j|} \Gamma_{\pi_j}[a_{\psi(1)}, \dots, a_{\psi(\ell_j)}].
\]
Consider the $\mathcal{B}$-valued polynomial
\[
p(t) = t \Gamma_\ell[a_1,\dots,a_\ell] + \sum_{j=1}^J t^{|\pi_j|} \beta_j \Gamma_{\pi_j}[a_{\psi(1)}, \dots, a_{\psi(\ell_j)}].
\]
Let $\tilde{p}$ be the corresponding polynomial for $\tilde{\Gamma}_\ell$.  Then we have for each $k$ that
\[
p(n^k) = E[a_1^{(k)} \dots a_\ell^{(k)}] = \tilde{p}(n^k)
\]
since $E[a_1^{(k)} \dots a_\ell^{(k)}]$ is defined independently of $\Gamma_\ell$.  Since these two polynomials agree at infinitely many values of $t$, we may equate their coefficients for each power of $t$.  Because we have $|\pi_j| > 2$ for each $j$, the linear term of $p$ is $\Gamma_\ell[a_1,\dots,a_\ell]$ and similarly the linear term of $\tilde{p}$ is $\tilde{\Gamma}_\ell[a_1,\dots,a_\ell]$, and hence $\Gamma_\ell[a_1,\dots,a_\ell] = \tilde{\Gamma}_\ell[a_1,\dots,a_\ell]$.
\end{proof}

\subsection{The Free, Monotone, and Boolean Cases} \label{subsec:cumulantexamples}

Now we show that in the free, Boolean, and monotone cases our definition of the cumulants agrees with the definitions given in previous literature for those special cases.  Of course, one could deduce this directly from Proposition \ref{prop:cumulantaxioms} because the free, Boolean, and monotone cumulants are known to satisfy these axioms (compare \cite{HS2011b}).  But we would rather give a direct computation of the coefficients $\alpha_{\mathcal{T},\pi}$ using \eqref{eq:partitioncoefficientidentity} because that will shed light on the intuition behind our construction of cumulants.

\begin{proposition} \label{prop:freecumulants}
We have $\alpha_{\mathcal{T}_{N,\free},\pi} = 1$ for every non-crossing partition $\pi$.  In particular, \eqref{eq:momentcumulantformula} reduces to the free moment-cumulant formula used in \cite{Speicher1994} \cite{Speicher1998}.  The coefficients and hence the cumulants for $\mathcal{T}_{N,\free}$ are independent of $N$, and we denote these cumulants by $K_{\free,\ell}$.
\end{proposition}

\begin{proof}
In this case $n = N$, and it suffices to show that the coefficients $\alpha_{\free,\pi} = 1$ satisfy the fixed point equation
\[
\alpha_{\free,\pi} = \sum_{\chi \in \mathcal{X}_w(\pi,\mathcal{T}_{N,\free})} \frac{1}{N^{|\pi|}} \prod_{\chi \text{-components } \pi'} \alpha_{\free,\pi'}.
\]
The left hand side is $1$.  On the right hand side, note that every coloring $\chi$ of $\pi$ is weakly compatible with $\pi$ and $\mathcal{T}_{N,\free}$, and the number of colorings is exactly $N^{|\pi|}$, so the right hand side is also $1$.
\end{proof}

\begin{proposition} \label{prop:booleancumulants}
We have $\alpha_{\mathcal{T}_{N,\Bool},\pi} = 1$ if $\pi$ is an interval partition and $\alpha_{\mathcal{T}_{N,\Bool},\pi} = 0$ otherwise.  In particular, \eqref{eq:momentcumulantformula} reduces to the Boolean moment-cumulant formula used in \cite{SW1997} \cite{Popa2009} \cite[\S 3]{PV2013}, and $K_{\mathcal{T}_{N,\Bool},\ell}$ agrees with $K_{\Bool,\ell}$ given in Definition \ref{def:Booleancumulants}.
\end{proposition}

\begin{proof}
It suffices to show that the coefficients $\alpha_{\Bool,\pi}$ given by $1$ if $\pi$ is an interval partition and zero otherwise satisfy the fixed point equation
\[
\alpha_{\Bool,\pi} = \sum_{\chi \in \mathcal{X}_w(\pi,\mathcal{T}_{N,\Bool})} \frac{1}{N^{|\pi|}} \prod_{\chi \text{-components } \pi'} \alpha_{\Bool,\pi'}.
\]
If $\pi$ is an interval partition, then for every coloring $\chi$, the partition $\pi / \chi$ is also an interval partition and hence contained in $\mathcal{NC}(\chi,\mathcal{T}_{N,\Bool})$, so that $\chi \in \mathcal{X}_w(\pi, \mathcal{T}_{N,\Bool})$.  Moreover, each $\chi$-component is an interval partition, and therefore every coloring of $\pi$ contributes a $1$ to the sum on the right hand side is $1$, so the right hand side is $1$.  On the other hand, suppose that $\pi$ is not an interval partition, so there are some blocks $V \prec W$.  For $\pi / \chi$ to be in $\mathcal{NC}(\chi,\mathcal{T}_{N,\Bool})$, it must be an interval partition, and thus $V$ and $W$ must be in the same $\chi$-component.  But then this $\chi$-component is not an interval partition, so this coloring $\chi$ contributes zero to the right hand side.  Thus, both sides of the equation are zero in this case.
\end{proof}

\begin{proposition} \label{prop:monotonecumulants}
Let $\Ord(\pi)$ be the set of total orders on the blocks of $\pi$ which extend the partial order $\prec$.  Then we have
\[
\alpha_{\mathcal{T}_{N,\mono},\pi} = \frac{|\Ord(\pi)|}{|\pi|!}.
\]
In particular, the $\mathcal{T}_{N,\mono}$-free cumulants agree with the monotone cumulants defined in \cite{HS2011a} \cite{HS2014}; compare \cite[\S 5]{HS2011b}, \cite[Theorem 3.4]{HS2014}, \cite[Definition 4.4]{AW2016}.  These cumulants are independent of $N$ and we denote them by $K_{\mono,\ell}$.
\end{proposition}

\begin{proof}
Denote $\alpha_{\mono,\pi} = |\Ord(\pi)| / |\pi|!$.   Let
\[
\Upsilon_\pi = \{t \in (0,1]^\pi: V \prec W \implies t_V < t_W\}.
\]
As in \cite[Remark 6.13]{Jekel2020}, we claim that $\alpha_{\mono,\pi} = |\Upsilon_\pi|$, where $| \cdot |$ denotes Lebesgue measure.  If $R$ is a total order on $\pi$, let $\Upsilon_R = \{t \in [0,1]^\pi: V\, R\, W \implies t_V < t_W$.  Then, up to sets of measure zero, $(0,1]^\pi$ is the disjoint union of the simplices $\Upsilon_R$ over every total order $R$ on $\pi$.  There are $|\pi|!$ such total orders and each $\Upsilon_R$ has the same measure, and thus $|\Upsilon_R| = 1 / |\pi|!$.  But $\Upsilon_\pi$ is the disjoint union of the sets $\Upsilon_R$ over all total orders $R$ that extend the partial order $\prec$, and thus $\alpha_{\mono,\pi} = |\Upsilon_\pi|$.

We must check the fixed point equation
\[
N^{|\pi|} \alpha_{\mono,\pi} = \sum_{\chi \in \mathcal{X}_w(\pi,\mathcal{T})} \prod_{\chi \text{-components } \pi'} \alpha_{\mono,\pi'}.
\]
The left hand side is equal to
\[
N^{|\pi|} \alpha_{\mono,\pi} = |N \Upsilon_\pi| = |\{t \in (0,N]^\pi: V \prec W \implies t_V < t_W\}|.
\]
For convenience, we will write $t \blacktriangleleft \pi$ to mean that $V \prec W \implies t_V < t_W$.  The set $(0,N]$ is the disjoint union of the subintervals $(j-1,j]$ for $j = 1$, \dots, $N$, and this induces a decomposition of $(0,N]^\pi$ into unit cubes.  Thus, we have
\[
N^{|\pi|} \alpha_{\mono,\pi} = \sum_{\chi \in [N]^\pi} |\{t \in \prod_{V \in \pi} (\chi(V) - 1, \chi(V)]: t \blacktriangleleft \pi \}|.
\]
The set on the right hand side will be empty unless $V \prec W \implies \chi(V) \leq \chi(W)$.  Moreover, the condition that $V \prec W \implies \chi(V) \leq \chi(W)$ is equivalent to $\pi$ and $\chi$ being weakly compatible with $\mathcal{T}_{N,\mono}$ as in the proof of Proposition \ref{prop:monotonemoments}.  Thus, we obtain
\[
N^{|\pi|} \alpha_{\mono,\pi} = \sum_{\chi \in \mathcal{X}_w(\pi,\mathcal{T}_{N,\mono})} |\{t \in \prod_{V \in \pi} (\chi(V) - 1, \chi(V)]: t \blacktriangleleft \pi \}|.
\]

If $\chi \in \mathcal{X}_w(\pi,\mathcal{T}_{N,\mono})$ and if $t \in \prod_{V \in \pi} (\chi(V) - 1, \chi(V)]$, then the condition $t \blacktriangleleft \pi$ is equivalent to $t|_{\pi'} \blacktriangleleft \pi'$ for every $\chi$-component $\pi'$.  Indeed, suppose that $t|_{\pi'} \blacktriangleleft \pi'$ for each $\pi'$ and that $V \prec W$ in $\pi$.  If $V$ and $W$ are in the same $\chi$-component $\pi'$, then $t_V < t_W$ because $t|_{\pi'} \blacktriangleleft \pi'$.  Otherwise, the $\pi / \chi$-block containing $V$ must be $\prec$ the $\pi / \chi$-block containing $W$, hence $\chi(V) < \chi(W)$ and so $t_V < t_W$.

It follows that
\[
\{t \in \prod_{V \in \pi} (\chi(V) - 1, \chi(V)]: t \blacktriangleleft \pi \} = \prod_{\chi \text{-components } \pi'} \{t \in (j(\pi') - 1,j(\pi')]^{\pi'}: t \blacktriangleleft \pi'\}.
\]
where $j(\pi')$ is the constant value of $\chi$ on $\pi'$.  Therefore,
\begin{align*}
N^{|\pi|} \alpha_{\mono,\pi} &= \sum_{\chi \in \mathcal{X}_w(\pi,\mathcal{T}_{N,\mono})} \prod_{\chi \text{-components } \pi'} |\{t \in (j(\pi') - 1,j(\pi')]^{\pi'}: t \blacktriangleleft \pi'\}| \\
&=  \sum_{\chi \in \mathcal{X}_w(\pi,\mathcal{T}_{N,\mono})} \prod_{\chi \text{-components } \pi'} |\Upsilon_{\pi'}|,
\end{align*}
where the last step follows from translation invariance of Lebesgue measure.  Now $|\Upsilon_{\pi'}| = \alpha_{\mono,\pi'}$, so the proof is complete.
\end{proof}

\subsection{Further Properties of the Coefficients $\alpha_{\mathcal{T},\pi}$} \label{subsec:cumulantcoefficientproperties}

In this section, we will reformulate the definition of the coefficients $\alpha_\pi$ in a more graph-theoretic way, along the lines of Remark \ref{rem:graphhomo}.  This will allow us in the next section to compute $\alpha_{\mathcal{T},\pi}$ when $\mathcal{T} = \Walk(G)$ for a $d$-regular digraph $G$.

\begin{definition}
For $\pi, \tau \in \mathcal{NC}(\ell)$, we say that $\tau$ is a \emph{quotient} of $\pi$ if there exists a coloring $\chi$ of $\pi$ such that $\tau = \pi / \chi$. 
\end{definition}

\begin{remark}
One can show that $\tau$ is a quotient of $\pi$ if and only if every block of $\tau$ is contained in a block of $\pi$, and each block of $\tau$ is the union of some blocks in $\pi$ which form a connected subgraph of $\graph(\pi)$.
\end{remark}

\begin{definition}
Let $\mathcal{T}$ and $\mathcal{T}'$ be rooted trees.  We say that $\phi: \mathcal{T} \to \mathcal{T}'$ is a \emph{rooted tree homomorphism} if $s \sim t$ in $\mathcal{T}$ implies $\phi(s) \sim \phi(t)$ in $\mathcal{T}'$ and the distance from $\phi(s)$ to the root of $\mathcal{T}'$ is the same as the distance from $s$ to the root of $\mathcal{T}$.  In other words, it is a graph homomorphism that respects the levels of the tree.  We denote the set of rooted tree homomorphisms from $\mathcal{T}$ to $\mathcal{T}'$ by $\Hom(\mathcal{T}, \mathcal{T}')$.
\end{definition}

\begin{lemma}
Let $\mathcal{T} \in \Tree(N)$ and $n = |[N] \cap \mathcal{T}|$.  Then
\begin{equation} \label{eq:partitioncoefficientidentity3}
\alpha_{\mathcal{T},\pi} = \frac{1}{n^{|\pi|}} \sum_{\tau \text{ quotient of } \pi} |\Hom(\graph(\tau),\mathcal{T})| \prod_{V \in \tau} \alpha_{\mathcal{T},\pi|_V}.
\end{equation}
\end{lemma}

\begin{proof}
Of course, the idea is to rephrase \ref{eq:partitioncoefficientidentity} in different language.

We will construct a bijection between colorings $\chi \in \mathcal{X}_w(\pi,\mathcal{T})$ and pairs $(\tau,\phi)$ where $\tau$ is a quotient of $\pi$ and $\phi: \graph(\tau) \to \mathcal{T}$ is a graph homomorphism which preserves the levels in the tree (that is, it sends a block at depth $d$ to a string of length $d$ in $\mathcal{T}$, and in particular sends the root vertex to the root vertex).  Given a coloring $\chi$, we define $\tau = \pi / \chi$; then since $\tau \in \mathcal{NC}(\chi,\mathcal{T})$, Remark \ref{rem:graphhomo} shows that there is a unique rooted graph homomorphism $\phi: \graph(\tau) \to \mathcal{T}$ such that $\chi(V)$ is the first letter of $\phi(V)$.

Conversely, given a quotient $\tau$ of $\pi$ and a homomorphism $\phi: \graph(\tau) \to \mathcal{T}$, we can define a coloring $\chi$ of $\tau$ by saying that $\chi(V)$ is the first letter of $\phi(V)$.  Then the coloring $\chi$ of $\tau$ can also be interpreted as a coloring of $\pi$ since each block of $\pi$ is contained in a block of $\tau$.  And we have precisely that $\tau = \pi / \chi$.  Indeed, two adjacent vertices $V$ and $W$ in $\graph(\pi)$ are in the same component of $\pi / \chi$ if and only if the first letters of $\phi(V)$ and $\phi(W)$ agree.  Now $\phi(V)$ and $\phi(W)$ are either adjacent or equal in $\mathcal{T}$; sine the strings are alternating, the first letters will disagree if they are adjacent and agree if they are equal.  Similarly, $V$ and $W$ are either in adjacent blocks or the same block of $\tau$, and they are in the adjacent (resp.\ equal) blocks if and only if $\phi(V)$ and $\phi(W)$ are adjacent (resp.\ equal).

In summary, we have a bijective correspondence between coloring $\chi \in \mathcal{X}_w(\pi,\mathcal{T})$ and pairs $(\tau,\phi)$ as above, and the blocks of $\tau$ are precisely the $\chi$-components of $\pi$.

Thus, the formula \eqref{eq:partitioncoefficientidentity}, which defines $\alpha_{\mathcal{T},\pi}$ inductively, can be written as
\[
\alpha_{\mathcal{T},\pi} = \frac{1}{n^{|\pi|}} \sum_{\tau \text{ quotient of } \pi} \sum_{\phi \in \Hom(\graph(\tau),\mathcal{T})} \prod_{V \in \tau} \alpha_{\mathcal{T},\pi|_V},
\]
which is exactly \eqref{eq:partitioncoefficientidentity3}.

The number $n$ only depends on the rooted isomorphism class of $\mathcal{T}$, since it is the number of neighbors of the root vertex.  Thus, it is clear upon inspection that this formula is invariant under rooted isomorphism.
\end{proof}

This formula has the following immediate consequences.

\begin{proposition} \label{prop:cumulantproperties}
Let $\mathcal{T} \in \Tree(N)$ with $n = |[N] \cap \mathcal{T}| \geq 2$.
\begin{enumerate}[(1)]
    \item The coefficients $\alpha_{\mathcal{T},\pi}$ only depend on $\pi$ and the isomorphism class of $\mathcal{T}$ as a rooted tree.
	\item If $\pi$ is a partition where each block has depth $\leq d$, then $\alpha_{\mathcal{T},\pi}$ only depends on the ball of radius $d$ around the root vertex of $\mathcal{T}$.
	\item If $\mathcal{T}' \subseteq \mathcal{T}$ and $\mathcal{T}' \cap [N] = \mathcal{T} \cap [N]$, then $\alpha_{\mathcal{T}',\ell} \leq \alpha_{\mathcal{T},\ell}$.
\end{enumerate}
\end{proposition}

\begin{proof}
(1) Note that the coefficients are defined inductively by
\begin{equation} \label{eq:partitioncoefficientidentity4}
\left(1 - \frac{1}{n^{|\pi|-1}} \right) \alpha_{\mathcal{T},\pi} = \frac{1}{n^{|\pi|}} \sum_{\substack{\tau \text{ quotient of } \pi \\ \tau \neq \pi}} |\Hom(\graph(\tau),\mathcal{T})| \prod_{V \in \tau} \alpha_{\mathcal{T},\pi|_V}.
\end{equation}
This formula clearly only depends on the rooted isomorphism class of $\mathcal{T}$ since it deals with homomorphisms into $\mathcal{T}$.

(2) The argument is similar.  If $\pi$ has depth bounded by $d$, then all the partitions $\tau$ that occur in the formula also have depth bounded by $d$, and hence any rooted tree homomorphism on $\graph(\pi)$ will only map into the vertices of $\mathcal{T}$ with a distance $\leq d$ from the root vertex.  Moreover, the partitions $\pi|_V$ for $V \in \tau$ will also have depth $\leq d$.  So the whole inductive procedure that evaluates $\alpha_{\mathcal{T},\pi}$ will only ever encounter partitions of depth $\leq d$ and vertices of $\mathcal{T}$ within a distance of $d$ from the root vertex.

(3) Note that since $\mathcal{T}' \subseteq \mathcal{T}$, we have
\[
|\Hom(\graph(\tau),\mathcal{T}')| \leq |\Hom(\graph(\tau),\mathcal{T})|
\]
for any non-crossing partition $\tau$.  Thus, the inequality follows by induction using \eqref{eq:partitioncoefficientidentity4}.
\end{proof}

\begin{corollary}
If $\mathcal{T} \in \Tree(N)$ and $\mathcal{T}' \in \Tree(N')$ are isomorphic as rooted trees, then
\[
\boxplus_{\mathcal{T}}(\mu, \dots, \mu) = \boxplus_{\mathcal{T}'}(\mu, \dots, \mu).
\]
Also, for any $\mathcal{B}$-valued non-commutative probability space $(\mathcal{A},E)$, we have $K_{\mathcal{T},\ell} = K_{\mathcal{T}',\ell}$.
\end{corollary}

\begin{proof}
The cumulants are uniquely determined by the coefficients $\alpha_{\mathcal{T},\pi}$ and hence so is the convolution of a law with itself.
\end{proof}

Not only are the coefficients a rooted-isomorphism invariant of $\mathcal{T}$, but they also behave well when we compose a tree $\mathcal{T}$ with itself, as the next proposition shows.  This proposition is technical and not essential to the main flow of our paper, but nonetheless we include it in case it is useful in clarifying the concepts or studying examples.

\begin{proposition} \label{prop:cumulantcomposition}
We have $\alpha_{\mathcal{T}(\mathcal{T},\dots,\mathcal{T}),\pi} = \alpha_{\mathcal{T},\pi}$.

More generally, suppose that $\mathcal{T} \in \Tree(M)$ with $|[M] \cap \mathcal{T}| = n \geq 2$, suppose that $\mathcal{T}_1, \dots, \mathcal{T}_M \in \Tree(N)$ with $|[N] \cap \mathcal{T}_j| = n \geq 2$ for each $j$, and suppose that $\alpha_{\mathcal{T}_j,\pi} = \alpha_{\mathcal{T},\pi}$ for every $j$ and $\pi$.  Then we have $\alpha_{\mathcal{T}(\mathcal{T}_1,\dots,\mathcal{T}_M),\pi} = \alpha_{\mathcal{T},\pi}$.
\end{proposition}

\begin{proof}
The first claim is a special case of the second because we can take $\mathcal{T}_j = \mathcal{T}$.  To the prove the second claim, let $\widehat{\mathcal{T}} = \mathcal{T}(\mathcal{T}_1,\dots,\mathcal{T}_M)$.  Then by Lemma \ref{lem:cumulantcoefficients}, it suffices to show that
\begin{equation} \label{eq:composedpartitionidentity}
\alpha_{\mathcal{T},\pi} = \sum_{\substack{\widehat{\chi} \in [MN]^\pi \\ \pi / \widehat{\chi} \in \mathcal{NC}(\widehat{\chi},\mathcal{T}^*)}} \frac{1}{(mn)^{|\pi|}} \prod_{\widehat{\chi} \text{-components } \pi'} \alpha_{\mathcal{T},\pi'}.
\end{equation}
Our goal is to express a coloring $\widehat{\chi} \in \mathcal{X}_w(\pi,\widehat{\mathcal{T}})$ in terms of colorings related to $\mathcal{T}$ and $\mathcal{T}_j$.

For $j = 1$, \dots, $M$, let $\iota_j: [N] \to [MN]$ be the $j$th inclusion given by $\iota_j(i) = i + N(j-1)$.  Define $\phi_1: [MN] \to [M]$ and $\phi_2: [MN] \to [N]$ by
\[
\phi_1(\iota_j(i)) = j, \qquad \phi_2(\iota_j(i)) = i.
\]
By definition of the composition, if a string $s$ satisfies $\red(s) \in \widehat{\mathcal{T}}$, then we also have $\red((\phi_1)_*(s))$ is in $\mathcal{T}$.  Let us call $s'$ is a $\phi_1$-component of $s$ if $s'$ is a maximal substring such that $\phi_1$ is constant on $s'$.  If $s'$ is a $\phi_1$-component of $s$ and $\phi_1|_{s'} = j$, then it follows that $\red((\phi_2)_*(s')) \in \mathcal{T}_j$.

This shows that every string $s$ such that $\red(s) \in \widehat{\mathcal{T}}$ is obtained as $(\iota_{j_1})_*(s_1) \dots (\iota_{j_\ell})_*(s_\ell)$ where $j_1 \dots j_\ell = \red((\phi_1)_*(s)) \in \mathcal{T}$ and $\red(s_i) \in \mathcal{T}_{j_i}$.  Conversely, every such string $(\iota_{j_1})_*(s_1) \dots (\iota_{j_\ell})_*(s_\ell)$ defined in this way will satisfy $\red(s) \in \widehat{\mathcal{T}}$.

Now $\widehat{\chi} \in \mathcal{X}_w(\pi,\widehat{\mathcal{T}})$ if and only if for every maximal chain $V_1 \succ \dots \succ V_d$, we have $\red(\widehat{\chi}(V_1) \dots \widehat{\chi}(V_d)) \in \widehat{\mathcal{T}}$.  In particular, this implies that $\phi_1 \circ \widehat{\chi} \in \mathcal{X}_w(\pi,\mathcal{T})$ in light of the foregoing observations.  Moreover, if $\pi'$ is a $\phi_1 \circ \widehat{\chi}$-component of $\widehat{\chi}$ and $\phi_1 \circ \widehat{\chi}|_{\pi'} = j$, then $\phi_2 \circ \widehat{\chi}|_{\pi'} \in \mathcal{X}_w(\pi',\mathcal{T}_j)$.  Indeed, if $C' = (V_1,\dots,V_d)$ is a maximal in $\pi'$, then $C'$ is contained in some maximal chain $C$ in $\pi$; moreover, $\widehat{\chi}|_{C'}$ is a $\phi_1$-component of $\widehat{\chi}|_C$, and thus $\red(\phi_2 \circ \widehat{\chi}|_{C'})$ is $\mathcal{T}_j$.

Therefore, for every $\widehat{\chi} \in \mathcal{X}_w(\pi,\widehat{\mathcal{T}})$, there exists $\chi \in \mathcal{X}_w(\pi,\mathcal{T})$ and $\chi' \in \mathcal{X}_w(\pi',\mathcal{T}_j)$ for each $\chi$-component of $\pi$ where $\chi|_{\pi'} = j$, such that
\[
\chi = \phi_1 \circ \widehat{\chi}, \qquad \chi' = \phi_2 \circ \widehat{\chi}|_{\pi'}
\]
Conversely, given such partitions, $\chi$ and $\chi'$, we may define $\chi \in \mathcal{X}_w(\pi,\widehat{\mathcal{T}})$ by
\[
\widehat{\chi}|_{\pi'} = \iota_j \circ \chi'
\]
whenever $\pi'$ is a $\chi$-component of $\pi$ with $\chi|_{\pi'} = j$.  Moreover, each $\widehat{\chi}$-component of $\pi$ is one of the $\chi$'-components of $\pi'$, where $\pi'$ is a $\chi$-component of $\pi$.  In other words, the subdivision of $\pi$ into $\widehat{\chi}$-components is obtained by first dividing into $\chi$-components in a way compatible with $\mathcal{T}$, and then coloring and subdividing each of those components in a way that is weakly compatible with the appropriate tree $\mathcal{T}_j$.

It follows that the right hand side of \eqref{eq:composedpartitionidentity} equals
\[
\sum_{\chi \in \mathcal{X}(\pi,\mathcal{T})} \frac{1}{(mn)^{|\pi|}} \prod_{\chi \text{-components } \pi'} \left( \sum_{\chi' \in \mathcal{X}_w(\pi',\mathcal{T}_{j(\pi')})} \prod_{\widehat{\chi'} \text{-components } \pi''} \alpha_{\mathcal{T},\pi''} \right),
\]
where $j(\pi')$ is the constant value of $\chi$ on $\pi'$.  Since $|\pi|$ is the sum of $|\pi'|$ over the $\chi$-components, this is equal to
\[
\sum_{\chi \in X_{\mathcal{T},\pi}} \frac{1}{m^{|\pi|}} \prod_{\chi \text{-components } \pi'} \left( \frac{1}{n^{|\pi'|}} \sum_{\chi' \in \mathcal{X}_w(\pi,\mathcal{T}_{j(\pi')})} \prod_{\widehat{\chi'} \text{-components } \pi''} \alpha_{\mathcal{T},\pi''} \right).
\]
We have assumed that $\alpha_{\mathcal{T}_j,\pi''} = \alpha_{\mathcal{T}_j,\pi'}$, and thus by \eqref{eq:partitioncoefficientidentity} this becomes
\[
\sum_{\chi \in \mathcal{X}_w(\pi,\mathcal{T})} \frac{1}{m^{|\pi|}} \prod_{\chi \text{-components } \pi'} \alpha_{\mathcal{T},\pi'}.
\]
By applying \eqref{eq:partitioncoefficientidentity} again, this time to $\mathcal{T}$, we obtain $\alpha_{\mathcal{T},\pi}$, which establishes \eqref{eq:composedpartitionidentity} as desired.
\end{proof}

\subsection{Coefficients for Regular Digraphs and Trees} \label{subsec:treecoefficients}

In this section, we will compute $\alpha_{\mathcal{T},\pi}$ for certain trees, including the walks on a regular digraph.

\begin{definition}
A rooted tree $\mathcal{T}$ is said to be $(n,d)$-regular if the root vertex has exactly $n$ neighbors and each other vertex has exactly $d$ children (hence, $d + 1$ neighbors overall). 
\end{definition}

\begin{example}
A digraph $G$ on a vertex set $[N]$ is said to be \emph{$d$-regular} if each vertex has exactly $d$ outgoing edges.  If $G$ is a $d$-regular digraph on $[N]$, then $\Walk(G)$ is an $(N,d)$-regular rooted tree.
\end{example}

\begin{definition}
Let $\pi$ be a non-crossing partition.  We say that a block is \emph{outer} if it has depth $1$ (that is, it is adjacent to the root vertex in $\graph(\pi)$) and it is  \emph{inner} otherwise.
\end{definition}

\begin{proposition} \label{prop:regulartree}
Let $\mathcal{T} \in \Tree(N)$ be an $(n,d)$-regular rooted tree.  Then
\[
\alpha_{\mathcal{T},\pi} = \left( \frac{d}{n-1} \right)^{I(\pi)},
\]
where $I(\pi)$ denotes the number of inner blocks of $\pi$.
\end{proposition}

\begin{proof}
It suffices to show that the coefficients $\beta_\pi = (d / (n - 1))^{I(\pi)}$ satisfy the fixed point formula \eqref{eq:partitioncoefficientidentity3}.  In other words, we must show that
\[
\left( \frac{d}{n-1} \right)^{I(\pi)} = \frac{1}{n^{|\pi|}} \sum_{\tau \text{ quotient of } \pi} |\Hom(\graph(\tau),\mathcal{T})| \prod_{V \in \tau} \left( \frac{d}{n-1} \right)^{I(\pi|_V)}.
\]

First, let us compute $|\Hom(\graph(\tau),\mathcal{T})|$.  A homomorphism $\phi$ must send each outer block of $\tau$ to a neighbor of the root vertex in $\mathcal{T}$, so there are $n$ choices of where to map each outer block.  Each inner block $V$ is a child of some other block $W$ in $\graph(\tau)$, and hence it must be mapped to a child of $\phi(W)$, which means that there are $d$-choices of where to map $V$ once $\phi(W)$ is chosen.  Thus, overall the number of homomorphisms is $n^{E(\tau)} d^{I(\tau)}$, where $O(\tau)$ denotes the number of outer blocks of $\tau$.  Furthermore, the outer blocks of $\tau$ are in bijection with the outer blocks of $\pi$, and hence
\[
|\Hom(\graph(\tau),\mathcal{T})| = n^{O(\pi)} d^{I(\tau)}.
\]

Next, consider $\prod_{V \in \tau} \left( \frac{d}{n-1} \right)^{I(\pi|_V)}$.  Suppose $W$ is a block of $\pi$ contained in the block $V$ of $\tau$.  Then $W$ is inner in $\pi$ if and only if either $W$ is inner in $\pi|_V$ or $V$ is inner in $\tau$.  Note that $\pi|_V$ is in $\mathcal{NC}^\circ(V)$ and hence has exactly one outer block; so the inner blocks of $\pi$ that are outer in some $\pi|_V$ are in bijection with inner blocks of $\tau$.  Thus,
\[
I(\pi) = I(\tau) + \sum_{V \in \tau} I(\pi|_V).
\]
Hence,
\[
\prod_{V \in \tau} \left( \frac{d}{n-1} \right)^{I(\pi|_V)} = \left( \frac{d}{n-1} \right)^{I(\pi) - I(\tau)},
\]
and
\[
|\Hom(\graph(\tau),\mathcal{T})| \prod_{V \in \tau} \left( \frac{d}{n-1} \right)^{I(\pi|_V)} = n^{O(\pi)} d^{I(\pi)} (n-1)^{I(\tau) - I(\pi)}.
\]
Hence, the equation we want to prove is
\[
d^{I(\pi)} (n - 1)^{-I(\pi)} = \frac{1}{n^{|\pi|}} \sum_{\tau \text{ quotient of } \pi} n^{E(\pi)} d^{I(\pi)} (n-1)^{I(\tau) - I(\pi)},
\]
which is equivalent to
\[
n^{I(\pi)} = \sum_{\tau \text{ quotient of } \pi} (n - 1)^{I(\tau)}.
\]

Let $S$ be the set of edges in $\graph(\pi)$.  There is a bijection between the quotients $\tau$ of $\pi$ and subsets $A$ of $S$ given by the relation that the edge between $V$ and $W$ in $\graph(\pi)$ is in the set $A$ if and only if $V$ and $W$ are in the same block of $\tau$.  Moreover, each time we add an edge between $V$ and $W \succ V$ to the set $A$, we reduce the number of inner blocks by $1$ by joining the inner block $W$ together with the block $V$.  Hence, $I(\tau) = I(\pi) - |A|$.  By the same token, $|S| = I(\pi)$ because when all the edges are included in the set $A$, there are no more inner blocks in $\tau$.  Hence,
\[
\sum_{\tau \text{ quotient of } \pi} (n - 1)^{I(\tau)} = \sum_{A \subseteq S} (n - 1)^{|S| - |A|} = [(n - 1) + 1]^{|S|} = n^{I(\pi)},
\]
where the equality $\sum_{A \subseteq S} (n-1)^{|S| - |A|} = [(n-1) + 1]^{|S|}$ follows from the binomial theorem.  Therefore, we have proved the desired equality and thus verified our formula for the coefficients.
\end{proof}

\begin{example}
The complete graph $K_N$ is $(N-1)$-regular on $N$ vertices, so the corresponding tree $\mathcal{T}_{N,\free}$ is $(N,N-1)$-regular.  Thus, the term $d / (n - 1)$ becomes $(N-1) / (N-1) = 1$ and we get $\alpha_{\free,\pi} = 1$.

The totally disconnected graph $K_N^c$ is $0$-regular on $N$ vertices, so the corresponding tree $\mathcal{T}_{N,\Bool}$ is $(N,0)$-regular, and we get $d / (n - 1) = 0 / (N - 1) = 0$.  Hence, $\alpha_{\Bool,\pi}$ is $1$ for a partition with no inner blocks and $0$ otherwise.

Hence, Proposition \ref{prop:regulartree} provides an interpolation between the free and Boolean cases.
\end{example}

The behavior of such regular trees under composition is straightforward to describe, and we leave the verification as an exercise.

\begin{observation}
Let $\mathcal{T} \in \Tree(N)$ be $(n,d)$-regular and for $j = 1, \dots, N$, let $\mathcal{T}_j \in \Tree(N_j)$ be $(n',d')$-regular.  Then $\mathcal{T}(\mathcal{T}_1,\dots,\mathcal{T}_N)$ is $(nn', dn' + d')$-regular.
\end{observation}

\begin{example} \label{ex:freebooleaniteration}
For integers $p$ and $q$, consider $(\mu^{\boxplus p})^{\uplus q}$.  This is the composition of a $(q,0)$-regular tree (on the outside) and a $(p,p-1)$-regular tree (on the inside, repeated $q$ times).  So it produces a $(pq,p-1)$-regular tree.

On the other hand, $(\mu^{\uplus q})^{\boxplus p}$ is the composition of a $(p,p-1)$-regular tree inside a $(q,0)$-regular tree, which produces a $(pq, (p-1)q)$-regular tree.

It follows that
\[
(\mu^{\boxplus p})^{\uplus q} = (\mu^{\uplus q'})^{\boxplus p'}
\]
provided that $pq = p'q'$ and $p - 1 = (p' - 1)q'$.  More generally, this identity was proved analytically generalized for real values of $p \geq 1$, $q \geq 0$, $p' \geq 1$, $q' \geq 0$ by \cite{BN2008a}, and for completely positive maps by \cite{Liu2018}.
\end{example}

\begin{example} \label{ex:regulartree}
Let $\mathcal{T}$ be $(n,d)$-regular.  Then we have
\[
\boxplus_{\mathcal{T}}(\mu^{\uplus (n-1)}, \dots \mu^{\uplus (n-1)})^{\uplus d} = ((\mu^{\uplus d})^{\boxplus n})^{\uplus (n-1)}.
\]
The proof is the same.  On the left-hand side, we compose a $(d,0)$-regular tree with an $(n,d)$-regular tree with an $(n-1,0)$-regular tree, while on the right-hand side, we compose an $(n-1,0)$-regular tree with an $(n,n-1)$-regular tree with a $(d,0)$-regular tree.  Both compositions result in an $(n(n-1)d, (n-1)d)$-regular tree, and since the rooted trees are isomorphic, the iterated convolutions are equal.
\end{example}

We shall return to this discussion later in \S \ref{subsec:BPbijection}.

\section{The Central Limit Theorem} \label{sec:CLT}

\subsection{Main Theorem}

We are now ready to prove the central limit theorem for $\mathcal{T}$-free independence.  This is (in some sense) an analogue of the free, Boolean, and monotone central limit theorems, which were first proved in \cite[Theorem 4.8]{Voiculescu1985}, \cite[Theorem 3.4]{SW1997}, and \cite{Muraki2001} respectively for the scalar-valued setting; the canonical references for the operator-valued setting are \cite[Theorem 8.4]{Voiculescu1995} and \cite[\S 4.2]{Speicher1998} for the free case, \cite[\S 2.1]{BPV2013} for the Boolean case, and \cite[\S 2.3]{BPV2013} and \cite[Theorem 3.6]{HS2014} for the monotone case.  We discuss these special cases further after Corollary \ref{cor:CLTestimate}.

To state the general $\mathcal{T}$-free theorem, we first define convolution powers of the a $\mu$ obtained from iterating the $\mathcal{T}$-free convolution operation.

\begin{definition} \label{def:integerconvolutionpowers}
Let $\mathcal{T} \in \Tree(N)$ and let $n = |[N] \cap \mathcal{T}| \geq 2$.  We define $\boxplus_{\mathcal{T}}^{n^k}(\mu)$ for $k \geq 0$ inductively by
\[
\boxplus_{\mathcal{T}}^1(\mu) = \mu, \qquad \boxplus_{\mathcal{T}}^{n^{k+1}}(\mu) = \boxplus_{\mathcal{T}}( \underbrace{\boxplus_{\mathcal{T}}^{n^k}(\mu), \dots, \boxplus_{\mathcal{T}}^{n^k}(\mu)}_{N \text{ copies}}).
\]
Similarly, we define $\mathcal{T}^{n^k}$ inductively by
\[
\mathcal{T}^1 = \mathcal{T}, \qquad \mathcal{T}^{n^{k+1}} = \mathcal{T}( \underbrace{\mathcal{T}^{n^k}, \dots, \mathcal{T}^{n^k}}_{N \text{ copies}}),
\]
so that
\[
\boxplus_{\mathcal{T}}^{n^k}(\mu) = \boxplus_{\mathcal{T}^{n^k}}(\mu,\dots,\mu).
\]
\end{definition}

\begin{remark}
We have chosen to use the notation $\boxplus_{\mathcal{T}}^{n^k}$ even though $\boxplus_{\mathcal{T}}$ has $N$ arguments because of the fact that the $\mathcal{T}$-free cumulants of $\boxplus_{\mathcal{T}}^{n^k}(\mu)$ are $n^k$ times the $\mathcal{T}$-free cumulants of $\mu$ (see Observation \ref{obs:lawcumulantextensivity} below).  The notation used here will be consistent with our notation for more general convolution powers in Definition \ref{def:convolutionpowers}.
\end{remark}

Our goal is to show that if $\mu$ has mean zero and variance $\eta$, then $\dil_{n^{-k/2}}(\boxplus_{\mathcal{T}}^{n^k}(\mu))$ converges to some universal distribution $\nu_{\mathcal{T},\eta}$ as $k \to \infty$, where the dilation $\dil_{n^{-k/2}}$ is given by Definition \ref{def:dilation}.  The limiting distribution will be stable under $\mathcal{T}$-free convolution in the following sense.

\begin{definition}
Let $\mathcal{T} \in \Tree(N)$ and let $n = |[N] \cap \mathcal{T}|$.  We say that $\nu$ is a \emph{$\mathcal{T}$-free central limit distribution} if
\[
\boxplus_{\mathcal{T}}(\nu,\dots,\nu) = \dil_{n^{1/2}}(\nu).
\]
\end{definition}

We will prove convergence of $\dil_{n^{-k/2}}(\boxplus_{\mathcal{T}}^{n^k}(\mu))$ to the central limit distribution using the cumulants of a law.

\begin{definition}
Let $\mu$ be a $\mathcal{B}$-valued law.  We define the $\mathcal{T}$-free cumulants of $\mu$ as the multilinear forms $\kappa_{\mathcal{T},\ell}(\mu): \mathcal{B}^{\ell-1} \to \mathcal{B}$ given by
\[
\kappa_{\mathcal{T},\ell}(\mu)[b_1,\dots,b_{\ell-1}] = K_{\mathcal{T},\ell}[Xb_1,Xb_2,\dots, Xb_{\ell-1},X],
\]
where $X$ is a $\mathcal{B}$-valued random variable with law $\mu$.  It is clear that this is independent of the particular choice of operator realizing the law $\mu$.  More generally, for $\pi \in \mathcal{NC}(\ell)$, we define
\[
\kappa_{\mathcal{T},\pi}(\mu)[b_1,\dots,b_{\ell-1}] = K_{\mathcal{T},\pi}[Xb_1,Xb_2,\dots,Xb_{\ell-1},X].
\]
\end{definition}

\begin{observation} \label{obs:lawcumulantextensivity}
Let $\mathcal{T} \in \Tree(N)$ with $n = |[N] \cap \mathcal{T}| \geq 2$.  As a consequence of Theorem \ref{thm:extensivity}, we have
\[
\kappa_{\mathcal{T},\ell}(\boxplus_{\mathcal{T}}^{n^k}(\mu)) = n^k \kappa_{\mathcal{T},\ell}(\mu).
\]
\end{observation}

From here it is easy to see that the cumulants (and hence the moments) of $\dil_{n^{-k/2}}(\boxplus_{\mathcal{T}}^{n^k}(\mu))$ will converge as $k \to \infty$ if $\mu$ has mean zero.  Indeed, since $K_{\mathcal{T},\ell}$ is homogeneous of degree $\ell$, we have
\[
\kappa_{\mathcal{T},\ell}(\dil_{n^{-k/2}}(\mu_k)) = n^{-\ell k / 2} n^k \kappa_{\mathcal{T},\ell}(\mu) \to \begin{cases} \eta, & \ell = 2 \\ 0, & \text{ otherwise.} \end{cases}
\]
A precise statement of the central limit theorem is as follows.

\begin{theorem} \label{thm:CLT}
Let $\mathcal{T} \in \Tree(N)$, let $n = |[N] \cap \mathcal{T}|$, and suppose that $n \geq 2$.  Let $\eta: \mathcal{B} \to \mathcal{B}$ be completely positive.
\begin{enumerate}[(1)]
	\item There is a unique $\mathcal{T}$-central limit law $\nu_{\mathcal{T},\eta}$ of mean zero and variance $\eta$.
	\item We have $\rad(\nu_{\mathcal{T},\eta}) \leq 2 \left(\frac{N-1}{n-1} \norm{\eta(1)} \right)^{1/2}$.
	\item The $\mathcal{T}$-free cumulants of $\nu_{\mathcal{T},\eta}$ are all zero except for the second cumulant, which is $\eta$.
	\item If $\mu$ is a $\mathcal{B}$-valued law of mean zero and variance $\eta$, then $\dil_{n^{-k/2}}(\boxplus_{\mathcal{T}}^{n^k}(\mu)) \to \nu_{\mathcal{T},\eta}$ in moments, and the radius of these laws is bounded as $k \to \infty$.
\end{enumerate}
\end{theorem}

We have already explained why the convergence in moments occurs.  However, we still need to estimate the radius of $\dil_{n^{-k/2}}(\boxplus_{\mathcal{T}}^{n^k}(\mu))$.  We will use Proposition \ref{prop:operatornormbound} together with the following bound on the degrees of vertices in $\mathcal{T}^{n^k}$ (which is only nontrivial in the case that $n < N$).

\begin{lemma} \label{lem:iterateddegreebound}
Let $\mathcal{T} \in \Tree(N)$ and $n = |[N] \cap \mathcal{T}|$ and suppose that $n \geq 2$.  Then
\[
\max_{s \in \mathcal{T}^{n^k}} |\{j: js \in \mathcal{T}^{n^k} \}| \leq \frac{N-1}{n-1} n^k.
\]
\end{lemma}

\begin{proof}
Recall that $\mathcal{T}^{n^k}$ is a set of strings on the alphabet $[N^k]$.  In the definition of the composition $\mathcal{T}^{n^{k+1}} = \mathcal{T}(\mathcal{T}^{n^k},\dots,\mathcal{T}^{n^k})$, the set $[N^{k+1}]$ is viewed as the disjoint union of $N$ copies of $[N^k]$ with respect to specified inclusion maps $\iota_j: [N^k] \to [N^{k+1}]$.  The disjoint union of $N$ copies of $[N^k]$ can also be viewed as $[N] \times [N^k]$.  Thus, when it iterating the composition, it will be convenient to transform $[N^k]$ into $[N]^k$ by expressing each letter $j \in [N^k]$ as
\[
j - 1 = N^{k-1} (\phi_1(j) - 1) + \dots + N (\phi_{k-1}(j) - 1) + (\phi_k(j) - 1),
\]
where $\phi_i(j) \in [N]$.  (This is equivalent to computing the $N$-ary expansion of $j - 1$.)  The index $j$ in the composed graph corresponds to the sequence of indices $\phi_1(j)$, \dots, $\phi_k(j)$ in the individual copies of $\mathcal{T}$, where $\phi_i(j)$ is the index in the $i$-th level of the composition (where $1$ is the outermost level and $k$ is the innermost).

The strings in the iterated composition $\mathcal{T}^{n^k}$ can be described as follows.  Under our map $[N^k] \cong [N]^k$, a string $s \in [N^k]^\ell$ corresponds to a tuple of strings $(s_1,\dots,s_k)$, where $s_i = (\phi_i)_*(s)$.  For each $j$, $[\ell]$ can be divided into maximal subintervals on which $(s_1,\dots,s_i)$ is constant, and we call these the \emph{level-$i$ components of $[\ell]$ with respect to $s$} (by convention, there is a single level-$0$ component, namely $[\ell]$).  If $I$ is a level-$j$ component of $s$, then we denote by $s_i|_I$ the substring obtained by restricting to this interval.  We claim that $s \in \mathcal{T}$ if and only if for every level-$j$ component $I$, we have $\red(s_{i+1}|_I) \in \mathcal{T}$ for $i + 1 < k$ and $s_{i+1}|_I \in \mathcal{T}$ for $i + 1 = k$.  This claim is proved by a straightforward induction from the definition of composition.

Let us fix $s \in \mathcal{T}^{n^k}$ and consider the possible values of $j$ such that $js \in \mathcal{T}^{n^k}$.  First suppose that $s \neq \emptyset$.  We partition the set $\{j: js \in \mathcal{T}^{n^k}\}$ based on the first index $i$ such that $\phi_i(j) \neq \phi_i((1))$ (that is, the first $i$ such that the first two indices of $[\ell+1]$ are in different level-$i$ components with respect to $js$).  More explicitly,
\begin{multline*}
\{j \in [n^k]: js \in \mathcal{T}^{n^k}\} \\
= \bigsqcup_{i=1}^k \{j: js \in \mathcal{T}^{n^k}, \phi_i(j) \neq \phi_i(s(1)), \phi_{i'}(j) = \phi_{i'}(s(1)) \text{ for } i' < i\}.
\end{multline*}
Suppose that $j$ is in the $i$th set on the right hand side.  Then $\phi_1(j)$, \dots, $\phi_{i-1}(j)$ are uniquely determined since they are equal to the values for $s(1)$.  The number of possibilities for $\phi_i(j)$ is at most $N - 1$ since $\phi_i(j) \neq \phi_i(s(1))$.  Finally, if $i' > i$, then the first index in $[\ell+1]$ (corresponding to the letter $j$ in $js$) is in its own level-$(i'-1)$ component, and therefore, we must have $\phi_{i'}(j) \in \mathcal{T}$.  So the number of possibilities for $\phi_{i'}(j)$ is $n$.

Therefore,
\[
|\{j: js \in \mathcal{T}, \phi_i(j) \neq \phi_i(s(1)), \phi_{i'}(j) = \phi_{i'}(s(1)) \text{ for } i' < i\}| \leq (N - 1)n^{k-i}.
\]
Hence,
\[
|\{j \in [n^k]: js \in \mathcal{T}^{n^k}\}| \leq \sum_{i=1}^k (N - 1) n^{k-i} = \frac{(N-1)(n^k-1)}{n-1} < \frac{N-1}{n-1} n^k.
\]
In the case where $s = \emptyset$, we have
\[
|\{j: js \in \mathcal{T}^{n^k}\}| = n^k \leq \frac{N-1}{n-1} n^k.
\]
\end{proof}

\begin{proof}[Proof of Theorem \ref{thm:CLT}]
Let $\mu$ be a law of mean zero and variance $\eta$.  Let
\[
\mu_k = \boxplus_{\mathcal{T}}^{n^k}(\mu) = \boxplus_{\mathcal{T}^{n^k}}(\mu,\dots,\mu)
\]
By Proposition \ref{prop:operatornormbound}, we have
\[
\rad(\mu_k) \leq 2 \sup_{s \in \mathcal{T}^{n^k}} \sqrt{ \sum_{js \in \mathcal{T}^{n^k}} \mu(X^2) } + \rad(\mu), 
\]
and thus by Lemma \ref{lem:iterateddegreebound},
\[
\rad(\mu_k) \leq 2 \left( \frac{N-1}{n-1} \norm{\eta(1)} \right)^{1/2} n^{k/2} + \rad(\mu).
\]
Therefore,
\[
\rad(\dil_{n^{-k/2}}(\mu_k)) \leq 2 \left( \frac{N-1}{n-1} \norm{\eta(1)} \right)^{1/2} + n^{-k/2} \rad(\mu).
\]
As explained above, as $k \to \infty$, we have
\[
\kappa_{\mathcal{T},\ell}(\dil_{n^{-k/2}}(\mu_k)) = n^{-\ell k / 2} n^k \kappa_{\mathcal{T},\ell}(\mu) \to \begin{cases} \eta, & \ell = 2 \\ 0, & \text{ otherwise.} \end{cases}
\]
It follows that $\dil_{n^{-k/2}}(\mu_k)$ converges in moments to some law $\nu$ satisfying (2) and (3).  Because there exists some law $\mu$ of mean zero and variance $\eta$ (e.g.\ the operator-valued Bernoulli distribution), the existence claim of (1) is proved.  The uniqueness claim of (1) follows from the observation that $\nu$ is a $\mathcal{T}$-central limit law, then its $\mathcal{T}$-free cumulants must satisfy (2).  Finally, the above argument also proved (4).
\end{proof}

\begin{remark}
It follows from Proposition \ref{prop:cumulantproperties} that for a fixed choice of $\eta$, the central limit law $\nu_{\mathcal{T},\eta}$ only depends on the isomorphism class of $\mathcal{T}$ as a rooted tree.  Also, we have $\nu_{\mathcal{T}^{n^k},\eta} = \nu_{\mathcal{T},\eta}$ as a consequence of Proposition \ref{prop:cumulantcomposition}.  Alternatively, $\nu_{\mathcal{T}^{n^k},\eta} = \nu_{\mathcal{T},\eta}$ by the uniqueness claim in the theorem because a central limit law for $\mathcal{T}$ must be a central limit law for $\mathcal{T}^{n^k}$ as well.
\end{remark}

Unfortunately, Theorem \ref{thm:CLT} as stated does not recapture the free, Boolean, and monotone central limit theorem.  For free, Boolean, and monotone independence, there exists an $N$-ary convolution power of the law $\mu$ for every $N$, and the known central limit theorems in those cases say that $\dil_{N^{-1/2}}(\mu^{\boxplus_{\mathcal{T}} N})$ converges as $N \to \infty$ to the appropriate central limit law.  But for a general tree $\mathcal{T}$, we had to restrict our attention to $n^k$ convolution powers for $k \in \N$.

However, there is a common generalization of the free, Boolean, and monotone central limit theorems and Theorem \ref{thm:CLT}, which we can state as follows.

\begin{proposition} \label{prop:CLT}
Fix a family of trees $(\mathcal{T}_k)_{k \in \N}$, where $\mathcal{T}_k \in \Tree(N_k)$.  Suppose that $n_k := |[N_k] \cap \mathcal{T}_k| \geq 2$, and let
\[
m_k = \max_{s \in \mathcal{T}_k} |\{j: js \in \mathcal{T}_k\}|.
\]
Let $\mu$ be a $\mathcal{B}$-valued law with mean zero and variance $\eta$.
\begin{enumerate}[(1)]
    \item If $\sup_k (m_k / n_k) < +\infty$, then
    \[
    \sup_k \rad(\dil_{n_k^{-1/2}}(\boxplus_{\mathcal{T}_k}^{n_k}(\mu)) < +\infty.
    \]
    \item Suppose that $\lim_{k \to \infty} n_k = +\infty$ and that for each partition $\pi$, the limit $\alpha_\pi := \lim_{k \to \infty} \alpha_{\mathcal{T}_k,\pi}$ exists.  Then the moments of $\dil_{n_k^{-1/2}}(\boxplus_{\mathcal{T}_k}^{n_k}(\mu))$ converge as $k \in \infty$, and the limit only depends on $\eta$ and $(\alpha_\pi)_\pi$.
\end{enumerate}
\end{proposition}

\begin{proof}
(1) This follows from Proposition \ref{prop:operatornormbound}.

(2) Because $\lim_{k \to \infty} \alpha_{\mathcal{T}_k,\pi}$ exists for each partition $\pi$, we can deduce from the moment-cumulant formula that $\lim_{k \to \infty} \kappa_{\mathcal{T}_k,\ell}(\mu)$ exists, using induction on $\ell$.  Hence, writing $\nu_k = \dil_{n_k^{-1/2}}(\boxplus_{\mathcal{T}_k}^{n_k}(\mu))$, we have
\[
\kappa_{\mathcal{T}_k,\ell}(\nu_k) = n_k^{-\ell/2} n_k \kappa_{\mathcal{T}_k,\ell}(\mu) \to \begin{cases} \eta, & \ell = 2 \\ 0, & \text{otherwise.} \end{cases}
\]
It follows that for a partition $\pi$, the limit
\[
\lim_{k \to \infty} \kappa_{\mathcal{T}_k,\pi}(\nu_k)
\]
exists, and it only depends upon $\pi$ and $\eta$ (specifically, it is zero unless $\pi$ is a pair partition, in which case, it is a multilinear form obtained by composing $\eta$ in a certain way described by the partition).  By the moment-cumulant formula,
\[
\nu_k(Xb_1X \dots b_{\ell-1}X) = \sum_{\pi \in \mathcal{NC}(\ell)} \alpha_{\mathcal{T}_k,\pi} \kappa_{\mathcal{T}_k,\pi}(\nu_k)[b_1,\dots,b_{\ell-1}]/
\]
Each of the terms on the right-hand side has a limit as $k \to \infty$, and hence the moment has a limit as $k \to \infty$, and it only depends on $\eta$ and the coefficients $(\alpha_\pi)_\pi$.
\end{proof}

Proposition \ref{prop:CLT} includes Theorem \ref{thm:CLT} as a special case because we can take $\mathcal{T}_k = \mathcal{T}^{n^k}$ and the coefficients $\alpha_{\mathcal{T}^{n^k},\pi} = \alpha_{\mathcal{T},\pi}$.  In the free, Boolean, and monotone cases, we can take $N_k = k$ and $\mathcal{T}_k$ to be the tree for $k$-ary free, Boolean, or monotone convolution.  We already know that the coefficients $\alpha_{\mathcal{T}_k,\pi}$ are independent of $k$ in this case.

Similarly, suppose that $\mathcal{T}_k$ is an $(n_k,d_k)$-regular tree as in \S \ref{subsec:treecoefficients}, that $n_k \to \infty$ and that $d_k / (n_k - 1) \to t$.  Then Proposition \ref{prop:CLT} (1) applies because $m_k = \max(n_k, d_k)$ and $d_k / n_k$ is bounded.  And Proposition \ref{prop:CLT} (2) applies because the cumulant coefficients $\alpha_{\mathcal{T}_k,\pi}$ are certain powers of $d_k / (n_k - 1)$ (Proposition \ref{prop:regulartree}).  Thus, we have a central limit theorem for the family $(\mathcal{T}_k)_{k \in \N}$.

However, Proposition \ref{prop:CLT} is merely a template for various central limit theorems, and it does not explain by itself how to check whether $\sup m_k / n_k < +\infty$ or $\lim_{k \to \infty} \alpha_{\mathcal{T}_k,\pi}$ exists.  The behavior of $\alpha_{\mathcal{T}_k,\pi}$ is a question of asymptotic combinatorics that must be answered directly based on the properties of the particular family $(\mathcal{T}_k)_{k \in \N}$.

For instance, Wysocza{\' n}ski has proved central limit theorems for mixtures of Boolean and monotone independence \cite{Wysoczanski2010}.  As discussed in \S \ref{subsec:digraphoperad}, these mixtures of Boolean and monotone independence are given by taking $\mathcal{T}_k = \Walk(G_k)$, where $G_k$ is a directed graph that forms a poset.  Wysocza{\' n}ski studies the case where the poset $G_k$ is obtained as a discretization at the scale $1/k$ of some fixed convex cone in Euclidean space, and here the asymptotics of moments in the central limit theorem are expressed in terms of volumes in a similar way to the monotone case which we discussed in Proposition \ref{prop:monotonecumulants}.  Kula and Wysocza{\' n}ski study mixtures of free and Boolean independence in a similar way in \cite{KW2013}.

\subsection{Refined Central Limit Estimates} \label{subsec:CLTestimate}

Next, we prove a refined version of the central limit theorem (Corollary \ref{cor:CLTestimate}), which considers the more general situation of non-identically distributed random variables and which gives an explicit estimate for
\[
[\dil_{n^{-k/2}} \boxplus_{\mathcal{T}^{n^k}}(\mu_1,\dots,\mu_{n^k})](f) - [\dil_{n^{-k/2}} \boxplus_{\mathcal{T}^{n^k}}(\nu_1,\dots,\nu_{n^k})](f)
\]
where $f \in \mathcal{B}\ip{X}$.  This comes as a consequence of the following result about coupling.  (A similar idea was used by the first author in \cite[\S 7.3]{Jekel2020}.)

\begin{theorem} \label{thm:CLT2}
Suppose that $\mu_1$, \dots, $\mu_{N^k}$ are laws with mean zero and variance $\eta$.  Then there exist self-adjoint random variables $Y$ and $Z$ in a $\mathcal{B}$-valued probability space $(\mathcal{A},E)$ such that
\begin{enumerate}[(1)]
	\item $Y \sim \dil_{n^{-k/2}} \boxplus_{\mathcal{T}^{n^k}}(\mu_1,\dots,\mu_{N^k})$,
	\item $Z \sim \dil_{n^{-k/2}} \boxplus_{\mathcal{T}}^{n^k}(\nu_{\Bool,\eta})$,
	\item $\norm{Z} \leq 2 \left(\frac{N-1}{n-1} \norm{\eta(1)} \right)^{1/2}$,
	\item $\norm{Y - Z} \leq n^{-k/2} \max_j \rad(\mu_j)$.
\end{enumerate}
\end{theorem}

\begin{proof}[Proof of Theorem \ref{thm:CLT2}]
Let $(\mathcal{H}_j,\xi_j)$ be the space $L^2(\mathcal{B}\ip{Y_j},\mu_j)$ with $\xi_j = 1$.  Let $P_j$ be the projection onto $\xi_j$ and let $Q_j = 1 - P_j$.  Then since $Y_j$ has expectation zero, we have
\[
Y_j = P_j Y_j Q_j + Q_j Y_j P_j + Q_j Y_j Q_j.
\]
The operator $Z_j := P_j Y_j Q_j + Q_j Y_j P_j$ is distributed according to the operator-valued Bernoulli distribution $\nu_{\Bool,\eta}$.  Let $(\mathcal{H},\xi)$ be the $\mathcal{T}^{n^k}$-free product of $(\mathcal{H}_1,\xi_1)$, \dots, $(\mathcal{H}_{N^k}, \xi_{N^k})$.  Let
\[
Y = n^{-k/2} \sum_{j=1}^{N^k} \lambda_j(Y_j), \qquad Z = n^{-k/2} \sum_{j=1}^{N^k} \lambda_j(Z_j).
\]
Then $Y$ and $Z$ have the asserted distributions.  Because $Q_j Z_j Q_j = 0$, the proof of Proposition \ref{prop:operatornormbound} shows that
\[
\norm{Z} = n^{-k/2} \norm*{ \sum_{j=1}^{N^k} \lambda_j(Z_j) } \leq 2 \left( \frac{N-1}{n-1} \norm{\eta(1)} \right)^{1/2},
\]
that is, we may discard the third term on the right hand side in Proposition \ref{prop:operatornormbound}.  Similarly, since $P_j(Y_j - Z_j)Q_j = 0$ and $P_j(Y_j - Z_j) Q_j = 0$, we may discard the first and second terms on the right hand side of Proposition \ref{prop:operatornormbound} and obtain
\[
\norm{Y - Z} \leq n^{-k/2} \max_j \norm{Y_j - Z_j} \leq n^{-k/2} \max_j \rad(\mu_j).
\]
\end{proof}

\begin{remark}
The proof shows that in fact $\sqrt{(N - 1) / (n - 1)}$ may be replaced by $m_k^{1/2} n^{-k/2}$, where $m_k = \sup_{s \in \mathcal{T}^{n^k}} |\{j: js \in \mathcal{T}^{n^k}\}|$.  The same applies to Theorem \ref{thm:CLT} above and Corollary \ref{cor:CLTestimate} below.
\end{remark}

\begin{definition}
Let $f \in \mathcal{B}\ip{X}$ and $R > 0$ and $\epsilon$.  We define the operator-valued modulus of continuity $\Mod_{f,R}(\epsilon)$ as the supremum of $\norm{E[f(X) - f(Y)]}$, where $X$ and $Y$ are self-adjoint elements of a $\mathcal{B}$-valued probability space $(\mathcal{A},E)$ satisfying $\norm{X} \leq R$, $\norm{Y} \leq R$, and $\norm{X - Y} \leq \epsilon$.
\end{definition}

\begin{remark}
We cannot technically quantify over $(\mathcal{A},E)$ because these $\mathcal{B}$-valued probability spaces do not form a set.  However, in light of (the multivariable version of) Theorem \ref{thm:realizationoflaw}, the supremum can be expressed instead by quantifying over all possible joint laws of $(X,Y)$.  More precisely, let $\mathcal{B}\ip{X,Y}$ be the universal unital $*$-algebra generated by two self-adjoint indeterminates $X$ and $Y$. Then $\Mod_{f,R}(\epsilon)$ is the supremum of $\norm{\mu(f(X) - f(Y))}$ where $\mu: \mathcal{B}\ip{X,Y} \to \mathcal{B}$ is a positive $\mathcal{B}$-$\mathcal{B}$-bimodule map satisfying the exponential bounds
\begin{align*}
	\norm{\mu(b_0Xb_1 \dots X b_\ell)} &\leq R^\ell \norm{b_0} \dots \norm{b_\ell} \\
	\norm{\mu(b_0Yb_1 \dots Y b_\ell)} &\leq R^\ell \norm{b_0} \dots \norm{b_\ell} \\
	\norm{\mu(b_0(X - Y)b_1 \dots (X - Y)b_\ell)} &\leq \epsilon^\ell \norm{b_0} \dots \norm{b_\ell}.
\end{align*}
\end{remark}

\begin{corollary} \label{cor:CLTestimate}
Let $\mathcal{T} \in \Tree(N)$ and suppose that $n = |[N] \cap \mathcal{T}| \geq 2$.  Let $\mu_1$, \dots, $\mu_{n^k}$ and $\nu_1$, \dots, $\nu_{n^k}$ be non-commutative laws of mean zero, variance $\eta$, and radius $\leq R$.  Let
\[
R' = 2 \sqrt{\frac{N-1}{n-1}} \norm{\eta(1)}^{1/2}.
\]
Then for $f \in \mathcal{B}\ip{X}$, we have
\begin{multline*}
\norm*{[\dil_{n^{-k/2}} \boxplus_{\mathcal{T}^{n^k}}(\mu_1,\dots,\mu_{n^k})](f) - [\dil_{n^{-k/2}} \boxplus_{\mathcal{T}^{n^k}}(\nu_1,\dots,\nu_{n^k})](f)} \\ 
\leq 2 \Mod_{f,R' + n^{-k/2}R}(n^{-k/2} R).
\end{multline*}
\end{corollary}

\begin{proof}
Using the operators $Y$ and $Z$ from the previous theorem, it follows that
\begin{multline*}
\norm*{[\dil_{n^{-k/2}} \boxplus_{\mathcal{T}^{n^k}}(\mu_1,\dots,\mu_{n^k})](f) - [\dil_{n^{-k/2}} \boxplus_{\mathcal{T}^{n^k}}(\nu_{\Bool,\eta},\dots,\nu_{\Bool,\eta})](f)} \\
\leq \Mod_{f,R' + n^{-k/2}R}(n^{-k/2} R).
\end{multline*}
The same holds with $\mu_j$ replaced by $\nu_j$, and thus we obtain the desired estimate by the triangle inequality.
\end{proof}

A few remarks are needed to explain how to apply this estimate, and how it relates to previous work on central limit theorems. First, Corollary \ref{cor:CLTestimate} gives us an explicit rate of convergence for Theorem \ref{thm:CLT}.  By taking $\nu_j = \nu_{\mathcal{T},\eta}$ in Corollary \ref{cor:CLTestimate}, we obtain the central limit estimate
\[
\norm*{[\dil_{n^{-k/2}} \boxplus_{\mathcal{T}^{n^k}}(\mu_1,\dots,\mu_{n^k})](f) - \nu_{\mathcal{T},\eta}(f)} \leq 2 \Mod_{f,R' + n^{-k/2}R}(n^{-k/2} R),
\]
where we have used the fact that $\nu_{\mathcal{T},\eta}$ is $\boxplus_{\mathcal{T}}$-stable.

Second, since we can always take $k = 1$ in Corollary \ref{cor:CLTestimate}, it gives us an explicit bound for the difference between $\dil_{n^{-1/2}}(\boxplus_{\mathcal{T}}(\mu_1,\dots,\mu_N))$ and $\dil_{n^{-1/2}}(\boxplus_{\mathcal{T}}(\nu_1,\dots,\nu_N))$ that can be applied in the context of any central limit theorem deriving from Proposition \ref{prop:CLT}.  For example, consider the free, Boolean, and monotone cases.  Suppose that $\mathcal{T}_N = \mathcal{T}_{N,\free}$, $\mathcal{T}_{N,\Bool}$, or $\mathcal{T}_{N,\mono}$.  Recall that because the cumulants for $\mathcal{T}_N$ are independent of $N$, the central limit law $\nu_\eta$ is independent of $N$, and only depends on the variance $\eta$.  We apply Corollary \ref{cor:CLTestimate} with $n = N$ and $k = 1$ to obtain the explicit estimate
\[
\norm*{[\dil_{N^{-1/2}} \boxplus_{\mathcal{T}_N}(\mu_1,\dots,\mu_N)](f) - \nu_{\eta}(f)} \leq 2 \Mod_{f,R' + N^{-1/2}R}(N^{-1/2} R),
\]
for the free, Boolean, and monotone cases.

Third, while we have stated the estimate Corollary \ref{cor:CLTestimate} for a non-commutative polynomial $f \in \mathcal{B}\ip{X}$, the same method can be used to estimate the difference $G_\mu^{(n)}(z) - G_\nu^{(n)}(z)$ in the operator-valued Cauchy transforms, where again $\mu = \dil_{n^{-k/2}} \boxplus_{\mathcal{T}^{n^k}}(\mu_1,\dots,\mu_{n^k})$ and $\nu = \dil_{n^{-k/2}} \boxplus_{\mathcal{T}^{n^k}}(\nu_1,\dots,\nu_{n^k})$.  All we need to do is consider a matrix-valued resolvent $(z - Y^{(n)})^{-1}$ rather than a polynomial $f(Y)$.

Consider first the case $\nu_j = \nu_{\Bool,\eta}$.   Then let $Y \sim \mu$ and $Z \sim \nu$ be given by Theorem \ref{thm:CLT2}.  Then $\norm{Y^{(n)} - Z^{(n)}} = \norm{Y - Z}$.  Let $z \in M_n(\mathcal{B})$ and suppose that $z - Y^{(n)}$ and $z - Z^{(n)}$ are invertible.  Then we have by the resolvent identity that
\[
\norm{(z - Y^{(n)})^{-1} - (z - Z^{(n)})^{-1}} \leq \norm{(z - Y^{(n)})^{-1}} \cdot \frac{R}{n^{k/2}} \cdot \norm{(z - Z^{(n)})^{-1}}.
\]
If we assume either that $\im z \geq \epsilon > 0$ or that
\[
\norm{z^{-1}} \leq \frac{1}{R' + R / n^{k/2} + \epsilon},
\]
then we obtain $\norm{(z - Y^{(n)})^{-1}} \leq 1/\epsilon$ and $\norm{(z - Z^{(n)})^{-1}} \leq 1 / \epsilon$.  Thus, for such values of $z$, we have
\[
\norm{(z - Y^{(n)})^{-1} - (z - Z^{(n)})^{-1}} \leq \frac{R}{\epsilon^2 n^{k/2}}
\]
and by taking the expectation, we obtain
\[
\norm{G_\mu^{(n)}(z) - G_\nu^{(n)}(z)} \leq \frac{R}{\epsilon^2 n^{k/2}}.
\]

For the case of a general $\nu_j$ with variance $\eta$ and $\rad(\nu_j) \leq R$, we apply the previous argument to both $\mu_j$ and $\nu_j$ and then use the triangle inequality to obtain $\norm{G_\mu^{(n)}(z) - G_\nu^{(n)}(z)} \leq 2R / \epsilon^2 n^{k/2}$.  In particular, this implies \cite[Theorem 1.1]{MS2013} in the free case (by taking $n = N$ and $k = 1$).

\begin{remark} \label{rem:sharpCLT}
In the case $\mathcal{B} = \C$, one could replace $f$ by an arbitrary continuous function on the real line using functional calculus.  The techniques of \cite{AP2010a}, \cite{AP2010b}, \cite{ANP2016}, \cite{AlPe2017} can be used to estimate the modulus of continuity of $f$ as a map $\mathcal{L}(\mathcal{H})_{sa} \to \mathcal{L}(\mathcal{H})$ in terms of its modulus of continuity as a map $\R \to \C$, and hence to estimate the quantity $\Mod_{f,R}$ defined above.  For instance, \cite[Corollary 7.5]{AP2010a} shows that if $f: \R \to \C$ is Lipschitz and $X$ and $Y$ are self-adjoint with spectrum in $[-R,R]$, then
\[
\norm{f(X) - f(Y)} \leq \text{const} \norm{f}_{\Lip} \norm{X - Y} \log \frac{2eR}{\norm{X - Y}}.
\]
Thus, in the situation of Corollary \ref{cor:CLTestimate} with $\mathcal{B} = \C$, if $\mu = \dil_{n^{-k/2}} \boxplus_{\mathcal{T}^{n^k}}(\mu_1,\dots,\mu_{n^k})$ and $\nu = \dil_{n^{-k/2}} \boxplus_{\mathcal{T}^{n^k}}(\nu_1,\dots,\nu_{n^k})$, then we would obtain
\[
d_{\text{Wasserstein}}(\mu,\nu) \leq \left(C + C' \log \frac{n^{k/2}}{M} \right) \frac{M}{n^{k/2}},
\]
where $C$ and $C'$ are universal constants.  The questions remain of whether we can remove the logarithmic factor, how to estimate other distances such as the Kolmogorov distance, and how to give estimates in terms of moments of $\mu_j$ rather than the operator norm.  These might be better addressed from the complex-analytic viewpoint.  For the sharpest known estimates in the scalar-valued free, Boolean, and monotone cases see \cite{CG2008}, \cite{AS2018}, \cite{ASW2018}.
\end{remark}

\section{Infinitely Divisible Laws and Fock Spaces} \label{sec:infinitelydivisible}

\subsection{Statement of Results}

In this section, we fix $\mathcal{T} \in \Tree(N)$ and suppose that $n = |[N] \cap \mathcal{T}| \geq 2$.

\begin{definition}
A $\mathcal{B}$-valued law $\mu$ is said to be \emph{infinitely divisible with bounded support} if there exist laws $\mu_{n^{-k}}$ for $k \geq 0$ such that
\[
\boxplus_{\mathcal{T}}^{n^k}(\mu_{n^{-k}},\dots,\mu_{n^{-k}}) = \mu
\]
and
\[
\liminf_{k \to \infty} \rad(\mu_{n^{-k}}) < +\infty.
\]
\end{definition}

Our main goal in this section is to characterize the laws that are infinitely divisible with bounded support.  This theorem generalizes previous work on non-commutative independences.  For the free case, see \cite[Theorem 4.3]{Voiculescu1986}, \cite{BV1992}, \cite{Biane1998}, \cite[\S 4.5 - 4.7]{Speicher1998}, \cite[\S 4]{PV2013}, \cite{ABFN2013}.  For the Boolean case, see \cite[Theorem 3.6]{SW1997}, \cite[\S 3]{PV2013}, \cite{ABFN2013}.  For the monotone case, see \cite{Muraki2001}, \cite[Ch.\ 3]{Belinschi2006}, \cite{Hasebe2010a}, \cite{Hasebe2010b}, \cite{HS2014}, \cite{AW2016}, \cite{Jekel2020}.

\begin{remark} \label{rem:IDradiusbound}
In the free, Boolean, and monotone cases, it is not necessary to assume the ``boundedness'' condition $\liminf_{k \to \infty} \rad(\mu_{n^{-k}}) < +\infty$ because it holds automatically (under the assumption that $\mu$ itself is exponentially bounded).  We do not know whether this is true for general trees $\mathcal{T}$.
\end{remark}

\begin{theorem} \label{thm:infinitelydivisible}
Let $\mathcal{T}$ be as above.
\begin{enumerate}[(1)]
	\item If $\mu$ is $\mathcal{T}$-freely infinitely divisible with bounded support, then there exists a unique self-adjoint $c \in \mathcal{B}$ and completely positive exponentially bounded $\sigma: \mathcal{B}\ip{Y} \to \mathcal{B}$ such that
\begin{equation} \label{eq:IDbijection}
\kappa_{\mathcal{T},\ell}(\mu)[b_1,\dots,b_{\ell-1}] =
\begin{cases} c, & \ell = 1, \\
\sigma(b_1 Y b_2 \dots Y b_{\ell-1}), & \ell \geq 2.
\end{cases}
\end{equation}
	We also have $\rad(\sigma) \leq \liminf_{k \to \infty} \rad(\mu_{n^{-k}})$.
	\item Conversely, given $c = \mathrm{C}^* \in \mathcal{B}$ and $\sigma: \mathcal{B}\ip{Y} \to \mathcal{B}$ completely positive and exponentially bounded, there exists a unique law $\mu$ that is infinitely divisible with bounded support and satisfies \eqref{eq:IDbijection}.  We also have
	\begin{equation} \label{eq:IDradiusbound}
	\rad(\mu_{n^{-k}}) \leq \frac{N-1}{n-1} n^{-k} \norm{c} + 2 \sqrt{ \frac{N-1}{n-1} } n^{-k/2} \norm{\sigma(1)}^{1/2} + \rad(\sigma).
	\end{equation}
\end{enumerate}
\end{theorem}

In the next subsection, we prove (1) using the Boolean cumulants.  The proof of (2) will occupy \S \ref{subsec:cumulantmeasures} - \S \ref{subsec:Fockspace}.  Given a completely positive exponentially bounded $\sigma$ and self-adjoint constant $c$ in $\mathcal{B}$, we will obtain the corresponding $\mu$ as the law of an operator $X$ explicitly constructed on a certain $\mathcal{B}$-$\mathcal{B}$-correspondence (the $\mathcal{T}$-free Fock space defined below).

The study of Fock spaces has a long history, and has been present in free probability theory since its inception; see \cite[\S 2 - 3]{Voiculescu1985}.  The relationship between Fock spaces and infinitely divisible laws was first described in Glockner, Sch{\"u}rmann, and Speicher \cite{GSS1992} and the operator-valued case is due to Speicher \cite[\S 4.7]{Speicher1998}.  The operator-valued Boolean case was studied in \cite[Lemma 2.9]{PV2013}.  The operator-valued monotone case was studied in \cite[\S 6]{Jekel2020}.  Also relevant to the free and Boolean cases is \cite[\S 7]{ABFN2013}.

\subsection{Positivity of Cumulants}

To prove $\implies$ of Theorem \ref{thm:infinitelydivisible}, we use the following lemma about the Boolean cumulants.  This lemma actually turns out to be the Boolean case of Theorem \ref{thm:infinitelydivisible} because all laws are Boolean infinitely divisible (see \cite[\S 2]{PV2013}).  For related statements, see \cite[Proposition 3.1]{SW1997}, \cite{BP1999}, \cite[Theorem 5.6, Remark 5.7]{PV2013}, \cite[\S 7]{ABFN2013}, \cite[Corollary 3.3]{Williams2017}.

\begin{lemma} \label{lem:BooleanID}
Let $\mu$ be a $\mathcal{B}$-valued law.  Then there exists a unique completely positive and exponentially bounded $\sigma: \mathcal{B}\ip{Y} \to \mathcal{B}$ such that for $\ell \geq 2$,
\[
\kappa_{\Bool,\ell}(\mu)[b_1,\dots,b_{\ell-1}] = \sigma(b_1 Y \dots Y b_{\ell - 1}).
\]
Conversely, given a self-adjoint $c \in \mathcal{B}$ and a generalized law $\sigma$, there exists a unique law $\mu$ with mean $c$ and Boolean cumulants for $\ell \geq 2$ given by $\sigma$ as above.  Moreover, we have
\[
\rad(\sigma) \leq \rad(\mu) \leq \rad(\sigma) + \norm{\sigma(1)}^{1/2} + \norm{\mu(X)}.
\]
\end{lemma}

\begin{proof}
Let $X$ be the operator of multiplication by $X$ on the $\mathcal{B}$-$\mathcal{B}$-correspondence $\mathcal{H} = \mathcal{B}\ip{X} \otimes_\sigma \mathcal{B}$, let $\xi = 1 \otimes 1$, and let $\mathcal{H}^\circ$ be the orthogonal complement of $\xi$.  Let $P$ be the projection onto $\xi$ and let $Q = 1 - P$.  Define the generalized law $\sigma: \mathcal{B}\ip{Y} \to \mathcal{B}$ by
\[
\sigma(f(Y)) = \ip{\xi, XQ f(QXQ) QX \xi}.
\]
In light of the explicit realization of $\sigma$ by an operator, it is clear that $\sigma$ is completely positive and $\rad(\sigma) \leq \norm{QXQ} \leq \rad(\mu)$.  Moreover, by Lemma \ref{lem:booleancumulants}, the moments of $\sigma$ give the Boolean cumulants of $\mu$.  Uniqueness of $\sigma$ is clear from the well-definedness of the Boolean cumulants.

Conversely, let $\sigma$ and $c$ be given.  Let $\mathcal{K} = \mathcal{B}\ip{Y} \otimes_\sigma \mathcal{B}$ as in Theorem \ref{thm:realizationoflaw} and let $\mathcal{H} = \mathcal{B} \oplus \mathcal{K}$.  Let $\xi = 1 \in \mathcal{B} \subseteq \mathcal{H}$.  Let $Y$ be the operator of multiplication by $Y$ on $\mathcal{K}$ (as in Theorem \ref{thm:realizationoflaw}), which we extend to an operator on $\mathcal{H}$ by setting it to zero on $\mathcal{B}$.  Let $L: \mathcal{H} \to \mathcal{H}$ be given by
\[
L(\xi b + f \otimes b') = 1 \otimes b
\]
and note that
\[
L^*(\xi b + f \otimes b') = \xi \sigma(f) b'.
\]
Let $P$ be the projection onto $\mathcal{B} \xi$ and $Q = 1 - P$.  We define
\[
X = cP + L + L^* + Y
\]
Then observe that $\ip{\xi, X \xi} = c$ and for $\ell \geq 2$,
\begin{align*}
K_{\Bool,\ell}(Xb_1,Xb_2,\dots,Xb_{\ell-1},X) &= \ip{\xi, Xb_1 Q \dots Q Xb_{\ell-1} Q X \xi} \\
&= \ip{\xi, L^* b_1 Y b_2 \dots Y b_{\ell-1} L\xi} \\
&= \sigma(b_1 Y \dots Y b_{\ell-1}).
\end{align*}
Thus, the law $\mu$ of the operator $X$ has the desired properties.  Moreover,
\begin{align*}
\mu \leq \norm{X} &\leq \norm{Y} + \norm{L + L^*} + \norm{c} \\
&\leq \rad(\sigma) + \norm{\sigma(1)}^{1/2} + \norm{c},
\end{align*}
where the estimate for $\norm{L + L^*}$ follows from the fact that $L$ maps $\mathcal{B}$ into $\mathcal{K}$ and $L^*$ does the reverse, and hence
\[
\norm{L + L^*} = \norm{L} = \norm{L^*L}^{1/2} = \norm{L\xi} = \norm{\sigma(1)}^{1/2}.
\]
\end{proof}

\begin{proof}[Proof of Theorem \ref{thm:infinitelydivisible} (1)]
By the previous lemma, there exists a completely positive exponentially bounded $\sigma_{n^{-k}}: \mathcal{B}\ip{Y} \to \mathcal{B}$ such that
\[
\kappa_{\Bool,\ell}(\mu_{n^{-k}})[b_1,\dots,b_{\ell-1}] = \sigma_{n^{-k}}(b_1 Y b_2 \dots Y b_{\ell-1}) \text{ for } \ell \geq 2.
\]
and $\rad(\sigma_{n^{-k}}) \leq \rad(\mu_{n^{-k}})$.  Since the $\mathcal{T}$-free cumulants of $\mu_{n^{-k}}$ are $n^{-k}$ times those of $\mu$, we have
\begin{align*}
n^k \kappa_{\Bool,\ell}(\mu_{n^{-k}})[b_1,\dots,b_{\ell-1}] &= n^k \sum_{\substack{\pi \in \mathcal{NC}(\ell) \\ 1 \sim_\pi \ell}} n^{-k|\pi|} \alpha_{\mathcal{T},\pi} \kappa_{\mathcal{T},\pi}[b_1,\dots,b_{\ell-1}] \\
&= \kappa_{\mathcal{T},\ell}[b_1,\dots,b_{\ell-1}] + O(n^{-k}).
\end{align*}
Therefore, $n^k \sigma_{n^{-k}}$ converges in moments as $k \to \infty$ to some completely positive $\sigma$, which is exponentially bounded since $\rad(\sigma) \leq \liminf_{k \to \infty} \rad(\sigma_{n^{-k}}) \leq \liminf_{k \to \infty} \rad(\mu_{n^{-k}})$, and we have
\[
\kappa_{\mathcal{T},\ell}(\mu)[b_1,\dots,b_{\ell-1}] = \sigma(b_1 X b_2 \dots X b_{\ell-1}).
\]
\end{proof}

\subsection{Measures Modeling the Cumulant Coefficients} \label{subsec:cumulantmeasures}

We begin by introducing certain measures $\theta_\pi$ associated to a partition $\pi$ which will have total variation $\norm{\theta_\pi} = \alpha_{\mathcal{T},\pi}$, where $\alpha_{\mathcal{T},\pi}$ is the coefficient in the moment-cumulant formula.  Of course, the measures $\theta_\pi$ also depend on $\mathcal{T}$, but we suppress the $\mathcal{T}$-dependence for the sake of brevity.

Let $\Omega = [N]^{\N}$ equipped with the product topology and Borel $\sigma$-algebra.  We define Borel measures $\theta_\pi$ on $\Omega^\pi$ as follows using induction on $|\pi|$.

First, if $|\pi| = 1$, then $\theta_\pi$ is the uniform probability distribution on $([N] \cap \mathcal{T})^{\N} \subseteq [N]^{\N}$.  More precisely, $\theta_\pi$ is the probability measure which is the product of $\N$-indexed copies of the probability measure $(1/n) \sum_{j \in [N] \cap \mathcal{T}} \delta_j$ on the finite set $[N]$.  (Recall that although infinite products of measures are not defined in general, they do make sense for \emph{Borel probability} measures on compact metric spaces.)

To set up the inductive step, note that for a coloring $\chi \in [N]^\pi$, we can express $\Omega^\pi$ as the Cartesian product of $\Omega^{\pi'}$ as $\pi'$ ranges over the $\chi$-components of $\pi$.  We can also express $\Omega^\pi$ as $[N]^\pi \times \Omega^\pi$ using the decomposition $\Omega = [N] \times \Omega$, where the $[N]$ factor corresponds to the first coordinate in $\Omega$.  In light of these facts, we have
\begin{equation} \label{eq:productdecomposition}
\Omega^\pi \cong [N]^\pi \times \Omega^\pi \cong [N]^\pi \times \prod_{\chi \text{-components } \pi'} \Omega^{\pi'}.
\end{equation}
We claim that we can define $\theta_\pi$ such that
\begin{equation} \label{eq:partitionmeasureidentity}
\theta_\pi = \frac{1}{n^{|\pi|}} \sum_{\chi \in \mathcal{X}_w(\pi,\mathcal{T})} \delta_{\chi} \times \prod_{\chi \text{-components } \pi'} \theta_{\pi'}.
\end{equation}
If $\graph(\pi)$ has multiple components, then the only partitions $\pi'$ that occur on the right hand side satisfy $|\pi'| < |\pi|$, so it is immediate that we can define $\theta_\pi$ by this relation.

On the other hand, suppose that $\graph(\pi)$ has only one component, and hence the term $\theta_\pi$ occurs both on the left hand side and also on the right hand side, with one occurrence on the right for every constant coloring $\chi$.  In particular, we want
\begin{equation} \label{eq:measurefixedpoint}
\theta_\pi = \frac{1}{n^{|\pi|}} \sum_{\substack{\chi \text{ constant} \\ \chi \in [N] \cap \mathcal{T}}} \delta_\chi \times \theta_\pi + \frac{1}{n^{|\pi|}}\sum_{\text{non-constant } \chi \in \mathcal{X}_w(\pi,\mathcal{T})} \delta_\chi \times \prod_{\chi \text{-components } \pi'} \theta_{\pi'}.
\end{equation}
By the inductive hypothesis, the measures $\theta_{\pi'}$ have been defined for $|\pi'| < |\pi|$.  The right hand side of \eqref{eq:measurefixedpoint}, viewed as a function of $\theta_\pi$ with $\theta_{\pi'}$ fixed for $|\pi'| < |\pi|$, defines a map on the space of finite Borel measures on $\Omega^\pi$.  Moreover, this map is Lipschitz with respect to the total variation metric with Lipschitz constant $n / n^{|\pi|} < 1$.  Therefore, by the Banach fixed point theorem, there is a unique fixed point, and we define $\theta_\pi$ to be this fixed point.

We also establish the convention that $\Omega^{\varnothing}$ consists of a single point space, $\theta_\varnothing$ is the probability measure on this space, and $\alpha_{\mathcal{T},\varnothing} = 1$.

\begin{lemma}
The total mass of $\theta_\pi$ satisfies $\norm{\theta_\pi} = \alpha_{\mathcal{T},\pi}$.
\end{lemma}

\begin{proof}
By evaluating the total mass of both sides of \eqref{eq:partitionmeasureidentity}, we obtain
\[
\norm{\theta_\pi} = \frac{1}{n^{|\pi|}} \sum_{\substack{\chi \in [N]^\pi \\ \pi / \chi \in \mathcal{NC}(\chi,\mathcal{T})}} \prod_{\chi \text{-components } \pi'} \norm{\theta_{\pi'}}.
\]
Thus, $\norm{\theta_\pi}$ satisfies the same identity that was used to define $\alpha_{\mathcal{T},\pi}$ in \eqref{eq:partitioncoefficientidentity}.  Also, for $|\pi| = 1$, we have $\norm{\theta_\pi} = 1 = \alpha_{\mathcal{T},\pi}$.  Thus, it follows by induction that $\norm{\theta_\pi} = \alpha_{\mathcal{T},\pi}$ for all $\pi$.
\end{proof}

Next, we provide a more explicit description of the measures $\theta_\pi$, both for the sake of our main goal of constructing the Fock space and to aid with the computation of examples later on.

Note that $\Omega^\pi \cong ([N]^\pi)^{\N}$, so that a point in $\Omega^\pi$ corresponds to an infinite sequence of colorings $\vec{\chi} = (\chi_1,\chi_2,\dots)$.  In particular, $(\chi_1,\dots,\chi_j)$ defines a coloring $\pi \to [N]^j$, so therefore it makes sense to talk about the $(\chi_1,\dots,\chi_j)$-components of $\pi$.  We define $\mathcal{X}_w^{(k)}(\pi,\mathcal{T}) \subseteq ([N]^\pi)^k$ to be the set of tuples $(\chi_1,\dots,\chi_k)$ such that
\[
\text{for every } j < k \text{ and every } (\chi_1,\dots,\chi_j) \text{-component } \pi', \text{ we have } \chi_{j+1} \in \mathcal{X}_w(\pi',\mathcal{T}).
\]

Let $\theta_\pi^{(k)}$ be the measure on $([N]^\pi)^k$ given by
\[
\theta_\pi^{(k)} = \frac{1}{n^{|\pi|k}} \sum_{(\chi_1,\dots,\chi_k) \in \mathcal{X}_w^{(k)}(\pi,\mathcal{T})} \delta_{(\chi_1,\dots,\chi_k)},
\]
that is, the uniform distribution on $\mathcal{X}_w^{(k)}(\pi,\mathcal{T})$ normalized to have total mass $|\mathcal{X}_w^{(k)}(\pi,\mathcal{T})| / n^k$.  Let $u_\pi$ be the probability measure on $([N]^\pi)^{\N}$ given by
\[
u_\pi = \prod_{k \in \N} \prod_{V \in \pi} \left( \frac{1}{n} \sum_{j \in [N] \cap \mathcal{T}} \delta_j \right),
\]
that is, the probability measure on infinite sequences $\vec{\chi}$ given by choosing $\chi_j(V)$ for each $j$ and $V$ independently from the uniform probability distribution on $[N] \cap \mathcal{T}$.  Considering the decomposition $\Omega^\pi \cong ([N]^\pi)^{\N} \cong ([N]^\pi)^k \times ([N]^\pi)^{\N}$, we may view the product measure $\theta_\pi^{(k)} \times u_\pi$ as a finite Borel measure on $\Omega^\pi$.

\begin{proposition} \label{prop:explicitmeasures}
If $\theta_\pi^{(k)}$ and $u_\pi$ are defined as above, then we have $\theta_\pi^{(k)} \times u_\pi \to \theta_\pi$ in total variation.
\end{proposition}

\begin{proof}
First, one can show by a direct computation that with respect to the product decomposition \eqref{eq:productdecomposition} we have
\[
\theta_\pi^{(k+1)} = \frac{1}{n^{|\pi|}} \sum_{\chi \in \mathcal{X}_w(\pi,\mathcal{T})} \delta_{\chi} \times \prod_{\chi \text{-components } \pi'} \theta_{\pi'}^{(k)}.
\]
In other words, the tuple $(\theta_\pi^{(k)})_\pi$ indexed by non-crossing partitions is obtained by iterating the function on such tuples given by the right hand side of \eqref{eq:partitionmeasureidentity}, whereas $(\theta_\pi)_{\pi}$ itself is a fixed point of this function.  Since we are changing the measures for each partition $\pi$ simultaneously, convergence does not immediately follow from the fact that the right hand side of \eqref{eq:partitionmeasureidentity} is a contraction with respect to the single variable $\theta_\pi$.

However, we will show directly that there exists a polynomial $f_\pi$ such that
\[
\norm{\theta_\pi^{(k)} \times u_\pi - \theta_\pi} \leq \frac{f_\pi(k)}{n^k}.
\]
We proceed by induction on $|\pi|$.  For the base case $|\pi| = 1$, we note that
\[
\theta_\pi^{(k)} \times u_\pi = u_\pi = \theta_\pi \text{ for all } k,
\]
so we can take $f_\pi = 0$.  For the induction step, observe that
\[
\theta_\pi^{(k+1)} - \theta_\pi = \frac{1}{n^{|\pi|}} \sum_{\chi \in \mathcal{X}_w(\pi,\mathcal{T})} \delta_{\chi} \times \left( \prod_{\chi \text{-components } \pi'} \theta_{\pi'}^{(k)} - \prod_{\chi \text{-components } \pi'} \theta_\pi \right).
\]
Letting $M_\pi := \sup_k \norm{\theta_\pi^{(k)}}$, we have
\begin{align*}
\norm*{ \prod_{\chi \text{-components } \pi'} \theta_{\pi'}^{(k)} - \prod_{\chi \text{-components } \pi'} \theta_\pi }
&\leq \sum_{\chi \text{-components } \pi'} \norm{\theta_{\pi'}^{(k)} \times u_\pi - \theta_{\pi'}} \prod_{\pi'' \neq \pi'} M_{\pi''}.
\end{align*}

In the case where $\pi$ is reducible (or equivalently $\graph(\pi) \setminus \{\emptyset\}$ has multiple components), every $\chi$-component $\pi'$ satisfies $|\pi'| < |\pi|$.  Thus, we may apply the induction hypothesis to $\pi'$.  The induction hypothesis applied to $\pi''$ above also implies that $M_{\pi''} < +\infty$.  Therefore, we have
\[
\norm{\theta_\pi^{(k+1)} - \theta_\pi} \leq \frac{1}{n^{|\pi|}} \sum_{\chi \in \mathcal{X}_w(\pi,\mathcal{T})} \sum_{\chi \text{-components } \pi'} \frac{f_{\pi'}(k)}{n^k} \prod_{\pi'' \neq \pi'} M_{\pi''},
\]
which proves the claim for $\pi$.  

In the case where $\pi$ is irreducible, we must separate out the $n$ constant colorings of $\pi$ in the sum and we get
\begin{align*}
\norm{\theta_\pi^{(k+1)} \times u_\pi - \theta_\pi} &\leq \frac{n}{n^{|\pi|}} \norm{\theta_\pi^{(k)} \times u_\pi - \theta_\pi} + \frac{1}{n^k} g(k) \\
&\leq \frac{1}{n} \norm{\theta^{(k)} \times u_\pi - \theta_\pi} + \frac{1}{n^k} g(k).
\end{align*}
where
\[
g(k) = \frac{1}{n^{|\pi|}} \sum_{\text{non-constant } \chi \in \mathcal{X}_w(\pi,\mathcal{T})} \sum_{\chi \text{-components } \pi'} f_{\pi'}(k) \prod_{\pi'' \neq \pi'} M_{\pi''}.
\]
It follows that
\[
n^{k+1} \norm{\theta_\pi^{(k+1)} \times u_\pi - \theta_\pi} \leq n^k \norm{\theta_\pi^{(k)} \times u_\pi - \theta_\pi} + g(k).
\]
We can define a polynomial $f_\pi$ by
\[
f_\pi(k) = \norm{\theta_\pi^{(0)} \times u_\pi - \theta_\pi} + \sum_{j=0}^{k-1} g(j),
\]
that is, $f_\pi$ is the discrete antiderivative of $g$.  Then we have by induction that
\[
n^k \norm{\theta_\pi^{(k)} \times u_\pi - \theta_\pi} \leq f_\pi(k)
\]
as desired.
\end{proof}

\begin{remark} \label{rem:measureswithN=n}
In the important special case that $n = N$, we have an even more explicit description of the measure.  Note that in this case $u_\pi$ is the uniform probability distribution on $\Omega^\pi \cong ([N]^\pi)^{\N}$.  Then we have
\[
d(\theta_\pi^{(k)} \times u_\pi) = 1_{\mathcal{X}_w^{(k)}(\pi,\mathcal{T}) \times \Omega^\pi} du_\pi,
\]
that is, the $\theta_\pi^{(k)} \times u_\pi$ is the uniform distribution $u_\pi$ restricted to $\mathcal{X}_w^{(k)}(\pi,\mathcal{T}) \times \Omega^\pi$.  By taking $k \to \infty$, we see that
\[
d\theta_\pi = 1_{\mathcal{X}_w^{(\infty)}} du_\pi,
\]
where
\[
\mathcal{X}_w^{(\infty)}(\pi,\mathcal{T}) = \bigcap_{k=0}^\infty \mathcal{X}_w^{(k)}(\pi,\mathcal{T}) \times \Omega^\pi,
\]
or equivalently $\mathcal{X}_w^{(\infty)}(\pi,\mathcal{T})$ is the set of tuples $\vec{\chi} = (\chi_1,\chi_2,\dots)$ such that for each $(\chi_1,\dots,\chi_j)$-component $\pi'$, we have $\chi_{j+1}|_{\pi'} \in \mathcal{X}_w(\pi',\mathcal{T})$.
\end{remark}

Next, we will prove a disintegration result for the measures $\theta_\pi$.  Suppose that $\pi \in \mathcal{NC}(\ell)$ and that $V$ is a maximal block of $\pi$ with respect to the nesting order $\prec$ (that is, $V$ does not surround any other blocks of $\pi$), and suppose that $\chain(V) = (V,V_1,\dots,V_\ell)$.  Then we will show (Lemma \ref{lem:disintegration}) that there is a disintegration of measures,
\[
d\theta_\pi((\omega_W)_{W \in \pi}) = d\gamma_{\omega_{V_1},\dots,\omega_{V_\ell}}(\omega_V) \, d\theta_{\pi \setminus V}((\omega_W)_{W \in \pi \setminus V}).
\]
where $\gamma_{\omega_1,\dots,\omega_\ell}$ is a family of measures indexed by $d$-tuples of elements of $\Omega$.

In probabilistic language, this means that choosing $(\omega_W)_{W \in \pi}$ according to $\theta_\pi$ is equivalent to first choosing $(\omega_W)_{W \in \pi \setminus V}$ according to $\theta_{\pi \setminus V}$ and then choosing $\omega_V$ according to a certain ``conditional distribution'' which only depends on $\omega|_{\chain(V)}$.  (Of course, since $\theta_\pi$ is not necessarily a probability measure, we are using the word ``conditional distribution'' loosely.)

In order to define $\gamma_{\omega_1,\dots,\omega_\ell}$, we use some auxiliary notation, closely related to the notation used in the proof of Lemma \ref{lem:iterateddegreebound}.  Let $\phi_i: \Omega \to [N]$ be the projection onto the $k$th coordinate in the product decomposition $\Omega = [N]^{\N}$.  For $(\omega_1,\dots,\omega_d) \in \Omega^\ell$, let $(\phi_i)_*(\omega_1,\dots,\omega_\ell)$ be the string $\phi_i(\omega_1) \dots \phi_i(\omega_\ell)$.  We define the \emph{level-$k$ components of $[\ell]$ with respect to $\omega$} as the maximal subintervals $I$ of $[\ell]$ on which $(\phi_1(\omega_j),\dots,\phi_k(\omega_j))$ is constant for $j \in I$.  By convention, $[\ell]$ is considered to be a single level-$0$ component.

We define $r_i = r_i(\omega_1,\dots,\omega_\ell)$ to be the index in $[\ell]$ such that $[r_i]$ is the level-$k$ component of $[\ell]$ containing $1$.  We define the string $t_i = t_i(\omega_\ell,\dots,\omega_1)$ by
\[
t_i = \red((\phi_i)_*(\omega_1,\dots,\omega_\ell)|_{[r_i]}) = \red(\phi_i(\omega_1),\dots,\phi_i(\omega_{r_i}).
\]
Then we define
\begin{equation} \label{eq:definegamma}
\gamma_{\omega_1,\dots,\omega_\ell} = \sum_{i=0}^\infty \frac{1}{n} \delta_{(\omega_1)_1} \times \dots \times \frac{1}{n} \delta_{(\omega_1)_i} \times \left( \frac{1}{n} \sum_{j: j t_i \in  \mathcal{T}} \delta_j \right) \times v^{\N},
\end{equation}
where $v = \frac{1}{n} \sum_{j \in [N] \cap \mathcal{T}} \delta_j$ is the uniform probability measure on $[N] \cap \mathcal{T}$.  In the case $\ell = 0$, we define $\gamma$ to be the measure $v^{\times \N}$.

\begin{lemma} \label{lem:disintegration} ~
\begin{enumerate}[(1)]
	\item The measure $\gamma_{\omega_1,\dots,\omega_\ell}$ depends Borel-measurably on $(\omega_1,\dots,\omega_\ell)$.
	\item We have $\norm{\gamma_{\omega_1,\dots,\omega_\ell}} \leq (N-1) / (n-1)$.
	\item Let $V$ be a maximal block of $\pi$, and let $V_\ell \prec V_{\ell-1} \prec \dots \prec V_1 \prec V$ be the blocks containing $V$.  Then we have a disintegration of measures
	\begin{equation} \label{eq:disintegration}
	d\theta_\pi( (\omega_W)_{W \in \pi} ) = d\gamma_{\omega_{V_1},\dots,\omega_{V_\ell}}(\omega_V) \,d\theta_{\pi \setminus V}((\omega_W)_{W \in \pi \setminus V}).
	\end{equation}
\end{enumerate}
\end{lemma}

\begin{proof}
(1) is an exercise.

(2) The case $\ell = 0$ is immediate since the total variation is $1$.  For $\ell > 0$, we observe that
\[
\norm{\gamma_{\omega_1,\dots,\omega_\ell}} = \sum_{i=0}^\infty \frac{1}{n^{i+1}} |\{j: j t_i \in \mathcal{T}\}| \leq \sum_{i=0}^\infty \frac{N-1}{n^{i+1}} = \frac{N-1}{n-1}.
\]
(Compare the proof of Lemma \ref{lem:iterateddegreebound}.)

(3) We proceed by proving an approximate version of the statement for $\theta_\pi^{(k)}$, which in turn follows from describing $\mathcal{X}_w^{(k)}(\pi,\mathcal{T})$ in terms of $\mathcal{X}_w^{(k)}(\pi \setminus V, \mathcal{T})$.  If $\chi$ is a coloring of $\pi$, then we have $\chi \in \mathcal{X}_w(\pi,\mathcal{T})$ if and only $\chi|_{\pi \setminus V} \in \mathcal{X}_w(\pi \setminus V, \mathcal{T})$ and $\red(\chi(\chain(V))) \in \mathcal{T}$ (similar to the proof of Lemma \ref{lem:iterateddegreebound}).  Based on this fact, one can show that for $(\chi_1,\dots,\chi_k) \in ([N]^\pi)^k$, the condition $(\chi_1,\dots,\chi_k) \in \mathcal{X}_w^{(k)}(\pi,\mathcal{T})$ is equivalent to
\begin{enumerate}[(a)]
	\item $(\chi_1|_{\pi \setminus V}, \dots, \chi_k|_{\pi \setminus V}) \in \mathcal{X}_w^{(k)}(\pi \setminus V,\mathcal{T})$;
	\item for each $i$, if $\pi'$ is the $(\chi_1,\dots,\chi_i)$-component of $\pi$ that contains $V$, then we have
	\[
	\red[\chi_{i+1}(\chain_{\pi'}(V))] \in \mathcal{T}.
	\]
\end{enumerate}
If $V$ is in its own $(\chi_1,\dots,\chi_i)$-component of $\pi$, then (b) is equivalent to $\chi_{i+1}(V) \in [N] \cap \mathcal{T}$.  Otherwise, we have $\chain_{\pi'}(V) = (V,V_1,\dots,V_m)$ for some $m$ with $1 \leq m \leq \ell$.  Thus, (b) is equivalent to $\red(\chi_{i+1}(V) \dots \chi_{i+1}(V_m)) \in \mathcal{T}$, which is in turn equivalent (assuming that (a) holds) to $\chi_{i+1}(V) = \chi_{i+1}(V_1)$ or $\chi_{i+1}(V) \red(\chi_{i+1}(V_1) \dots \chi_{i+1}(V_m)) \in \mathcal{T}$.

The above argument describes all the possible ways to extend a tuple $(\chi_1',\dots,\chi_k') \in \mathcal{X}_w^{(k)}(\pi \setminus V, \mathcal{T})$ to a tuple in $\mathcal{X}_w^{(k)}(\pi,\mathcal{T})$ (and every tuple in $\mathcal{X}_w^{(k)}(\pi,\mathcal{T})$ is obtained in this way).  It follows that
\begin{multline} \label{eq:initialdisintegration}
d\theta_\pi^{(k)}(\chi_1,\dots,\chi_k) = d\Gamma_{\chi_1|_{\chain(V)}, \dots, \chi_k|_{\chain(V)}}^{(k)}(\chi_1(V),\dots,\chi_k(V)) \\
d\theta_{\pi \setminus V}^{(k)}(\chi_1|_{\pi \setminus V}, \dots, \chi_k|_{\pi \setminus V}),
\end{multline}
where $\Gamma_{\chi_1|_{\chain(V)},\dots,\chi_k|_{\chain(V)}}$ is the measure on $[N]^k$ given by
\begin{align*}
\Gamma_{\chi_1|_{\chain(V)},\dots,\chi_k|_{\chain(V)}}^{(k)} &= \sum_{i=0}^{k-1} \frac{1}{n} \delta_{\chi_1(V_1)} \times \dots \times \frac{1}{n} \delta_{\chi_i(V_1)} \times \left( \frac{1}{n} \sum_{j: j \red(\chi(V_1),\dots,\chi(V_{m_i})) \in \mathcal{T}} \delta_j \right) \times v^{\times (k-1-i)} \\
& \quad + \frac{1}{n^k} \delta_{\chi_1(V_1)} \times \dots \delta_{\chi_k(V_1)},
\end{align*}
where $m_i$ is the index such that $V_1$, \dots, $V_{m_i}$ are the blocks of $\chain(V)$ which are in the same $(\chi_1,\dots,\chi_i)$-component of $V$.

We must now translate \eqref{eq:initialdisintegration} from $([N]^\pi)^k$ coordinates into $([N]^k)^\pi$ coordinates (and hence $\Omega^\pi$ coordinates).  Suppose that $(\chi_1,\dots,\chi_k)$ is the first $k$ coordinates of a tuple $\vec{\chi} \in (\Omega^\pi)^{\N}$ corresponding to a point $(\omega_V)_{V \in \pi} \in (\Omega^{\N})^\pi$.  Then the string $\red(\chi_{j+1}(V_1) \dots \chi_{j+1}(V_m)) = \red(\chi(\chain_{\pi'}(V)))$ used in condition (b) is precisely $t_j(\omega_{V_1},\dots,\omega_{V_\ell}) = t_j(\omega|_{\chain_\pi(V)})$.  This implies that
\[
d[\theta_\pi^{(k)} \times u^\pi]((\omega_W)_{W\in\pi}) = d\gamma_{\omega_{V_1},\dots,\omega_{V_\ell}}^{(k)}(\omega_V) \, d[\theta_{\pi \setminus V} \times u^{\pi \setminus V}]((\omega_W)_{W \in \pi \setminus V}),
\]
where $\gamma_{\omega_1,\dots,\omega_\ell}^{(k)}$ is the measure on $\Omega$ given by
\begin{align*}
\gamma_{\omega_1,\dots,\omega_\ell}^{(k)} &= \sum_{i=0}^k \frac{1}{n} \delta_{(\omega_1)_1} \times \dots \times \frac{1}{n} \delta_{(\omega_1)_i} \times \left( \frac{1}{n} \sum_{j: j t_i \in  \mathcal{T}} \delta_j \right) \times v^{\times \N} \\
& \quad + \frac{1}{n^k} \delta_{(\omega_1)_1} \times \dots \times \delta_{(\omega_1)_k} \times v^{\times \N}.
\end{align*}
In light of the estimates in the proof of (2), we see that $\gamma_{\omega_1,\dots,\omega_\ell}^{(k)} \to \gamma_{\omega_1,\dots,\omega_\ell}$ in total variation as $k \to \infty$, and in fact the rate of convergence is independent of $(\omega_1,\dots,\omega_\ell)$.  Because $\theta_\pi^{(k)} \times u^\pi \to \theta_\pi$, we obtain \eqref{eq:disintegration} in the limit.
\end{proof}

\subsection{Operators Modeling the Cumulant Coefficients}

Let us define a measure $\gamma_\ell$ on $\Omega^\ell$ inductively by setting $\gamma_0$ to be the (unique) probability measure on the one-point space $\Omega^0$ and setting
\[
d\gamma_{\ell+1}(\omega_1,\dots,\omega_{\ell+1}) = d\gamma_{\omega_2,\dots,\omega_{\ell+1}}(\omega_1)\,d\gamma_\ell(\omega_2,\dots,\omega_{\ell+1}).
\]

We define the operator $S_\ell: L^2(\Omega^\ell,\gamma_\ell) \to L^2(\Omega^{\ell+1},\gamma_{\ell+1})$ by
\[
(S_\ell f)(\omega_1,\dots,\omega_{\ell+1}) = f(\omega_2,\dots,\omega_{\ell+1}).
\]
Note that $S_\ell$ is a well-defined bounded operator because $\gamma_{\ell+1}$ is defined by integrating against $\gamma_\ell$ the measures $\gamma_{\omega_2,\dots,\omega_{\ell+1}}$ which each have total mass $\leq (N - 1) / (n - 1)$ by Lemma \ref{lem:disintegration}.  More precisely, $\norm{S_\ell} \leq \sqrt{(N - 1) / (n - 1)}$.  We also have
\[
(S_\ell^*f)(\omega_1,\dots,\omega_\ell) = \int_{\Omega} f(\omega_0,\omega_1,\dots,\omega_\ell)\,d\gamma_{\omega_1,\dots,\omega_\ell}(\omega_0).
\]

\begin{lemma} \label{lem:Fockcoefficients}
Let $\pi \in \mathcal{NC}(k)$.  For $j \in V \in \pi$, denote
\[
T_j = \begin{cases}
S_{\depth_\pi(V)-1}^* S_{\depth_\pi(V)-1}, & |V| = 1 \\
S_{\depth_\pi(V)-1}^*, & |V| > 1, j = \min V \\
S_{\depth_\pi(V)-1}, & |V| > 1, j = \max V \\
1, & \text{ otherwise,}
\end{cases}
\]
where in the last case, $1$ represents the identity on $L^2(\Omega^\ell,\gamma_\ell)$ for $\ell = \depth(V)$.  Then the domain of $T_j$ equals the codomain of $T_{j+1}$, so that the composition $T_1 \dots T_k$ is a well-defined operator on $L^2(\Omega^0,\gamma_0) = \C$.  Since $\mathcal{L}(\C) \cong \C$, we may view $T_1 \dots T_k$ as a scalar.  Then
\[
T_1 \dots T_k = \alpha_{\mathcal{T},\pi}.
\]
\end{lemma}

Although this lemma is the only fact we need for the Fock space construction in the next subsection, it will be helpful for the sake of induction to prove a more general statement.  For $\omega \in \Omega^\pi$, we denote
\[
\omega|_{\chain(V)} = (\omega_V, \omega_{V_1},\dots,\omega_{V_\ell}).
\]
Moreover, for $f \in C(\Omega^\ell)$, let us define
\[
M_\ell(f): L^2(\Omega^\ell,\gamma_\ell) \to L^2(\Omega^\ell,\gamma_\ell)
\]
to be the operator of multiplication by $f$.

\begin{lemma} \label{lem:Fockcoefficients2}
Let $\pi \in \mathcal{NC}(k)$.  Let $V_j$ denote the block of $\pi$ containing $j$ (here the $V_j$'s are not necessarily distinct).  Fix $f_j \in C(\Omega^{\depth(V_j)})$ for each $j$, and then define
\[
T_j = \begin{cases}
S_{\depth_\pi(V_j)-1}^* M_{\depth(V_j)}(f_j) S_{\depth_\pi(V_j)-1}, & |V_j| = 1 \\
S_{\depth_\pi(V_j)-1}^* M_{\depth(V_j)}(f_j), & |V_j| > 1, j = \min V_j \\
M_{\depth(V_j)}(f_j)S_{\depth_\pi(V_j)-1}, & |V_j| > 1, j = \max V_j \\
M_{\depth(V_j)}(f_j), & \text{ otherwise.}
\end{cases}
\]
Then $T_1 \dots T_k$ makes sense and is a map $\C \to \C$ given by some scalar, and we have
\begin{equation} \label{eq:cumulantmeasureintegration}
T_1 \dots T_k = \int_{\Omega^\pi} \prod_{j=1}^k f_j(\omega|_{\chain(V_j)})\,d\theta_\pi(\omega).
\end{equation}
\end{lemma}

This immediately implies Lemma \ref{lem:Fockcoefficients} because we can take $f_j = 1$ for all $j$.

\begin{proof}[Proof of Lemma \ref{lem:Fockcoefficients2}]
We proceed by induction on $|\pi|$.  We take as the base case $\pi = \varnothing$.  Now suppose that $|\pi| \geq 1$.  Choose a block $V$ of $\pi$ that is maximal with respect to $\prec$.  Then $V$ is an interval block and hence can be expressed as $\{i+1,\dots,j\}$ for some $1 \leq i < j \leq k$.  Let $d = \depth(V)$.  A direct computation shows that
\[
S_{d-1}^* M_d(f_{i+1}) \dots M_d(f_j) S_{d-1} = M_{d-1}(g),
\]
where
\[
g(\omega_1,\dots,\omega_{d-1}) = \int_{\Omega} (f_{i+1} \dots f_j)(\omega_0,\omega_1,\dots,\omega_{d-1})\,d\gamma_{\omega_1,\dots,\omega_{d-1}}(\omega_0).
\]

In the case where $d = 1$, we have
\[
M_0(g) = \int_{\Omega^{\{V\}}} (f_{i+1} \dots f_j)(\omega_V)\,d\theta_{\{V\}}(\omega_V)
\]
and applying the inductive hypothesis to $\pi \setminus V$, we have
\begin{align*}
T_1 \dots T_k &= T_1 \dots T_i T_{j+1} \dots T_k \int_{\Omega^{\{V\}}} (f_{i+1} \dots f_j)(\omega_V)\,d\theta_{\{V\}}(\omega_V) \\
&= \int_{\Omega^{\pi \setminus V}} f_1(\omega|_{\chain(V_1)}) \dots f_i(\omega|_{\chain(V_i)}) f_{j+1}(\omega|_{\chain(V_{j+1})}) \dots f_k(\omega|_{\chain(V_k)})\,d\theta_{\pi \setminus V}(\omega) \\
& \qquad \int_{\Omega^{\{V\}}} (f_{i+1} \dots f_j)(\omega_V)\,d\theta_{\{V\}}(\omega_V) \\
&= \int_{\Omega^\pi} \prod_{j=1}^k f_j(\omega|_{\chain(V_j)})\,d\theta_\pi(\omega).
\end{align*}
since $\theta_{\pi \setminus V} = \theta_\pi \times \theta_{\{V\}}$ in this case.

On the other hand, suppose that $d > 1$.  Let $V'$ be the maximal block $\prec V$ (that is, the parent of $V$ in the rooted tree $\graph(\pi)$), and note that $i \in V'$ and $i < \max(V')$.  We can apply the inductive hypothesis to $\pi \setminus V$ with the list of functions $(f_1,\dots,f_{i-1}, f_i g, f_{j+1}, \dots, f_k)$ and obtain
\begin{multline*}
T_1 \dots T_i M_{d-1}(g) T_{j+1} \dots T_k \\
= \int_{\Omega^{\pi \setminus V}} f_1(\omega|_{\chain(V_1)}) \dots f_i(\omega|_{\chain(V_i)}) g(\omega|_{\chain(V')}) f_{j+1}(\omega|_{\chain(V_{j+1})}) \dots f_k(\omega|_{\chain(V_k)})\,d\theta_{\pi \setminus V}(\omega)
\end{multline*}
By our choice of $g$, this is equal to
\begin{multline*}
\int_{\Omega^{\pi \setminus V}}  f_1(\omega|_{\chain(V_1)}) \dots f_i(\omega|_{\chain(V_i)}) \\
\left( \int_{\Omega^{\{V\}}} f_{i+1}(\omega|_{\chain(V)}) \dots f_j(\omega|_{\chain(V)}) d\gamma_{\omega|_{\chain(V')}}(\omega_V) \right) \\
f_{j+1}(\omega|_{\chain(V_{j+1})} \dots f_k(\omega|_{\chain(V_k)})\,d\theta_{\pi \setminus V}(\omega)
\end{multline*}
But by Lemma \ref{lem:disintegration}, this is equal to the right hand side of \eqref{eq:cumulantmeasureintegration}, which completes the proof.
\end{proof}

\subsection{The $\mathcal{T}$-free Fock Space} \label{subsec:Fockspace}

Let $\mathcal{K}$ be a $\mathcal{B}$-$\mathcal{B}$-correspondence (for instance, we will take $\mathcal{K} = \mathcal{B}\ip{Y} \otimes_\sigma \mathcal{B}$ later on to prove Theorem \ref{thm:infinitelydivisible} (2)).  We define the \emph{$\mathcal{T}$-free Fock space over $\mathcal{K}$} to be the space
\begin{equation}
\mathcal{F}_{\mathcal{T}}(\mathcal{K}) = \mathcal{B} \xi \oplus \bigoplus_{\ell=1}^\infty \bigl( \underbrace{\mathcal{K} \otimes_{\mathcal{B}} \dots \otimes_{\mathcal{B}} \mathcal{K}}_\ell \bigr) \otimes_{\C} L^2(\Omega^\ell, \gamma_\ell),
\end{equation}
where $\mathcal{B}\xi$ is a copy of the $\mathcal{B}$-$\mathcal{B}$-correspondence $\mathcal{B}$.  We will define four types of operators on $\mathcal{F}(\mathcal{K})$, including creation and annihilation operators, and two types of multiplication operators.

We define creation and annihilation operators on $\mathcal{F}(\mathcal{K})$ as follows.  For $\zeta \in \mathcal{K}$, we define
\[
L_\ell(\zeta): \mathcal{K}^{\otimes_{\mathcal{B}} \ell} \to \mathcal{K}^{\otimes_{\mathcal{B}} \ell + 1}
\]
by
\[
L(\zeta) [(\zeta_1 \otimes \dots \otimes \zeta_\ell) \otimes f] = (\zeta \otimes \zeta_1 \otimes \dots \otimes \zeta_\ell) \otimes S_\ell f,
\]
In the case $\ell = 0$, the vector $\zeta_1 \otimes \dots \otimes \zeta_n$ is to be interpreted as an element of $\mathcal{B} \xi$.  Observing that
\[
\norm{\zeta \otimes \zeta_1 \otimes \dots \otimes \zeta_\ell}^2 = \norm{ \ip{\zeta,\zeta}^{1/2} \zeta_1 \otimes \dots \otimes \zeta_\ell}^2,
\]
we see that
\[
\norm{L_\ell(\zeta)} \leq \norm{\zeta}.
\]
Moreover, $L_\ell(\zeta)$ is adjointable with its adjoint being given by
\[
L_\ell(\zeta)^* (\zeta_1 \otimes \dots \otimes \zeta_\ell) = \begin{cases} 0, & \ell = 0 \\ \ip{\zeta,\zeta_1} \zeta_2 \otimes \dots \otimes \zeta_\ell, & \ell > 0.  \end{cases} 
\]
If we let $S_\ell$ be the operator defined in the previous subsection, then there is a bounded adjointable operator $L(\zeta): \mathcal{F}(\mathcal{K}) \to \mathcal{F}(\mathcal{K})$ given by
\begin{align*}
L(\zeta)|_{\mathcal{K}^{\otimes_{\mathcal{B}} \ell} \otimes L^2(\Omega^\ell,\gamma_\ell)} &= L_\ell(\zeta) \otimes S_\ell \\
L(\zeta)^*|_{\mathcal{K}^{\otimes_{\mathcal{B}} \ell} \otimes L^2(\Omega^\ell,\gamma_\ell)} &= \begin{cases} 0, & \ell = 0 \\ L_{\ell-1,\free}(\zeta)^* \otimes S_\ell^*, & \ell > 0. \end{cases}
\end{align*}
We call $L(\zeta)$ the \emph{creation operator} and $L(\zeta)^*$ the \emph{annihilation operator} associated to $\zeta$.  Note that
\[
\norm{L(\zeta)} = \norm{L(\zeta)^*} \leq \sqrt{\frac{N-1}{n-1}} \norm{\zeta}.
\]

For $x \in \mathcal{L}(\mathcal{K})$, there is a bounded operator $M_\ell(x): \mathcal{K}^{\otimes_{\mathcal{B}} \ell} \to \mathcal{K}^{\otimes_{\mathcal{B}} \ell}$ given by
\[
M_\ell(x)(\zeta_1 \otimes \dots \otimes \zeta_\ell) = \begin{cases} 0, & \ell = 1 \\
 (x \zeta_1 \otimes \zeta_2 \otimes \dots \otimes \zeta_\ell), & \ell > 0.
\end{cases}
\]
We define the \emph{multiplication operator} $M(x)$ as the direct sum of the operators $M_\ell(x) \otimes \id_{L^2(\Omega^\ell,\gamma_\ell)}$.  Note that $M$ defines a $*$-homomorphism $\mathcal{L}(\mathcal{K}) \to \mathcal{L}(\mathcal{F}_{\mathcal{T}}(\mathcal{K}))$.

For $b \in \mathcal{B}$, we define the \emph{multiplication operator} $M'(b)$ as the direct sum of the operators $M_\ell'(b) \otimes S_\ell^* S_\ell$ on $\mathcal{K}^{\otimes_{\mathcal{B}} \ell} \otimes L^2(\Omega^\ell,\gamma_\ell)$, where $M_\ell'(b)$ denotes the left multiplication action of $b$ on $\mathcal{K}^{\otimes_{\mathcal{B}} \ell}$ from the $\mathcal{B}$-$\mathcal{B}$-correspondence structure.  Here we take $M_0'(b) b'\xi = bb' \xi$; note in contrast that in the case above we took $M_0(x) = 0$.

\begin{theorem}
Suppose that
\[
a_j = M'(b_j) + L(\zeta_j)^* + L(\zeta_j') + M(x_j) \in \mathcal{L}(\mathcal{F}_{\mathcal{T}}(\mathcal{K})),
\]
where $x_j \in \mathcal{L}(\mathcal{K})$ and $b_j \in \mathcal{B}$ and $\zeta_j, \zeta_j' \in \mathcal{K}$ for $j = 1$, \dots, $\ell$.  Then we have
\begin{equation} \label{eq:Fockcumulantformula}
K_{\mathcal{T},\ell}[a_1,\dots,a_\ell] = \begin{cases} b_1, & \ell = 1 \\ \ip{\zeta_1, x_2 \dots x_{\ell-1} \zeta_\ell'}, & \ell > 1. \end{cases}
\end{equation}
\end{theorem}

\begin{remark}
The Boolean case of this statement was given in \cite[Lemma 2.9]{PV2013} and the free case was given in \cite[Lemma 3.7]{PV2013}.
\end{remark}

\begin{proof}
Let $\Lambda_\ell[a_1,\dots,a_\ell]$ be the right hand side of \eqref{eq:Fockcumulantformula}.  This definition makes sense because $b_j$, $\zeta_j$, $\zeta_j'$, and $x_j$ are uniquely determined by $a_j$ by the following argument.  From our assumption that $n \geq 2$, we know $\mathcal{T}$ is not the trivial tree $\{\emptyset\}$, and hence $L^2(\Omega^1,\gamma_1)$ is nontrivial.  It follows that $b_j$, $\zeta_j$, $\zeta_j'$, and $x_j$ are uniquely determined by $T_j$.  Indeed, $b_j$ can be found from $\ip{\xi,T_j\xi}$ and $\zeta_j$ and $\zeta_j'$ can be found by evaluating $T_j \xi$ and $T_j^* \xi$, and $x_j$ is determined from the compression of $T_j$ by the projection onto $\mathcal{K} \otimes L^2(\Omega^1,\gamma_1)$.

Operators of the form $M'(b) + L(\zeta)^* + L(\zeta') + M(x)$ form a $\mathcal{B}$-$\mathcal{B}$-bimodule because we have
\begin{multline*}
[M'(b_1) + L(\zeta_1)^* + L(\zeta_1') + M(x_1)] + [M'(b_2) + L(\zeta_2)^* + L(\zeta_2') + M(x_2)] \\
= [M'(b_1+b_2) + L(\zeta_1+\zeta_2)^* + L(\zeta_1'+\zeta_2') + M(x_1 + x_2)]  
\end{multline*}
and
\begin{multline*}
b[M'(b_1) + L(\zeta_1)^* + L(\zeta_1') + M(x_1)]b' \\
= M'(bb_1b') + L((b')^*\zeta_1 b^*)^* + L(b \zeta_1 b') + M(bx_1b').
\end{multline*}
where $b$ and $b' \in \mathcal{B}$ are viewed on the left hand side of the equation as left multiplication operators on $\mathcal{F}_{\mathcal{T}}(\mathcal{K})$.  Moreover, the maps $\Lambda_\ell$ are $\mathcal{B}$-quasi-multilinear maps on this $\mathcal{B}$-$\mathcal{B}$-bimodule, and therefore, for $\pi \in \mathcal{NC}(\ell)$, the composition
\[
\Lambda_\pi[a_1,\dots,a_\ell]
\]
is well-defined by Definition \ref{def:picomposition}.  In order to prove the theorem, it suffices to show that
\[
\ip{\xi, a_1 \dots a_\ell \xi} = \sum_{\pi \in \mathcal{NC}(\ell)} \alpha_{\mathcal{T},\pi} \Lambda_\pi[a_1,\dots,a_\ell]
\]
because the cumulants are uniquely determined by the relation \eqref{eq:momentcumulantformula} (see Lemma \ref{lem:partitionMobiusinversion}).

In order to evaluate $\ip{\xi, a_1 \dots a_\ell \xi}$, we proceed along similar lines to the proof of Theorem \ref{thm:combinatorics}.  Denote
\begin{align*}
a_j^{(0,0)} &= M'(b_j) \\
a_j^{(0,1)} &= L(\zeta_j)^* \\
a_j^{(1,0)} &= L(\zeta_j') \\
a_j^{(1,1)} &= M(x_j).
\end{align*}
Observe that
\[
\ip{\xi, a_1 \dots a_\ell \xi} = \sum_{(\delta_j,\epsilon_j) \in \{0,1\}} \ip{\xi, a_1^{(\delta_1,\epsilon_1)} \dots a_\ell^{(\delta_\ell,\epsilon_\ell)} \xi}.
\]

Let $\mathcal{G}$ be the undirected multigraph with vertex set $\N_0 = \{0,1,2,\dots\}$ with an edge from $k$ to $k+1$, a self-loop $e_k^{(0)}$ at each vertex $k \geq 0$ and a distinct self-loop $e_k^{(1)}$ at each vertex $k \geq 1$.  We adopt the convention that the edges from $k$ to $k + 1$ have two possible orientations, while the self-loops have only one possible orientation.  We denote by $e_j^-$ the source vertex of $e_j$ and by $e_j^+$ the target vertex of $e_j$.

Define four sets of oriented edges
\begin{align*}
\mathcal{E}^{(0,0)} &= \{e_k^{(0)}: k \in \N_0\} \\
\mathcal{E}^{(0,1)} &= \{(k,k+1): k \in \N_0\} \\
\mathcal{E}^{(1,0)} &= \{(k+1,k): k \in \N_0\} \\
\mathcal{E}^{(1,1)} &= \{e_k^{(1)}: k \geq 1\}.
\end{align*}
A path in $\mathcal{G}$ will be given by a sequence of oriented edges $e_1$, \dots, $e_\ell$ where the source of $e_i$ is the target of $e_{i-1}$.  We say that a sequence $(\delta_1,\epsilon_1)$, \dots, $(\delta_\ell,\epsilon_\ell)$ and a path $e_1$, \dots, $e_\ell$ in $\mathcal{G}$ are \emph{compatible} if $e_j \in \mathcal{E}^{(\delta_j,\epsilon_j)}$ for each $j$.

Observe that if $(\delta_1,\epsilon_1)$, \dots, $(\delta_\ell,\epsilon_\ell)$ does not have a compatible path in $\mathcal{G}$, then 
\[
\ip{\xi, a_1^{(\delta_1,\epsilon_1)} \dots a_\ell^{(\delta_\ell,\epsilon_\ell)} \xi} = 0.
\]
On the other hand, if there is a compatible path, then for each $j$, the element $a_j^{(\delta_j,\epsilon_j)} \dots a_\ell^{(\delta_\ell,\epsilon_\ell)} \xi$ is in the $e_j^-$ indexed direct summand of $\mathcal{F}_{\mathcal{T}}(\mathcal{K})$.  Moreover, if there is a compatible path, then the choice of $(\delta_i,\epsilon_i)$ is uniquely determined by the path.  We call the paths that arise in this way \emph{admissible}.

Similar to the proof of Theorem \ref{thm:combinatorics}, there is bijective correspondence between admissible paths $e_1$, \dots, $e_\ell$ and partitions $\pi \in \mathcal{NC}(\ell)$ such that if $j \in V \in \pi$, then
\[
e_j \in \begin{cases}
\mathcal{E}^{(0,0)}, & V = \{j\} \\
\mathcal{E}^{(0,1)}, & |V| > 1, j = \min V \\ 
\mathcal{E}^{(1,0)}, & |V| > 1, j = \max V \\
\mathcal{E}^{(1,1)}, & \text{otherwise.}
\end{cases}
\]
and such that $\depth(V) = \max(e_j^-, e_j^+)$ if $|V| > 1$ and $\depth(V) = e_j^+ - 1 = e_j^- - 1$ if $V = \{j\}$.  Because we know that $a_j^{(\delta_j,\epsilon_j)} \dots a_\ell^{(\delta_\ell,\epsilon_\ell)} \xi$ is in the $e_j^-$ indexed direct summand of $\mathcal{F}_\mathcal{T}(\mathcal{K})$, we may replace the operator $a_j^{(\delta_j,\epsilon_j)}$ defined on the whole Fock space with its restriction to an operator
\[
\mathcal{K}^{\otimes_{\mathcal{B}} e_j^+} \otimes L^2(\Omega^{e_j^+},\gamma_{e_j^+}) \to \mathcal{K}^{\otimes_{\mathcal{B}} e_j^-} \otimes L^2(\Omega^{e_j^-},\gamma_{e_j^-}).
\]
This restriction is given by $Y_j \otimes T_j$, where $Y_j$ and $T_j$ are defined by the relation that if $j \in V \in \pi$, then
\[
Y_j = \begin{cases}
M_{\depth_\pi(V)-1}'(b_j), & V = \{j\} \\
L_{\depth_\pi(V)-1}(\zeta_j)^*, & |V| > 1, j = \min V \\ 
L_{\depth_\pi(V)-1}(\zeta_j'), & |V| > 1, j = \max V \\
M_{\depth_\pi(V)}(x_j), & \text{otherwise.}
\end{cases}
\]
and
\[
T_j = \begin{cases}
S_{\depth_\pi(V)-1}^* S_{\depth_\pi(V)-1}, & V = \{j\} \\
S_{\depth_\pi(V)-1}^*, & |V| > 1, j = \min V \\ 
S_{\depth_\pi(V)-1}, & |V| > 1, j = \max V \\
1, & \text{otherwise.}
\end{cases}
\]
Thus, we have
\begin{align*}
\ip{\xi, a_1^{(\delta_1,\epsilon_1)} \dots a_\ell^{(\delta_\ell,\epsilon_\ell)} \xi} &= \ip{\xi, (Y_1 \otimes T_1) \dots (Y_\ell \otimes T_\ell) \xi} \\
&= \ip{\xi, Y_1 \dots Y_\ell \xi} \cdot \ip{1, T_1 \dots T_\ell 1}.
\end{align*}
It follows from Lemma \ref{lem:Fockcoefficients} that $\ip{1, T_1 \dots T_\ell 1} = \alpha_{\mathcal{T},\pi}$.  Therefore, to complete the proof, it suffices to show that for a path and the corresponding partition $\pi$, we have
\[
\ip{\xi, Y_1 \dots Y_\ell \xi} = \Lambda_\pi[a_1,\dots,a_\ell].
\]

We verify this by induction for $|\pi| \geq 1$.  Let $V$ be a block of $\pi$ which is maximal with respect to $\prec$ and let $d = \depth(V)$.  Then $V$ can be written as $\{j+1,\dots,k\}$.  If $|V| = 1$, then $T_k = M_{d-1}'(b_k)$, while if $|V| > 1$, we have
\[
Y_{j+1} \dots Y_k = L_{d-1}(\zeta_{j+1})^* M_d(x_{j+2}) \dots M_d(x_{k-1}) L_{d-1}(\zeta_k') = M_{d-1}'(\ip{\zeta_{j+1}, x_{j+2} \dots x_{k-1} \zeta_k'}.
\]
In either case $Y_{j+1} \dots Y_k = M_{d-1}'(\Lambda_\pi[a_{j+1},\dots,a_k])$, and hence
\[
Y_1 \dots Y_\ell = Y_1 \dots Y_j \Lambda_\pi[a_{j+1},\dots,a_k] Y_{k+1} \dots Y_\ell.
\]
In the base case $|\pi| = 1$, we have $j = 0$ and $k = \ell$, so the proof is already complete.  Otherwise, we may group the scalar $\Lambda_\pi[a_{j+1},\dots,a_k] \in \mathcal{B}$ together with $Y_j$ or $Y_{k+1}$ and apply the inductive hypothesis for $\pi \setminus V$.
\end{proof}

\begin{proof}[Proof of Theorem \ref{thm:infinitelydivisible} (2)]
Let $\sigma: \mathcal{B}\ip{Y} \to \mathcal{B}$ be completely positive and exponentially bounded and let $c \in \mathcal{B}$ be self-adjoint.  Let $\mathcal{K} = \mathcal{B}\ip{Y} \otimes_\sigma \mathcal{B}$, and let $Y$ denote the operator of multiplication by $Y$ on $\mathcal{K}$.  Define
\[
X = M'(c) + L(1 \otimes 1)^* + L(1 \otimes 1) + M(Y) \in \mathcal{L}(\mathcal{F}_{\mathcal{T}}(\mathcal{K})).
\]
Let $\mu_{c,\sigma}$ be the law of $X$ with respect to $\xi$.  Then it follows from the previous theorem that the $\mathcal{T}$-free cumulants of $\mu_{c,\sigma}$ are given by $c$ and $\sigma$ as in \eqref{eq:IDbijection}.  Moreover, we have
\begin{align*}
\rad(\mu_{c,\sigma}) \leq \norm{X} &\leq \norm{M'(c)} + 2 \norm{L(1 \otimes 1)} + \norm{Y} \\
&\leq \frac{N-1}{n-1} \norm{c} + 2 \sqrt{\frac{N-1}{n-1}} \norm{\sigma(1)}^{1/2} + \rad(\sigma).
\end{align*}
It follows from Theorem \ref{thm:extensivity} that
\[
\mu_{c,\sigma} = \boxplus_{\mathcal{G}}^{n^k}(\mu_{n^{-k}c, n^{-k} \sigma}),
\]
and it follows from our previous estimate that
\[
\rad(\mu_{n^{-k}c,n^{-k}\sigma}) \leq n^{-k} \norm{c} + 2 \sqrt{\frac{N-1}{n-1}} n^{-k/2} \norm{\sigma(1)}^{1/2} + \rad(\sigma).
\]
Therefore, $\mu_{c,\sigma}$ is infinitely divisible with bounded support.
\end{proof}

\subsection{The Free, Boolean, and Monotone Cases} \label{subsec:infdivexamples}

We now explain how the constructions in this section work themselves out in the free, Boolean, and monotone cases.

\begin{example}[Free case]
Consider the tree $\mathcal{T}_{N,\free} \in \Tree(N)$.  Because $\mathcal{X}_w(\pi,\mathcal{T}_{N,\free})$ is all of $[N]^\pi$, we see that $\mathcal{X}_w^{(\infty)}(\pi,\mathcal{T}_{N,\free})$ is all of $\Omega^\pi$, and $\theta_\pi$ is the uniform distribution $u^\pi$.  It follows that $\gamma_{\omega_1,\dots,\omega_\ell}$ is the uniform distribution on $\Omega$, and the measure $\gamma_\ell = u^{\times \ell}$.  The operator $S_\ell$ is given by
\[
S_\ell f(\omega_1,\dots,\omega_{\ell+1}) = f(\omega_2,\dots,\omega_\ell).
\]
This satisfies $S_\ell^* S_\ell = 1$.

The Fock space given by our construction is
\[
\mathcal{F}_{N,\free}(\mathcal{K}) = \bigoplus_{\ell=0}^\infty \mathcal{K}^{\otimes_{\mathcal{B}} \ell} \otimes L^2(\Omega^\ell,u^{\times \ell}).
\]
In this case, $S_\ell$ and $S_\ell^*$ map constant functions to constant functions.  Therefore, as far as the joint law of the creation, annihilation, and multiplication operators is concerned, we might as well replace $L^2(\Omega^\ell, u^\ell)$ by the subspace of constant functions.  This amounts to replacing $L^2(\Omega^\ell, u^{\times \ell})$ by $\C$ and replacing $S_\ell$ by $1$.  These replacements will produce the space
\[
\bigoplus_{\ell=0}^\infty \mathcal{K}^{\otimes_{\mathcal{B}} \ell}
\]
which is the free Fock space defined in previous work \cite[\S 4.7]{Speicher1998}.
\end{example}

\begin{example}[The Boolean Case]
In the case of $\mathcal{T}_{N,\Bool}$, the measure $\theta_\pi$ is the uniform distribution if $\pi$ is an interval partition and zero otherwise.  The measure $\gamma_{\omega_1,\dots,\omega_\ell}$ is the uniform distribution if $\ell = 0$, and otherwise it is zero.  The measure $\gamma_1 = u$ and $\gamma_\ell = 0$ for $\ell > 1$.  In the Fock space, one may replace $L^2(\Omega^\ell,\gamma_\ell)$ by $\C$ if $\ell = 1$ and by zero if $\ell > 1$ and replace $S_\ell$ by $1$ for $\ell = 1$ and zero for $\ell > 1$.  This replacements will produce the space
\[
\mathcal{B}\xi \oplus \mathcal{K},
\]
which is the Boolean Fock space considered in previous work and in Lemma \ref{lem:BooleanID}.
\end{example}

\begin{example}[The Monotone Case]
Consider $\mathcal{T}_{N,\mono}$.  In light of Remark \ref{rem:measureswithN=n}, $\theta_\pi$ is given by restricting the uniform distribution on $\Omega^\pi$ to the set $\mathcal{X}_w^{(\infty)}(\pi,\mathcal{T}_{N,\mono})$.  Now $\chi \in \mathcal{X}_w(\pi,\mathcal{T}_{N,\mono})$ if and only if $V \prec W$ in $\pi$ implies that $\chi(V) \leq \chi(W)$.  From this we can see that $\vec{\chi} = (\chi_1,\chi_2,\dots)$ is in $\mathcal{X}_w^{(\infty)}(\pi,\mathcal{T}_{N,\mono})$ if and only if $V \prec W$ in $\pi$ implies that $\vec{\chi}(V) \leq \vec{\chi}(W)$ in the lexicographical order on $[N]^{\N}$.

Rephrasing this in terms of points $(\omega_V)_{V \in \pi} \in \Omega^\pi$, this means that $V \prec W$ in $\pi$ implies that $\omega_V \leq \omega_W$ in the lexicographical order on $\Omega$.  Now there is an isomorphism of measure spaces $\Omega \to [0,1]$ given by
\[
(j_1,j_2,\dots) \mapsto \sum_{i=1}^\infty \frac{1}{N^i} (j_i - 1),
\]
where the measurable inverse map is given by taking the $N$-ary expansion of numbers in $[0,1]$ and adding one to each digit.  This isomorphism carries the lexicographical order on $\Omega$ to the standard order on $[0,1]$ (up to null sets).  Hence, it maps $\mathcal{X}_w^{(\infty)}(\pi,\mathcal{T}_{N,\mono})$ onto the set
\[
\{t \in [0,1]^\pi: V \prec W \implies t_V \leq t_W\},
\]
which is equal (up to null sets) to the set $\Upsilon_\pi$ from the proof of Proposition \ref{prop:monotonecumulants}.  Thus, we obtain an alternative proof that $\alpha_{N,\mono} = |\Upsilon_\pi|$.

If we choose $\omega_1, \dots, \omega_\ell \in \Omega$ and let $t_1$, \dots, $t_\ell$ be the corresponding points in $[0,1]$. The measures $\gamma_\ell$ correspond to the Lebesgue measure restricted to the set
\[
\Upsilon_\ell = \{(t_1,\dots,t_\ell) \in [0,1]^\ell: t_1 \geq t_2 \geq \dots \geq t_\ell\}.
\]
Moreover, if $\omega_1 \leq \dots \leq \omega_\ell$, then $\gamma_{\omega_1,\dots,\omega_\ell}$ corresponds to the Lebesgue measure restricted to $[t_1,1]$.  Under this change of coordinates, the Fock space given by our construction becomes
\[
\bigoplus_{\ell=0}^\infty \mathcal{K}^{\otimes_{\mathcal{B}} \ell} \otimes L^2(\Upsilon_\ell,\text{Leb}).
\]
The scalar-valued creation operator $S_\ell$ satisfies
\begin{align*}
S_\ell f(t_1,\dots,t_{\ell+1}) &= f(t_2,\dots,t_\ell) \\
S_\ell^* f(t_1,\dots,t_\ell) &= \int_{t_1}^1 f(t,t_1,\dots,t_\ell)\,dt.
\end{align*}
Thus, our construction reduces to the constructions in previous literature; see \cite{Lu1997}, \cite{Muraki1997}, and \cite[\S 6.1 - 6.4]{Jekel2020}.
\end{example}

\begin{example}[Digraphs]
Suppose that $\mathcal{T} = \Walk(G)$, where $G$ is a digraph on the vertex set $[N]$.  Let $\sim_G$ denote the directed adjacency relation of $G$.  Then we claim that the condition $jt_i \in \mathcal{T}$ in the definition of $\gamma_{\omega_1,\dots,\omega_\ell}$ reduces to $(\omega_1)_i \sim_G j$.   More precisely, for $(\omega_1,\dots,\omega_\ell)$ in the support of $\gamma_\ell$, we have
\begin{equation} \label{eq:digraphconditionaldistribution}
\gamma_{\omega_1,\dots,\omega_\ell} = \sum_{i=0}^\infty \frac{1}{N} \delta_{(\omega_1)_1} \times \dots \times \frac{1}{N} \delta_{(\omega_1)_i} \times \left( \frac{1}{N} \sum_{j: (\omega_1)_i \sim j} \delta_j \right) \times v^{\N}.
\end{equation}
To verify the claim, first observe that if $(\omega_1,\dots,\omega_\ell)$ is in the support of $\gamma_\ell$, then the string $t_i$ must be in $\mathcal{T}$ (that is, the reverse of $t_i$ defines a path in $G$); this can be verified by induction on $\ell$.  It follows that $jt_i \in \mathcal{T}$ if and only if there is an edge in $G$ from the first letter of $t_i$ to $j$.  And the first letter of $t_i$ is $(\omega_1)_i$, which proves our claim.

So we have shown that the condition $j t_i \in \mathcal{T}$ in the definition of $\gamma_{\omega_1,\dots,\omega_\ell}$, which depends on $(\omega_1,\dots,\omega_\ell)$, can be replaced by the condition $(\omega_1)_i \sim_G j$, which only depends on $\omega_1$.  Thus, the measures $\gamma_\ell$ satisfy a ``Markov property'' in that the ``conditional distribution'' of the newest coordinate (which is the leftmost coordinate) only depends on the value of the second coordinate, and indeed these measures are obtained in a similar way to a random walk on the digraph $G$.
\end{example}

\subsection{Bercovici-Pata Bijections} \label{subsec:BPbijection}

As explained in the introduction, Bercovici and Pata \cite{BP1999} studied the bijection between infinitely divisible laws in the classical, free, and Boolean settings, and their work has since been extended to the monotone case \cite{AW2014}.  These bijections were adapted to operator-valued free, Boolean, and monotone independence in the case of (exponentially bounded) $\mathcal{B}$-valued laws \cite{PV2013} \cite{AW2016}.  These latter bijections adapt directly to $\mathcal{T}$-free convolutions as follows.

Theorem \ref{thm:infinitelydivisible} defines a bijection between laws $\mu$ that are $\mathcal{T}$-freely infinitely divisible with bounded support and pairs $(c,\sigma)$ where $c$ is a self-adjoint element of $\mathcal{B}$ and $\sigma: \mathcal{B}\ip{Y} \to \mathcal{B}$ is completely positive and exponentially bounded.  We denote by $\mathbb{BP}_{\mathcal{T}}$ the map sending $(c,\sigma)$ to the corresponding infinitely divisible law $\mu$.

As a corollary, given $\mathcal{T} \in \Tree(N)$ and $\mathcal{T}' \in \Tree(N')$ satisfying $n := |[N] \cap \mathcal{T}| \geq 2$ and $n' := |[N'] \cap \mathcal{T}'| \geq 2$, there is a bijection between the laws that are infinitely divisible with bounded support for $\mathcal{T}$ and those for $\mathcal{T}'$ given by $\mathbb{BP}_{\mathcal{T}',\mathcal{T}} = \mathbb{BP}_{\mathcal{T}'} \circ \mathbb{BP}_{\mathcal{T}}^{-1}$.  We call this map $\mathbb{BP}_{\mathcal{T}',\mathcal{T}}$ a \emph{Bercovici-Pata bijection}.  Note that the $\mathbb{T}'$-free cumulants of $\mathbb{BP}_{\mathcal{T}',\mathcal{T}}(\mu)$ are equal to the $\mathcal{T}$-free cumulants of $\mu$.

Theorem \ref{thm:infinitelydivisible} also allows us to define operator-valued convolution powers of a law that is infinitely divisible with bounded support.

\begin{definition} \label{def:convolutionpowers}
Let $\eta: \mathcal{B} \to \mathcal{B}$ be a completely positive map.  If $\mu$, $\nu \in \Sigma(\mathcal{B})$, we say that $\nu = \boxplus_{\mathcal{T}}^\eta(\mu)$ if we have $\kappa_{\mathcal{T},\ell}(\nu) = \eta \circ \kappa_{\mathcal{T},\ell}(\mu)$.  In particular, if $t \in [0,+\infty)$, we say that $\nu = \boxplus_{\mathcal{T}}^t(\mu)$ if $\kappa_{\mathcal{T},\ell}(\nu) = t \kappa_{\mathcal{T},\ell}(\mu)$.  The free, Boolean, and monotone $\eta$-convolution powers of $\mu$ are denoted $\mu^{\boxplus \eta}$, $\mu^{\uplus \eta}$, and $\mu^{\rhd \eta}$ respectively.
\end{definition}

\begin{remark}
By Observation \ref{obs:lawcumulantextensivity}, we have $\kappa_{\mathcal{T},\ell}(\boxplus_{\mathcal{T}}(\mu,\dots,\mu)) = n \kappa_{\mathcal{T},\ell}(\mu)$ and therefore the two definitions of $\boxplus_{\mathcal{T}}^{n^k}(\mu)$ given by Definitions \ref{def:integerconvolutionpowers} and \ref{def:convolutionpowers} agree.
\end{remark}

\begin{observation}
It follows from Theorem \ref{thm:infinitelydivisible} that
\[
\mathbb{BP}_{\mathcal{T}}(\eta(c), \eta \circ \sigma) = \boxplus_{\mathcal{T}}^\eta(\mathbb{BP}(c,\sigma)).
\]
In particular, if $\mu$ is infinitely divisible with bounded support, then $\boxplus_{\mathcal{T}}^\eta$ is defined for every completely positive $\eta: \mathcal{B} \to \mathcal{B}$.  Moreover, the Bercovici-Pata bijections $\mathbb{BP}_{\mathcal{T}',\mathcal{T}}$ respect operator-valued convolution powers.
\end{observation}

\begin{remark}
It follows from Proposition \ref{prop:cumulantproperties} that $\mathbb{BP}_{\mathcal{T}}$ only depends on the isomorphism class of $\mathcal{T}$ as a rooted tree, and hence $\mathbb{BP}_{\mathcal{T}',\mathcal{T}} = \id$ if $\mathcal{T}'$ and $\mathcal{T}$ are isomorphic as rooted trees.  Similarly, by Proposition \ref{prop:cumulantcomposition}, we have $\mathbb{BP}_{\mathcal{T}^{n^k},\mathcal{T}} = \id$ for every $k$.
\end{remark}

It is a remarkable fact that the Boolean-to-free Bercovici-Pata bijection is given by $\mathbb{BP}_{\free,\Bool}(\mu) = (\mu^{\boxplus 2})^{\uplus 1/2}$; this is a special case of the results in \cite[\S 1.2]{BN2008a}, \cite{BN2009}, and \cite{ABFN2013}.  Actually, we can find a similar expression for the Bercovici-Pata bijection from Boolean independence to $\mathcal{T}$-independence when $\mathcal{T}$ is an $(n,d)$-regular tree as in \S \ref{subsec:treecoefficients}.  (Recall we defined the non-standard terminology ``$(n,d)$-regular tree'' in \S \ref{subsec:treecoefficients}, and it means that the root has $n$ neighbors, and all the other vertices have $d$ children each.)  The proposition below in particular gives a combinatorial proof that $\mathbb{BP}_{\free,\Bool}(\mu) = (\mu^{\boxplus 2})^{\uplus 1/2}$.

\begin{proposition} \label{prop:regulartreebijection}
Let $\mathcal{T} \in \Tree(N)$ be an $(n,d)$-regular tree with $n \geq 2$.  Then we have
\[
\mathbb{BP}_{\mathcal{T},\Bool}(\mu) = \boxplus_{\mathcal{T}}(\mu^{\uplus \frac{1}{n-1}}, \dots, \mu^{\uplus \frac{1}{n-1}})^{\uplus \frac{n-1}{n}}.
\]
\end{proposition}

\begin{proof}
Fix $\mu$ and let $\nu$ be the law on the right-hand side.  We will prove that $\nu = \mathbb{BP}_{\mathcal{T},\Bool}(\mu)$ by showing that their Boolean cumulants are equal using Corollary \ref{cor:combinatorics2} and Lemma \ref{lem:TBcumulantconversion}.

Let $X$ be an operator on $(\mathcal{H},\xi)$ which realizes the law $\mu^{\uplus \frac{1}{n-1}}$.  Let $(\mathcal{K},\zeta) = \assemb_{\mathcal{T}}[(\mathcal{H},\xi),\dots,(\mathcal{H},\xi)]$, let $\lambda_1, \dots, \lambda_N: \mathcal{L}(\mathcal{H}) \to \mathcal{L}(\mathcal{K})$ be the two inclusions, and let $Y = \sum_{j=1}^N \lambda_j(X)$, so that $Y \sim \boxplus_{\mathcal{T}}(\mu^{\uplus \frac{1}{n-1}}, \dots, \mu^{\uplus \frac{1}{n-1}})$.  By Corollary \ref{cor:combinatorics2}, we have
\begin{align*}
K_{\Bool,\ell} & [Yb_1,\dots,Yb_{\ell-1},Y] \\
&= \sum_{\chi \in [N]^{[\ell]}} K_{\Bool,\ell}[\lambda_{\chi(1)}(X)b_1, \dots, \lambda_{\chi(\ell-1)}(X) b_{\ell-1}, \lambda_{\chi(\ell)}(X)] \\
&= \sum_{\chi \in [N]^{[\ell]}} \sum_{\pi \in \mathcal{NC}^\circ(\chi,\mathcal{T})} \kappa_{\Bool,\pi}(\mu^{\uplus \frac{1}{n-1}})[b_1, \dots, b_{\ell-1}].
\end{align*}
Here on the right hand side, we mean the $\pi$-composition of the Boolean cumulants of $\mu^{\uplus \frac{1}{n-1}}$ for arbitrary partitions, not only interval partitions.  This expression $\kappa_{\Bool,\pi}(\mu^{\uplus \frac{1}{n-1}})$ is what we get from evaluating the term
\[
\Lambda_{\chi,\pi}(Xb_1,\dots,Xb_{\ell-1},X)
\]
from Corollary \ref{cor:combinatorics2}.

Recalling that $\nu$ is the $(n-1)/n$ Boolean convolution power of the law of $Y$, we get
\begin{align*}
\kappa_{\Bool,\ell}(\nu)[b_1,\dots,b_{\ell-1}] &= \frac{n-1}{n} \sum_{\chi \in [N]^{[\ell]}} \sum_{\pi \in \mathcal{NC}^\circ(\chi,\mathcal{T})} \kappa_{\Bool,\pi}(\mu^{\uplus \frac{1}{n-1}})[b_1, \dots, b_{\ell-1}] \\
&= \frac{n-1}{n} \sum_{\pi \in \mathcal{NC}^\circ(\ell)} \sum_{\chi \in \mathcal{X}(\pi,\mathcal{T})} \frac{1}{(n-1)^{|\pi|}} \kappa_{\Bool,\pi}(\mu)[b_1, \dots, b_{\ell-1}],
\end{align*}
where we have exchanged the order of summation and substituted that $\kappa_{\Bool,\ell}(\mu^{\uplus \frac{1}{n-1}}) = \frac{1}{n-1} \kappa_{\Bool,\ell}(\mu)$.

Next, recall that the colorings in $\mathcal{X}(\pi,\mathcal{T})$ are equivalent to rooted tree homomorphisms from $\graph(\pi)$ to $\mathcal{T}$.  As we saw in \S \ref{subsec:treecoefficients}, the number of such homomorphisms is $n$ to the number of outer blocks of $\pi$ times $d$ to the number of inner blocks of $\pi$.  Since $\pi$ is irreducible, it has only one outer block.  So we get
\begin{align*}
\kappa_{\Bool,\ell}(\nu)[b_1,\dots,b_{\ell-1}] &= \frac{n-1}{n} \sum_{\pi \in \mathcal{NC}^\circ(\ell)} \frac{n d^{|\pi| - 1}}{(n-1)^{|\pi|}} \kappa_{\Bool,\pi}(\mu)[b_1, \dots, b_{\ell-1}] \\
&= \sum_{\pi \in \mathcal{NC}^\circ(\ell)} \left( \frac{d}{n-1}\right)^{|\pi|-1} \kappa_{\Bool,\pi}(\mu)[b_1, \dots, b_{\ell-1}].
\end{align*}
Recall that $(d / (n-1))^{|\pi| - 1} = \alpha_{\mathcal{T},\pi}$ by Proposition \ref{prop:regulartree}.  Meanwhile, the Bercovici-Pata bijection $\mathbb{BP}_{\mathcal{T},\Bool}(\mu)$ is defined so that the $\mathcal{T}$-free cumulants equal to the Boolean cumulants of $\mu$, which means that
\begin{align*}
\kappa_{\Bool,\ell}(\nu)[b_1,\dots,b_{\ell-1}] &= \sum_{\pi \in \mathcal{NC}^\circ(\ell)} \alpha_{\mathcal{T},\pi} \kappa_{\mathcal{T},\pi}(\mathbb{BP}_{\mathcal{T},\Bool}(\mu))[b_1, \dots, b_{\ell-1}] \\
&= \kappa_{\Bool,\ell}(\mathbb{BP}_{\mathcal{T},\Bool}(\mu)),
\end{align*}
where the last line is Lemma \ref{lem:TBcumulantconversion}.  Thus, $\nu$ and $\mathbb{BP}_{\mathcal{T},\Bool}(\mu)$ have the same Boolean cumulants, so they are equal.
\end{proof}

Actually, for an $(n,d)$-regular tree, we can express the $\mathcal{T}$-convolution powers and the Bercovici-Pata bijection purely in terms of free and Boolean convolution powers.  First, by substituting $\mu^{\uplus \frac{1}{n-1}}$ for $\mu$ in Example \ref{ex:regulartree}, we have for $d > 0$ that
\begin{equation} \label{eq:Tfreeconvolutionconversion}
\boxplus_{\mathcal{T}}(\mu,\dots,\mu) = ((\mu^{\uplus \frac{d}{n-1}})^{\boxplus n})^{\uplus \frac{n-1}{d}}.
\end{equation}

Next, we claim that $\mathbb{BP}_{\mathcal{T},\Bool}$ can be expressed in terms of the \emph{Belinschi-Nica semigroup}.  This semigroup was defined in \cite{BN2008b} by the formula
\[
\mathbb{BN}_t(\mu) = (\mu^{\boxplus (1 + t)})^{\uplus \frac{1}{1+t}}.
\]
The definition relies on the fact from free probability that the $(1 + t)$-free convolution power of a non-commutative law is defined for any $t \geq 0$, or in other words, there exists a law $\mu^{\boxplus (1 + t)}$ whose free cumulants are $1 + t$ times the free cumulants of $\mu$.  This can be proved using either the $R$-transform or certain operator models; see \cite[Cor.\ 1.14]{NS1996}, \cite[Thm.\ 1]{Shlyakhtenko1997}, \cite[Thm.\ 8.4]{ABFN2013}, \cite[Thm.\ 2.3]{Shlyakhtenko2013}.

Furthermore, \cite{BN2008b,ABFN2013,Liu2018} showed that $(\mathbb{BN}_t)_{t \geq 0}$ forms a semigroup, that is,
\[
\mathbb{BN}_s \circ \mathbb{BN}_t = \mathbb{BN}_{s+t}.
\]
This can be verified using the identity
\begin{equation} \label{eq:freebooleanconversion}
(\mu^{\boxplus p})^{\uplus q} = (\mu^{\uplus q'})^{\boxplus p'} \text{ whenever } pq = p'q' \text{ and } p - 1 = (p' - 1)q'
\end{equation}
for real $p, p' \geq 1$ and $q, q' \geq 0$, which was proved in \cite{BN2008b,Liu2018}, as we alluded to in Example \ref{ex:freebooleaniteration}.  If we take the formula from Proposition \ref{prop:regulartreebijection} and substitute \eqref{eq:Tfreeconvolutionconversion} and then use \eqref{eq:freebooleanconversion}, we obtain the following corollary.

\begin{corollary}
If $\mathcal{T}$ is an $(n,d)$-regular tree, then
\[
\mathbb{BP}_{\mathcal{T},\Bool}(\mu) = \mathbb{BN}_{d/(n-1)}(\mu).
\]
\end{corollary}

Thus, not only does $\mathbb{BN}_1$ give the Boolean-to-free Bercovici-Pata bijection, but actually $\mathbb{BN}_t$ at any rational time $t$ gives the Boolean-to-$\mathcal{T}$-free Bercovici-Pata bijection for some tree $\mathcal{T}$.  More precisely, if we write the rational number $t$ as $d / (n-1)$, then we may construct an $(n,d)$ regular tree $\mathcal{T}$ in $\Tree(N)$, where $N = \max(n,d+1)$, and then $\mathbb{BN}_{d/(n-1)}$ is the Bercovici-Pata bijection for this tree.

This also gives an explicit formula for the central limit distribution for this tree for the scalar-valued setting $\mathcal{B} = \C$.  Indeed, the central limit law is $\mathbb{BN}_{d/(n-1)}$ applied to the Bernoulli distribution.  So we first compute the $(1+d/(n-1))$-free convolution power of the Bernoulli distribution using the $R$-transform, and then take the $1 / (1 + d/(n-1))$-Boolean convolution power.  We leave the details of the computation as an exercise for those familiar with the analytic transforms.

\begin{corollary}
Consider the case $\mathcal{B} = \C$.  The central limit distribution for an $(n,d)$-regular tree $\mathcal{T}$ is the probability measure $\mu$ on $\R$ given as follows.  Let $t = d / (n - 1)$.  If $t < 1/2$, then
\[
d\mu(x) = \frac{\sqrt{4t - x^2}}{2 \pi [(t - 1)x^2 + 1]} \chi_{(-2\sqrt{t}, 2\sqrt{t})}(x)\,dx + \frac{1-2t}{2(1-t)}\left(d\delta_{-\frac{1}{\sqrt{1-t}}}(x) + d\delta_{\frac{1}{\sqrt{1-t}}}(x) \right), 
\]
and if $t \geq 1/2$, then
\[
d\mu(x) = \frac{\sqrt{4t - x^2}}{2 \pi [(t - 1)x^2 + 1]} \chi_{(-2\sqrt{t}, 2\sqrt{t})}(x)\,dx.
\]
\end{corollary}

\begin{remark}
These probability distributions also occur as the first law in the pair $(\mu,\nu)$ giving the c-free central limit distribution with $\Var(\mu) = 1$ and $\Var(\nu) = t$.  See \cite[Thm.\ 4.3]{BLS1996}.  The relationship between the c-free central limit laws and free/Boolean convolution powers can be seen by studying the analytic transforms \cite[\S 5]{BLS1996} and \cite[\S 4]{BPV2012}.
\end{remark}

\section{Concluding Remarks} \label{sec:conclusion}

We mention a few results and directions for future research which we did not fully develop in this paper.

\subsection{Laws with Finite Moments}

Many of the results of this paper would generalize to laws with finite moments, that is, unital, completely positive $\mathcal{B}$-$\mathcal{B}$-bimodule maps $\mathcal{B}\ip{X} \to \mathcal{B}$ that are not necessarily exponentially bounded, provided that we assume each moment $\mu(b_0 X b_1 \dots X b_\ell)$ is a bounded multilinear map $\mathcal{B}^{\ell+1} \to \mathcal{B}$.  The construction of product spaces and Fock spaces would proceed in the same way except that we would use $\mathcal{B}$-valued semi-inner product modules rather than Hilbert $\mathcal{B}$-modules and use the algebraic direct sums and tensor products without taking separation-completions.  The convolutions would be well-defined and satisfy the same moment formulas and operad properties, and the central limit theorem and Bercovici-Pata bijections would work in the same way.

Furthermore, one could drop the assumption that $\mu$ is completely positive and only work with unital $\mathcal{B}$-$\mathcal{B}$-bimodule maps $\mathcal{B}\ip{X} \to \mathcal{B}$ such that each moment is bounded.  In this case, one would also drop the positivity condition from the semi-inner products.  In this setting, essentially purely algebraic, the central limit theorem would still hold.  However, the characterization of infinitely divisible laws would be trivial since every law would be infinitely divisible in the algebraic sense.  Hence, the Bercovici-Pata bijections would be globally defined.  Also, arbitrary convolution powers by a bounded linear map $\eta: \mathcal{B} \to \mathcal{B}$ would be defined.

We assumed complete positivity and boundedness throughout because we believe that the results about positivity and operator-norm estimates have inherent interest.  Besides, including several variants of every result would have added more length than content, but the reader can easily adapt our proofs to the more algebraic spaces of laws for all the results that still apply.

\subsection{Analytic Viewpoint and Sharp Estimates}

As mentioned in the introduction, we did not fully develop the complex-analytic viewpoint on $\mathcal{T}$-free independences.  Moreover, at least in the scalar-valued setting, the complex-analytic viewpoint should allow the study of $\mathcal{T}$-free convolution of arbitrary probability measures, discovery of the optimal rate of convergence in the central limit theorem (see Remark \ref{rem:sharpCLT}), and the classification of infinitely divisible and stable distributions with unbounded support (as in \cite{BP1999}, \cite{Belinschi2006}, \cite{AW2014}, \cite{HSW2018}).  We would also like to know under what conditions the estimate
\[
\max_j (\rad(\mu_j)) \leq \rad(\boxplus_{\mathcal{T}}(\mu_1,\dots,\mu_N))
\]
or the estimate
\[
\rad(\mu) \leq \rad(\boxplus_{\mathcal{T}}(\mu,\dots,\mu))
\]
holds (see Remark \ref{rem:IDradiusbound}).

\subsection{Multiplicative Convolutions}

Parallel to the theory of additive convolutions, there is a theory of multiplicative convolutions based on multiplying rather than adding independent variables; see e.g.\ \cite{Voiculescu1995,Bercovici2005b,Franz2006,Franz2008,BW2008,Franz2009,AA2017}.  Of course, the product of self-adjoint operators is not necessarily self-adjoint, but there are several natural settings for multiplicative convolution | for instance, multiplication of unitaries and symmetrized multiplication of positive operators.


In general, for a (non-self-adjoint) operator $a$ in a $\mathcal{B}$-valued non-commutative probability space $(\mathcal{A},E)$, the \emph{$*$-distribution of $a$} is the map $\mathcal{B}\ip{Z,Z^*} \to \mathcal{B}$ given by $p(Z,Z^*) \mapsto E[p(a,a^*)]$.  Here $\mathcal{B}\ip{Z,Z^*}$ is the $*$-algebra of non-commutative polynomials in $Z$ and $Z^*$ with the obvious involution that maps $Z$ to $Z^*$ and vice versa.  We let $\Upsilon(\mathcal{B})$ denote the set of $*$-distributions that can be realized by a unitary operator $a$ in $(\mathcal{A},E)$.

The free convolution of two laws in $\Upsilon(\mathcal{B})$ is easy to define.  Indeed, suppose that for $j = 1$, $2$, we have a unitary operator $U_j$ on a $\mathcal{B}$-$\mathcal{B}$-correspondence with $\mathcal{B}$-central unit vector $(\mathcal{H}_j,\xi_j)$ which realizes the law $\mu_j$.  Let $(\mathcal{H},\xi)$ be the free product of these two correspondences. Then $\lambda_{\free,j}(U_j)$ is unitary because $\lambda_{\free,j}$ is a \emph{unital} $*$-homomorphism.  So we can define the free multiplicative convolution $\mu_1 \boxtimes \mu_2$ to be the law of $\lambda_{\free,1}(U_1) \lambda_{\free,2}(U_2)$.

However, for general trees $\mathcal{T}$ (and even in the Boolean and monotone cases) the map $\lambda_{\mathcal{T},j}$ might not be unital, and hence will not send unitaries to unitaries.  One solution to this problem (as in \cite{Bercovici2005b,Franz2006,Franz2008}) is to consider $\lambda_{\mathcal{T},j}(U_j - 1) + 1$ rather than $\lambda_{\mathcal{T},j}(U_j)$.  This element can be equivalently expressed as $\lambda_{\mathcal{T},j}(U_j) + (1 - \lambda_{\mathcal{T},j}(1))$, which is unitary because $\lambda_{\mathcal{T},j}(U_j)$ is a partial isometry with left and right support given by the projection $\lambda_{\mathcal{T},j}(1)$.

\begin{definition}
Let $\mathcal{T} \in \Tree(N)$ and $\mu_1$, \dots, $\mu_N \in \Upsilon(\mathcal{B})$.  Let $U_j$ be a unitary operator on $(\mathcal{H}_j, \xi_j)$ realizing the $*$-distribution $\mu_j$.  Then we define the \emph{$\mathcal{T}$-free unitary multiplicative convolution} $\boxtimes_{\mathcal{T}}(\mu_1,\dots,\mu_N)$ to be the $*$-distribution of
\[
[\lambda_{\mathcal{T},1}(U_1 - 1) + 1] \dots [\lambda_{\mathcal{T},N}(U_N - 1) + 1].
\]
\end{definition}

The analogue of Corollary \ref{cor:operadmorphism} does hold for unitary multiplicative convolution.

\begin{proposition} \label{prop:multiplicativeoperadmorphism}
$\mathcal{T} \mapsto \boxtimes_{\mathcal{T}}$ is an operad morphism.  In other words, given tree $\mathcal{T} \in \Tree(k)$ and $\mathcal{T}_j \in \Tree(n_j)$ for $j = 1, \dots, k$, we have
\[
\boxtimes_{\mathcal{T}}(\boxtimes_{\mathcal{T}_1}, \dots, \boxtimes_{\mathcal{T}_k}) = \boxtimes_{\mathcal{T}(\mathcal{T}_1,\dots,\mathcal{T}_k)}.
\]
\end{proposition}

To prove this, we use the following combinatorial observation:
\begin{equation} \label{eq:weirdproductexpansion}
[\lambda_{\mathcal{T},1}(U_1 - 1) + 1] \dots [\lambda_{\mathcal{T},N}(U_N - 1) + 1] - 1 = \sum_{\substack{S \subseteq [N] \\ S \neq \varnothing}} \prod_{j \in S} \lambda_{\mathcal{T},j}(U_j - 1).
\end{equation}
Here the product $\prod_{j \in S}$ has to be interpreted carefully because multiplication is not necessarily commutative.  Our convention is that the terms in the product will be multiplied in order from left to right according to the standard order on the natural numbers, so that for instance if $S = \{2,3,5\}$, then we write $\lambda_{\mathcal{T},2}(U_2 - 1) \lambda_{\mathcal{T},3}(U_3 - 1) \lambda_{\mathcal{T},5}(U_5 - 1)$.

\begin{proof}[Proof of Proposition \ref{prop:multiplicativeoperadmorphism}]
Fix laws $\mu_{j,i} \in \Upsilon(\mathcal{B})$ for $j = 1$, \dots, $k$ and $i = 1$, \dots, $n_j$.  Let $U_{j,i}$ be an operator $(\mathcal{H}_{j,i}, \xi_{j,i})$ which realizes the law $\mu_{j,i}$.  Use all the same notation as in Theorem \ref{thm:composition}, and in particular, $(\mathcal{H},\xi)$ will be the product of the $(\mathcal{H}_{j,i},\xi_{j,i})$'s according to $\assemb_{\mathcal{T}}(\assemb_{\mathcal{T}_1},\dots,\assemb_{\mathcal{T}_k})$, and $(\mathcal{K},\zeta)$ will be the product according to $\assemb_{\mathcal{T}(\mathcal{T}_1,\dots,\mathcal{T}_k)}$.  We also denote $\mathcal{T}' = \mathcal{T}(\mathcal{T}_1,\dots,\mathcal{T}_k)$ and $N = n_1 + \dots + n_k$.

If we first convolve $(\mu_{j,i})_i$ according to $\mathcal{T}_j$ and then convolve these laws according to $\mathcal{T}$, we get the law of the operator
\[
\prod_{j=1}^k \left[ \lambda_{\mathcal{T},j} \left( \prod_{i=1}^{n_j} [\lambda_{\mathcal{T}_j,i}(U_{j,i}-1) + 1] - 1 \right) + 1 \right]
\]
on $(\mathcal{H}, \xi)$.  Applying \eqref{eq:weirdproductexpansion} both to the inner products and the outer product, we obtain
\[
1 + \sum_{\substack{S \subseteq [k] \\ S \neq \varnothing}} \prod_{j \in S} \lambda_{\mathcal{T},j} \left( \sum_{\substack{S_j \subseteq [n_j] \\ S_j \neq \varnothing}} \lambda_{\mathcal{T}_j,i}(U_{j,i}-1) \right).
\]
By Theorem \ref{thm:combinatorics}, this operator on $(\mathcal{H},\xi)$ corresponds to the operator on $(\mathcal{K}, \zeta)$ given by
\[
1 + \sum_{\substack{S \subseteq [k] \\ S \neq \varnothing}} \prod_{j \in S} \sum_{\substack{S_j \subseteq [n_j] \\ S_j \neq \varnothing}} \lambda_{\mathcal{T}',\iota_j(i)}(U_{j,i}-1),
\]
where $\iota_j: [n_j] \to [N]$ is given by $\iota_j(i) = n_1 + \dots + n_{j-1} + i$.  By elementary combinatorics, this is equal to
\[
1 + \sum_{\substack{S' \subseteq [N] \\ S' \neq \varnothing}} \prod_{\alpha \in S'} \lambda_{\mathcal{T}',\alpha}(U_{j(\alpha),i(\alpha)}),
\]
where $i(\alpha)$ and $j(\alpha)$ are the indices such that $\iota_{j(\alpha)}(i(\alpha)) = \alpha$.  By \eqref{eq:weirdproductexpansion} again, this is equal to
\[
\prod_{\alpha \in [N]} [\lambda_{\mathcal{T}',\alpha}(U_{j(\alpha),i(\alpha)} - 1 ) + 1],
\]
and the law of this operator is $\boxplus_{\mathcal{T}'}(\mu_{1,1}, \dots, \mu_{1,n_1}, \dots \dots, \mu_{k,1}, \dots, \mu_{k,n_k})$.  Therefore, the two laws agree as desired.
\end{proof}

However, $\mathcal{T} \mapsto \boxtimes_{\mathcal{T}}$ is not a morphism of \emph{symmetric} operads.  The problem is that unlike addition, multiplication is not commutative and the definition of convolution involves fixing a certain order in which to multiply the independent operators.  If we were to permute the arguments in $\boxtimes_{\mathcal{T}}$, this would not only permute the letters used in the $\mathcal{T}$ but it would also permute the order of multiplication.

Thus, one cannot expect the analogue of Corollary \ref{cor:convolutionidentity} to hold for general surjective maps $\psi: [N'] \to [N]$.  However, the argument does go through if $\psi$ is an increasing function.

\begin{proposition} \label{prop:multiplicativeconvolutionidentity}
Let $\mathcal{T} \in \Tree(N)$ and $\mathcal{T}' \in \Tree(N')$.  Suppose that $\psi: [N'] \to [N]$ is increasing and surjective and $\psi_*$ defines a bijection from $\mathcal{T}'$ to $\mathcal{T}$.  Then for $\mu_1$, \dots, $\mu_N \in \Upsilon(\mathcal{B})$, we have
\[
\boxtimes_{\mathcal{T}'}(\mu_{\psi(1)}, \dots, \mu_{\psi(N)}) = \boxtimes_{\mathcal{T}}(\mu_1, \dots, \mu_N).
\]
\end{proposition}

\begin{proof}
For $j = 1$, \dots, $N$, fix an operator $U_j$ on $(\mathcal{H}_j,\xi_j)$ with the law $\mu_j$.  Denote
\begin{align*}
(\mathcal{H},\xi) &= \assemb_{\mathcal{T}}[(\mathcal{H}_1,\xi_1),\dots,(\mathcal{H}_N,\xi_N)] \\
(\mathcal{K},\zeta) &= \assemb_{\mathcal{T}'}[(\mathcal{H}_{\psi(1)}, \xi_{\psi(1)}), \dots, (\mathcal{H}_{\psi(N')}, \xi_{\psi(N')}].
\end{align*}
Note that $\boxtimes_{\mathcal{T}'}(\mu_{\psi(1)}, \dots, \mu_{\psi(N)})$ is the law of the operator $V$ on $(\mathcal{K},\zeta)$ given by
\[
V = \prod_{i=1}^{N'} [\lambda_{\mathcal{T}',i}(U_{\psi(i)} - 1) + 1] = \prod_{j=1}^N \prod_{i \in \psi^{-1}(j)} [\lambda_{\mathcal{T}',i}(U_j - 1) + 1],
\]
where we continue to use the conventions established above for the ordering of non-commutative products and here we rely on the fact that $\psi$ is increasing.

Now if we fix $j$ and $a \in \mathcal{L}(\mathcal{H}_j)$, then Theorem \ref{thm:permutation} tells us that the operator $\lambda_{\mathcal{T},j}(a)$ on $(\mathcal{H},\xi)$ corresponds under the isomorphism $(\mathcal{H},\xi) \to (\mathcal{K},\zeta)$ to the operator
\[
\tilde{\lambda}_j(a) = \sum_{i \in \psi^{-1}(j)} \lambda_{\mathcal{T}',i}(a).
\]
Moreover, $\lambda_{\mathcal{T}',i}(a) \lambda_{\mathcal{T}',i'}(a') = 0$ for two distinct indices $i$ and $i'$ in $\psi^{-1}(j)$ and for $a, a' \in \mathcal{L}(\mathcal{H}_j)$; this follows for instance because $(\lambda_{\mathcal{T}',i})_{i \in \psi^{-1}(j)}$ are a family of projections which add up to the projection $\tilde{\lambda}_j(1)$, and hence $(\lambda_{\mathcal{T}',i})_{i \in \psi^{-1}(j)}$ are mutually orthogonal.  Because of this orthogonality property,
\[
\prod_{i \in \psi^{-1}(j)} [\lambda_{\mathcal{T}',i}(U_j - 1) + 1] = \sum_{i \in \psi^{-1}(j)} \lambda_{\mathcal{T}',i}(U_j - 1) + 1 = \tilde{\lambda}_j(U_j - 1) + 1.
\]
Therefore,
\[
V = \prod_{j \in [N]} [\tilde{\lambda}_j(U_j - 1) + 1],
\]
and so it corresponds by the isomorphism $(\mathcal{H},\xi) \cong (\mathcal{K},\zeta)$ to the operator $U$ on $(\mathcal{H},\xi)$ given by
\[
U = \prod_{j = 1}^N [\lambda_{\mathcal{T},j}(U_j - 1) + 1],
\]
which by definition has the law $\boxtimes_{\mathcal{T}}(\mu_1,\dots,\mu_N)$, which completes the proof. 
\end{proof}

This means that any additive convolution identities that we proved using only operad composition and Corollary \ref{cor:convolutionidentity} for increasing surjective $\psi$ will still hold for unitary multiplicative convolution by the same argument.  Thus, for instance, we obtain the multiplicative analogue of the formulas $\mu \rhd \nu = (\mu \vdash \nu) \uplus \nu$ and $\mu \boxplus \nu = (\mu \boxright \nu) \lhd \nu$, which are
\begin{align*}
\boxtimes_{\mathcal{T}_{2,\mono}}(\mu,\nu) &= \boxtimes_{\mathcal{T}_{2,\Bool}}(\boxtimes_{\mathcal{T}_{\orth}}(\mu,\nu), \nu) \\
\boxtimes_{\mathcal{T}_{2,\free}}(\mu,\nu) &= \boxtimes_{\mathcal{T}_{2,\mono \dagger}}(\boxtimes_{\mathcal{T}_{\sub}}(\mu,\nu),\nu);
\end{align*}
see also \cite{Lenczewski2008,BSTV2014,AA2017}.

There are many other questions related to multiplicative convolution that warrant further study, which we do not have time to address here.
\begin{enumerate}
    \item For various types of independence, one can study the multiplicative convolution operation where we subtract off the mean of an operator rather than subtracting off $1$ (see \cite{Bercovici2005b}).
    \item There is also a symmetric multiplicative convolution for positive elements, given by taking independent positive operators $X_1$, \dots, $X_N$ and studying the law of
    \[
    X_1^{1/2} \dots X_N^{1/2} X_N^{1/2} \dots X_1^{1/2}.
    \]
    However, we do not expect that in general this multiplicative convolution would satisfy the analogues of Propositions \ref{prop:multiplicativeoperadmorphism} and \ref{prop:multiplicativeconvolutionidentity}.
    \item Furthermore, for any type of multiplicative convolution considered, one can ask what the analogue of the central limit theorem and the L{\'e}vy-Hin{\v c}in formula, and how to realize the infinitely divisible laws on a Fock space.
\end{enumerate}

\subsection{Functoriality}

One can show that direct sums and tensor products are functors on the category of $\mathcal{B}$-$\mathcal{B}$-correspondences, where the morphisms are given by adjointable right $\mathcal{B}$-module maps that are also left $\mathcal{B}$-modular.  This implies that the $\mathcal{T}$-free product of $\mathcal{B}$-$\mathcal{B}$-correspondences is functorial on the category of pairs $(\mathcal{H},\xi)$ where the morphisms are given by $\mathcal{B}$-$\mathcal{B}$-bimodule maps that are adjointable, contractive, and unit-vector-preserving (see \cite[\S 5.3]{Voiculescu1985} for a similar statement in the free case).  Also, the maps $\Phi$ and $\Psi$ constructed in Theorems \ref{thm:composition} and \ref{thm:permutation} respectively are natural transformations.

However, we do not know what the best framework is to study functoriality of the $\mathcal{T}$-free product on probability spaces $(\mathcal{A},E)$.  For instance, is the $\mathcal{T}$-free product functorial for unital, completely positive, expectation-preserving maps?

We should also mention that the Fock space construction $\mathcal{K} \mapsto \mathcal{F}(\mathcal{K})$ is functorial on $\mathcal{B}$-$\mathcal{B}$-correspondences, where for the input variable $\mathcal{K}$, the morphisms are contractive, adjointable $\mathcal{B}$-$\mathcal{B}$-bimodule maps, and for the output variable, the morphisms are contractive, adjointable, unit-vector-preserving $\mathcal{B}$-$\mathcal{B}$-bimodule maps.  This property is well-known for standard examples of Fock spaces, and it was proved in \cite{GM2002} for a general class of Fock spaces different than the ones studied here.


\subsection{Other Notions of Independence}

Although the independences introduced in this paper are new and quite general, there are surely further generalizations.  Here are a few vague suggestions for further investigation.

First, as mentioned before, we can study trees on an infinite alphabet, and hence join infinitely many algebras together in a $\mathcal{T}$-free manner (and in particular $\mathcal{T}$ might be the set of walks on an infinite digraph).  Theorem \ref{thm:combinatorics} would generalize without difficulty.  However, the convolution operations for infinite trees would require more care to study.  Indeed, convolving infinitely many laws requires adding up infinitely many independent variables, and thus we need additional conditions to make this sum converge.  For instance, using the arguments of Proposition \ref{prop:operatornormbound}, we could consider $\sum_{j=1}^\infty \lambda_{\mathcal{T},j}(a_j)$ where $\norm{a_j}$ is bounded and the mean and variances of $a_j$ are summable.  To adapt the results of \S \ref{sec:operad}, we would need to generalize operads to include infinitely many arguments.  However, \S \ref{sec:cumulants} - \ref{sec:infinitelydivisible} could not generalize at all; indeed, the definition of the cumulants no longer makes sense in the infinitary setting because there could be infinitely many colorings of a partition $\pi$.

Second, one could add weights to the edges of the tree $\mathcal{T}$ and to multiply $\lambda_{\mathcal{T},j}(x)|_{\mathcal{H}_s^\circ \oplus \mathcal{H}_{js}^\circ}$ by the weight of the edge $(s,js)$.  We conjecture that the results of this paper could be adapted if the weights are positive and uniformly bounded from above.

Third, it would be interesting to see whether there is a common framework that includes our $\mathcal{T}$-free independence together with the bi-free independence of \cite{Voiculescu2014} or the free-Boolean independence of \cite{Liu2017}, which are independence relations for pairs of algebras acting on the same Hilbert space.  Similarly, one could study operations on pairs of laws such as c-free independence (see Example \ref{ex:cfree}), which arise from pairs of algebras acting on pairs of Hilbert spaces.

Fourth, in the scalar-valued setting, one could hope for a generalization of $\mathcal{T}$-free independence that also includes classical independence, if we allow the operators in $\mathcal{L}(\mathcal{H}_j)$ to act on the free product Hilbert space in other ways, e.g.\ by acting not only on the left-most tensorands, but also on the middle and right tensorands of each subspace $\mathcal{H}_s^\circ$.  Such a construction ought to include the mixtures of free and classical independence studied in \cite{Mlotkowski2004,SW2016} and perhaps connect to the Fock spaces studied in \cite{GM2002}.

However, we caution that in this generality there will not be so close a resemblance to the free case as in this paper.  For instance, the operator norm bounds in Proposition \ref{prop:operatornormbound} fail drastically in the classical setting and the central limit distribution has unbounded support.  Moreover, we should not expect a Boolean-orthogonal decomposition to hold for a larger class of independences; indeed, iterating Boolean and orthogonal convolution as in \S \ref{subsec:BOdecomp} will only ever produce operations within the operad $\Tree$.

\end{document}